\documentclass[a4paper,10pt]{article}

\usepackage[utf8x]{inputenc}
\usepackage{amssymb}
\usepackage{amsthm}
\usepackage{amsmath}
\usepackage{amsfonts}
\usepackage{bbm}
\usepackage[T1]{fontenc}
\usepackage{graphicx}
\usepackage[dvips]{epsfig}
\usepackage[english]{babel}
\usepackage{array}
\graphicspath{{images/}}
\usepackage[a4paper]{geometry}
\usepackage{color}
\usepackage{setspace}
\usepackage{fancybox}

\newcommand{\Z}{\mathbb Z} 

\newcommand{\R}{\mathbb R}

\newcommand{\la}{\lambda}
\newcommand{\Ac}{\mathcal{A}}

\newcommand{\Cc}{\mathcal{C}}

\newcommand{\cM}{\mathcal{M}}

\newcommand{\Rc}{\mathcal{R}}

\newcommand{\bs}{\boldsymbol}

\newcommand{\ep}{\varepsilon}
\newcommand{\La}{\Lambda}
\newcommand{\Aug}{\mathcal{A}ug}

\DeclareMathOperator{\ind}{ind}
\DeclareMathOperator{\id}{id}

\DeclareMathOperator{\mfm}{\mathfrak{m}}

\DeclareMathOperator{\Cth}{Cth}
\DeclareMathOperator{\Fuk}{\mathcal{F}uk}
\DeclareMathOperator{\CZ}{CZ}

\newtheorem{lem}{Lemma}
\newtheorem{teo}{Theorem}
\newtheorem{coro}{Corollary}
\newtheorem{prop}{Proposition}
\theoremstyle{definition}
\newtheorem{defi}{Definition}
\newtheorem{ex}{Example}
\newtheorem{rem}{Remark}
\newtheorem{nota}{Notations}
\begin{document}
\title{$A_\infty$-category of Lagrangian cobordisms in the symplectization of $P\times\R$.}
\author{No\'emie Legout}
\date{}
\maketitle
\begin{abstract}
	We define a unital $A_\infty$-category $\Fuk(\R\times Y)$ whose objects are exact Lagrangian cobordisms in the symplectization of $Y=P\times\R$, with negative cylindrical ends over Legendrians equipped with augmentations. The morphism spaces $\hom_{\Fuk(\R\times Y)}(\Sigma_0,\Sigma_1)$ are given in terms of Floer complexes $\Cth_+(\Sigma_0,\Sigma_1)$ which are
	versions of the Rabinowitz Floer complex defined by Symplectic Field Theory (SFT) techniques.
\end{abstract}
\tableofcontents

\section{Introduction}

This paper deals with Lagrangian cobordisms in the symplectization $(\R\times Y,d(e^t\alpha))$ of a contact manifold $(Y,\alpha)$. These cobordisms are properly embedded Lagrangian submanifolds admitting cylindrical ends on Legendrian submanifolds of $Y$, and here $Y$ will be the contactization $(P\times\R,dz+\beta)$ of a Liouville manifold $(P,\beta)$.

Our goal is to define through SFT-techniques introduced in \cite{EGH} a unital $A_\infty$-category $\Fuk(\R\times Y)$ whose objects are Lagrangian cobordisms equipped with an augmentation of the Chekanov-Eliashberg algebra (CE-algebra) of its negative end, and whose morphism spaces are given by certain Floer-type complexes $\Cth_+(\Sigma_0,\Sigma_1)$. In particular, when $\Sigma_0$ is a cylinder over a Legendrian equipped with an augmentation and $\Sigma_1$ a parallel copy the complex $\Cth_+(\Sigma_0,\Sigma_1)$ is an SFT-formulation of the Lagrangian Rabinowitz Floer complex due to Merry \cite{Merry}.
There already exists several versions of Fukaya categories whose objects are (non compact) exact Lagrangians, notably the (partially) wrapped Fukaya categories of a Liouville domain (see \cite{AS,Abouzaid,ZS}) and more recently of Liouville sectors \cite{GPS}. In this paper we instead consider Lagrangian submanifolds in a trivial Liouville cobordism, meaning a trivial cylinder over a contact manifold. The main difference between this and the case of Liouville domains is that we have a non empty concave end.
It is known that additional assumptions are necessary in order to define Floer complexes in this setting (Lagrangian cobordisms with loose negative ends are known to satisfy some flexibility results). For that reason, we impose some restriction on the Lagrangians. More precisely, we consider only exact Lagrangian cobordisms with negative cylindrical ends over Legendrian submanifolds whose CE-algebra admits an augmentation. In particular, an exact Lagrangian filling implies the existence of an augmentation \cite{EHK}.

Generalizing the structures we define in this paper to the case of Lagrangians in a more general Liouville cobordism should be possible  using the latest techniques of virtual perturbations in \cite{Pardon} or the polyfold technology developed in \cite{HWZ}.
Note that Cieliebak-Oancea in \cite{OC} have defined a version of Rabinowitz Floer homology for Lagrangians in a Liouville cobordism under the assumption that this cobordism admits a filling.
The Floer complex we define in this paper are similar to the ones defined by Cieliebak-Oancea; for instance, there is an identification on the level of generators. The main difference is that their differential is defined by Floer strips with a Hamiltonian perturbation term that corresponds to wrapping, while the differential considered here is defined in terms of honest SFT-type pseudo-holomorphic discs. It is expected that the two theories give quasi-isomorphic complexes. \\

We start by contrasting the complex considered here with the Floer-type complex for pairs of Lagrangian cobordisms considered in \cite{CDGG2}. Namely, given a pair of transverse exact Lagrangian cobordisms $(\Sigma_0,\Sigma_1)$ where $\Sigma_i$ has positive and negative cylindrical ends over Legendrians $\La_i^+$ and $\La_i^-$ respectively and $\La_i^-$ are equipped with augmentations, the authors in \cite{CDGG2} define the Floer complex $(\Cth(\Sigma_0,\Sigma_1),\mathfrak{d})$ whose underlying vector space is given by
\begin{alignat*}{1}
	\Cth(\Sigma_0,\Sigma_1)=C(\La_0^+,\La_1^+)\oplus CF(\Sigma_0,\Sigma_1)\oplus C(\La_0^-,\La_1^-)
\end{alignat*}
where $C(\La_0^\pm,\La_1^\pm)$ is generated by Reeb chords from $\La_1^\pm$ to $\La_0^\pm$ and $CF(\Sigma_0,\Sigma_1)$ is generated by intersection points in $\Sigma_0\cap\Sigma_1$. This complex is actually the cone of a map
\begin{alignat*}{1}
	f_1:CF_{-\infty}(\Sigma_0,\Sigma_1):=CF(\Sigma_0,\Sigma_1)\oplus C(\La_0^-,\La_1^-)\to C(\La_0^+,\La_1^+)
\end{alignat*}
In \cite{L} the author defined a product structure on $CF_{-\infty}(\Sigma_0,\Sigma_1)$, as well as higher order maps satisfying the $A_\infty$-equations. Moreover, in the same paper it is proved that the map $f_1$ generalizes to a family of maps $\{f_d\}_{d\geq1}$,
\begin{alignat*}{1}
	f_d:\Cth(\Sigma_{d-1},\Sigma_d)\otimes\dots\otimes\Cth(\Sigma_0,\Sigma_1)\to C(\La_0^+,\La_d^+)
\end{alignat*}
defined for a $(d+1)$-tuple of pairwise transverse exact Lagrangian cobordisms $(\Sigma_0,\dots,\Sigma_d)$, and satisfying the $A_\infty$-functor equations; where the $A_\infty$-structure maps on $C(\La_{d-1}^+,\La_d^+)\otimes\dots\otimes C(\La_0^+,\La_1^+)$ are given by the structure maps of the augmentation category $\Aug_-(\La_0^+\cup\dots\cup\La_d^+)$, see \cite{BCh}.
However, there exists no non trivial $A_\infty$-structure on the whole complex $\Cth(\Sigma_0,\Sigma_1)$ for example for degree reasons: the grading of Reeb chord generators in the positive end in $\Cth(\Sigma_0,\Sigma_1)$ is given by the Conley-Zehnder index plus $1$ (see Subsections \ref{sec:grad} and \ref{Cth-}) so a count of rigid pseudo-holomorphic discs with boundary on the positive cylindrical ends, with two negative Reeb chord asymptotics and one positive Reeb chord asymptotic would not provide a degree $0$ order $2$ map for example.

In this article, we use similar techniques for constructing a version 
of the Rabinowitz complex $(\Cth_+(\Sigma_0,\Sigma_1),\mfm_1)$, on which it will be possible to define higher order structure maps. The underlying vector space is:
\begin{alignat*}{1}
\Cth_+(\Sigma_0,\Sigma_1)=C(\La_1^+,\La_0^+)\oplus CF(\Sigma_0,\Sigma_1)\oplus C(\La_0^-,\La_1^-)
\end{alignat*}
so the only difference with $\Cth(\Sigma_0,\Sigma_1)$ is the generators we consider in the positive end; unlike in the complex $\Cth(\Sigma_0,\Sigma_1)$ these generators consist of the chords which \emph{start} at $\La_0^+$ and \emph{end} at $\La_1^+$. 
The differential is defined by a count of pseudo-holomorphic discs with boundary on the cobordisms and asymptotic to Reeb chords and intersection points such that
\begin{itemize}
	\item $C(\La_0^-,\La_1^-)$ is a subcomplex which is the linearized Legendrian contact cohomology complex of $\La_0^-\cup\La_1^-$ restricted to chords from $\La_1^-$ to $\La_0^-$,
	\item $C(\La_1^+,\La_0^+)$ is a quotient complex which is the linearized Legendrian contact homology complex of $\La_0^+\cup\La_1^+$ restricted to chords from $\La_0^+$ to $\La_1^+$.
\end{itemize}
In the case $\Sigma_0=\R\times\La_0$ and $\Sigma_1$ is a cylinder over a perturbed copy of $\La_0$ translated far in the positive Reeb direction, the complex $\Cth_+(\Sigma_0,\Sigma_1)$ is the complex of the $2$-copy of $\La_0$ considered in \cite{EESa}.\\

Then we proceed to investigate the properties of this complex. Some of them resemble  properties satisfied by the complex $\Cth(\Sigma_0,\Sigma_1)$ but there are also some significant differences.\\

\noindent\textbf{Acyclicity:}
Contrary to $\Cth(\Sigma_0,\Sigma_1)$, the complex $\Cth_+(\Sigma_0,\Sigma_1)$ is not always acyclic.
For example if $Y=J^1M$ for a closed manifold $M$, $\Sigma_0$ is a cylinder over the $0$-section in $J^1M$ and $\Sigma_1$ a cylinder over a Morse perturbation of the $0$-section, the homology of $\Cth_+(\Sigma_0,\Sigma_1)$ does not vanish but equals instead the Morse homology of $M$.
However, in the case $Y=P\times\R$ where any compact subset of $P$ is Hamiltonian displaceable, for example $Y=\R^{2n+1}$, the complex $\Cth_+(\Sigma_0,\Sigma_1)$ is always acyclic. It is also always acyclic whenever $\La_0^-=\La_1^-=\emptyset$ as in this case the complex is actually the same as the dual complex of $\Cth(\Sigma_1,\Sigma_0)$.\\

\noindent\textbf{Structure maps and continuation element:} 
The new Cthulhu complex carries structure maps which satisfy the $A_\infty$-equations. More precisely, for any $(d+1)$-tuple $(\Sigma_0,\dots,\Sigma_d)$ of pairwise transverse exact Lagrangian cobordisms, we define a map
\begin{alignat*}{1}
	\mfm_d:\Cth_+(\Sigma_{d-1},\Sigma_d)\otimes\dots\otimes\Cth_+(\Sigma_0,\Sigma_1)\to\Cth_+(\Sigma_0,\Sigma_d)
\end{alignat*}
by counts of SFT-buildings consisting of pseudo-holomorphic discs with boundary on the $\Sigma_i$'s and asymptotic to Reeb chords and intersection points.  Then, for any $1\leq k\leq d$ and sub-tuple $(\Sigma_{i_0},\dots,\Sigma_{i_k})$ with $0\leq i_0<\dots< i_k\leq d$, one has
\begin{alignat}{1}
	\sum\limits_{j=1}^k\sum\limits_{n=0}^{k-j}\mfm_{k-j+1}\big(\id^{\otimes k-j+1}\otimes\mfm_j\otimes\id^{\otimes n}\big)=0\label{Ainf-eq}
\end{alignat}
where the inner $\mfm_j$ has domain $\Cth_+(\Sigma_{i_{n+j-1}},\Sigma_{i_{n+j}})\otimes\dots\otimes\Cth_+(\Sigma_{i_n},\Sigma_{i_{n+1}})$ and $\mfm_{k-j+1}$ has domain $\Cth_+(\Sigma_{i_{k-1}},\Sigma_{i_k})\otimes\dots\otimes\Cth_+(\Sigma_{i_{n}},\Sigma_{i_{n+j}})\otimes\dots\otimes\Cth_+(\Sigma_{i_0},\Sigma_{i_1})$.\\

In the case when $\Sigma_1$ is a suitable small Hamiltonian perturbation of $\Sigma_0$ one establishes the existence of a \textit{continuation element} in $\Cth_+(\Sigma_0,\Sigma_1)$ (see Section \ref{sec:unit} for a precise description of the Hamiltonian perturbation $\Sigma_1$):

\begin{teo} There exists an element $e_{\Sigma_0,\Sigma_1}\in\Cth_+(\Sigma_0,\Sigma_1)$ satisfying that for any exact Lagrangian cobordism $\Sigma_2$ transverse to $\Sigma_0$ and $\Sigma_1$ the map
\begin{alignat*}{1}
	\mfm_2(\,\cdot\,,e_{\Sigma_0,\Sigma_1}):\Cth_+(\Sigma_1,\Sigma_2)\to\Cth_+(\Sigma_0,\Sigma_2)
\end{alignat*}
is a quasi-isomorphism.
\end{teo}

Finally we use these ingredients to construct a unital $A_\infty$-category $\Fuk(\R\times Y)$ via localization, in the same spirit as the construction of the wrapped Fukaya category  of Liouville sectors in \cite{GPS}:
\begin{teo}
	There exists a unital $A_\infty$-category $\Fuk(\R\times Y)$ whose objects are Lagrangian cobordisms equipped with augmentations of its negative ends and whose morphism spaces in the cohomological category satisfy $H^*\hom_{\Fuk(\R\times Y)}(\Sigma_0,\Sigma_1)\cong H^*(\Cth_+(\Sigma_0,\Sigma_1),\mfm_1)$ whenever $\Sigma_0$ and $\Sigma_1$ are transverse.
\end{teo}

The homology of the Rabinowitz complex $\Cth_+(\Sigma_0,\Sigma_1)$ is invariant under cylindrical at infinity Hamiltonian isotopies (in particular under Legendrian isotopies of its ends). This implies that the quasi-equivalence class of the category $\Fuk(\R\times Y)$ does not depend on choices of representatives of Hamiltonian isotopy classes of Lagrangian cobordisms involved in its construction by localization (see Section \ref{sec:Fuk}).\\

\noindent\textbf{Behaviour under concatenation:}
Given a pair of concatenated cobordisms $(V_0\odot W_0,V_1\odot W_1)$, we describe the complex $\Cth_+(V_0\odot W_0,V_1\odot W_1)$ in terms of the complexes  $\Cth_+(V_0,V_1)$ and $\Cth_+(W_0,W_1)$ and some \textit{transfer maps} fitting into a diagram
\begin{alignat*}{1}
	\Cth_+(V_0,V_1)\xleftarrow{\bs{\Delta}_1^W}\Cth_+(V_0\odot W_0,V_1\odot W_1)\xrightarrow{\bs{b}_1^V}\Cth_+(W_0,W_1)
\end{alignat*}
We prove that $\bs{\Delta}_1^W$ and $\bs{b}_1^V$ are chain maps which induce a Mayer-Vietoris sequence and moreover preserve the continuation element in homology.\\

In addition, the transfer maps generalize also to families of maps $\{\bs{\Delta}_d\}_{d\geq1}$ and $\{\bs{b}_d\}_{d\geq1}$ satisfying the $A_\infty$-functor equations. That is to say, for a $(d+1)$-tuple of concatenated cobordisms $(V_0\odot W_0,\dots,V_d\odot W_d)$ there are maps
\begin{alignat*}{1}
	&\mfm_d^{V\odot W}:\Cth_+(V_{d-1}\odot W_{d-1},V_d\odot W_d)\otimes\dots\otimes\Cth_+(V_0\odot W_0,V_1\odot W_1)\to\Cth_+(V_0\odot W_0,V_d\odot W_d)\\
	&\bs{\Delta}_d^W:\Cth_+(V_{d-1}\odot W_{d-1},V_d\odot W_d)\otimes\dots\otimes\Cth_+(V_0\odot W_0,V_1\odot W_1)\to\Cth_+(V_0,V_d)\\
	&\bs{b}_d^V:\Cth_+(V_{d-1}\odot W_{d-1},V_d\odot W_d)\otimes\dots\otimes\Cth_+(V_0\odot W_0,V_1\odot W_1)\to\Cth_+(W_0,W_d)
\end{alignat*}
such that for all $1\leq k\leq d$ and sub-tuple $(V_{i_0}\odot W_{i_0},\dots,V_{i_k}\odot W_{i_k})$ with $0\leq i_0<\dots< i_k\leq d$, the maps $\{\mfm_k^{V\odot W}\}_{1\leq k\leq d}$ satisfy the $A_\infty$-equations \eqref{Ainf-eq}, and the maps $\{\bs{\Delta}_k^W\}_{1\leq k\leq d}$ and $\{\bs{b}_k^V\}_{1\leq k\leq d}$ satisfy
\begin{alignat*}{1}
	&\sum_{s=1}^k\sum_{j_1+\dots+j_s=k}\mfm_s^V\big(\bs{\Delta}_{j_s}^W\otimes\dots\otimes\bs{\Delta}_{j_1}^W\big)+\sum_{j=1}^{k}\sum_{n=0}^j\bs{\Delta}_{k-j+1}^W\big(\id^{\otimes k-j+1}\otimes\mfm_j^{V\odot W}\otimes\id^{\otimes n}\big)=0\\
	&\sum_{s=1}^k\sum_{j_1+\dots+j_s=k}\mfm_s^W\big(\bs{b}_{j_s}^V\otimes\dots\otimes\bs{b}_{j_1}^V\big)+\sum_{j=1}^{k}\sum_{n=0}^j\bs{b}_{k-j+1}^V\big(\id^{\otimes k-j+1}\otimes\mfm_j^{V\odot W}\otimes\id^{\otimes n}\big)=0.
\end{alignat*}
\vspace{2mm}

\noindent\textbf{Acknowledgements} The author warmly thanks Baptiste Chantraine, Georgios Dimitroglou-Rizell and Paolo Ghiggini for helpful discussions and comments on earlier versions of the paper, as well as Alexandru Oancea for stimulating discussions. The author was partly supported by the grant KAW 2016.0198 from the Knut and Alice Wallenberg Foundation and the Swedish Research Council under the grant no. 2016-03338.

\section{Background}
\subsection{Geometric set-up}
Throughout the paper we will be working with a contact manifold $(Y,\alpha)$ given by the contactization of a \textit{Liouville manifold}. We briefly recall the definition of these terms.
A \textit{Liouville domain} $(\widehat{P},\theta)$ is the data of a $2n$-dimensional manifold with boundary $\widehat{P}$ as well as a $1$-form $\theta$ on $\widehat{P}$ such that $d\theta$ is symplectic, and the Liouville vector field $V$ defined by $\iota_Vd\theta=\theta$ is required to point outward on the boundary $\partial\widehat{P}$. In particular, $\theta_{|\partial\widehat{P}}$ is a contact form on $\partial\widehat{P}$.
The \textit{completion} of $(\widehat{P},\theta)$ is the exact symplectic manifold $(P=\widehat{P}\cup_{\partial \widehat{P}}[0,\infty)\times\partial \widehat{P},\omega=d\beta)$, where $\beta$ equals $\theta$ in $\widehat{P}$ and $e^\tau\theta_{|\partial\widehat{P}}$ on $[0,\infty)\times\partial\widehat{P}$ where $\tau$ is the coordinate on $[0,\infty)$. The Liouville vector field smoothly extends to the whole manifold $P$. We call $(P,\beta)$ a \textit{Liouville manifold}.

The \textit{contactization} of a Liouville manifold $(P,\beta)$ is the contact manifold $(Y,\alpha)$ where $Y$ is the $2n+1$-dimensional manifold $Y=P\times\R$ and $\alpha=dz+\beta$, where $z$ is the $\R$-coordinate. The Reeb vector field of $\alpha$ is given by $R_\alpha=\partial_z$ so in particular there are no closed Reeb orbits in $Y$. A \textit{Legendrian submanifold} of $(Y,\alpha)$ is a $n$-dimensional submanifold $\La$ satisfying $\alpha_{|T\La}=0$, and \textit{Reeb chords} of $\La$ are trajectories of the Reeb flow starting and ending on $\La$. We consider only Legendrian with a finite number of isolated Reeb chords, and denote $\Rc(\La)$ the set of Reeb chords of $\La$. These are called pure Reeb chords. Given two Legendrian $\La_0$ and $\La_1$, we denote $\Rc(\La_1,\La_0)$ the set of Reeb chords starting on $\La_0$ and ending on $\La_1$, these are called mixed Reeb chords.

The main objects under consideration in this article are exact Lagrangian cobordisms between Legendrian submanifolds of $Y$. These are Lagrangian submanifolds in the \textit{symplectization} of $(Y,\alpha)$ which is the symplectic manifold $(\R\times Y,d(e^t\alpha))$ where $t$ is the $\R$-coordinate.
\begin{defi}
	Given $\La^-,\La^+\subset Y$ Legendrian, an \textit{exact Lagrangian cobordism} from $\La^-$ to $\La^+$ is a submanifold $\Sigma\subset\R\times Y$ such that there exists
	\begin{itemize}
		\item $T>0$ such that
		\begin{enumerate}
			\item $\Sigma\cap[T,\infty)\times Y=[T,\infty)\times\La^+$,
			\item $\Sigma\cap(-\infty,-T]\times Y=(-\infty,-T]\times\La^-$,
			\item $\Sigma\cap[-T,T]\times Y$ is compact.
		\end{enumerate}
		\item $f:\Sigma\to\R$ a smooth function called a \textit{primitive} of $\Sigma$, satisfying
		\begin{enumerate}
			\item $e^t\alpha_{|T\Sigma}=df$,
			\item $f$ is constant on $[T,\infty)\times\La^+$ and $(-\infty,-T]\times\La^-$.
		\end{enumerate}
	\end{itemize}
\end{defi}

In all the paper, we will assume that the coefficient field is $\Z_2$. Moreover, we assume that $2c_1(P)=0$, and that the Legendrian submanifolds and Lagrangian cobordisms between them have Maslov number $0$. This will ensure a well-defined $\Z$ grading for the various complexes that will appear. 

\subsection{Almost complex structure}
Given a family of pairwise transverse Lagrangian cobordisms $(\Sigma_0,\dots,\Sigma_d)$ with Legendrian cylindrical ends $\R\times\La_i^\pm$, $0\leq i\leq d$, we consider several types of moduli spaces of pseudo-holomorphic discs with boundary on those Lagrangian cobordisms. Those discs are asymptotic to intersection points and/or Reeb chords of the links $\La_0^\pm\cup\dots\cup\La_d^\pm$. First, let us describe briefly the almost complex structure we consider on the symplectic manifold $(\R\times Y,d(e^t\alpha))$, in order to define the moduli spaces mentioned above and achieve transversality.

An almost complex structure $J$ on $(\R\times Y,d(e^t\alpha))$ is called \textit{cylindrical} if
\begin{itemize}
	\item it is compatible with $d(e^t\alpha)$,
	\item $J(\partial_t)=R_\alpha$,
	\item $J(\xi)=\xi$,
	\item $J$ is invariant by translation along the $t$-coordinate axis.
\end{itemize}
We denote $\mathcal{J}^{cyl}(\R\times Y)$ the set of cylindrical almost complex structures on $\R\times Y$.
An almost complex structure on $P$ is called \textit{admissible} if it is cylindrical on $P\backslash\widehat{P}$ outside of a compact set.
The \textit{cylindrical lift} of an admissible almost complex structure $J_P$ on $P$ is the unique cylindrical almost complex structure $\widetilde{J}_P$ on $\R\times(P\times\R)$ making the projection $\pi_P:\R\times(P\times\R)\to P$ holomorphic. 

Let $J^+$ and $J^-$ be two cylindrical almost complex structures which coincide outside of $\R\times K$ for some compact $K\subset Y$. Assuming that the cobordisms we consider are all cylindrical outside of $[-T,T]\times Y$ for some fixed $T>0$, we take an almost complex structure $J$ which is equal to $J^-$ on $(-\infty,-T)\times Y$, to $J^+$ on $(T,+\infty)\times Y$, and to the cylindrical lift of an admissible complex structure $J_P$ on $[-T,T]\times (Y\backslash K)$. We denote $\mathcal{J}_{J^+,J^-}(\R\times Y)$ this class of almost complex structures on $\R\times Y$.

In order to achieve transversality for the moduli spaces later on, we will finally need domain dependent almost complex structures with values in $\mathcal{J}_{J^+,J^-}(\R\times Y)$, i.e. families of almost complex structures in $\mathcal{J}_{J^+,J^-}(\R\times Y)$ parametrized by the domains of the pseudo-holomorphic discs (punctured Riemann discs), which is part of a \textit{universal choice of perturbation data}, see \cite[Section (9h)]{S}.

\subsection{Moduli spaces of curves with boundary on Lagrangian cobordisms}\label{sec:mod}
Let $\mathcal{R}^{d+1}$ be the space of $d+1$ cyclically ordered points $\bs{y}=(y_0,\dots,y_d)\in (S^1)^{d+1}$ quotiented by the automorphisms of the unit disc $D^2$. This is the Deligne-Mumford space. For $\bs{y}\in\mathcal{R}$, let us denote $S_{\bs{y}}=D^2\backslash\{y_0,\dots,y_d\}$. In a sufficiently small neighborhood of the punctures $y_i$ in the disc, we have strip-like end coordinates $[s_i,t_i]\in(0,+\infty)\times[0,1]$, $0\leq i\leq d$.

Let us denote $\Sigma_{0...d}=(\Sigma_0,\dots,\Sigma_d)$ a $d+1$-tuple of Lagrangian cobordisms satisfying the following:
\begin{itemize}
	\item if $\Sigma_0=\Sigma_d$, then $\Sigma_i=\Sigma_0$ for all $1\leq i\leq d$,
	\item if $\Sigma_0\neq\Sigma_d$, then the ordered family $\Sigma_0,\dots,\Sigma_d$ is of the form
	$$\Sigma_{i_0},\dots,\Sigma_{i_0},\Sigma_{i_1},\dots,\Sigma_{i_1},\Sigma_{i_2},\dots,\Sigma_{i_k},$$ with $\Sigma_{i_0}:=\Sigma_0$ and $\Sigma_{i_k}:=\Sigma_d$, and such that $\Sigma_{i_0},\Sigma_{i_1},\dots,\Sigma_{i_d}$ are pairwise transverse. In other words, we allow only consecutive repetition of a given Lagrangian. 
\end{itemize} 
The set of \textit{asymptotics} $A(\Sigma_{i-1},\Sigma_i)$ associated to the pair $(\Sigma_{i-1},\Sigma_i)$ consists of Reeb chords in $\Rc(\La_{i-1}^\pm\cup\La_i^\pm)$, and intersection points in $\Sigma_{i-1}\cap\Sigma_i$ when the cobordisms are transverse. 
Consider a $d+1$-tuple of asymptotics $(a_0,\dots,a_d)$, with $a_i\in A(\Sigma_{i-1},\Sigma_i)$, $\Sigma_{-1}:=\Sigma_d$ and $\La_{-1}^\pm:=\La_d^\pm$. If $\Sigma_{i-1}=\Sigma_i$, then $a_i$ is called a \textit{pure asymptotic}, and it is a pure Reeb chord of $\La_{i-1}^\pm=\La_i^\pm$, while if $\Sigma_{i-1}\neq\Sigma_i$, then $a_i$ is called a \textit{mixed asymptotic}. 
Given $J$ an almost complex structure on $\R\times Y$, we denote $\cM_{\Sigma_{0,...,d},J}(a_0;a_1,\dots,a_d)$ the set of pairs $(\bs{y},u)$ where
\begin{enumerate}
	\item $\bs{y}\in\mathcal{R}^{d+1}$,
	\item $u:(S_{\bs{y}},j)\to(\R\times Y,J)$ is a pseudo-holomorphic map (with $j$ the standard almost complex structure on $D^2$),
	\item $u$ maps the boundary of $S_r$ contained between $y_i$ and $y_{i+1}$ for $0\leq i\leq d$ ($y_{d+1}:=y_0$) to $\Sigma_i$,
	\item $\lim_{z\to y_i}u(z)=a_i$.
\end{enumerate}
Let us specify the condition (4) in the case $a_i$ is a Reeb chord, for which we also denote $a_i:[0,1]\to Y$ a parametrization. We say that
\begin{itemize}
	\item $u$ has a \textit{positive asymptotic to $a_i$ at $y_i$} if $\lim\limits_{s_i\to+\infty}u(s_i,t_i)=a_i(t_i)$,
	\item $u$ has a \textit{negative asymptotic to $a_i$ at $y_i$} if $\lim\limits_{s_i\to+\infty}u(s_i,t_i)=a_i(1-t_i)$.
\end{itemize}

\begin{rem}\label{rem:label} 
Note that the fact that a mixed Reeb chord asymptotic is a positive or a negative asymptotic is entirely determined by the \textquotedblleft jump\textquotedblright\, of the chord. Namely, positive mixed Reeb chord asymptotics are mixed chords of $\La_i^+\cup\La_{i+1}^+$ from $\La_i^+$ to $\La_{i+1}^+$, and negative mixed Reeb chord asymptotics are mixed Reeb chords of $\La_i^-\cup\La_{i+1}^-$ from $\La_{i+1}^-$ to $\La_i^-$.
\end{rem}

\begin{nota}
From now on, we denote the Lagrangian boundary condition for discs only by the family $(\Sigma_{i_0},\Sigma_{i_2},\dots,\Sigma_{i_k})$, even though the pseudo-holomorphic discs we will consider can have pure Reeb chords asymptotics too.
\end{nota}
In the following two subsections we describe the several types of moduli spaces we will make use of later.

\subsubsection{Moduli spaces of curves with cylindrical boundary conditions}\label{mod_cyl}

The Lagrangian boundary conditions we consider here are trivial 
cylinders over Legendrians, and we take an almost complex structure $J\in\mathcal{J}^{cyl}(\R\times Y)$. If the boundary conditions consists of only one cylinder $\R\times\La$ then we denote 
\begin{alignat*}{1}
	\cM_{\R\times\La,J}(\gamma^+;\gamma_1,\dots,\gamma_d)
\end{alignat*}
the moduli space of discs with boundary on $\R\times\La$, with a positive asymptotic to $\gamma^+$ and negative asymptotics at $\gamma_i$ for $1\leq i\leq d$. We call discs in such moduli spaces \textit{pure}, as all asymptotics are pure.
In case the Lagrangian conditions is a family of distinct transverse cylinders $\R\times\La_{0...d}:=(\R\times\La_0,\dots,\R\times\La_d)$ with $d>0$, we consider the:
\begin{enumerate}
	\item Banana-type moduli spaces:
	\begin{alignat*}{1}
	\cM_{\R\times\La_{0...d},J}(\gamma_{d,0};\bs{\delta}_0,\gamma_1,\bs{\delta}_1,\dots,\gamma_d,\bs{\delta}_d)
	\end{alignat*}
	 where $\gamma_{d,0}\in\Rc(\La_0,\La_d)$ is a mixed Reeb chord from $\La_d$ to $\La_0$, $\gamma_i\in\Rc(\La_{i-1},\La_{i})\cup\Rc(\La_{i},\La_{i-1})$ are mixed chords of $\La_{i-1}\cup\La_i$ and $\bs{\delta}_i$ are words of Reeb chords of $\La_i$ and are negative asymptotics. Note that according to Remark \ref{rem:label}, $\gamma_{d,0}$ is a positive Reeb chord asymptotic and then $\gamma_i$ is a positive asymptotic if it is in $\Rc(\La_{i},\La_{i-1})$ and a negative one if it is in $\Rc(\La_{i-1},\La_{i})$.
	\item $\Delta$-type moduli spaces:
	\begin{alignat*}{1}
	\cM_{\R\times\La_{0...d},J}(\gamma_{0,d};\bs{\delta}_0,\gamma_1,\bs{\delta}_1,\dots,\gamma_d,\bs{\delta}_d)
	\end{alignat*}
	where $\gamma_{0,d}\in\Rc(\La_d,\La_0)$ is a negative Reeb chord asymptotic and with the same condition as above on asymptotics $\gamma_i$ and $\bs{\delta}_i$.
\end{enumerate}
The discs in moduli spaces of type (1) and (2) are called \textit{mixed} as $d+1$ asymptotics are mixed.
There is a $\R$-action by translation on moduli spaces with cylindrical Lagrangian boundary condition, we use the notation $\widetilde{\cM}$ to denote the quotient of the moduli space $\cM$ by $\R$. 

\subsubsection{Moduli spaces of curves with boundary on non-cylindrical Lagrangians}\label{mod_non_cyl}
The Lagrangian boundary conditions consist of Lagrangian cobordisms $(\Sigma_0,\dots,\Sigma_d)$ such that at least one is not a trivial cylinder. Denote $\mathbf{J}$ a domain dependent almost complex structure with values in $\mathcal{J}_{J^+,J^-}(\R\times Y)$. If $d=0$, $\Sigma:=\Sigma_0$ is a non trivial cobordism from $\La^-$ to $\La^+$ and we denote
\begin{alignat*}{1}
	\cM_{\Sigma,\mathbf{J}}(\gamma^+;\gamma_1,\dots,\gamma_d)
\end{alignat*}
the moduli space of discs where $\gamma\in\Rc(\La^+)$ is a positive Reeb chord asymptotic and $\gamma_i\in\Rc(\La^-)$ are negative Reeb chord asymptotics. We call again those discs \textit{pure}. If the Lagrangian boundary condition consists of several distinct Lagrangians $\Sigma_{0...d}=(\Sigma_0,\dots,\Sigma_d)$ where $d>0$ and $\Sigma_i$ is a cobordism from $\La_i^-$ to $\La_i^+$, then we consider the following \textit{mixed} moduli spaces:
\begin{enumerate}
	\item Banana-type moduli space: $\cM_{\Sigma_{0...d},\mathbf{J}}(\gamma_{d,0};\bs{\delta}_0,a_1,\bs{\delta}_1,\dots,a_d,\bs{\delta}_d)$,
	\item $\mfm_0$-type moduli space:
	$\cM_{\Sigma_{0...d},\mathbf{J}}(x;\bs{\delta}_0,a_1,\bs{\delta}_1,\dots,a_d,\bs{\delta}_d)$,
	\item $\Delta$-type moduli space
	$\cM_{\Sigma_{0...d},\mathbf{J}}(\gamma_{0,d};\bs{\delta}_0,a_1,\bs{\delta}_1,\dots,a_d,\bs{\delta}_d)$,
\end{enumerate}
where $\gamma_{d,0}\in\Rc(\La_0^+,\La_d^+)$ is a positive Reeb chord asymptotic, $\gamma_{0,d}\in\Rc(\La_d^-,\La_0^-)$ is a negative Reeb chord asymptotic, $a_i$ are intersection points in $\Sigma_{i-1}\cap\Sigma_i$ or mixed Reeb chord asymptotics in $\Rc(\La_i^+,\La_{i-1}^+)\cup\Rc(\La_{i-1}^-,\La_{i}^-)$, and $\bs{\delta}_i$ are words of pure Reeb chords of $\La_i^-$.

\subsection{Action and energy}
Consider a $d+1$-tuple of pairwise disjoint cobordisms $(\Sigma_0,\dots,\Sigma_d)$ with cylindrical ends over $\La_i^{\pm}$. Let $T>0$ and $\varepsilon>0$ such that all cobordisms are cylindrical outside of $[-T+\varepsilon,T-\varepsilon]\times Y$. The \textit{length} of a chord $\gamma$ is defined by $\ell(\gamma):=\int_\gamma\alpha$. Then, the \textit{action} of asymptotics is defined as follows:
\begin{alignat*}{1}
&\mathfrak{a}(\gamma)=e^T\ell(\gamma)+\mathfrak{c}_j-\mathfrak{c}_i\,\,\mbox{ for } \gamma\in \Rc(\La_i^+,\La_j^+),\\
&\mathfrak{a}(x)=f_j(x)-f_i(x)\,\,\mbox{ for }\,x\in\Sigma_i\cap\Sigma_j,\,i<j,\\
&\mathfrak{a}(\gamma)=e^{-T}\ell(\gamma)\,\,\mbox{ for } \gamma\in \Rc(\La_i^-,\La_j^-),\\
&\mathfrak{a}(\gamma)=e^{\pm T}\ell(\gamma)\,\,\mbox{ for } \gamma\in \Rc(\La^\pm).
\end{alignat*}
 Given a function $\chi:\R\to\R$ satisfying for some $\varepsilon>0$
$$\chi(t)=\left\{\begin{array}{ccl}
	e^T&\mbox{ for } &t\geq T\\
	e^t&\mbox{ for }& -T+\varepsilon\leq t\leq T-\varepsilon\\
	e^{-T}&\mbox{ for } &t\leq-T
	\end{array}\right.$$
and $\chi'(t)\geq0$, one defines the \textit{energy} of a pseudo-holomorphic disc $u$ to be:
\begin{alignat*}{1}
	E(u)=\int_ud(\chi(t)\alpha)
\end{alignat*}
This energy is always positive and vanishes if and only if the disc is constant. The energy of the pseudo-holomorphic discs considered in this paper is finite and can be expressed in terms of the action of the asymptotics.

\begin{prop} For the moduli spaces described in Sections \ref{mod_cyl} and \ref{mod_non_cyl}, we have the following:
	\begin{enumerate}
		\item If $u\in\cM_{\R\times\La}(\gamma^+;\gamma_1,\dots,\gamma_d)$, then
		$E(u)=\mathfrak{a}(\gamma^+)-\sum_i\mathfrak{a}(\gamma_i)$.

		\item If $u\in\cM_{\R\times\La_{0...d}}(\gamma_{d,0};\bs{\delta}_0,\gamma_1,\bs{\delta}_1,\dots,\gamma_d,\bs{\delta}_d)$, 
$$E(u)=\mathfrak{a}(\gamma_{d,0})+\sum_{\gamma_i\in\Rc(\La_{i},\La_{i-1})}\mathfrak{a}(\gamma_i)-\sum_{\gamma_i\in\Rc(\La_{i-1},\La_{i})}\mathfrak{a}(\gamma_i)-\sum_{i=0}^d\mathfrak{a}(\bs{\delta}_i).$$

		\item If $u\in\cM_{\R\times\La_{0...d}}(\gamma_{0,d};\bs{\delta}_0,\gamma_1,\bs{\delta}_1,\dots,\gamma_d,\bs{\delta}_d)$, 
$$E(u)=-\mathfrak{a}(\gamma_{0,d})+\sum_{\gamma_i\in\Rc(\La_{i},\La_{i-1})}\mathfrak{a}(\gamma_i)-\sum_{\gamma_i\in\Rc(\La_{i-1},\La_{i})}\mathfrak{a}(\gamma_i)-\sum_{i=0}^d\mathfrak{a}(\bs{\delta}_i).$$

		\item If $u\in\cM_\Sigma(\gamma^+;\gamma_1,\dots,\gamma_d)$, then $E(u)=\mathfrak{a}(\gamma^+)-\sum_i\mathfrak{a}(\gamma_i)$.

		\item If $u\in\cM_{\Sigma_{0...d}}(\gamma_{d,0};\bs{\delta}_0,a_1,\bs{\delta}_1,\dots,a_d,\bs{\delta}_d)$,
			$$E(u)=\mathfrak{a}(\gamma_{d,0})+\sum_{a_i\in\Rc(\La_{i}^+,\La_{i-1}^+)}\mathfrak{a}(a_i)-\sum_{a_i\in\Rc(\La_{i-1}^-,\La_{i}^-)}\mathfrak{a}(a_i)-\sum_{a_i\in\Sigma_{i-1}\cap\Sigma_i}\mathfrak{a}(a_i)-\sum_{i=0}^d\mathfrak{a}(\bs{\delta}_i).$$

		\item If $u\in\cM_{\Sigma_{0...d}}(x;\bs{\delta}_0,a_1,\bs{\delta}_1,\dots,a_d,\bs{\delta}_d)$,
			$$E(u)=\mathfrak{a}(x)+\sum_{a_i\in\Rc(\La_{i}^+,\La_{i-1}^+)}\mathfrak{a}(a_i)-\sum_{a_i\in\Rc(\La_{i-1}^-,\La_{i}^-)}\mathfrak{a}(a_i)-\sum_{a_i\in\Sigma_{i-1}\cap\Sigma_i}\mathfrak{a}(a_i)-\sum_{i=0}^d\mathfrak{a}(\bs{\delta}_i).$$

		\item If $u\in\cM_{\Sigma_{0...d}}(\gamma_{0,d};\bs{\delta}_0,a_1,\bs{\delta}_1,\dots,a_d,\bs{\delta}_d)$,
			$$E(u)=-\mathfrak{a}(\gamma_{0,d})+\sum_{a_i\in\Rc(\La_{i}^+,\La_{i-1}^+)}\mathfrak{a}(a_i)-\sum_{a_i\in\Rc(\La_{i-1}^-,\La_{i}^-)}\mathfrak{a}(a_i)-\sum_{a_i\in\Sigma_{i-1}\cap\Sigma_i}\mathfrak{a}(a_i)-\sum_{i=0}^d\mathfrak{a}(\bs{\delta}_i).$$
	\end{enumerate}
\end{prop}

\subsection{Grading}\label{sec:grad}
Given cobordisms $(\Sigma_0,\dots,\Sigma_d)$ as above, we associate a grading to the asymptotics of pseudo-holomorphic discs with boundary on $\Sigma_{0...d}$ using the Conley-Zehnder index. We refer for example to \cite{EES1} for the definition of this index.

\begin{enumerate}
	\item \textbf{Grading of Reeb chords:} Consider $\La\subset Y$ a connected Legendrian, then the grading of a Reeb chord $\gamma\in\Rc(\La)$ is defined to be
	\begin{alignat*}{1}
	|\gamma|=\CZ(\gamma)-1
	\end{alignat*}
	where $\CZ(\gamma)$ denotes the Conley-Zehnder index of a capping path for $\gamma$. Note that it does not depend on the choice of capping path as by hypothesis we consider Maslov $0$ Legendrians, and it does not depend neither on a choice of symplectic trivialization of $TP$ along the capping path, as $2c_1(P)=0$. If the Legendrian $\La$ is not connected, there are no capping paths for Reeb chords connecting two distinct components so some additional choices are needed (see \cite{DR}). If $\gamma$ is a chord from $\La_j$ to $\La_i$, one fixes points $p_j\in\La_j$ and $p_i\in\La_i$ and a path $\Gamma_{i,j}$ from $p_i$ to $p_{j}$ as well as a path of Lagrangians from $T_{p_i}\pi_P(\La_{i})$ to $T_{p_{j}}\pi_P(\La_{j})$. Then, one takes as capping path for $\gamma$ a path from the ending point of $\gamma$ to $p_i$, followed by $\Gamma_{i,j}$, followed by a path from $p_{j}$ to the starting point of $\gamma$. The grading of mixed chords depends on those additional paths but the difference in grading of two chords does not.
	
	\item \textbf{Grading of intersection points:} 
	Let $p\in\Sigma_i\cap\Sigma_j$, for $i<j$. Generically, the immersed Lagrangian $\Sigma_i\cup\Sigma_j$ lifts to an embedded Legendrian submanifold $\widetilde{\Sigma}_i\cup\widetilde{\Sigma}_j\subset\big((\R\times Y)\times\R_u,du+e^t\alpha\big)$ and $p$ is the projection of a Reeb chord $\gamma_p$ of $\widetilde{\Sigma}_i\cup\widetilde{\Sigma}_j$. If $\gamma_p$ is a chord from $\widetilde{\Sigma}_j$ to $\widetilde{\Sigma}_i$, then $|p|=\CZ(\gamma_p)$. If $\gamma_p$ is a chord from $\widetilde{\Sigma}_i$ to $\widetilde{\Sigma}_j$, then $|p|=n+1-\CZ(\gamma_p)$. These Conley-Zehnder indices are computed after a choice of path connecting $\widetilde{\Sigma}_i$ and $\widetilde{\Sigma}_j$ as explained above for the non-connected case.
	Again, the vanishing of the Maslov number for Lagrangian cobordisms, and of the first Chern class of $P$ imply that the grading of intersection points does not depend on the choices made, except paths to connect any two distinct components of the Legendrian lift.
\end{enumerate}


The expected dimension of the moduli spaces described in Sections \ref{mod_cyl} and \ref{mod_non_cyl} can then be expressed in terms of the grading of asymptotics; this is the purpose of the next proposition. 
\begin{prop}\label{teo:grading} Consider the moduli spaces described in Sections \ref{mod_cyl} and \ref{mod_non_cyl}. For those where it applies denote $j^+$ the number of positive mixed Reeb chord asymptotics among $\{\gamma_1,\dots,\gamma_d\}$ or $\{a_1,\dots,a_d\}$, and $l$ the number of intersection points asymptotics among $\{a_1,\dots,a_d\}$. Moreover, we assume that negative asymptotics to pure Reeb chords $\bs{\delta}_i$ have degree $0$. Then we have:
	\begin{alignat*}{1}\displaystyle
	&\dim\widetilde{\cM}_{\R\times\La}(\gamma^+;\gamma_1,\dots,\gamma_d)=|\gamma^+|-\sum|\gamma_i|-1,\\
	&\dim\widetilde{\cM}_{\R\times\La_{0...d}}(\gamma_{d,0};\bs{\delta}_0,\gamma_1,\bs{\delta}_1,\dots,\gamma_d,\bs{\delta}_d)\\
&\hspace{1cm}=|\gamma_{d,0}|+\sum_{\gamma_i\in\Rc(\La_{i},\La_{i-1})}|\gamma_i|-\sum_{\gamma_i\in\Rc(\La_{i-1},\La_{i})}|\gamma_i|+(2-n)j^+-1,\\
	&\dim\widetilde{\cM}_{\R\times\La_{0...d}}(\gamma_{0,d};\bs{\delta}_0,\gamma_1,\bs{\delta}_1,\dots,\gamma_d,\bs{\delta}_d)\\
	&\hspace{1cm}=-|\gamma_{0,d}|+\sum_{\gamma_i\in\Rc(\La_{i},\La_{i-1})}|\gamma_i|-\sum_{\gamma_i\in\Rc(\La_{i-1},\La_{i})}|\gamma_i|+(2-n)(j^+-1)-1,\\
	&\dim\cM_\Sigma(\gamma^+;\gamma_1,\dots,\gamma_d)=|\gamma^+|-\sum|\gamma_i|,\\
	&\dim\cM_{\Sigma_{0...d}}(\gamma_{d,0};\bs{\delta}_0,a_1,\bs{\delta}_1,\dots,a_d,\bs{\delta}_d)\\
	&\hspace{1cm}=|\gamma_{d,0}|+\sum_{a_i\in\Rc(\La_{i}^+,\La_{i-1}^+)}|a_i|-\sum_{a_i\in\Sigma_{i-1}\cap\Sigma_i}|a_i|-\sum_{a_i\in\Rc(\La_{i-1}^-,\La_{i}^-)}|a_i|+(2-n)j^++l,\\
	&\dim\cM_{\Sigma_{0...d}}(x;\bs{\delta}_0,a_1,\bs{\delta}_1,\dots,a_d,\bs{\delta}_d)\\
	&\hspace{1cm}=|x|+\sum_{a_i\in\Rc(\La_{i}^+,\La_{i-1}^+)}|a_i|-\sum_{a_i\in\Sigma_{i-1}\cap\Sigma_i}|a_i|-\sum_{a_i\in\Rc(\La_{i-1}^-,\La_{i}^-)}|a_i|+(2-n)j^++l-2,\\
	&\dim\cM_{\Sigma_{0...d}}(\gamma_{0,d};\bs{\delta}_0,a_1,\bs{\delta}_1,\dots,a_d,\bs{\delta}_d)\\
	&\hspace{1cm}=-|\gamma_{0,d}|+\sum_{a_i\in\Rc(\La_{i}^+,\La_{i-1}^+)}|a_i|-\sum_{a_i\in\Sigma_{i-1}\cap\Sigma_i}|a_i|-\sum_{a_i\in\Rc(\La_{i-1}^-,\La_{i}^-)}|a_i|+n+(2-n)j^++l-2.
	\end{alignat*}
\end{prop} 

\begin{nota}
	Given a moduli space $\cM_{\Sigma_{0...d}}(a_0;a_1,\dots,a_d)$, we add an exponent indicating the (expected) dimension of it as a smooth manifold: $\cM^i_{\Sigma_{0...d}}(a_0;a_1,\dots,a_d)$. This dimension is equal to the index of the Fredholm operator obtained by linearizing $\bar\partial$ at a pseudo-holomorphic disc $u$.
\end{nota}

\subsection{Compactness}\label{sec:structure}
This section sums up what will be used to prove almost all the results in this paper. Namely, once transversality is achieved simultaneously for all moduli spaces considered above, which is possible using a domain dependent almost complex structure, $0$-dimensional (eventually after quotienting by an $\R$-action) moduli spaces are compact manifolds. We call discs in these moduli spaces \textit{rigid discs}.
Then, $1$-dimensional moduli spaces are not necessarily compact and can be compactified by adding \textit{broken discs}. The goal of this section is to describe the types of broken discs one can find in the boundary of the compactification of the moduli spaces in Sections \ref{mod_cyl} and \ref{mod_non_cyl}.

Consider a $1$-dimensional moduli space $\cM^1_{L_{0...d}}(a_0;\bs{\delta}_0,a_1,\bs{\delta}_1,\dots,\bs{\delta}_{d-1},a_d,\bs{\delta}_d)$ of curves where $L_{0...d}=\R\times\La_{0...d}$ or $\Sigma_{0...d}$ (in the first case $\cM^1_{L_{0...d}}=\widetilde{\cM^2}_{L_{0...d}}$), with mixed asymptotics $a_i$ and asymptotics to words of pure Reeb chords $\bs{\delta}_i$.
By results in \cite{BEHWZ},\cite[Theorem 3.20]{Abbas}, a sequence of curves in such a moduli space admits a subsequence converging to a \textit{pseudo-holomorphic building} (see the cited references for a precise definition) consisting of several pseudo-holomorphic discs together with \textit{nodes} connecting these components and choices of asymptotics for these nodes, satisfying the following:
\begin{itemize}
		\item each disc in the pseudo-holomorphic building has positive energy, so in particular a component with only Reeb chords asymptotics must have at least one positive Reeb chord asymptotic,
		\item each disc has a non negative Fredholm index because of the regularity of the almost complex structure, 
		\item if the building consists of the discs $u_1,\dots,u_k$, then the \textit{glued solution} $u$ has index
		 $\ind(u)=\nu+\sum_i\ind(u_i)$, where $\nu$ is the number of nodes asymptotic to intersection points.
\end{itemize}

Let us now precise what these conditions imply in particular in the case of moduli spaces described in the previous sections.
\begin{enumerate}
	\item \textbf{Cylindrical boundary condition:} 
	
Consider a $1$-dimensional moduli space $\widetilde{\cM^2}_{\R\times\La_{0...d}}(\gamma_0;\bs{\delta}_0,\gamma_1,\bs{\delta}_1,\dots,\bs{\delta}_{d-1},\gamma_d,\bs{\delta}_d)$ 
The conditions described above imply that a sequence of discs in this moduli space admits a subsequence which limits to a pseudo-holomorphic building consisting two index $1$ pseudo-holomorphic discs with boundary on cylinders, and which glue together along a node asymptotic to a Reeb chord. Remark that this Reeb chord can be pure or mixed, we will later be interested only in the case of nodes asymptotic to mixed Reeb chords, see Remark \ref{rem:delta}.

	\item \textbf{Non-cylindrical boundary conditions:}
	
Consider a $1$-dimensional moduli space $\cM^1_{\Sigma_{0...d}}(a_0;\bs{\delta}_0,a_1,\bs{\delta}_1,\dots,\bs{\delta}_{d-1},a_d,\bs{\delta}_d)$ of curves with boundary on the cobordisms $\Sigma_0,\dots,\Sigma_d$, with mixed asymptotics $a_i$ and asymptotics to words of pure Reeb chords $\bs{\delta}_i$.
A sequence of discs in such a moduli space admits a subsequence converging to a pseudo-holomorphic building which is
	\begin{enumerate}
		\item either a pseudo-holomorphic building with two index $0$ components (which are not trivial strips) with boundary on non cylindrical parts of the cobordisms and connected by node asymptotic to an intersection point, or
		\item a pseudo-holomorphic building consisting of some index $0$ components with boundary on the non cylindrical parts of the cobordisms, and one index $1$ component with boundary on the positive or negative cylindrical ends of the cobordisms, connected to each index $0$ component via a node asymptotic to a Reeb chord.
	\end{enumerate} 
\end{enumerate} 

\begin{nota}
	We denote $\partial\overline{\cM^2}_{\R\times\La_{0...d}}:=\partial\overline{\widetilde{\cM^2}}_{\R\times\La_{0...d}}$, and $\partial\overline{\cM^1}_{\Sigma_{0...d}}$ the set of pseudo-holomorphic buildings arising at the boundary of the compactification of the corresponding moduli spaces.
\end{nota}

\subsection{Legendrian contact homology}\label{sec:leg}
Consider a compact Legendrian submanifold $\La\subset Y$. We denote by $C(\La)$ the $\Z_2$-vector space generated by Reeb chords of $\La$. 
The Legendrian contact homology of $\La$ is an invariant of $\La$ up to Legendrian isotopy, introduced by Chekanov in \cite{Che} and Eliashberg \cite{E} (see also \cite{EES1,EES2}). It is the homology of a differential graded algebra (DGA) $(\Ac(\La),\partial)$ associated to $\La$. The algebra $\Ac(\La)$, called the Chekanov-Eliashberg algebra of $\La$, is the unital tensor algebra of $C(\La)$, i.e. $\Ac(\La)=\bigoplus_{i\geq0}C(\La)^{\otimes i}$ with $C(\La)^{\otimes0}:=\Z_2$.
The grading of Reeb chords is as defined in Section \ref{sec:grad}.
Given a cylindrical almost complex structure, the differential $\partial$ is defined as follows. For $\gamma\in\Rc(\La)$ we have
\begin{alignat*}{1}
	\partial\gamma=\sum_{d\geq 0}\sum_{\gamma_i\in\Rc(\La)}\widetilde{\cM^1}_{\R\times\La}(\gamma^+;\gamma_1,\dots,\gamma_d)\gamma_1\gamma_2\dots\gamma_d
\end{alignat*}
The differential $\partial$ extends to the whole algebra by Leibniz rule, and satisfies $\partial^2=0$. We consider the SFT definition of the differential here, which has been proved in \cite{DR} to give the same invariant as the original version with a differential defined by a count of discs in $P$ with boundary on $\pi_P(\Lambda)$.
For a generic cylindrical almost complex structure, the moduli spaces $\widetilde{\cM^1}_{\R\times\La}(\gamma;\gamma_1,\dots,\gamma_d)$ are compact $0$-dimensional manifolds. The Legendrian contact homology of $\La$, denoted $LCH_*(\La)$, does not depend on a generic choice of almost complex structure and is invariant under Legendrian isotopy.

Consider now an almost complex structure $J\in\mathcal{J}_{J^+,J^-}(\R\times Y)$. It is proved in \cite{EHK} that an exact Lagrangian cobordism $\Sigma$ from $\La^-$ to $\La^+$ induces a DGA-map $\Phi_\Sigma:\Ac(\La^+)\to\Ac(\La^-)$, defined by
\begin{alignat*}{1}
	\Phi_\Sigma(\gamma)=\sum_{d\geq 0}\sum_{\gamma_i\in\Rc(\La)}\cM^0_{\Sigma}(\gamma^+;\gamma_1,\dots,\gamma_d)\gamma_1\gamma_2\dots\gamma_d
\end{alignat*}
When $\Sigma$ is a Lagrangian filling of $\La$, i.e. $\La^-=\emptyset$, then $\Phi_\Sigma$ is a map from $\Ac(\La)$ to $\Ac(\emptyset):=\Z_2$. It is an instance of an \textit{augmentation} of $\Ac(\La)$. More generally, we have the following definition.
\begin{defi}
	An augmentation of $(\Ac(\La),\partial)$ over $\Z_2$ is a map $\ep\colon\Ac(\La)\to\Z_2$ satisfying
	\begin{enumerate}
		\item $\ep\circ\partial=0$
		\item $\ep(\gamma)=0$ if $|\gamma|\neq0$,
		\item $\ep(1)=1$,
		\item $\ep(\gamma_1\gamma_2)=\ep(\gamma_1)\ep(\gamma_2)$,
	\end{enumerate}
In other words, it is a unital DGA-map when considering $\Z_2$ as a DGA with a vanishing differential.
\end{defi}
\begin{rem}
	The condition 2. in the definition above means that the augmentations we consider are graded (more precisely $0$-graded). In this paper we will only consider $0$-graded augmentations, although we could as well take them $\rho$-graded for $\rho$ a positive integer (meaning that all chords with degree $0$ mod $\rho$ can potentially be augmented) or even ungraded (any chord can be augmented). If we consider $\rho$-graded, resp. ungraded, augmentations, the complexes appearing later in the paper and involving augmentations become $\rho$-graded, resp. ungraded, as well as some higher order operations being part of some $A_\infty$-structure.
\end{rem}
Chekanov made use of augmentations to linearize the DGA $(\Ac(\La),\partial)$, leading to finite dimensional invariants called \textit{linearized Legendrian contact homologies}. Bourgeois and Chantraine \cite{BCh} generalized this idea using two augmentations instead of one:
assume $\Ac(\La)$ admits augmentations $\ep_0,\ep_1$, then there is a complex $(LCC_*^{\ep_0,\ep_1}(\La),\partial^{\ep_0,\ep_1}_1)$ where $LCC_*^{\ep_0,\ep_1}(\La):=C(\La)$ and for a Reeb chord $\gamma$,
\begin{alignat*}{1}
	\partial_1^{\ep_0,\ep_1}(\gamma)=\sum_{d\geq 0}\sum_{\gamma_1,\dots,\gamma_d\in\Rc(\La)}\sum_{i=1}^d\widetilde{\cM^1}_{\R\times\La}(\gamma^+;\gamma_1,\dots,\gamma_d)\ep_0(\gamma_1\dots\gamma_{i-1})\ep_1(\gamma_{i+1}\dots\gamma_d)\gamma_i
\end{alignat*}
This map satisfies $(\partial_1^{\ep_0,\ep_1})^2=0$.
The \textit{Legendrian contact homology of $\La$ bilinearized by $(\ep_0,\ep_1)$} is the homology of this complex, denoted $LCH_*^{\ep_0,\ep_1}(\La)$.
The dual complex $(LCC^*_{\ep_0,\ep_1}(\La),\mu_{\ep_0,\ep_1}^1)$ is the complex of the \textit{bilinearized Legendrian contact cohomology of $\La$}, $LCH^*_{\ep_0,\ep_1}(\La)$. For Reeb chords $\gamma,\beta\in\Rc(\La)$, if we denote by $\langle \partial_1^{\ep_0,\ep_1}(\beta),\gamma\rangle$ the coefficient of $\gamma$ in $\partial_1^{\ep_0,\ep_1}(\beta)$, then we have
\begin{alignat*}{1}
	\mu_{\ep_0,\ep_1}^1(\gamma)=\sum_{\beta\in\Rc(\La)}\langle \partial_1^{\ep_0,\ep_1}(\beta),\gamma\rangle \beta
\end{alignat*}
When $\ep_0=\ep_1$, these complexes correspond to the linearized Legendrian contact (co)homology complexes defined by Chekanov.
Finally, given an augmentation $\ep^-$ of $\Ac(\La^-)$ and an exact Lagrangian cobordism $\La^-\prec_\Sigma\La^+$, it induces an augmentation $\ep^+:=\ep^-\circ\Phi_\Sigma$ of $\Ac(\La^+)$.

\subsection{The Cthulhu complex $\Cth$}\label{Cth-}

The Cthulhu homology is the homology of a Floer-type complex defined in \cite{CDGG2} for a pair $\La_0^-\prec_{\Sigma_0}\La_0^+$ and $\La_1^-\prec_{\Sigma_1}\La_1^+$ of transverse exact Lagrangian cobordisms in $\R\times Y$ such that the algebras $\Ac(\La_0^-)$ and $\Ac(\La_1^-)$ admit augmentations $\ep_0^-$ and $\ep_1^-$ respectively. The Cthulhu complex $\big(\Cth(\Sigma_0,\Sigma_1),\mathfrak{d}_{\ep_0^-,\ep_1^-}\big)$ has three types of generators,
\begin{alignat*}{1}
\Cth(\Sigma_0,\Sigma_1)=C(\La_0^+,\La_1^+)[2]\oplus CF(\Sigma_0,\Sigma_1)\oplus C(\La_0^-,\La_1^-)[1]
\end{alignat*}
where $C(\La_0^+,\La_1^+)[2]$ denotes the $\Z_2$-vector space generated by Reeb chords from $\La_1^+$ to $\La_0^+$ with a grading shift, namely if $\gamma\in C(\La_0^+,\La_1^+)[2]$ then $|\gamma|_{\Cth(\Sigma_0,\Sigma_1)}=|\gamma|+2$, $CF(\Sigma_0,\Sigma_1)$ is the $\Z_2$-vector space generated by intersection points in $\Sigma_0\cap\Sigma_1$, and $C(\La_0^-,\La_1^-)$ is generated by Reeb chords from $\La_1^-$ to $\La_0^-$.
The differential is given by the matrix
\begin{alignat*}{1}
\mathfrak{d}_{\ep_0^-,\ep_1^-}=\left(\begin{matrix} d_{++}&d_{+0}&d_{+-}\\
0&d_{00}&d_{0-}\\
0&d_{-0}&d_{--}\\
\end{matrix}\right)
\end{alignat*}
It is a degree $1$ map defined by a count of rigid  pseudo-holomorphic discs with boundary on the cobordisms, as schematized on Figure \ref{fig:cthulhu_diff}. The study of broken discs arising at the boundary of the compactification of $1$-dimensional moduli spaces gives that $\mathfrak{d}_{\ep_0^-,\ep_1^-}$ squares to $0$, see \cite[Theorem 4.1]{CDGG2}.

\begin{figure}[ht]  
	\begin{center}\includegraphics[width=15cm]{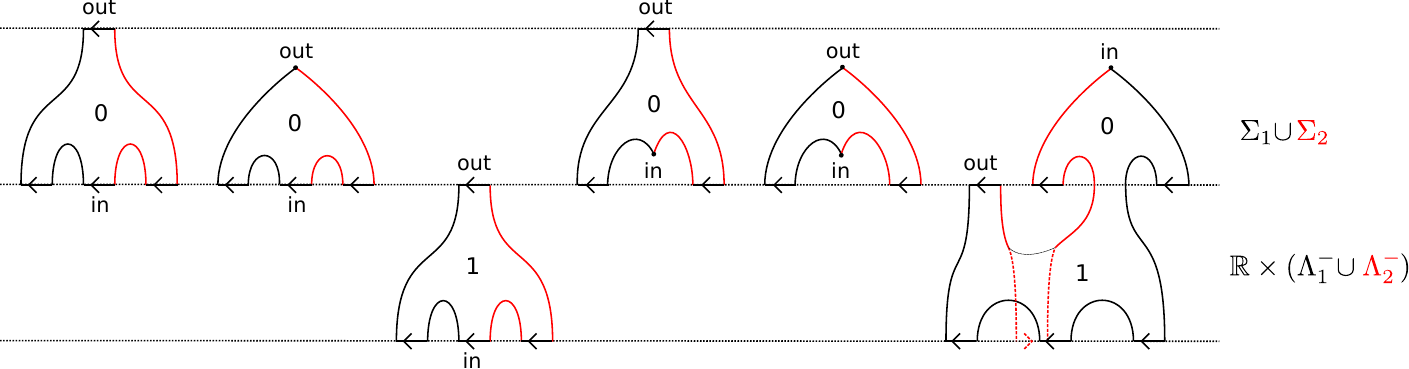}\end{center}
	\caption{Curves contributing to the differential $\partial_{\ep_0^-,\ep_1^-}$, "\textit{in}" stands for input and "\textit{out}" for output. The "$0$" and "$1$" indicate the Fredholm index of the respective discs.}
	\label{fig:cthulhu_diff}
\end{figure}

Denote by $\big(CF_{-\infty}(\Sigma_0,\Sigma_1),\mfm_1^{-\infty})$ the quotient complex of the Cthulhu complex $\Cth(\Sigma_0,\Sigma_1)$, with $CF_{-\infty}(\Sigma_0,\Sigma_1)=CF(\Sigma_0,\Sigma_1)\oplus C(\La_0^-,\La_1^-)[1]$ and $$\mfm_1^{-\infty}=\left(\begin{matrix}
d_{00}&d_{0-}\\
d_{-0}&d_{--}\\
\end{matrix}\right)$$
In \cite{L} the author proved that given a triple of pairwise transverse cobordisms $\Sigma_0,\Sigma_1$ and $\Sigma_2$, there is a non trivial map:
\begin{alignat*}{1}
	\mfm_2^{-\infty}:CF_{-\infty}(\Sigma_1,\Sigma_2)\otimes CF_{-\infty}(\Sigma_0,\Sigma_1)\to CF_{-\infty}(\Sigma_0,\Sigma_2)
\end{alignat*}
satisfying the Leibniz rule $\mfm_2^{-\infty}(\mfm_1^{-\infty}\otimes\id)+\mfm_2^{-\infty}(\id\otimes\mfm_1^{-\infty})+\mfm_1^{-\infty}\circ\mfm_2^{-\infty}=0$, see Section \ref{section:def_product} for more details.

\section{The complex $\Cth_+$}\label{sec:Cth}

\subsection{Definition of the complex}\label{section:def_complex}

In this section, we define the Rabinowitz complex $\Cth_+(\Sigma_0,\Sigma_1)$ for a pair of transverse exact Lagrangian cobordisms $\La_0^-\prec_{\Sigma_0}\La_0^+$ and $\La_1^-\prec_{\Sigma_1}\La_1^+$. We assume again that $\Ac(\La_i^-)$ admit augmentations $\ep_i^-$ for $i=0,1$, inducing augmentations $\ep_i^+$ of $\Ac(\La_i^+)$. The complex $\Cth_+(\Sigma_0,\Sigma_1)$ is generated by three types of generators:
\begin{alignat*}{1}
\Cth_+(\Sigma_0,\Sigma_1)=C(\La_1^+,\La_0^+)^\dagger[n-1]\oplus CF(\Sigma_0,\Sigma_1)\oplus C(\La_0^-,\La_1^-)[1]
\end{alignat*}
If we denote $|\cdot|_{\Cth_+}$ the grading of generators in $\Cth_+(\Sigma_0,\Sigma_1)$, then we have 
\begin{alignat*}{1}
	&|\gamma_{01}|_{\Cth_+}=n-1-|\gamma_{01}|,\mbox{ for }\gamma_{01}\in C(\La_1^+,\La_0^+)^\dagger[n-1]\\
	&|x|_{\Cth_+}=|x|,\mbox{ for }x\in CF(\Sigma_0,\Sigma_1)\\
	&|\xi_{10}|_{\Cth_+}=|\xi_{10}|+1,\mbox{ for }\xi_{10}\in C(\La_0^-,\La_1^-)[1]
\end{alignat*}
The difference on generators between $\Cth_+(\Sigma_0,\Sigma_1)$ and the original Cthulhu complex $\Cth(\Sigma_0,\Sigma_1)$ in \cite{CDGG2} is that the generators that are Reeb chords in the positive end are chords from $\La_0^+$ to $\La_1^+$ in $\Cth_+(\Sigma_0,\Sigma_1)$, whereas they are chords from $\La_1^+$ to $\La_0^+$ in $\Cth(\Sigma_0,\Sigma_1)$. The differential on $\Cth_+(\Sigma_0,\Sigma_1)$ is then given by
\begin{alignat*}{1}
\mfm_1^{\ep^-_0,\ep^-_1}
=\left(
\begin{matrix}
\Delta_1^{+}&0&0\\
d_{0+}&d_{00}&d_{0-}\\
b_1^-\circ\Delta_1^\Sigma&b_1^-\circ\Delta_1^\Sigma&b_1^-
\end{matrix}\right)
\end{alignat*}
where:
\begin{enumerate}
	\item $\Delta_1^+:C(\La_1^+,\La_0^+)^\dagger[n-1]\to C(\La_1^+,\La_0^+)^\dagger[n-1]$ is defined by
	\begin{alignat*}{1}
	&\Delta_1^+(\gamma^+_{01})=\sum\limits_{\gamma^-_{01}}\sum\limits_{\bs{\zeta}_i}\#\widetilde{\cM^1}_{\R\times\La^+_{01}}(\gamma_{01}^-;\bs{\zeta}_0,\gamma_{01}^+,\bs{\zeta}_1)\cdot\ep_0^+(\bs{\zeta}_0)\ep_1^+(\bs{\zeta}_1)\cdot\gamma^-_{01}	
	\end{alignat*}
	and is of degree $1$ according to Proposition \ref{teo:grading}.
	\item $\mfm_1^0:=d_{0+}+d_{00}+d_{0-}:\Cth_+(\Sigma_{0},\Sigma_1)\to CF(\Sigma_0,\Sigma_1)$ with
	\begin{alignat*}{1}
	&d_{0+}(\gamma_{01}^+)=\sum\limits_{x}\sum\limits_{\bs{\delta}_i}\#\cM^0_{\Sigma_{0},\Sigma_1}(x;\bs{\delta}_0,\gamma_{01}^+,\bs{\delta}_1)\cdot\ep_0^-(\bs{\delta}_0)\ep_1^-(\bs{\delta}_1)\cdot x\\
	&d_{00}(q)=\sum\limits_{x}\sum\limits_{\bs{\delta}_i}\#\cM^0_{\Sigma_{0},\Sigma_1}(x;\bs{\delta}_0,q,\bs{\delta}_1)\cdot\ep_0^-(\bs{\delta}_0)\ep_1^-(\bs{\delta}_1)\cdot x\\
	&d_{0-}(\gamma_{10}^-)=\sum\limits_{x}\sum\limits_{\bs{\delta}_i}\#\cM^0_{\Sigma_{0},\Sigma_1}(x;\bs{\delta}_0,\gamma_{10}^-,\bs{\delta}_1)\cdot\ep_0^-(\bs{\delta}_0)\ep_1^-(\bs{\delta}_1)\cdot x
	\end{alignat*}
	is also of degree $1$.
	\item $\Delta_1^\Sigma:\Cth^*_+(\Sigma_{0},\Sigma_1)\to C_{n-1-*}(\La_1^-,\La_0^-)$ is defined, for $a\in\Cth_+(\Sigma_0,\Sigma_1)$, by:
	\begin{alignat*}{1}
	&\Delta_1^\Sigma(a)=\sum\limits_{\gamma_{01}}\sum\limits_{\bs{\delta}_i}\#\cM^0_{\Sigma_{0},\Sigma_1}(\gamma_{01};\bs{\delta}_0,a,\bs{\delta}_1)\cdot\ep_0^-(\bs{\delta}_0)\ep_1^-(\bs{\delta}_1)\cdot\gamma_{01}	
	\end{alignat*}
	so in particular it vanishes for energy reasons on $C(\La_0^-,\La_1^-)$. This map is of degree $0$, i.e. $|\gamma_{01}|=n-1-|a|_{\Cth_+}$ where $|\gamma_{01}|$ is as defined in Section \ref{sec:grad}.
	\item Let us denote $\mathfrak{C}^*(\La_0^-,\La_1^-)=C_{n-1-*}(\La_1^-,\La_0^-)\oplus C^{*-1}(\La_0^-,\La_1^-)$, one finally defines the map $b_1^-:\mathfrak{C}^*(\La_0^-,\La_1^-)\to C^{*-1}(\La_0^-,\La_1^-)$ by
	\begin{alignat*}{1}
	b_1^-(\gamma)=\sum\limits_{\gamma^+_{10}}\sum\limits_{\bs{\delta}_i}\#\widetilde{\cM^1}_{\R\times\La^-_{01}}(\gamma_{10};\bs{\delta}_0,\gamma,\bs{\delta}_1)\cdot\ep_0^-(\bs{\delta}_0)\ep_1^-(\bs{\delta}_1)\cdot\gamma_{10}	
	\end{alignat*}
	where $\gamma$ is a positive asymptotic if it is in $C(\La_1^-,\La_0^-)$ and a negative asymptotic if it is in $C(\La_0^-,\La_1^-)$. This map is of degree $1$.
\end{enumerate}
On Figure \ref{fig:curves_diff} are schematized the pseudo-holomorphic curves contributing to the differential $\mfm_1^{\ep_0^-,\ep_1^-}$. 
\begin{figure}[ht]  
	\begin{center}\includegraphics[width=11cm]{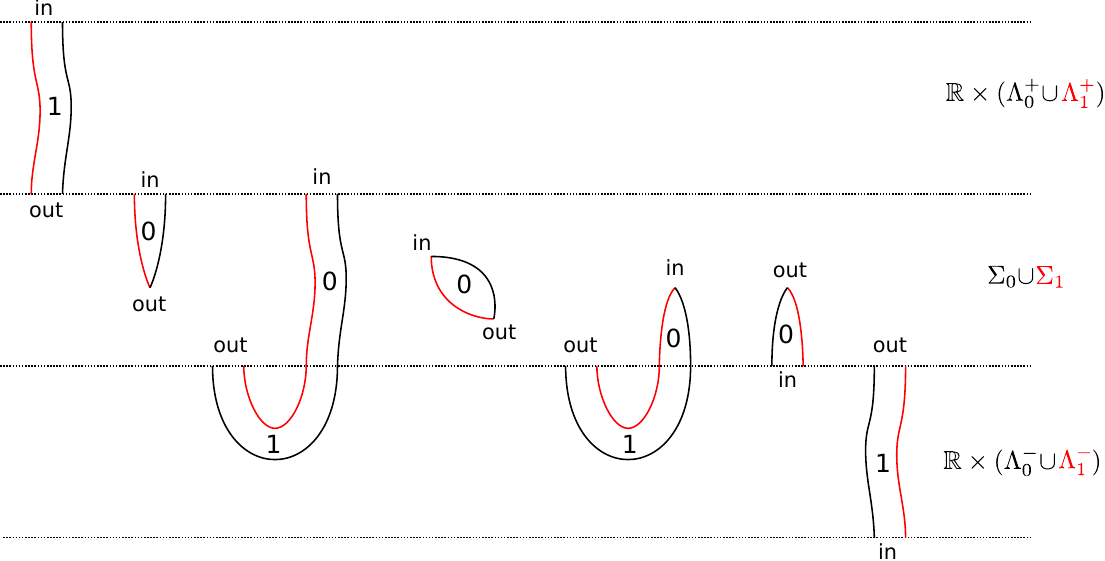}\end{center}
	\caption{Curves contributing to the differential $\mfm_1^{\ep^-_0,\ep^-_1}$.}
	\label{fig:curves_diff}
\end{figure}

\begin{rem}\label{rem:relation}
In the definition of $\mfm_1^{\ep_0^-,\ep_1^-}$, all components are related to components of the differential of $\Cth(\Sigma_0,\Sigma_1)$ or $\Cth(\Sigma_1,\Sigma_0)$ as follows:
\begin{itemize}
	\item the map $\Delta_1^+$ is the dual of $d_{++}$ in $\Cth(\Sigma_1,\Sigma_0)$, and it is the differential of the bilinearized Legendrian contact homology of $\La_0^+\cup\La_1^+$ restricted to $C(\La_1^+,\La_0^+)$,
	\item the map $d_{0+}$ is the dual of $d_{+0}$ in $\Cth(\Sigma_1,\Sigma_0)$,
	\item the map $\Delta_1^\Sigma$ restricted to the positive Reeb chords is the dual of $d_{+-}$ in $\Cth(\Sigma_1,\Sigma_0)$ and restricted to intersection points it is the dual of $d_{0-}$ in $\Cth(\Sigma_1,\Sigma_0)$,
	\item $b_1^-$ restricted to $C(\La_1^-,\La_0^-)$ is the banana map in $\Cth(\Sigma_0,\Sigma_1)$, and restricted to $C(\La_0^-,\La_1^-)$ it is the map $d_{--}$ in $\Cth(\Sigma_0,\Sigma_1)$ that is to say the differential of the Legendrian contact cohomology of $\La_0^-\cup\La_1^-$ restricted to $C(\La_0^-,\La_1^-)$.
\end{itemize}
In particular, the Floer complex $(CF_{-\infty}(\Sigma_0,\Sigma_1),\mfm_1^{-\infty})$ is a subcomplex of $\Cth_+(\Sigma_0,\Sigma_1)$.
\end{rem}

\begin{teo}\label{diff}
	$\mfm_1^{\ep_0^-,\ep_1^-}$ is a degree $1$ map satisfying $\mfm_1^{\ep_0^-,\ep_1^-}\circ\mfm_1^{\ep_0^-,\ep_1^-}=0$.
\end{teo}

\begin{rem}[$\partial$-breaking]\label{rem:delta}
Before we prove the theorem, let us give some precision about some types of pseudo-holomorphic buildings arising as limit of a sequence of pseudo-holomorphic discs in a $1$-dimensional moduli space. Namely, the buildings containing a non-trivial pure disc are a bit special.

Consider again the case 1. in Section \ref{sec:structure}, i.e. the limit of a sequence of discs with boundary on trivial cylinders. As we said, it consists of two index $1$ discs connected by a Reeb chord node. If this node is a pure Reeb chord $\gamma\in\Rc(\La)$, then the (non trivial) disc $u$ in the building for which this node is a positive Reeb chord asymptotic is a pure disc, given the condition we take on the Lagrangian boundary conditions in Section \ref{sec:mod}, and thus contributes to the differential of $\gamma$ in the Legendrian contact homology DGA of $\La$. 
Then, applying an augmentation $\ep$ of $\Ac(\La)$ to all the negative pure Reeb chord asymptotics of $u$ results in a pseudo-holomorphic curve count contributing to $\ep\circ\partial(\gamma)$. 
The sum over all possible negative pure Reeb chord asymptotics of a pure disc with positive asymptotic $\gamma$ leads in a curve count giving the whole term $\ep\circ\partial(\gamma)$ which vanishes by definition of an augmentation.

Consider now the case 2.(b) in Section \ref{sec:structure}, and the subcase where the index $1$ curve that we denote $u$ has boundary on the negative ends of the cobordisms and has one positive Reeb chord asymptotic which is a pure Reeb chord $\gamma\in\Rc(\La^-)$. In this case, for the same reason as above, $u$ is a pure disc, and the contribution of such discs vanish once we apply an augmentation to pure negative Reeb chords.

In the subcase of 2.(b) where the index $1$ curve $u$ has boundary on the positive ends of the cobordisms, $u$ can not have a positive asymptotic to a pure Reeb chord because this is not a node and thus would imply that the sequence of discs we started with had a positive pure Reeb chord asymptotic. However, there can be one (or several) index $0$ pure curve with a positive asymptotic to a pure chord $\gamma$ of $\La^+$, and with boundary on a cobordism $\La^-\prec_\Sigma\La^+$. Such an index $0$ curve, call it $v$, contributes thus to $\Phi_\Sigma(\gamma)$ where $\Phi_\Sigma:\Ac(\La^+)\to\Ac(\La^-)$ is the map induced by the cobordism. Applying the augmentation $\ep^-$ to negative Reeb chord asymptotics of $v$ leads to a curve count contributing to $\ep^-\circ\Phi_\Sigma(\gamma)=\ep^+(\gamma)$ by definition of $\ep^+$.
Fixing $\gamma$ and summing over all possible negative Reeb chords of $\La^-$ leads in a curve count giving the term $\ep^+(\gamma)$.

In this paper, every time we define a map via a count of mixed pseudo-holomorphic discs in some moduli spaces, we sum over all possible pure negative Reeb chord asymptotics, to which we apply then the given augmentations of the Legendrian negative ends. Thus, 
\begin{enumerate}
	\item[(A)] the contribution of broken discs having a non trivial pure disc component with boundary on cylindrical ends will vanish,
	\item[(B)] applying the augmentation $\ep_i^-$ to negative pure chords in $\Rc(\La_i^-)$ corresponds to applying $\ep_i^+$ to the potential pure chords asymptotics in $\Rc(\La_i^+)$ of a disc with boundary on the positive cylindrical ends.
\end{enumerate}
This being said, we will now ignore the broken discs of case (A), and the use of the induced augmentations $\ep_i^+$ when describing the boundary of the compactification of moduli spaces refers to breakings in case (B).
\end{rem}

\begin{proof}[Proof of Theorem \ref{diff}]
The degree of $\mfm_1^{\ep_0^-,\ep_1^-}$ follows from Proposition \ref{teo:grading}.
Then, we have
\begin{alignat*}{1}
	&\mfm_1^{\ep_0^-,\ep_1^-}\circ\mfm_1^{\ep_0^-,\ep_1^-}\\
\small
&=\left(
\begin{matrix}
\Delta_1^{+}\circ\Delta_1^+&0&0\\
d_{0+}\circ\Delta_1^++d_{00}\circ d_{0+}+d_{0-}\circ b_1^-\circ\Delta_1^\Sigma&d_{00}^2+d_{0-}\circ b_1^-\circ\Delta_1^\Sigma&d_{00}\circ d_{0-}+d_{0-}\circ b_1^-\\
b_1^-\circ\Delta_1^\Sigma\circ\Delta_1^++b_1^-\circ\Delta_1^\Sigma\circ d_{0+}+(b_1^-)^2\circ\Delta_1^\Sigma&b_1^-\circ\Delta_1^\Sigma\circ d_{00}+(b_1^-)^2\circ\Delta_1^\Sigma&b_1^-\circ\Delta_1^\Sigma\circ d_{0-}+(b_1^-)^2
\end{matrix}\right)
\end{alignat*}
\normalsize
\begin{enumerate}
	\item $\Delta_1^{+}\circ\Delta_1^+$ vanishes because for any $\gamma_{01}\in C(\La_1^+,\La_0^+)$, the discs contributing to $\Delta_1^{+}\circ\Delta_1^+(\gamma_{01})$ are in one-to-one correspondence with broken curves in the boundary of the compactification of moduli spaces
	$\widetilde{\cM^2}_{\R\times\La^+_{01}}(\xi_{01};\bs{\zeta}_0,\gamma,\bs{\zeta}_1)$, for all possible chord $\xi_{01}\in C(\La_1^+,\La_0^+)$ and words of pure Reeb chords $\bs{\zeta}_i$. Observe that $\Delta_1^+$ is in fact the bilinearized Legendrian homology differential of $\La_0^+\cup\La_1^+$ restricted to the subcomplex $C(\La_1^+,\La_0^+)$.
	\item For $\gamma_{01}\in C(\La_1^+,\La_0^+)$, the term $\big(d_{0+}\circ\Delta_1^++d_{00}\circ d_{0+}+d_{0-}\circ b_1^-\circ\Delta_1^\Sigma\big)(\gamma_{01})$ is given exactly by the count of broken curves in 
	$\partial\overline{\cM^1}_{\Sigma_{0},\Sigma_1}(p;\bs{\delta}_0,\gamma_{01},\bs{\delta}_1)$ for all $p\in\Sigma_0\cap\Sigma_1$ and words $\bs{\delta}_i$.
	\item For $\gamma_{01}\in C(\La_1^+,\La_0^+)$, the term $\big(b_1^-\circ\Delta_1^\Sigma\circ\Delta_1^++b_1^-\circ\Delta_1^\Sigma\circ d_{0+}+(b_1^-)^2\circ\Delta_1^\Sigma\big)(\gamma_{01})$ is given by the count of curves in
	\begin{alignat*}{1}
	&\widetilde{\cM^1}_{\R\times\La^-_{01}}(\xi_{10};\bs{\delta}_0',\beta_{01},\bs{\delta}_1')\times\partial\overline{\cM^1}_{\Sigma_{0},\Sigma_1}(\beta_{01};\bs{\delta}_0,\gamma_{01},\bs{\delta}_1)\\
	\mbox{ and }\,&\partial\overline{\cM^2}_{\R\times\La^-_{01}}(\xi_{10};\bs{\delta}_0',\beta_{01},\bs{\delta}_1')\times\cM^0_{\Sigma_{0},\Sigma_1}(\beta_{01};\bs{\delta}_0,\gamma_{01},\bs{\delta}_1)
	\end{alignat*}
	for all $\xi_{10}\in C(\La_0^-,\La_1^-),\beta_{01}\in C(\La_1^-,\La_0^-)$ and words of pure Reeb chords $\bs{\delta}_i,\bs{\delta}_i'$ of $\La_i^-$. Indeed, the study of $\partial\overline{\cM^2}_{\R\times\La^-_{01}}(\xi_{10};\bs{\delta}_0',\beta_{01},\bs{\delta}_1')$ gives that the map $b_1^-$ restricted to $C(\La_1^+,\La_0^+)$ satisfies $(b_1^-)^2+b_1^-\circ\Delta_1^-=0$, where $\Delta_1^-$ is the obvious analogue of $\Delta_1^+$ but defined on $C(\La_1^-,\La_0^-)$. Then one can write 
	\begin{alignat*}{1}
	&\big(b_1^-\circ\Delta_1^\Sigma\circ\Delta_1^++b_1^-\circ\Delta_1^\Sigma\circ d_{0+}+(b_1^-)^2\circ\Delta_1^\Sigma\big)(\gamma_{01})=b_1^-\big(\Delta_1^\Sigma\circ\Delta_1^++\Delta_1^\Sigma\circ d_{0+}+\Delta_1^-\circ\Delta_1^\Sigma\big)(\gamma_{01})
	\end{alignat*}
	and there is a one-to-one correspondence between broken discs contributing to $\Delta_1^\Sigma\circ\Delta_1^++\Delta_1^\Sigma\circ d_{0+}+\Delta_1^-\circ\Delta_1^\Sigma$ and broken discs in $\partial\overline{\cM^1}_{\Sigma_{0},\Sigma_1}(\beta_{01};\bs{\delta}_0,\gamma,\bs{\delta}_1)$.
	\item $\big(d_{00}^2+d_{0-}\circ b_1^-\circ\Delta_1^\Sigma\big)(x)$, for $x\in CF(\Sigma_0,\Sigma_1)$, counts broken curves in $\partial\overline{\cM^1}_{\Sigma_{0},\Sigma_1}(p;\bs{\delta}_0,x,\bs{\delta}_1)$, for all $p\in\Sigma_0\cap\Sigma_1$ and words $\bs{\delta}_i$.
	\item $\big(b_1^-\circ\Delta_1^\Sigma\circ d_{00}+(b_1^-)^2\circ\Delta_1^\Sigma\big)(x)$ counts broken curves in
	\begin{alignat*}{1}
		&\widetilde{\cM^1}_{\R\times\La^-_{01}}(\xi_{10};\bs{\delta}_0',\beta_{01},\bs{\delta}_1')\times\partial\overline{\cM^1}_{\Sigma_{0},\Sigma_1}(\beta_{01};\bs{\delta}_0,x,\bs{\delta}_1)\\
		\mbox{ and }\,&\partial\overline{\cM^2}_{\R\times\La^-_{01}}(\xi_{10};\bs{\delta}_0',\beta_{01},\bs{\delta}_1')\times\cM^0_{\Sigma_{0},\Sigma_1}(\beta_{01};\bs{\delta}_0,x,\bs{\delta}_1)
	\end{alignat*}
	for all $\xi_{10},\beta_{01}$ and words of Reeb chords $\bs{\delta}_i$ and $\bs{\delta}_i'$ as above.
	\item $\big(d_{00}\circ d_{0-}+d_{0-}\circ b_1^-\big)(\gamma_{10})$, for $\gamma_{10}\in C(\La_0^-,\La_1^-)$, counts broken curves in $\partial\overline{\cM^1}_{\Sigma_{0},\Sigma_1}(p;\bs{\delta}_0,\gamma_{10},\bs{\delta}_1)$.
	\item $\big(b_1^-\circ\Delta_1^\Sigma\circ d_{0-}+(b_1^-)^2\big)(\gamma_{10})=(b_1^-)^2(\gamma_{10})$ for energy reasons, and vanishes as $b_1^-$ restricted to $C(\La_0^-,\La_1^-)$ is the bilinearized Legendrian contact cohomology differential of $\La_0^-\cup\La_1^-$ restricted to the subcomplex $C(\La_0^-,\La_1^-)$ (observe otherwise that the broken discs contributing to $(b_1^-)^2(\gamma_{10})$ are exactly the one appearing in $\partial\overline{\cM^2}_{\R\times\La^-_{01}}(\xi_{10};\bs{\delta}_0,\gamma_{10},\bs{\delta}_1)$, for all $\xi_{10},\bs{\delta}_0,\bs{\delta}_1$).
\end{enumerate}
\end{proof}

\begin{nota} A few remarks about notations of maps:
	\begin{enumerate}
		\item Very rigorously, we should write the augmentations involved in the definition of each map all the time, but we drop it to enlighten the notation.
		\item Given a pair of cobordisms $(\Sigma_0,\Sigma_1)$, we will then write $\mfm_1^\Sigma$ for the differential on $\Cth_+(\Sigma_0,\Sigma_1)$, so without specifying the augmentations, and \textquotedblleft$\Sigma$\textquotedblright\, stands for the ordered pair $(\Sigma_0,\Sigma_1)$. If we want to explicit the order, we will sometimes write $\mfm_1^{\Sigma_0,\Sigma_1}$ or $\mfm_1^{\Sigma_{01}}$. Similarly, we write $\Delta_1^\Sigma$ instead of $\Delta_1^{\Sigma_0,\Sigma_1}$, and finally $b_1^-$ is a short notation for $b_1^{\Lambda_0^-,\La_1^-}$ and $\Delta_1^+$ is a short notation for $\Delta_1^{\Lambda_0^+,\La_1^+}$.
		\item We write $\mfm_1^{\Sigma,+}$, $\mfm_1^{\Sigma,0}$ and $\mfm_1^{\Sigma,-}$ (or simply $\mfm_1^{+}$, $\mfm_1^{0}$ and $\mfm_1^{-}$ when the pair of cobordisms is clear from the context) the components of the differential with values in $C(\La_1^+,\La_0^+)^\dagger[n-1]$, $CF(\Sigma_0,\Sigma_1)$ and $C(\La_0^-,\La_1^-)[1]$ respectively, and then $\mfm_1^{\Sigma,ij}:=\mfm_1^{\Sigma,i}+\mfm_1^{\Sigma,j}$, for $i,j\in\{+,0,-\}$ distinct.
		\item We will sometimes denote $CF_{+\infty}(\Sigma_0,\Sigma_1):= C(\La_1^+,\La_0^+)^\dagger[n-1]\oplus CF(\Sigma_0,\Sigma_1)$, but observe that contrary to $CF_{-\infty}(\Sigma_0,\Sigma_1)$, this is not a complex.
	\end{enumerate}
\end{nota}
Considering the notations above, the components $\mfm_1^+$ and $\mfm_1^-$ of $\mfm_1^\Sigma$ can be expressed as:
\begin{alignat}{1}
&\mfm_1^+=\Delta_1^+\\
&\mfm_1^-=b_1^-\circ\bs{\Delta}_1^\Sigma
\end{alignat}
where $\bs{\Delta}_1^\Sigma:\Cth_+(\Sigma_0,\Sigma_1)\to C_{n-1-*}(\La_1^-,\La_0^-)\oplus C^{*-1}(\La_0^-,\La_1^-)$ is defined by:
\begin{alignat*}{1}
&\bs{\Delta}_1^\Sigma(a)=
\left\{\begin{array}{cl} a&\mbox{ if }a\in C(\La_0^-,\La_1^-)\\
\Delta_1^\Sigma(a)&\mbox{ otherwise}\end{array}\right.
\end{alignat*}

\begin{ex}[Case of concordances]\label{ex1}
Consider a compact non-degenerate Legendrian submanifold $\La\subset Y$, admitting augmentations $\ep_0,\ep_1$. 
Consider a $2$-copy $\La^{(2)}$ of $\La$ consisting of the components $\La_0$ and $\La_1$ such that $\La_1$ is a copy of $\La_0:=\La$ perturbed by a small negative Morse function $f$, i.e. $\La_1$ is identified with $j^1(f)$ in a neighborhood of $\La$ identified with a neighborhood of the $0$-section of $J^1(\La)$, see Figure \ref{2-copy}. The Legendrian $\La_1$ inherits a Maslov potential from $\La_0$ and the mixed Reeb chords of $\La_0\cup\La_1$ are of three types:
\begin{itemize}
	\item $p$-chords: long chords from $\La_1$ to $\La_0$ corresponding to pure chords of $\La_0$,
	\item $q$-chords: long chords from $\La_0$ to $\La_1$ corresponding to pure chords of $\La_0$,
	\item \textit{Morse} chords: short chords from $\La_1$ to $\La_0$ corresponding to critical points of $f$. Note that an index $k$ critical point of $f$ corresponds to a Morse Reeb chord of LCH degree $n-k-1$ as $f$ is negative.
\end{itemize}
We have:
\begin{alignat*}{1}
\Cth_+(\R\times\La_0,\R\times\La_1)=C(\La_1,\La_0)^\dagger[n-1]\oplus C(\La_0,\La_1)[1]=\mathfrak{C}^*(\La_0,\La_1)
\end{alignat*}
where $C(\La_1,\La_0)^\dagger[n-1]$ is generated by $q$-chords and  $C(\La_0,\La_1)[1]$ is generated by $p$ and \textit{Morse} chords.
The differential takes the form:
\begin{alignat*}{1}
	\mfm_1^{\ep_0,\ep_1}=\left(
	\begin{matrix}
	\Delta_1^+&0\\
	b_1^-\circ\Delta_1^\Sigma&b_1^-
	\end{matrix}\right)=\left(
	\begin{matrix}
	\Delta_1^+&0\\
	b_1^-&b_1^-
	\end{matrix}\right)
\end{alignat*}
because the cobordisms are trivial cylinders so the map $\Delta_1^\Sigma$ is the identity map.

Consider another 2-copy of $\La$ which we denote $\overline{\La^{(2)}}$, consisting of $\La_0\cup\overline{\La_1}$, such that $\La_0:=\La$ and $\overline{\La_1}$ a perturbation of a push-off of $\La_1$ far in the positive Reeb direction so that it lies entirely \textit{above} $\La_0$ (the $z$-coordinate of any point in $\overline{\La}_1$ is greater than the $z$ coordinate of any point in $\La_0$). This is the 2-copy of $\La$ considered in \cite{EESa}. 
The mixed Reeb chords of $\La_0\cup\overline{\La_1}$ are all from $\La_0$ to $\overline{\La_1}$ but still of three different types:
\begin{itemize}
	\item $\bar q$-chords: long chords corresponding to pure chords of $\La_0$,
	\item $\bar p$-chords: short chords corresponding to pure chords of $\La_0$,
	\item \textit{Morse} chords corresponding to critical points of $f$. Note that in this case an index $k$ critical point of $f$ corresponds to a Morse Reeb chord of LCH degree $k-1$.
\end{itemize}
We have
$$\big(\Cth_+(\R\times\La_0,\R\times\overline{\La_1}),\mfm_1^{\ep_0,\ep_1}\big)=(C(\La_1,\La_0)^\dagger[n-1],\Delta_1^+)$$
and this complex is the complex of the bilinearized Legendrian contact homology of $\La_0\cup\overline{\La_1}$ restricted to mixed chords, but with a choice of grading making the differential a degree $1$ map.
\begin{figure}[ht]  
	\begin{center}\includegraphics[width=5cm]{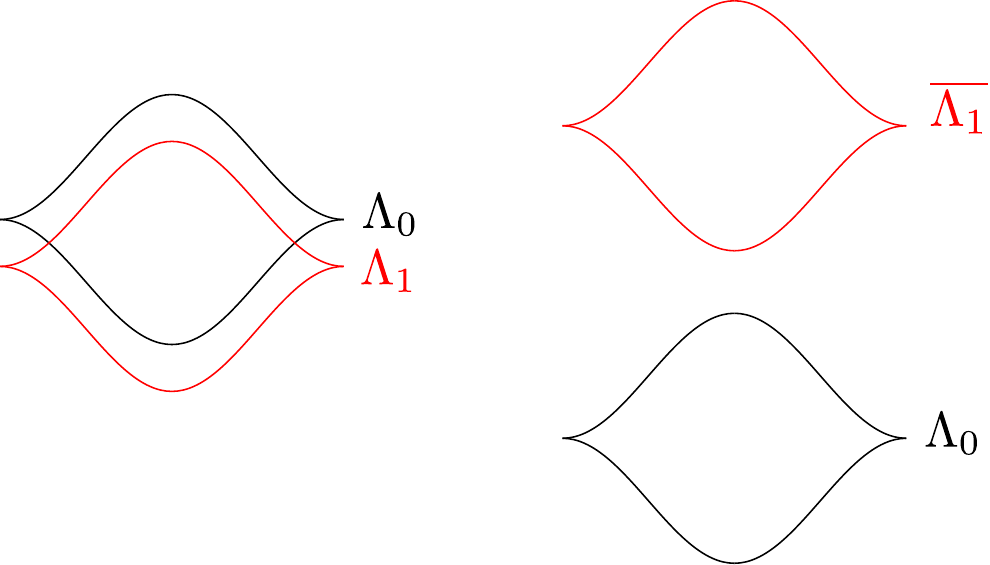}\end{center}
	\caption{Left: $2$-copy $\La^{(2)}$ of $\La$; right: $2$-copy $\overline{\La^{(2)}}$.}
	\label{2-copy}
\end{figure}
There is a canonical isomorphism of complexes
\begin{alignat*}{1}
\big(\Cth_+(\R\times\La_0,\R\times\La_1),\left(
\begin{matrix}
\Delta_1^+&0\\
b_1^-&b_1^-
\end{matrix}\right)\big)\xrightarrow[]{\simeq}(\Cth_+(\R\times\La_0,\R\times\overline{\La_1}),\Delta_1^+)
\end{alignat*}
sending a $q$-chord to its corresponding $\bar q$-chord, a $p$-chord to its corresponding $\bar p$-chord and a Morse chord $c$ to its corresponding Morse chord which we denote $\bar c$. Moreover it is of degree $0$ according to the $\Cth_+$ grading. This isomorphism of graded vector spaces extends to an isomorphism of complexes. We refer for example to \cite[Proposition 5.4]{NRSSZ} for  a detailed proof, we recall here the correspondence of pseudo-holomorphic discs with boundary on one of the 2-copy or the other.
Pseudo-holomorphic discs with boundary on $\R\times(\La_0\cup\La_1)$ with one positive asymptotic to a $q$-chords $q_1$ and one negative asymptotic to a $q$-chord $q_2$ are identified with discs with boundary on $\R\times(\La_0\cup\overline{\La_1})$ with one positive, resp. negative asymptotic to the chords $\bar q_1$, resp. $\bar q_2$.
 These discs correspond to $\Delta_1^+$ in $\Cth_+(\R\times\La_0,\R\times\La_1)$ and $\Delta_1^+$ restricted to $\bar q$ chords and with values in $\bar q$ chords  in $\Cth_+(\R\times\La_0,\R\times\overline{\La_1})$. Then, the component $b^-_1$ restricted to $q$-chords is defined by a count of bananas with two positive asymptotics, one at a $q$-chord which is an input and one at a $p$- or a Morse chord which is the output. Such bananas correspond exactly to discs with boundary on $\R\times(\La_0\cup\overline{\La_1})$ with one positive input asymptotic at the corresponding $\bar q$-chord and one negative output asymptotic to the $\bar p$- or Morse chord. These are discs contributing to $\Delta_1^+$ restricted to $\bar q$-chords and taking values in $\bar p$- and Morse chords in  $\Cth_+(\R\times\La_0,\R\times\overline{\La_1})$.
Finally, a disc contributing to $b^-_1$ restricted to $p$- and Morse chords is a disc with a negative input asymptotic at a $p$- or Morse chord $\gamma_1$ and a positive output asymptotic at a $p-$ or Morse chord $\gamma_2$. These become discs with boundary on $\R\times(\La_0\cup\overline{\La_1})$ with a positive input asymptotic at $\bar \gamma_1$ and a negative output asymptotic at $\bar \gamma_2$.

\end{ex}

\begin{ex}[$0$-section of a jet space]\label{ex:jet}
	 Consider the jet space $J^1(M)=T^*M\times\R$ of a smooth manifold $M$, endowed with the standard contact form $dz-\lambda$ where $z$ is the $\R$ coordinate and $\la$ the canonical form on $T^*M$. Then the $0$-section is a Legendrian $\La_0:=M$. Take a small push-off of $\La$ in the positive Reeb ($\partial_z$) direction and perturb it by a small Morse function $f:\La_0\to\R$. Denote this Legendrian $\La_1$. Consider the trivial augmentations $\ep_i$, $i=0,1$. Then, the complex $\Cth_+(\R\times\La_0,\R\times\La_1),\mfm_1^{\ep_0^-,\ep_1^-})$ is just the complex $(C(\La_1,\La_0)^\dagger[n-1],\Delta_1^{+})$ which is canonically identified with the Morse complex of $f$.
\end{ex}

\subsection{Concatenation of cobordisms}\label{sec:conc}

\subsubsection{Definition of the complex $\Cth_+(V_0\odot W_0,V_1\odot W_1)$}

Consider Legendrian submanifolds $\La_i^-,\La_i,\La_i^+$ for $i=0,1$, and cobordisms $\La_i^-\prec_{V_i}\La_i$ and $\La_i\prec_{W_i}\La_i^+$. As the positive end of $V_i$ is a cylinder over $\La_i$, as well as the negative end of $W_i$, one can perform the concatenation of $V_i$ and $W_i$ denoted $V_i\odot W_i$, which is an exact Lagrangian cobordism from $\La_i^-$ to $\La_i^+$, see for example \cite[Section 5.1]{CDGG2}. Assume that $\Ac(\La_i^-)$ admit augmentations $\ep_0^-,\ep_1^-$. These augmentations induce augmentations $\ep_0$ and $\ep_1$ of $\Ac(\La_0)$ and $\Ac(\La_1)$ respectively, and augmentations $\ep_0^+$ and $\ep_1^+$ of $\Ac(\La_0^+)$ and $\Ac(\La_1^+)$.
Assuming that the cobordisms $V_0\odot W_0$ and $V_1\odot W_1$ intersect transversely, the Cthulhu complex of the pair $(V_0\odot W_0,V_1\odot W_1)$ has four types of generators
\begin{alignat*}{1}
\Cth_+(V_0\odot W_0,V_1\odot W_1)&=C(\La_1^+,\La_0^+)\oplus CF(W_0,W_1)\oplus CF(V_0,V_1)\oplus C(\La_0^-,\La_1^-)\\
&=CF_{+\infty}(W_0,W_1)\oplus CF_{-\infty}(V_0,V_1)
\end{alignat*}
and the differential is given by
\begin{alignat*}{2}
\mfm^{V\odot W}_1
&=\left(\begin{matrix}
\mfm_1^{W,+}&0&0&0\\
\mfm_1^{W,0}(\id+\,b_1^V\circ\Delta_1^{W})&\mfm_1^{W,0}(\id+\,b_1^V\circ\Delta_1^{W})&\mfm_1^{W,0}\circ b_1^V&\mfm_1^{W,0}\circ b_1^V\\
\mfm_1^{V,0}\circ\Delta_1^{W}&\mfm_1^{V,0}\circ\Delta_1^{W}&\mfm_1^{V,0}&\mfm_1^{V,0}\\
\mfm_1^{V,-}\circ\Delta_1^{W}&\mfm_1^{V,-}\circ\Delta_1^{W}&\mfm_1^{V,-}&\mfm_1^{V,-}
\end{matrix}\right)	
\end{alignat*}
where $b_1^V:\Cth_+(V_0,V_1)\to C(\La_0,\La_1)[1]$ is the degree $0$ map defined by
\begin{alignat*}{1}
&b_1^V(a)=\sum\limits_{\gamma_{10}}\sum\limits_{\bs{\delta}_i}\#\cM^0_{V_{0},V_1}(\gamma_{10};\bs{\delta}_0,a,\bs{\delta}_1)\cdot\ep_0^-(\bs{\delta}_0)\ep_1^-(\bs{\delta}_1)\cdot\gamma_{10}	
\end{alignat*}
We will extend the definitions of the maps $b_1^V$ and $\Delta_1^V$ to $\Cth_+(V_0\odot W_0,V_1\odot W_1)$ in order to obtain a compact formula for $\mfm_1^{V\odot W}$.
Namely, we define
$\bs{b}_1^V:\Cth_+(V_0\odot W_0,V_1\odot W_1)\to\Cth_+(W_0,W_1)$
by
$$\bs{b}_1^V(a)=\left\{
\begin{array}{cl}
 a+b_1^V\circ\Delta_1^W(a)&\mbox{ for } a\in CF_{+\infty}(W_0,W_1)\\
 b_1^V(a) &\mbox{ for } a\in CF_{-\infty}(V_0,V_1)
\end{array}\right.$$
and we define $\bs{\Delta}_1^W:\Cth_+(V_0\odot W_0,V_1\odot W_1)\to\Cth_+(V_0,V_1)$
by
$$\bs{\Delta}_1^W(a)=\left\{
\begin{array}{cl}
\Delta_1^W(a)&\mbox{ for } a\in CF_{+\infty}(W_0,W_1)\\
a &\mbox{ for } a\in CF_{-\infty}(V_0,V_1)
\end{array}\right.$$
This may seem confusing because we have already defined a map $\bs{\Delta}_1^\Sigma$ for the case of a pair of cobordisms $(\Sigma_0,\Sigma_1)$ in the previous section. However, this map $\bs{\Delta}_1^\Sigma$ can be recovered from the map $\bs{\Delta}_1^W$ for the pair $(V_0\odot W_0,V_1\odot W_1)$ where $(V_0,V_1)=(\R\times\La_0^-,\R\times\La_1^-)$ and $(W_0,W_1)=(\Sigma_0,\Sigma_1)$, see Section \ref{special_case} for more details. In the remaining of this section, to make it clear we write $\bs{\Delta}_1^{W\subset W}$ when we consider the map for the pair $(W_0,W_1)$ not in the concatenation.

One can thus write the product in the following more compact way:
\begin{alignat*}{1}
\mfm_1^{V\odot W}=\mfm_1^{W,+0}\circ\,\bs{b}_1^V+\mfm_1^{V,0-}\circ\bs{\Delta}_1^W
\end{alignat*}
Let us now check that $\mfm_1^{V\odot W}$ is indeed a differential.
Using the definition of $\bs{b}_1^V$ and $\bs{\Delta}_1^W$, we have:
\begin{alignat*}{2}
\big(\mfm_1^{V\odot W}\big)^2=&\Big(\mfm_1^{W,+0}\circ\,\bs{b}_1^V+\mfm_1^{V,0-}\circ\,\bs{\Delta}_1^W\Big)\circ\Big(\mfm_1^{W,+0}\circ\,\bs{b}_1^V+\mfm_1^{V,0-}\circ\,\bs{\Delta}_1^W\Big)\\
=&\mfm_1^{W,+0}\circ\mfm_1^{W,+0}\circ\,\bs{b}_1^V+\mfm_1^{W,+0}\circ\big(b_1^V\circ\Delta_1^W\big)\circ\mfm_1^{W,+0}\circ\,\bs{b}_1^V+\mfm_1^{W,+0}\circ\,b_1^V\circ\mfm_1^{V,0-}\circ\,\bs{\Delta}_1^W\\
&+\mfm_1^{V,0-}\circ\,\Delta_1^W\circ\mfm_1^{W,+0}\circ\,\bs{b}_1^V+\mfm_1^{V,0-}\circ\mfm_1^{V,0-}\circ\,\bs{\Delta}_1^W
\end{alignat*}
where by definition the term $\mfm_1^{W,+}\circ\,b_1^V\circ\bs{\Delta}_1^W$ vanishes but we keep it in the formula to make it more homogeneous. We use then the following

\begin{lem}\label{lem:rel}
	The maps 
	\begin{enumerate}
		\item $\Delta_1^W\circ\mfm_1^{W,+0}\circ\,\bs{b}_1^V+\mfm_1^{V,+}\circ\,\bs{\Delta}_1^W:\Cth_+(V_0\odot W_0,V_1\odot W_1)\to C_{n-1-*}(\La_1,\La_0)$\label{rel1}, and
		\item $b_1^V\circ\mfm_1^{V}\circ\,\bs{\Delta}_1^W+\,b_1^\La\circ\bs{\Delta}_1^{W\subset W}\circ\bs{b}_1^V:\Cth_+(V_0\odot W_0,V_1\odot W_1)\to C^{*-1}(\La_0,\La_1)$\label{rel2}
	\end{enumerate}
	vanish.
\end{lem}

\begin{proof}
	\textit{1.} For $a\in CF_{-\infty}(V_0,V_1)$ we have 
\begin{alignat*}{1}
	\Delta_1^W\circ\mfm_1^{W,+0}\circ\,\bs{b}_1^V(a)+\mfm_1^{V,+}\circ\,\bs{\Delta}_1^W(a)=\Delta_1^W\circ\mfm_1^{W,+0}\circ\,b_1^V(a)+\mfm_1^{V,+}(a)
\end{alignat*}
and the first term vanishes for energy reason and the second one by definition.
Then, for $a\in CF_{+\infty}(W_0,W_1)$ we have
\begin{alignat*}{1}
	\Delta_1^W\circ\mfm_1^{W,+0}\circ\,\bs{b}_1^V(a)+\mfm_1^{V,+}\circ\,\bs{\Delta}_1^W(a)&=\Delta_1^W\circ\mfm_1^{W,+0}(a+b_1^V\circ\Delta_1^W(a))+\mfm_1^{V,+}\circ\,\Delta_1^W(a)\\
	&=\Delta_1^W\circ\mfm_1^{W,+0}(a)+\mfm_1^{V,+}\circ\,\Delta_1^W(a)
\end{alignat*}
because $\Delta_1^W\circ\mfm_1^{W,+0}\circ b_1^V\circ\Delta_1^W(a)$ vanishes for energy reason. Consider the boundary of the one-dimensional moduli space $\overline{\cM^1}_{W_{0},W_1}(\beta_{01};\bs{\delta}_0,a,\bs{\delta}_1)$ for $\beta_{01}\in C(\La_1,\La_0)$. The broken discs arising in the boundary (schematized on Figure \ref{broken_lemma} for the case $a=\gamma_{01}\in C(\La_1^+,\La_0^+)$) contribute exactly to 
\begin{alignat}{1}
	\big\langle(\Delta_1^W\circ\mfm_1^{W,+0}+\mfm_1^{V,+}\circ\,\Delta_1^W)(a),\beta_{01}\big\rangle\label{rel}
\end{alignat}
\begin{figure}[ht]  
		\begin{center}\includegraphics[width=9cm]{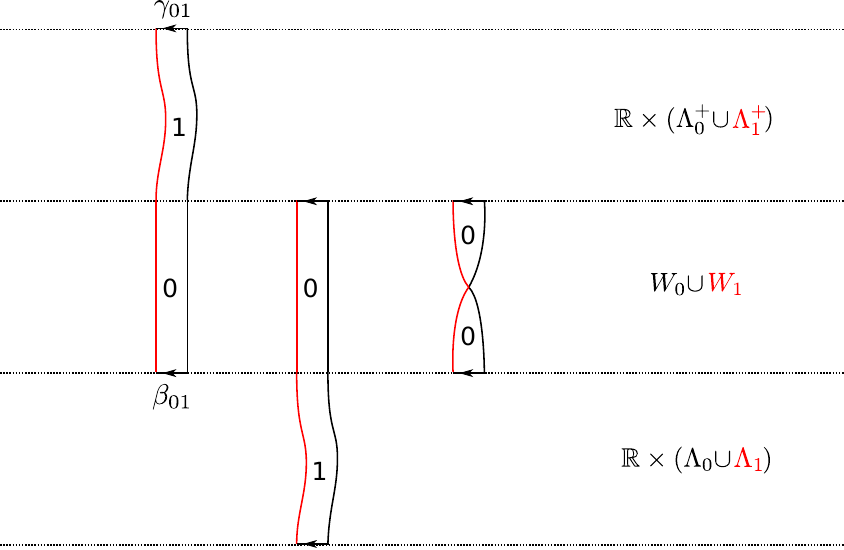}\end{center}
		\caption{Types of broken discs in the boundary of $\overline{\cM^1}_{W_{0},W_1}(\beta_{01};\bs{\delta}_0,\gamma_{01},\bs{\delta}_1)$.}
		\label{broken_lemma}
\end{figure}
	
\textit{2.} For $a\in CF_{-\infty}(V_0,V_1)$ we have
	\begin{alignat*}{1}
		b_1^V\circ\mfm_1^{V}\circ\,\bs{\Delta}_1^W(a)+b_1^\La\circ\bs{\Delta}_1^{W\subset W}\circ\bs{b}_1^V(a)=b_1^V\circ\mfm_1^V(a)+b_1^\La\circ b_1^V(a)
	\end{alignat*}
and for $a\in CF_{+\infty}(W_0,W_1)$ we have
	\begin{alignat*}{1}
		b_1^V\circ\mfm_1^{V}\circ\bs{\Delta}_1^W(a)+b_1^\La\circ\bs{\Delta}_1^{W\subset W}\circ\bs{b}_1^V(a)&=	b_1^V\circ\mfm_1^{V}\circ\Delta_1^W(a)+b_1^\La\circ\bs{\Delta}_1^{W\subset W}(a+b_1^V\circ\Delta_1^W(a))\\
		&=b_1^V\circ\mfm_1^{V}\circ\Delta_1^W(a)+b_1^\La\circ\Delta_1^W(a)+b_1^\La\circ b_1^V\circ\Delta_1^W(a)
	\end{alignat*}
To conclude that this map vanishes, one has to consider the broken curves in the boundary of the compactification of moduli spaces of bananas with boundary on $V$, namely the boundary of $\overline{\cM^1}_{V_0,V_1}(\gamma_{10};\bs{\delta}_0,a,\bs{\delta}_1)$ for $\gamma_{10}\in C(\La_0,\La_1)$, and $a\in\Cth_+(V_0\odot W_0,V_1\odot W_1)$, see Figure \ref{broken_bV} for the case $a=x\in CF(W_0,W_1)$.
\end{proof} 

\begin{figure}[ht]  
	\begin{center}\includegraphics[width=11cm]{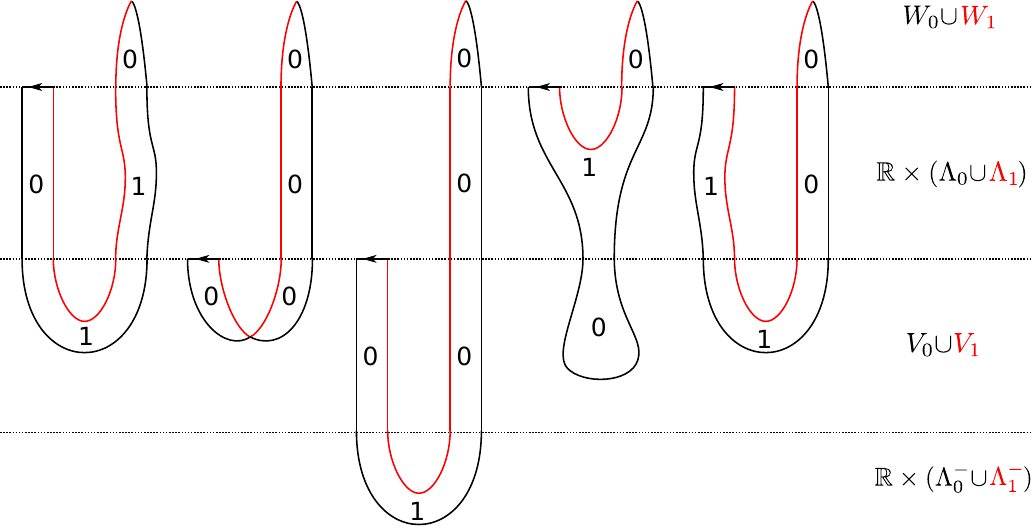}\end{center}
	\caption{Types of broken discs in $\partial\overline{\cM^1}_{V_0,V_1}(\gamma_{10};\bs{\delta}_0,\gamma_{01},\bs{\delta}_1)$.}
	\label{broken_bV}
\end{figure}
Thus, using part $1.$ of the lemma we can rewrite
\begin{alignat*}{2}
\big(\mfm_1^{V\odot W}\big)^2
=&\mfm_1^{W,+0}\circ\mfm_1^{W,+0}\circ\,\bs{b}_1^V+\mfm_1^{W,+0}\circ\,b_1^V\circ\mfm_1^{V,+}\circ\bs{\Delta}_1^V+\mfm_1^{W,+0}\circ\,b_1^V\circ\mfm_1^{V,0-}\circ\,\bs{\Delta}_1^W\\
&+\mfm_1^{V,0-}\circ\mfm_1^{V,+}\circ\,\bs{\Delta}_1^V+\mfm_1^{V,0-}\circ\mfm_1^{V,0-}\circ\,\bs{\Delta}_1^W\\
&=\mfm_1^{W,+0}\circ\mfm_1^{W,+0}\circ\,\bs{b}_1^V+\mfm_1^{W,+0}\circ\,b_1^V\circ\mfm_1^{V}\circ\bs{\Delta}_1^V+\mfm_1^{V,0-}\circ\mfm_1^{V}\circ\,\bs{\Delta}_1^V
\end{alignat*}
Now, using $\mfm_1^{W,+0}\circ\mfm_1^{W,+0}=\mfm_1^{W,+0}\circ\mfm_1^{W,-}=\mfm_1^{W,+0}\circ\,b_1^\La\circ\bs{\Delta}_1^{W\subset W}$, part $2.$ of Lemma \ref{lem:rel}, and the fact that $\mfm_1^{V,0-}\circ\mfm_1^V=0$, one gets that $\big(\mfm_1^{V\odot W}\big)^2=0$.

\subsubsection{Transfer maps}

The maps $\bs{b}_1^V$ and $\bs{\Delta}_1^W$ defined in the previous section are in fact what we will call \textit{transfer maps}. In particular, they are chain maps as we prove now.

\begin{prop}\label{prop:Phi}
	$\bs{b}_1^V:\Cth_+(V_0\odot W_0,V_1\odot W_1)\to\Cth_+(W_0,W_1)$ is a chain map.
\end{prop}

\begin{proof}
	We need to prove that
\begin{alignat}{1}
	\bs{b}_1^V\circ\mfm_1^{V\odot W}+\mfm_1^W\circ\,\bs{b}_1^V=0\label{eq:cotransfer}
\end{alignat}
By definition of $\mfm_1^{V\odot W}$ we have that the left-hand side of \eqref{eq:cotransfer} is equal to
\begin{alignat*}{1}
	\bs{b}_1^V\circ\mfm_1^{W,+0}\circ\,\bs{b}_1^V&+\bs{b}_1^V\circ\mfm_1^{V,0-}\circ\,\bs{\Delta}_1^W+\mfm_1^W\circ\,\bs{b}_1^V\\
	&=\mfm_1^{W,+0}\circ\,\bs{b}_1^V+b_1^V\circ\Delta_1^W\circ\mfm_1^{W,+0}\circ\,\bs{b}_1^V+b_1^V\circ\mfm_1^{V,0-}\circ\,\bs{\Delta}_1^W+\mfm_1^W\circ\,\bs{b}_1^V\\
	&=b_1^V\circ\Delta_1^W\circ\mfm_1^{W,+0}\circ\,\bs{b}_1^V+b_1^V\circ\mfm_1^{V,0-}\circ\,\bs{\Delta}_1^W+\mfm_1^{W,-}\circ\,\bs{b}_1^V
\end{alignat*}
Using the first part of Lemma \ref{lem:rel} on the first term and the second part of the lemma on the second and third terms, recalling that $\mfm_1^{W,-}=b_1^\La\circ\bs{\Delta}_1^{W\subset W}$, we get that the sum above vanishes.
\end{proof}

\begin{prop}\label{prop:Delta}
	$\bs{\Delta}_1^W:\Cth_+(V_0\odot W_0,V_1\odot W_1)\to\Cth_+(V_0,V_1)$ is a chain map.
\end{prop}

\begin{proof}
	We have to prove that
	\begin{alignat}{1}
	\bs{\Delta}_1^W\circ\mfm_1^{V\odot W}+\mfm_1^{V}\circ\,\bs{\Delta}_1^W=0
	\end{alignat}
The left-hand side of the equation is
\begin{alignat*}{1}
	\bs{\Delta}_1^W\circ\mfm_1^{W,+0}\circ\,\bs{b}_1^V+\bs{\Delta}_1^W\circ\mfm_1^{V,0-}\circ\,\bs{\Delta}_1^W+\mfm_1^{V}\circ\,\bs{\Delta}_1^W
\end{alignat*}
whose first term equals $\mfm_1^{V,+}\circ\,\bs{\Delta}_1^W$ by Lemma \ref{lem:rel} and the second term equals $\mfm_1^{V,0-}\circ\,\bs{\Delta}_1^W$ by definition of $\bs{\Delta}_1^W$. Thus the sum vanishes.
\end{proof}

\subsubsection{Special cases}\label{special_case}

Let us have a look at the two following cases for the pair $(V_0\odot W_0,V_1\odot W_1)$:
\begin{enumerate}
	\item $(W_0,W_1)=(\R\times\La_0,\R\times\La_1)$,
	\item $(V_0,V_1)=(\R\times\La_0,\R\times\La_1)$
\end{enumerate}
In the first case, one has 
$$\Cth_+(V_0\odot(\R\times\La_0),V_1\odot(\R\times\La_1))=C(\La_1,\La_0)^\dagger[n-1]\oplus CF(V_0,V_1)\oplus C(\La_0^-,\La_1^-)[1]=\Cth_+(V_0,V_1)$$
and we actually have an equality of complexes as $\Delta_1^W$ on $C(\La_1,\La_0)$ is the identity map (as it counts index $0$ discs with boundary on Lagrangian cylinders so it can only count trivial strips). Thus, the map $\bs{b}_1^V$ defined for a general pair of concatenated cobordisms before gives in this case a map
\begin{alignat*}{1}
	\bs{b}_1^V:\Cth_+(V_0,V_1)\to\Cth_+(\R\times\La_0,\R\times\La_1)=\mathfrak{C}^*(\La_0,\La_1)
\end{alignat*}
satisfying $\bs{b}_1^V(a)=a+b_1^V(a)$ for $a\in C(\La_1,\La_0)$ and $\bs{b}_1^V(a)=b_1^V(a)$ for $a\in CF_{-\infty}(V_0,V_1)$.
In the second case, one has 
$$\Cth_+((\R\times\La_0)\odot W_0,(\R\times\La_1)\odot W_1)=C(\La_1^+,\La_0^+)^\dagger[n-1]\oplus CF(W_0,W_1)\oplus C(\La_0,\La_1)[1]=\Cth_+(W_0,W_1)$$
and again this equality holds in terms of complexes as $b_1^V$ is the identity map on $C(\La_0,\La_1)$ and vanishes on $C(\La_1,\La_0)$ (no index $0$ banana with boundary on $\R\times(\La_0\cup\La_1)$ and two positive Reeb chord asymptotics), and $\Delta_1^V$ is the identity map on $C(\La_1,\La_0)$.
For such a pair of concatenated cobordisms, we get the map
\begin{alignat*}{1}
\bs{\Delta}_1^W:\Cth_+(W_0,W_1)\to\mathfrak{C}^*(\La_0,\La_1)
\end{alignat*}
satisfying $\bs{\Delta}_1^W(a)=\Delta_1^W(a)$ for $a\in CF_{+\infty}(W_0,W_1)$ and $\bs{\Delta}_1^W(a)=a$ for $a\in C(\La_0,\La_1)$, which recovers exactly the definition we gave at the end of the Section \ref{section:def_complex}.

\begin{nota}
	From now on, we use the maps $\bs{b}_1^V$ and $\bs{\Delta}_1^W$ without specifying if we are in the case of a pair $(V_0\odot W_0,V_1\odot W_1)$, $\big(V_0\odot(\R\times\La_0),V_1\odot(\R\times\La_1)\big)$ or $\big((\R\times\La_0)\odot W_0,(\R\times\La_1)\odot W_1\big)$.
\end{nota}

\subsubsection{Mayer-Vietoris long exact sequence}

From the previous special cases, one deduces a Mayer-Vietoris sequence. Consider a pair of concatenations $(V_0\odot W_0,V_1\odot W_1)$. By definition, we have
\begin{lem}\label{bcircdelta}
	$\bs{b}_1^V\circ\bs{\Delta}_1^W+\bs{\Delta}_1^W\circ\bs{b}_1^V:\Cth_+(V_0\odot W_0,V_1\odot W_1)\to\mathfrak{C}^*(\La_0,\La_1)$ vanishes.
\end{lem}
\begin{proof}
	First, remember that $\bs{\Delta}_1^V:\Cth_+(V_0\odot W_0,V_1\odot W_1)\to\Cth_+(V_0,V_1)$, so the term $\bs{b}_1^V\circ\bs{\Delta}_1^W$ should be read as $\bs{b}_1^{V\subset V}\circ\bs{\Delta}_1^{W\subset V\odot W}$. Similarly,
	$\bs{b}_1^V:\Cth_+(V_0\odot W_0,V_1\odot W_1)\to\Cth_+(W_0,W_1)$, thus the term $\bs{\Delta}_1^W\circ\bs{b}_1^V$ should be read as being $\bs{\Delta}_1^{W\subset W}\circ\bs{b}_1^{V\subset V\odot W}$.
	Then we have for $a\in CF_{+\infty}(W_0,W_1)$:
	\begin{alignat*}{1}
	&\bs{b}_1^V\circ\bs{\Delta}_1^W(a)=\bs{b}_1^V\circ\Delta_1^W(a)=\Delta_1^W(a)+b_1^V\circ\Delta_1^W(a)\\
	&\bs{\Delta}_1^W\circ\bs{b}_1^V(a)=\bs{\Delta}_1^W(a+b_1^V\circ\Delta_1^W(a))=\Delta_1^W(a)+\bs{\Delta}_1^W\circ b_1^V\circ\Delta_1^W(a)=\Delta_1^W(a)+b_1^V\circ\Delta_1^W(a)
	\end{alignat*}
	and for $a\in CF_{-\infty}(V_0,V_1)$,
	\begin{alignat*}{1}
	&\bs{b}_1^V\circ\bs{\Delta}_1^W(a)=\bs{b}_1^V(a)=b_1^V(a)\\
	&\bs{\Delta}_1^W\circ\bs{b}_1^V(a)=\bs{\Delta}_1^W\circ b_1^V(a)=b_1^V(a)
	\end{alignat*}
\end{proof}
From this, we get a short exact sequence of complexes
\begin{alignat*}{1}
0\to\Cth^*_+(V_0\odot W_0,V_1\odot W_1)\xrightarrow[]{\small(\bs{\Delta}_1^W,\bs{b}_1^V)}\Cth^*_+(V_0,V_1)\oplus\Cth^*_+(W_0,W_1)\xrightarrow[]{\small\bs{b}_1^V+\bs{\Delta}_1^W}\Cth^*_+(\R\times\La_0,\R\times\La_1)\to0
\end{alignat*}
which gives rise to a Mayer-Vietoris sequence
\begin{alignat*}{2}
\dots\to H^*\Cth_+(V_0\odot W_0,V_1\odot W_1)&\to H^*\Cth_+(V_0,V_1)\oplus H^*\Cth_+(W_0,W_1)\\
&\to H^*\Cth_+(\R\times\La_0,\R\times\La_1)\xrightarrow[]{g_*} H^{*+1}\Cth_+(V_0\odot W_0,V_1\odot W_1)\dots
\end{alignat*}
The connecting morphism $g_*$ is given on the chain level by 
$\mfm_1^{W,0}\circ\,b_1^V+\mfm_1^{V,0-}$
on $C(\La_1,\La_0)^\dagger[n-1]\subset\Cth_+(\R\times\La_0,\R\times\La_1)$ and by $\mfm_1^{W,0}$
on $C(\La_0,\La_1)[1]\subset\Cth_+(\R\times\La_0,\R\times\La_1)$. Below we check that $g$ is indeed a chain map so induces a well-defined map in homology, and that the sequence is exact.

We need to prove that $g\circ\mfm_1^{\R\times\La}=\mfm_1^{V\odot W}\circ g$. Instead of writing big matrices, let us prove it for the two types of generators separately. Consider $\gamma_{01}\in C(\La_1,\La_0)$, we have
\begin{alignat*}{1}
g\circ\mfm_1^{\R\times\La}(\gamma_{01})+&\mfm_1^{V\odot W}\circ g(\gamma_{01})\\
=&g\big(\Delta_1^\La(\gamma_{01})+b_1^\La(\gamma_{01})\big)+\big(\mfm_1^{W,+0}\circ\,\bs{b}_1^V+\mfm_1^{V,0-}\circ\bs{\Delta}_1^W\big)\circ\big(\mfm_1^{W,0}\circ\,b_1^V+\mfm_1^{V,0-}\big)(\gamma_{01})\\
=&\big(\mfm_1^{W,0}\circ\,b_1^V+\mfm_1^{V,0-}\big)\circ\Delta_1^\La(\gamma_{01})+\mfm_1^{W,0}\circ\,b_1^\La(\gamma_{01})
+\mfm_1^{W,0}\circ \mfm_1^{W,0}\circ\,b_1^V(\gamma_{01})\\
&+\mfm_1^{W,0}\circ\,b_1^V\circ\mfm_1^{V,0-}(\gamma_{01})+\mfm_1^{V,0-}\circ\mfm_1^{V,0-}(\gamma_{01})
\end{alignat*}
where we have removed terms vanishing for energy reasons. Now, observe that $\mfm_1^{V,0-}\circ\Delta_1^\La(\gamma_{01})+\mfm_1^{V,0-}\circ\mfm_1^{V,0-}(\gamma_{01})=0$ as $\Delta_1^\La(\gamma_{01})=\mfm_1^{V,+}(\gamma_{01})$ and $\mfm_1^V$ is a differential. Finally, the remaining terms are the algebraic contributions of the broken curves arising in the boundary of products of moduli spaces of the following type
\begin{alignat*}{1}
&\overline{\cM^1}_{W_0,W_1}(x,\bs{\xi}_0,\gamma_{10},\bs{\xi}_1)\times\cM_{V_0,V_1}^0(\gamma_{10},\bs{\delta}_0,\gamma_{01},\bs{\delta}_1)\\
&\cM_{W_0,W_1}^0(x,\bs{\xi}_0,\gamma_{10},\bs{\xi}_1)\times\overline{\cM^1}_{V_0,V_1}(\gamma_{10},\bs{\delta}_0,\gamma_{01},\bs{\delta}_1)
\end{alignat*}
for $x\in CF(W_0,W_1)$, $\gamma_{10}\in C(\La_0,\La_1)$, $\bs{\delta}_i$ words of pure Reeb chords of $\La_i^-$ and $\bs{\xi}_i$ words of pure Reeb chords of $\La_i$.

Then, consider $\gamma_{10}\in C(\La_0,\La_1)$, we have
\begin{alignat*}{1}
g\circ\mfm_1^{\R\times\La}(\gamma_{10})+&\mfm_1^{V\odot W}\circ g(\gamma_{10})\\
=&g\circ \,b_1^\La(\gamma_{10})+\big(\mfm_1^{W,+0}\circ\,\bs{b}_1^V+\mfm_1^{V,0-}\circ\bs{\Delta}_1^W\big)\circ\mfm_1^{W,0}(\gamma_{10})\\
=&\mfm_1^{W,0}\circ \,b_1^\La(\gamma_{10})+\mfm_1^{W,0}\circ\mfm_1^{W,0}(\gamma_{10})
\end{alignat*}
where we have removed terms vanishing for energy reason. Then the two remaining terms are algebraic contributions of the broken configurations in the boundary of the compactification of moduli spaces $\cM^1_{W_0,W_1}(x,\bs{\xi}_0,\gamma_{10},\bs{\xi}_1)$.

Now we check the exactness of
\small
\begin{alignat*}{1}
&\dots\to H^{*-1}\Cth_+(\R\times\La_0,\R\times\La_1)\xrightarrow[]{g_{*-1}} H^{*}\Cth_+(V_0\odot W_0,V_1\odot W_1)\xrightarrow[]{\small(\bs{\Delta}_1^W,\bs{b}_1^V)}H^{*}\Cth_+(V_0,V_1)\oplus H^{*}\Cth_+(W_0,W_1)\to\dots
\end{alignat*}
\normalsize
Consider a cycle $\gamma_{01}+\gamma_{10}\in\Cth_+(\R\times\La_0,\R\times\La_1)$, so in particular it means $\Delta_1^\La(\gamma_{01})=0$ and $b_1^\La(\gamma_{01})+b_1^\La(\gamma_{10})=0$. We need to check that in homology
$\bs{\Delta}_1^W\circ g_*(\gamma_{01}+\gamma_{10})=\bs{b}_1^V\circ g_*(\gamma_{01}+\gamma_{10})=0$.
We have 
\begin{alignat*}{1}
\bs{\Delta}_1^W\circ g_*(\gamma_{01}+\gamma_{10})&=\bs{\Delta}_1^W\Big(\mfm_1^{W,0}(b_1^V(\gamma_{01})+\gamma_{10})+\mfm_1^{V,0-}(\gamma_{01})\Big)\\
&=\Delta_1^W\circ\mfm_1^{W,0}\big(b_1^V(\gamma_{01})+\gamma_{10}\big)+\mfm_1^{V,0-}(\gamma_{01})\\
&=\mfm_1^{V,0-}(\gamma_{01})
\end{alignat*}
for energy reason, and then $\mfm_1^{V,0-}(\gamma_{01})=\mfm_1^V(\gamma_{01})$ because $\mfm_1^{V,+}(\gamma_{01})=\Delta_1^\La(\gamma_{01})=0$ by assumption. Thus $\bs{\Delta}_1^W\circ g_*(\gamma_{01}+\gamma_{10})\in\Cth_+(V_0,V_1)$ is a boundary so vanishes in homology. Then 
\begin{alignat*}{1}
\bs{b}_1^V\circ g_*(\gamma_{01}+\gamma_{10})&=\bs{b}_1^V\Big(\mfm_1^{W,0}\big(b_1^V(\gamma_{01})+\gamma_{10}\big)+\mfm_1^{V,0-}(\gamma_{01})\Big)\\
&=\mfm_1^{W,0}\big(b_1^V(\gamma_{01})+\gamma_{10}\big)+b_1^V\circ\Delta_1^W\circ\mfm_1^{W,0}\big(b_1^V(\gamma_{01})+\gamma_{10}\big)+b_1^V\circ\mfm_1^{V,0-}(\gamma_{01})\\
&=\mfm_1^{W,0}\big(b_1^V(\gamma_{01})+\gamma_{10}\big)+b_1^V\circ\mfm_1^{V,0-}(\gamma_{01})
\end{alignat*}
Then, by the study of index $1$ bananas with boundary on $V_0\cup V_1$ as above, one has
$b_1^V\circ\mfm_1^{V,0-}(\gamma_{01})=b_1^\La\circ b_1^V(\gamma_{01})+b_1^V\circ\Delta_1^\La(\gamma_{01})+b_1^\La(\gamma_{01})$. But $\Delta_1^\La(\gamma_{01})=0$ by assumption, as well as $b_1^\La(\gamma_{01})=b_1^\La(\gamma_{10})$ and then by definition $b_1^\La(\gamma_{10})=\mfm_1^{W,-}(\gamma_{10})$. Thus, we get
\begin{alignat*}{1}
\bs{b}_1^V\circ g_*(\gamma_{01}+\gamma_{10})&=\mfm_1^{W,0}\circ\, b_1^V(\gamma_{01})+\mfm_1^{W,0}(\gamma_{10})+b_1^\La\circ b_1^V(\gamma_{01})+\mfm_1^{W,-}(\gamma_{10})
\end{alignat*}
Finally, $b_1^\La\circ b_1^V(\gamma_{01})=\mfm_1^{W,-}\circ\,b_1^V(\gamma_{01})$ by definition, and one can add the terms $\mfm_1^{W,+}\circ\,b_1^V(\gamma_{01})$ and $\mfm_1^{W,+}(\gamma_{10})$ which vanish to obtain that $\bs{b}_1^V\circ g_*(\gamma_{01}+\gamma_{10})=\mfm_1^W\big(b_1^V(\gamma_{01})+\gamma_{10}\big)$ is a boundary in $\Cth_+(W_0,W_1)$.

Finally, let us check the exactness of
\small
\begin{alignat*}{1}
&\dots\to H^{*}\Cth_+(V_0,V_1)\oplus H^{*}\Cth_+(W_0,W_1)\xrightarrow[]{\small\bs{b}_1^V+\bs{\Delta}_1^W} H^{*}\Cth_+(\R\times\La_0,\R\times\La_1)\xrightarrow[]{g_{*}} H^{*}\Cth_+(V_0\odot W_0,V_1\odot W_1)\to\dots
\end{alignat*}
\normalsize
Given two cycles $\gamma_{01}+c_{-\infty}^V\in\Cth_+(V_0,V_1)$ and $c_{+\infty}^W+\gamma_{10}\in\Cth_+(V_0,V_1)$, where $\gamma_{01}\in C(\La_1,\La_0)$, $c_{-\infty}^V\in CF_{-\infty}(V_0,V_1)$, $c_{+\infty}^W\in CF_{+\infty}(W_0,W_1)$ and $\gamma_{10}\in C(\La_0,\La_1)$, we want to show that in homology 
\begin{alignat*}{1}
	g_*\circ\bs{b}_1^V(\gamma_{01}+c_{-\infty}^V)+g_*\circ\bs{\Delta}_1^W(c_{+\infty}^W+\gamma_{10})=0.
\end{alignat*}
We will  actually show that $g_*\circ\bs{b}_1^V(\gamma_{01}+c_{-\infty}^V)+g_*\circ\bs{\Delta}_1^W(c_{+\infty}^W+\gamma_{10})=\mfm_1^{V\odot W}(c_{-\infty}^V+c_{+\infty}^W)$.
By definition we have
\begin{alignat*}{1}
g_*\circ\bs{b}_1^V(\gamma_{01}+c_{-\infty}^V)&+g_*\circ\bs{\Delta}_1^W(c_{+\infty}^W+\gamma_{10})=g\big(\gamma_{01}+b_1^V(\gamma_{01})+b_1^V(c_{-\infty}^V)\big)+g\big(\Delta_1^W(c_{+\infty}^W)+\gamma_{10}\big)\\
&=\mfm_1^{W,0}\circ b_1^V(\gamma_{01})+\mfm_1^{V,0-}(\gamma_{01})+\mfm_1^{W,0}\circ b_1^V(\gamma_{01})+\mfm_1^{W,0}\circ b_1^V(c_{-\infty}^V)\\
&+\mfm_1^{W,0}\circ b_1^V\circ\Delta_1^W(c_{+\infty}^W)+\mfm_1^{V,0-}\circ\Delta_1^W(c_{+\infty}^W)+\mfm_1^{W,0}(\gamma_{10})\\
&=\mfm_1^{V,0-}(\gamma_{01})+\mfm_1^{W,0}\circ b_1^V(c_{-\infty}^V)\\
&+\mfm_1^{W,0}\circ b_1^V\circ\Delta_1^W(c_{+\infty}^W)+\mfm_1^{V,0-}\circ\Delta_1^W(c_{+\infty}^W)+\mfm_1^{W,0}(\gamma_{10})
\end{alignat*}
The second and fourth terms in this last sum are equal to $\mfm_1^{W,+0}\circ\bs{b}_1^V(c_{-\infty}^V)$ and $\mfm_1^{V,0-}\circ\bs{\Delta}_1^W(c_{+\infty}^W)$ respectively, by definition of $\bs{b}_1^V$ and $\bs{\Delta}_1^W$ and because $\mfm_1^{W,+}\circ b_1^V(c_{-\infty}^V)=0$ so we can add it. The remaining terms in the sum are
\begin{alignat}{1}
\mfm_1^{V,0-}(\gamma_{01})+\mfm_1^{W,0}\circ b_1^V\circ\Delta_1^W(c_{+\infty}^W)+\mfm_1^{W,0}(\gamma_{10})\label{sumMV}
\end{alignat}
But by assumption, $\mfm_1^{V,0-}(\gamma_{01})+\mfm_1^{V,0-}(c_{-\infty}^V)=0$ because $\gamma_{01}+c_{-\infty}^V$ is a cycle, and similarly using also the fact that $\mfm_1^{W,+}(\gamma_{10})=0$, we have $\mfm_1^{W,+0}(c_{+\infty}^W)+\mfm_1^{W,+0}(\gamma_{10})=0$. Thus the sum \eqref{sumMV} is equal to
\begin{alignat*}{1}
\mfm_1^{V,0-}(c_{-\infty}^V)+\mfm_1^{W,0}\circ b_1^V\circ\Delta_1^W(c_{+\infty}^W)+\mfm_1^{W,+0}(c_{+\infty}^W)=\mfm_1^{V,0-}\circ\bs{\Delta}_1^W(c_{-\infty}^V)+\mfm_1^{W,+0}\circ\bs{b}_1^V(c_{+\infty}^W)
\end{alignat*}
where the last equality holds by definition of $\bs{b}_1^V$ and $\bs{\Delta}_1^W$ and because $\mfm_1^{W,+}\circ b_1^V\circ\Delta_1^W(c_{+\infty}^W)=0$. So we get 
\begin{alignat*}{1}
g_*\circ\bs{b}_1^V(\gamma_{01}+c_{-\infty}^V)+g_*\circ\bs{\Delta}_1^W(c_{+\infty}^W+\gamma_{10})&=\mfm_1^{W,+0}\circ\bs{b}_1^V(c_{-\infty}^V)+\mfm_1^{V,0-}\circ\bs{\Delta}_1^W(c_{+\infty}^W)\\
&+\mfm_1^{V,0-}\circ\bs{\Delta}_1^W(c_{-\infty}^V)+\mfm_1^{W,+0}\circ\bs{b}_1^V(c_{+\infty}^W)\\
&=\mfm_1^{V\odot W}(c_{-\infty}^V+c_{+\infty}^W)
\end{alignat*}
which completes the proof of exactness.

\section{Acyclicity for horizontally displaceable Legendrian ends}\label{sec:acy}

The acyclicity of the complex $\Cth_+(\Sigma_0,\Sigma_1)$ is proved in the same way as the acyclicity of the complex $\Cth(\Sigma_0,\Sigma_1)$ in \cite{CDGG2}. However, in the case of $\Cth_+$ we need some horizontal displaceability assumption of at least one of the two Legendrian ends to achieve acyclicity.
\begin{defi}
	Two Legendrian submanifolds $\La_0,\La_1\subset Y=P\times\R$ are \textit{horizontally displaceable} if there exists an Hamiltonian isotopy $\varphi_t$ of $P$ which displace the Lagrangian projections $\Pi_P(\La_0)$ and $\Pi_P(\La_1)$, i.e. $\Pi_P(\La_0)$ and $\varphi_1(\Pi_P(\La_1))$ are contained in two disjoint balls.
	A Legendrian is called \textit{horizontally displaceable} if it can be displaced from itself.	
\end{defi}

The goal of the next subsections is to prove the following:
\begin{teo}\label{teo:acyclicity}
	Let $\La_0^-,\La_1^-,\La_0^+,\La_1^+\subset Y$ be Legendrian submanifolds such that $\La_0^-$ and $\La_1^-$, or $\La_0^+$ and $\La_1^+$ are horizontally displaceable. Assume moreover that $\Ac(\La_0^-)$ and $\Ac(\La_1^-)$ admit augmentations $\ep_0^-$ and $\ep_1^-$. Then, for any pair of transverse exact Lagrangian cobordisms $\La_0^-\prec_{\Sigma_0}\La_0^+$ and $\La_1^-\prec_{\Sigma_1}\La_1^+$, the complex $\big(\Cth_+(\Sigma_0,\Sigma_1),\mfm_1^{\ep_0^-,\ep_1^-}\big)$ is acyclic.
\end{teo}
The hypothesis of horizontal displaceability is necessary. Indeed, in the setting of Example \ref{ex:jet} the $0$-section of a jet space is not horizontally displaceable, and in fact the complex is not acyclic.\\

When $\Cth_+(\Sigma_0,\Sigma_1)$ is acyclic, one recovers long exact sequences obtained in \cite{CDGG2}. Indeed, the complex $\Cth_+(\Sigma_0,\Sigma_1)$ is the cone of the degree $1$ map $d_{0+}+b_1^-\circ\Delta_1^\Sigma:C(\La_1^+,\La_0^+)^\dagger[n-1]\to CF_{-\infty}(\Sigma_0,\Sigma_1)$, which is then a quasi-isomorphism since the complex is acyclic, i.e. we have
\begin{alignat}{1}\label{eq:qiso}
	H^{*}\big(C(\La_1^+,\La_0^+)^\dagger[n-1]\big)\simeq HF_{-\infty}^{*+1}(\Sigma_0,\Sigma_1)
\end{alignat}
Assume first that the Legendrian submanifolds $\La_0^+$ and $\La_1^+$ are horizontally displaceable. Then, the acyclicity of $\Cth_+(\R\times\La_0^+,\R\times\La_1^+)$ yields
\begin{alignat}{1}\label{eq:iso2}
	H^*\big(C(\La_1^+,\La_0^+)^\dagger[n-1]\big)\simeq H^*(C(\La_0^+,\La_1^+))
\end{alignat}
as there are no intersection point generators, and the Legendrians in the negative end are also $\La_0^+$ and $\La_1^+$.
When $(\Sigma_0,\Sigma_1)$ is a \textit{directed} pair, then $d_{0-}=0$ and $CF^{*}_{-\infty}(\Sigma_0,\Sigma_1)$ is the cone of $d_{-0}:CF^{*}(\Sigma_0,\Sigma_1)\to C^{*}(\La_0^-,\La_1^-)$. When $(\Sigma_0,\Sigma_1)$ is a \textit{V-shaped} pair, then $d_{-0}=0$ and $CF^{*}_{-\infty}(\Sigma_0,\Sigma_1)$ is the cone of $d_{0-}:C^{*-1}(\La_0^-,\La_1^-)\to CF^{*+1}(\Sigma_0,\Sigma_1)$. The long exact sequence of a cone, together with the isomorphisms \eqref{eq:qiso} and \eqref{eq:iso2}, and the fact that by definition $H^*(C(\La_0^+,\La_1^+))= LCH^*_{\ep_0^+,\ep_1^+}(\La_0^+,\La_1^+)$ and $H^{*}(C(\La_0^-,\La_1^-))=LCH^{*}_{\ep_0^-,\ep_1^-}(\La_0^-,\La_1^-)$, give
$$\begin{array}{rcl}
\dots\to LCH^{k-1}_{\ep_0^+,\ep_1^+}(\La_0^+,\La_1^+)\to& HF^{k}(\Sigma_0,\Sigma_1)& \\
&\downarrow&\\
& LCH^{k}_{\ep_0^-,\ep_1^-}(\La_0^-,\La_1^-)&\to LCH^{k}_{\ep_0^+,\ep_1^+}(\La_0^+,\La_1^+)\to\dots
\end{array}$$
for a directed pair, and
$$\begin{array}{rcl}
\dots\to LCH^{k-1}_{\ep_0^+,\ep_1^+}(\La_0^+,\La_1^+)\to&LCH^{k-1}_{\ep_0^-,\ep_1^-}(\La_0^-,\La_1^-)& \\
&\downarrow&\\
&  HF^{k+1}(\Sigma_0,\Sigma_1)&\to LCH^{k}_{\ep_0^+,\ep_1^+}(\La_0^+,\La_1^+)\to\dots
\end{array}$$
for a V-shaped pair.
These are the long exact sequences in \cite[Corollary 1.3]{CDGG2}.\\

In the case where $\La_0^+$ and $\La_1^+$ are not horizontally displaceable but $\La_0^-$ and $\La_1^-$ are, one gets the same exact sequences from the acyclicity of the dual complex $\Cth_+^{dual}(\Sigma_0,\Sigma_1)$:
\begin{alignat*}{1}
	\Cth_+^{dual}(\Sigma_0,\Sigma_1)=C^*(\La_1^+,\La_0^+)^\dagger[n-1]\oplus CF_*(\Sigma_0,\Sigma_1)\oplus C_*(\La_0^-,\La_1^-)
\end{alignat*}
with the degree $-1$ differential 
$$\left(\begin{matrix}
	b_1^+&b_1^{\Sigma_1,\Sigma_0}&b_1^{\Sigma_1,\Sigma_0}\circ b_1^-\\
	0&d^{\Sigma_1,\Sigma_0}_{00}&d^{\Sigma_1,\Sigma_0}_{0-}\\
	0&\Delta_1^{\Sigma_1,\Sigma_0}&\Delta_1^-
\end{matrix}\right)$$
However as we do not especially need this dual complex in this article we will not give more details here.

\subsection{Wrapping the ends}\label{wrapping}
Given a pair of cobordisms $(\Sigma_0,\Sigma_1)$ cylindrical outside $[-T,T]\times Y$, we will wrap the positive and negative ends of $\Sigma_1$ in order to get a pair of cobordisms such that the associated $\Cth_+$ complex has only intersection points generators. The wrapping is done by Hamiltonian isotopy. A smooth function $h:\R\to\R$ gives rise to a Hamiltonian $H:\R\times P\times\R\to\R$ defined by $H(t,p,z)=h(t)$. The corresponding Hamiltonian vector field $X_h$ is defined through the equation $d(e^t\alpha)(X_h,\,\cdot\,)=-dH$, and its Hamiltonian flow $\varphi^s_h$ takes the following simple form
\begin{alignat*}{1}
\varphi^s_h(t,p,z)=(t,p,z+se^{-t}h'(t))
\end{alignat*}
Moreover, the image of an exact Lagrangian cobordism $\Sigma$ with primitive $f_\Sigma$ by an Hamiltonian isotopy $\varphi_h^s$ as above is still an exact Lagrangian cobordism $\Sigma_s=\varphi_h^s(\Sigma)$, with a primitive $f_{\Sigma_s}$ given by
	\begin{alignat}{1}
		f_{\Sigma_s}&=f_\Sigma+s(h'-h)\circ\pi_\R
	\end{alignat}
where $\pi_\R:\R\times Y\to\R$ is the projection on the symplectization coordinate $t$.
Given $N>T$, consider a function $h_{T,N}^+:\R\to\R$ satisfying:
$$\left\{\begin{array}{l}
h_{T,N}^+(t)=0\,\mbox{ for }\,t\leq T+N,\\
h_{T,N}^+(t)=-e^t\,\mbox{ for }\,t\geq T+N+1,\\
(h_{T,N}^+)'(t)\leq0,\\
\end{array}\right.$$
such that the Hamiltonian vector field takes the form $\rho_{T,N}^+(t)\partial_z$ where $\rho_{T,N}^+:\R\to\R$ satisfies
$$\left\{\begin{array}{l}
\rho_{T,N}^+(t)=0\mbox{ for }t< T+N,\\
\rho_{T,N}^+(t)=-1\mbox{ for }t> T+N+1,\\
(\rho_{T,N}^+)'(t)\leq 0.\\
\end{array}\right.$$

\begin{figure}[ht]  
	\begin{center}\includegraphics[width=5cm]{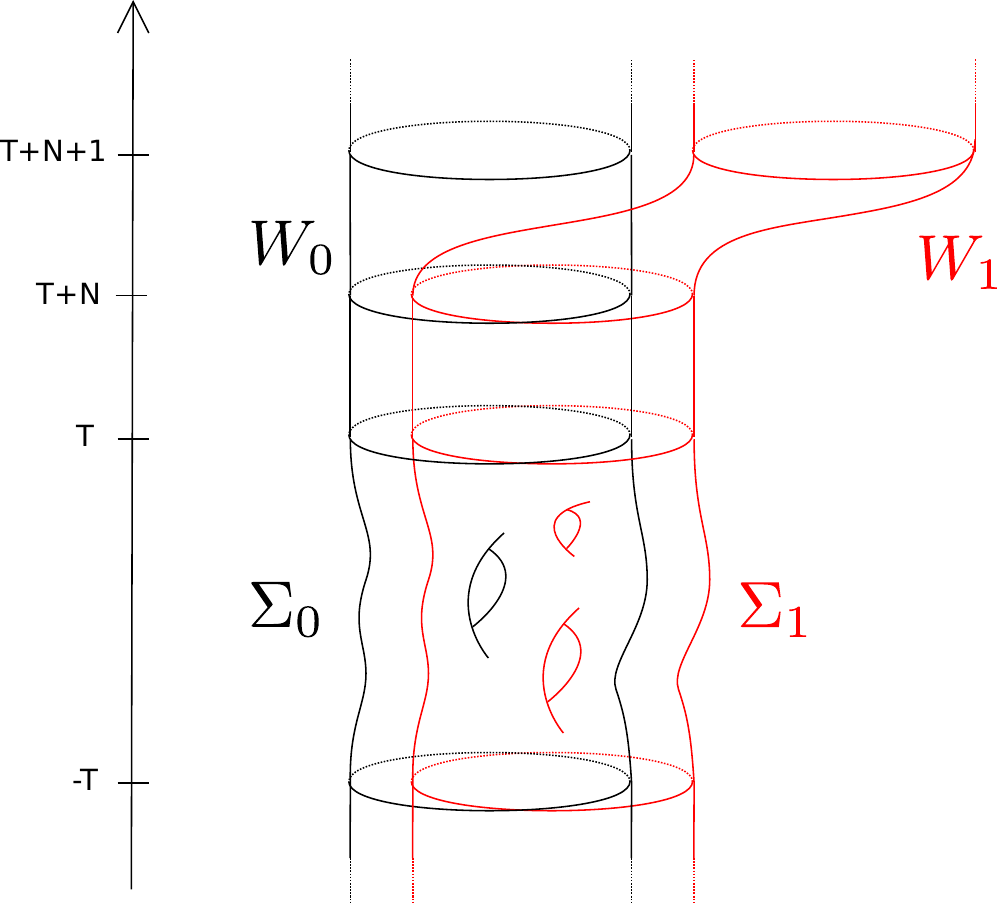}\end{center}
	\caption{Wrapping the positive end of $\Sigma_1$.}
	\label{fig:wrapping+}
\end{figure}
Let $S_+>0$ greater than the length of the longest Reeb chord from $\La_0^+$ to $\La_1^+$. We set $W_1:=\varphi_{h_{T,N}^+}^{S_+}(\R\times\La_1^+)$, and $W_0:=\R\times\La_0^+$, and consider the pair $(\Sigma_0\odot W_0,\Sigma_1\odot W_1)$, where in fact $\Sigma_0\odot W_0=\Sigma_0$, see Figure \ref{fig:wrapping+}. The complex
$\Cth_+(\Sigma_0\odot W_0,\Sigma_1\odot W_1)$ has only three types of generators, namely
$$\Cth_+(\Sigma_0\odot W_0,\Sigma_1\odot W_1)=CF(W_0,W_1)\oplus CF(\Sigma_0,\Sigma_1)\oplus C(\La_0^-,\La_1^-)[1]$$
Under this decomposition, the transfer map $\bs{\Delta}_1^W:\Cth_+(\Sigma_0\odot W_0,\Sigma_1\odot W_1)\to\Cth_+(\Sigma_0,\Sigma_1)$ is equal to the matrix
\begin{alignat*}{1}
\bs{\Delta}_1^W=\left(\begin{matrix}
\Delta_1^W&0&0\\
0&\id&0\\
0&0&\id\\
\end{matrix}\right)	
\end{alignat*}
We have then the following
\begin{prop}\label{prop:transfer}
	The transfer map $\bs{\Delta}_1^W:\Cth_+(\Sigma_0\odot W_0,\Sigma_1\odot W_1)\to\Cth_+(\Sigma_0,\Sigma_1)$ is an isomorphism.
\end{prop}

\begin{proof}
The proof is the same as the proof of \cite[Proposition 8.2]{CDGG2}. After wrapping, each Reeb chord from $\La_0^+$ to $\La_1^+$ creates an intersection point in $W_0\cap W_1$, and  observing that the wrapping in the positive end makes the Conley-Zehnder index increasing by $1$, there is a canonical identification of graded modules:
\begin{alignat}{1}
	CF^*(W_0,W_1)=C^*(\La_1^+,\La_0^+)^\dagger[n-1]\subset\Cth_+(\Sigma_0,\Sigma_1)\label{ident2}
\end{alignat}
If $p\in CF(W_0,W_1)$ we denote by $\gamma_p\in C^*(\La_1^+,\La_0^+)^\dagger[n-1]$ the corresponding Reeb chord.
The goal is to prove that this identification also applies at the level of complexes. We will show that under the identification \eqref{ident2}, the map $\bs{\Delta}_1^W$ is the identity map.

We consider the component 
$\Delta_1^{W}:CF^*(W_0,W_1)\to C_{n-1-*}(\La_1^+,\La_0^+)$ of $\bs{\Delta}_1^W$ which is of degree $0$. Let $u\in\cM^0_{W_0,W_1}(\gamma_{01};\bs{\delta}_0,p,\bs{\delta}_1)$ where $p\in W_0\cap W_1$, $\gamma_{01}\in\Rc(\La_1^+,\La_0^+)$ is a negative Reeb chord asymptotic, and $\bs{\delta}_i$ are words of degree $0$ pure Reeb chords which are also negative asymptotics. This disc contributes to $\Delta_1^{W}(p)$. By rigidity of $u$, we have 
\begin{alignat*}{1}
	n-1-|\gamma_{01}|-|p|_{\Cth_+(W_0,W_1)}=0
\end{alignat*}
Now, the projection of $u$ to $P$ is a pseudo-holomorphic map in $\cM_{\pi_P(\La_0^+),\pi_P(\La_1^+)}(\gamma_{01};\bs{\delta}_0,\pi_P(p),\bs{\delta}_1)$ which has dimension $|\pi_P(p)|-|\gamma_{01}|-1=|\gamma_p|-|\gamma_{01}|-1$, but we have
\begin{alignat*}{1}
	0=n-1-|\gamma_{01}|-|p|_{\Cth_+(W_0,W_1)}=n-1-|\gamma_{01}|-(n-1-|\gamma_p|)=|\gamma_p|-|\gamma_{01}|
\end{alignat*}
where we have used the identification \eqref{ident2}.
This implies that $\pi_P(u)$ is in a moduli space of dimension $-1$ so it must be constant. Hence, $\gamma_{01}=\gamma_p$. On the other side, for each intersection point $p\in W_0\cap W_1$ a strip over $\gamma_p$ lifts to a disk in $\cM^0_{W_0,W_1}(\gamma_p;\bs{\delta}_0,p,\bs{\delta}_1)$. We obtain that $\bs{\Delta}_1^W$ is the identity map.
\end{proof}

Next, we wrap the negative end of $\Sigma_1\odot W_1$ as schematized on Figure \ref{fig:wrapping-}, using a Hamiltonian defined by a function $h_{T,N}^-:\R\to\R$ satisfying:
$$\left\{\begin{array}{l}
h_{T,N}^-(t)=e^t\,\mbox{ for }\,t<-T-N-1,\\
h_{T,N}^-(t)=D\,\mbox{ for }\,t> -T-N,\\
(h_{T,N}^-)'(t)\geq0,\\
\end{array}\right.$$
for some positive constant $D\geq e^{-T-N}$, such that the Hamiltonian vector field is given by $\rho_{T,N}^-(t)\partial_z$ where $\rho_{T,N}^-:\R\to\R$ satisfies $\rho_{T,N}^-(t)=1$ for $t\leq -T-N$, $\rho_{T,N}^-(t)=0$ for $t\geq -T$, and $(\rho_{T,N}^-)'(t)\leq 0$.
Let $S_->0$ be greater than the length of the longest chord from $\La_1^-$ to $\La_0^-$ and define $V_1:=\varphi_{h_{T,N}^-}^{S_-}(\R\times\La_1^-)$ and set $V_0:=\R\times\La_0^-$. After concatenation, we obtain a pair $(V_0\odot\Sigma_0\odot W_0,V_1\odot\Sigma_1\odot W_1)=(\Sigma_0,V_1\odot\Sigma_1\odot W_1)$.
\begin{figure}[ht]  
	\begin{center}\includegraphics[width=5cm]{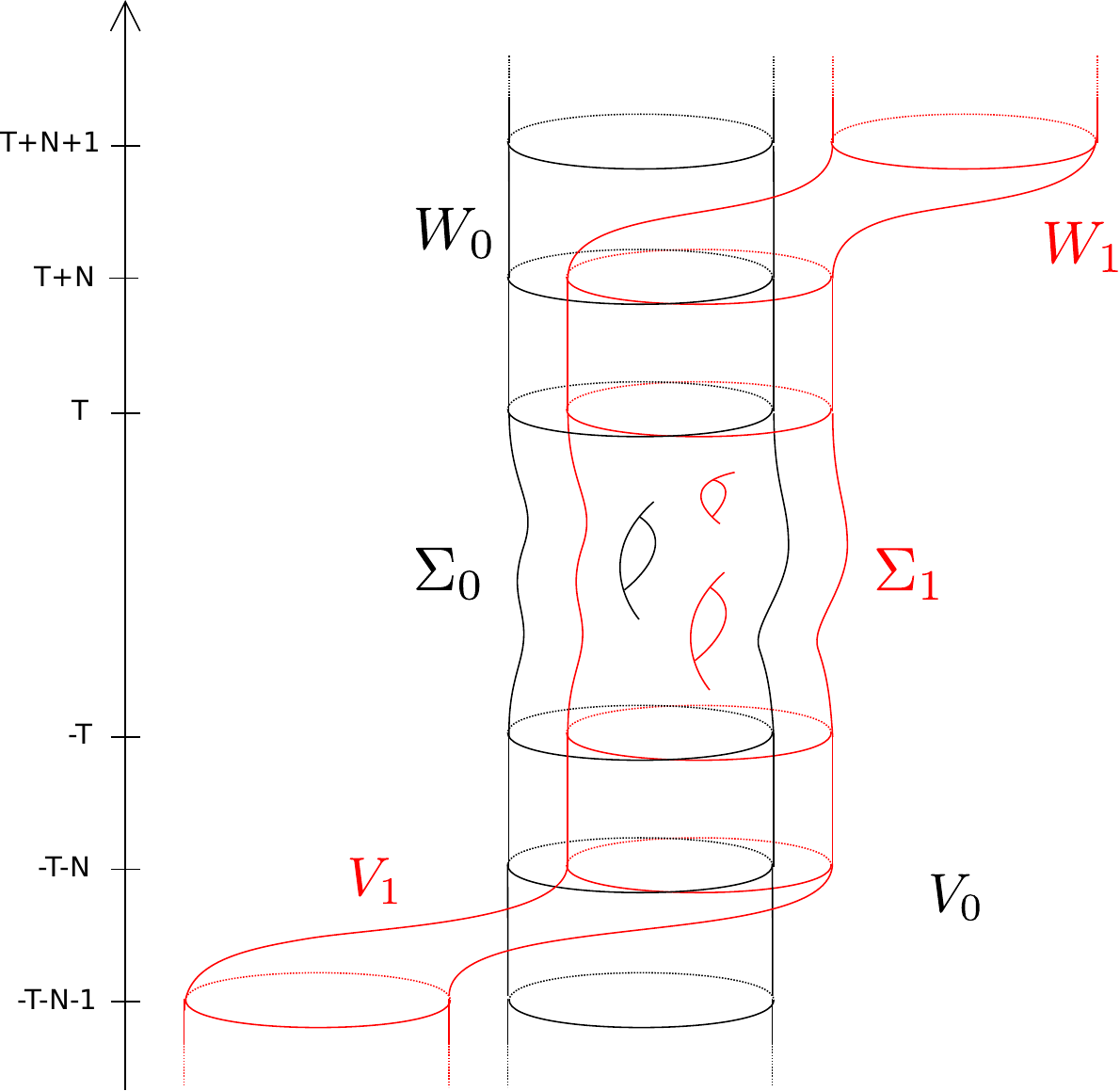}\end{center}
	\caption{Wrapping the negative end of $\Sigma_1\odot W_1$.}
	\label{fig:wrapping-}
\end{figure}
The Cthulhu complex of the pair $\big(\Sigma_0,V_1\odot(\Sigma_1\odot W_1)\big)$ has only intersection points generators, we have
$$\Cth_+\big(\Sigma_0,V_1\odot(\Sigma_1\odot W_1)\big)=CF(\Sigma_0,\Sigma_1\odot W_1)\oplus CF(V_0,V_1)$$
Under this decomposition, the map $\bs{b}_1^V:\Cth_+\big(\Sigma_0,V_1\odot(\Sigma_1\odot W_1)\big)\to\Cth_+(\Sigma_0,\Sigma_1\odot W_1)$
is given by
\begin{alignat*}{1}
\bs{b}_1^V=\left(\begin{matrix}
\id&0\\
b_1^V\circ\Delta_1^{\Sigma\odot W}&b_1^V\\
\end{matrix}\right)	
\end{alignat*}
\begin{prop}\label{prop:transfer2}
	The map $\bs{b}_1^V$ above is an isomorphism.
\end{prop}
\begin{proof}
It is the same kind of proof as for Proposition \ref{prop:transfer}. In this case we have a canonical identification:
\begin{alignat}{1}
		CF(V_0,V_1)=C(\La_0^-,\La_1^-)[1]\subset\Cth_+(\Sigma_0,\Sigma_1\odot W_1)\label{ident}
\end{alignat}
Let us consider the component $b_1^V:CF^*(V_0,V_1)\to C^{*-1}(\La_0^-,\La_1^-)$ of $\bs{b}_1^V$ which is of degree $0$. Let $u\in\cM^0_{V_0,V_1}(\gamma_{10};\bs{\delta}_0,p,\bs{\delta}_1)$ where $p\in V_0\cap V_1$, $\gamma_{10}\in\Rc(\La_0^-,\La_1^-)$ is a positive Reeb chord asymptotic, and $\bs{\delta}_i$ are words of degree $0$ pure Reeb chords which are also negative asymptotics, contributing to this component. By rigidity, we have
\begin{alignat*}{1}
		(|\gamma_{10}|+1)-|p|_{CF(V_0,V_1)}=0
\end{alignat*}
The projection of $u$ to $P$ is a pseudo-holomorphic map in $\cM_{\pi_P(\La_0^-),\pi_P(\La_1^-)}(\gamma_{10};\bs{\delta}_0,\pi_P(p),\bs{\delta}_1)$ which has dimension $|\gamma_{10}|-|\gamma_p|-1$.
Using the identification \eqref{ident} we have
$$0=|\gamma_{10}|-|p|_{CF(V_0,V_1)}+1=|\gamma_{10}|-(|\gamma_p|+1)+1=|\gamma_{10}|-|\gamma_p|$$
So the disk $\pi_P(u)$ must be constant and $\gamma_{10}=\gamma_p$.
\end{proof}

\subsection{Invariance by compactly supported Hamiltonian isotopy}

Let us consider a pair $(\Sigma_0,\Sigma_1)$ of exact Lagrangian cobordisms and a path of exact Lagrangian cobordisms $\Sigma_0^s$ for $s\in[0,1]$ induced by a compactly supported Hamiltonian isotopy, with $\Sigma_0^0:=\Sigma_0$. In particular, for all $s\in[0,1]$, $\Sigma_0^s$ have positive and negative cylindrical ends over $\La_0^\pm$.
\begin{prop}\label{prop:iso}
	The complexes $\big(\Cth_+(\Sigma_0^0,\Sigma_1),\mfm_1^{\ep_0^-,\ep_1^-}\big)$ and $\big(\Cth(\Sigma_0^1,\Sigma_1),\mfm_1^{\ep_0^-,\ep_1^-}\big)$ are homotopy equivalent.
\end{prop}
\begin{proof}
	First, wrap the positive and negative ends of $\Sigma_1$ in the negative and positive Reeb direction respectively, as done in the previous section. One gets the pair of cobordisms $(\Sigma_0,V_1\odot\Sigma_1\odot W_1)$, whose Cthulhu complex is isomorphic to that of the pair $(\Sigma_0,\Sigma_1)$ by Propositions \ref{prop:transfer} and \ref{prop:transfer2}. Then, all along the isotopy the complex $(\Sigma_0^s,V_1\odot\Sigma_1\odot W_1)$ as only intersection point generators and the bifurcation analysis explained in \cite[Proposition 8.4]{CDGG2long} (see also \cite{E1} for the case of fillings) proves that the complexes $\Cth_+(\Sigma_0^0,V_1\odot\Sigma_1\odot W_1)$ and $\Cth_+(\Sigma_0^1,V_1\odot\Sigma_1\odot W_1)$ are homotopy equivalent. Finally, unwrapping the ends of $\Sigma_1$ leads again to an isomorphism of complexes.
\end{proof}

\subsection{Proof of Theorem \ref{teo:acyclicity}}

Consider a pair of Lagrangian cobordisms $(\Sigma_0,\Sigma_1)$ satisfying the hypothesis of the Theorem. We assume without loss of generality that $\La_0^-$ and $\La_1^-$ are horizontally displaceable (in the case $\La_i^+$ are horizontally displaceable but $\La_i^-$ are not, the same type of argument works but moving the wrapping in the positive end instead of the negative end, see below). By wrapping the cylindrical ends of $\Sigma_1$ we get the pair $(\Sigma_0,V_1\odot\Sigma_1\odot W_1)$ such that:
\begin{enumerate}
	\item $\Sigma_0$ and $V_1\odot\Sigma_1\odot W_1$ are cylindrical outside $[-T-N,T+N]\times Y$.
	\item $\Cth_+(\Sigma_0,V_1\odot\Sigma_1\odot W_1)$ has only intersection points generators.
\end{enumerate}
By a Hamiltonian isotopy $\varphi^1_{h_c}$ compactly supported in $[-T-N,T+N]\times Y$, we perturb $V_1\odot\Sigma_1\odot W_1$ in such a way that all the intersection points are in fact contained in $[-T-N,-T]\times Y$, and are in bijective correspondence with mixed chords of $\La_0^-\cup\La_1^-$, as schematized on Figure \ref{fig:compact_ham_iso}. For this purpose we use for example the Hamiltonian $H_c(t,p,z)=h_c(t)$ with $h_c:\R\to\R$ satisfying
$$\left\{\begin{array}{l}
 h_c(t)=-e^t+C\mbox{ for }t\in[-T,T],\\
 (-\infty,-T-N)\cup(T+N,\infty)\subset(h_c')^{-1}(0)\\
 h_c'(t)\leq0
\end{array}\right.$$
with $C>0$ constant such that $h_c(t)=0$ for $t\leq-T-N$, to ensure the primitive of the perturbed cobordism to still vanish on the negative cylindrical end. The Hamiltonian vector field is given by $\rho_c(t)\partial_z$ with $\rho_c(t)=-1$ on $[-T,T]$ and $0$ on $(-\infty,-T-N)\cup(T+N,\infty)$.
\begin{figure}[ht]  
	\begin{center}\includegraphics[width=5cm]{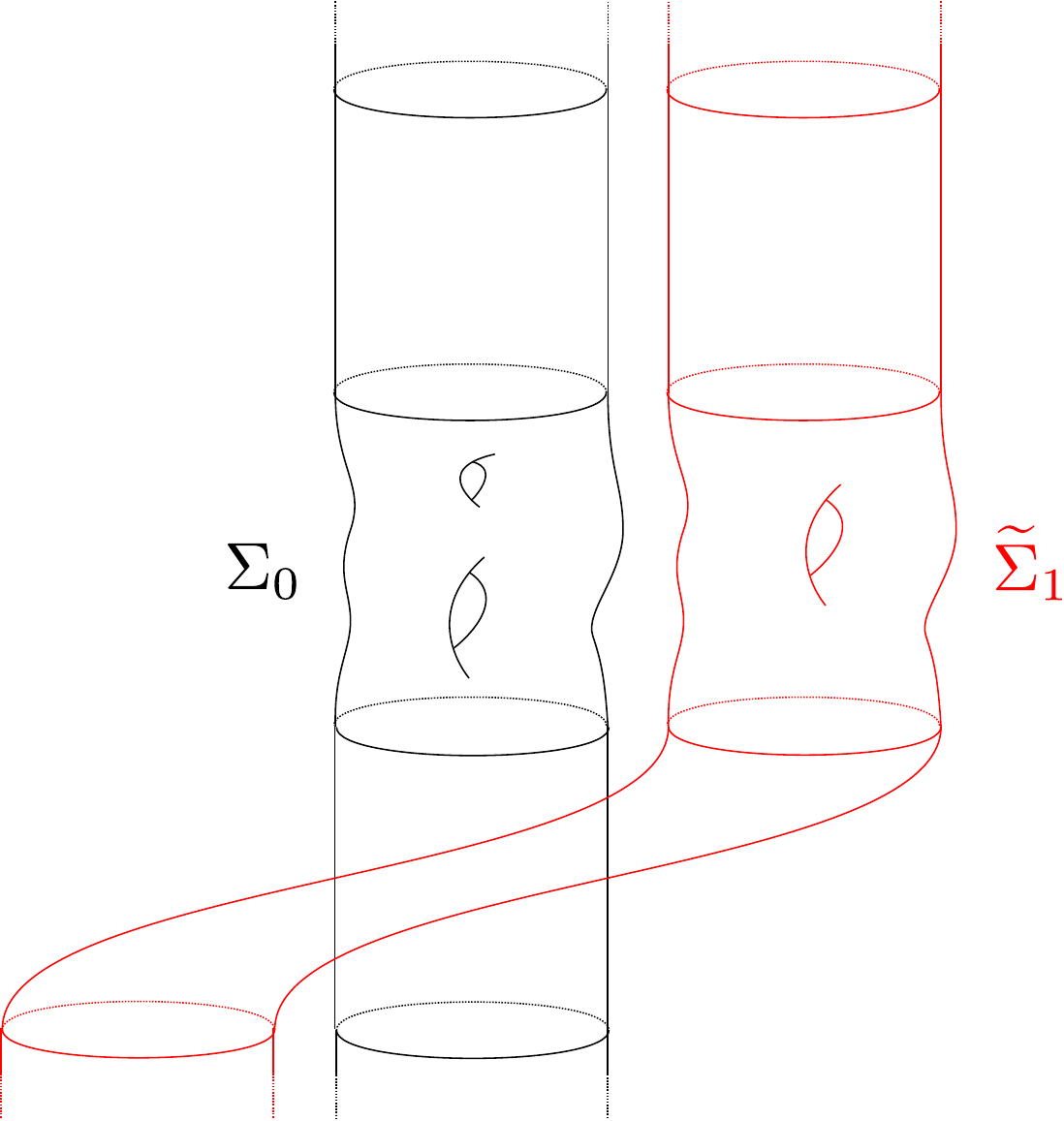}\end{center}
	\caption{Deformation by a compactly supported Hamiltonian isotopy.}
	\label{fig:compact_ham_iso}
\end{figure}
Let us denote $\widetilde{\Sigma}_1=\varphi_{h_c}^S(V_1\odot\Sigma_1\odot W_1)$, with $S$ big enough so that there are no intersection points in $[-T,T+N]\times Y$ anymore. This $S$ exists as $\Sigma_0\cap[-T-N,T+N]\times Y$ and  $V_1\odot\Sigma_1\odot W_1\cap[-T-N,T+N]\times Y$ are compact.
By Proposition \ref{prop:iso}, the complexes $\Cth_+(\Sigma_0,\Sigma_1)$ and $\Cth_+(\Sigma_0,\widetilde{\Sigma}_1)$ have the same homology. Now we prove that $\Cth_+(\Sigma_0,\widetilde{\Sigma}_1)$ is acyclic.
Given the Hamiltonian we used to perturb $V_1\odot\Sigma_1\odot W_1$, we have the canonical identification:
\begin{alignat*}{1}
	\Cth_+(\Sigma_0,\widetilde{\Sigma}_1)=\Cth_+(\R\times\La_0^-,\varphi_{h_c}^S(V_1))
\end{alignat*}
Then, we unwrap the negative end of $\varphi_{h_c}^S(V_1)$, and thus $\Cth_+(\R\times\La_0^-,\varphi_{h_c}^S(V_1))$ is isomorphic to $\Cth_+(\R\times\La_0^-,\R\times\widetilde{\La}_1^-)$ where $\widetilde{\La}_1^-$ is a translation of $\La_1^-$ in the negative Reeb direction and lies entirely below $\La_0^-$, see Figure \ref{concordances}.
In this case, we have $\Cth_+(\R\times\La_0^-,\R\times\widetilde{\La}_1^-)=C(\La_0^-,\widetilde{\La}_1^-)[1]$ with the differential $b_1^-$ being the Legendrian contact cohomology differential bilinearized by $\ep_0^-$ and $\ep_1^-$. But as the pair $(\La_0^-,\La_1^-)$ is a pair of horizontally displaceable Legendrians, so this complex is acyclic (observe that $\pi_P(\La_1^-)=\pi_P(\widetilde{\La}_1^-)$).
	

\begin{figure}[ht]  
	\begin{center}\includegraphics[width=9cm]{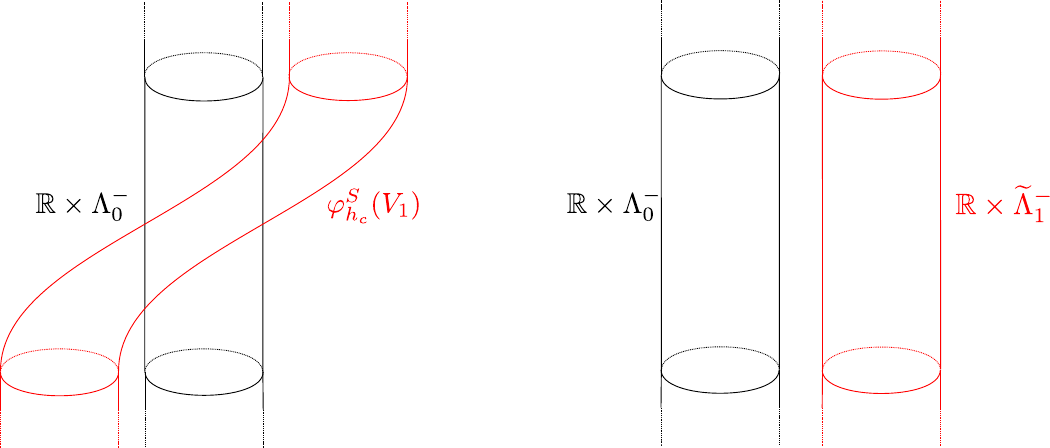}\end{center}
	\caption{Left: pair of concordances $\big(\R\times\La_0^-,\varphi_{h_c}^S(V_1)\big)$; right: pair $(\R\times\La_0^-,\R\times\widetilde{\La}_1^-)$.}
	\label{concordances}
\end{figure}

\section{Product structure}

\subsection{Definition of the map}\label{section:def_product}

Given $\La_i^-\prec_{\Sigma_i}\La_i^+$ for $i=0,1,2$ three exact Lagrangian cobordisms that are pairwise transverse such that $\Ac(\La_i^-)$ admit augmentations $\ep_0^-$, $\ep_1^-$ and $\ep_2^-$, we will define a map
\begin{alignat*}{1}
\mfm_2:\Cth_+(\Sigma_1,\Sigma_2)\otimes\Cth_+(\Sigma_0,\Sigma_1)\to\Cth_+(\Sigma_0,\Sigma_2)
\end{alignat*}
and we prove that it satisfies the Leibniz rule.
Let us denote the components of the product $\mfm_2$ by $\mfm_{ij}^k$, with $i,j,k\in\{+,0,-\}$ such that $\mfm_{ij}^k$ takes as arguments a generator of type $i$ in $\Cth_+(\Sigma_1,\Sigma_2)$, a generator of type $j$ in $\Cth_+(\Sigma_0,\Sigma_1)$ and has for output a generator of type $k$ in $\Cth_+(\Sigma_0,\Sigma_2)$. For example, $\mfm_{+-}^0$ is the component $C(\La_2^+,\La_1^+)^\dagger[n-1]\otimes C(\La_0^-,\La_1^-)[1]\to CF(\Sigma_0,\Sigma_2)$. We define $\mfm_2$ as follows. First, the eight components corresponding to the map
\begin{alignat*}{1}
	CF_{-\infty}(\Sigma_1,\Sigma_2)\otimes CF_{-\infty}(\Sigma_0,\Sigma_1)\to CF_{-\infty}(\Sigma_0,\Sigma_2)
\end{alignat*}
are the same components as those defining the product $\mfm_2^{-\infty}$ in \cite{L}, we start by recalling its definition (see also Figure \ref{m_infini}). For a pair of asymptotics $(a_2,a_1)$ which is equal to $(x_{12},x_{01}),(x_{12},\gamma_{10}),(\gamma_{21},x_{01})$ or $(\gamma_{21},\gamma_{10})$ in $CF_{-\infty}(\Sigma_1,\Sigma_2)\otimes CF_{-\infty}(\Sigma_0,\Sigma_1)$ ($x_{ij}$ is an intersection point in $\Sigma_i\cap\Sigma_j$ and $\gamma_{ij}$ is a chord from $\La_i^-$ to $\La_j^-$) we have
\begin{alignat*}{1}
&\mfm_2^{0}(a_2,a_1)=\sum\limits_{p_{20}\in\Sigma_0\cap\Sigma_2,\bs{\delta}_i}\#\cM^0_{\Sigma_{012}}(p_{20};\bs{\delta}_0,a_1,\bs{\delta}_1,a_2,\bs{\delta}_2)\ep_0^-(\bs{\delta}_0)\ep_1^-(\bs{\delta}_1)\ep_2^-(\bs{\delta}_2)\cdot p_{20}
\end{alignat*}
where the sum is over all intersection points $p_{20}\in\Sigma_0\cap\Sigma_2$ and words $\bs{\delta}_i$ of pure Reeb chords of $\La_i^-$. Then, for a pair $(x_{12},x_{01})\in CF(\Sigma_1,\Sigma_2)\otimes CF(\Sigma_0,\Sigma_1)$, we have
\small
\begin{alignat*}{1}
&\mfm_{00}^{-}(x_{12},x_{01})=\sum\limits_{\substack{\gamma_{20},\gamma_{02}\\\bs{\delta}_i,\bs{\delta}_i'}}\#\widetilde{\cM^1}_{\R\times\La^-_{012}}(\gamma_{20};\bs{\delta}_0,\gamma_{02},\bs{\delta}_2)\#\cM^0_{\Sigma_{012}}(\gamma_{02};\bs{\delta}_0',x_{01},\bs{\delta}_1',x_{12},\bs{\delta}_2')\cdot\ep_i^-(\bs{\delta}_i\bs{\delta}_i')\cdot \gamma_{20}\\
&+\sum\limits_{\substack{\gamma_{20},\gamma_{01},\gamma_{12}\\\bs{\delta}_i,\bs{\delta}_i',\bs{\delta}_i''}}\#\widetilde{\cM^1}_{\R\times\La^-_{012}}(\gamma_{20};\bs{\delta}_0,\gamma_{01},\bs{\delta}_1,\gamma_{12},\bs{\delta}_2)\#\cM^0_{\Sigma_{01}}(\gamma_{01};\bs{\delta}_0',x_{01},\bs{\delta}_1')\#\cM^0_{\Sigma_{12}}(\gamma_{12};\bs{\delta}_1'',x_{12},\bs{\delta}_2'')\cdot\ep_i^-(\bs{\delta}_i\bs{\delta}_i'\bs{\delta}_i'')\cdot \gamma_{20}
\end{alignat*}
\normalsize
where $\ep_i^-(\bs{\delta}_i)$ stands for the product of the augmentations applied to the corresponding pure chords. For a pair $(x_{12},\gamma_{10})\in CF(\Sigma_1,\Sigma_2)\otimes C(\La_0^-,\La_1^-)$, we have
\small
\begin{alignat*}{1}
&\mfm_{0-}^{-}(x_{12},\gamma_{10})=\sum\limits_{\substack{\gamma_{20},\gamma_{02}\\\bs{\delta}_i,\bs{\delta}_i'}}\#\widetilde{\cM^1}_{\R\times\La^-_{012}}(\gamma_{20};\bs{\delta}_0,\gamma_{02},\bs{\delta}_2)\#\cM^0_{\Sigma_{012}}(\gamma_{02};\bs{\delta}_0',\gamma_{10},\bs{\delta}_1',x_{12},\bs{\delta}_2')\cdot\ep_i^-(\bs{\delta}_i\bs{\delta}_i')\cdot \gamma_{20}\\
&\hspace{2cm}+\sum\limits_{\substack{\gamma_{20},\gamma_{12}\\\bs{\delta}_i,\bs{\delta}_i'}}\#\widetilde{\cM^1}_{\R\times\La^-_{012}}(\gamma_{20};\bs{\delta}_0,\gamma_{10},\bs{\delta}_1,\gamma_{12},\bs{\delta}_2)\#\cM^0_{\Sigma_{12}}(\gamma_{12};\bs{\delta}_1',x_{12},\bs{\delta}_2')\cdot\ep_i^-(\bs{\delta}_i\bs{\delta}_i')\cdot \gamma_{20}
\end{alignat*}
\normalsize
and the obvious symmetric formula for a pair $(\gamma_{21},x_{01})$,
and finally for a pair of Reeb chords $(\gamma_{21},\gamma_{10})\in C(\La_1^-,\La_2^-)\otimes C(\La_0^-,\La_1^-)$,
\begin{alignat*}{1}
&\mfm_{--}^{-}(\gamma_{21},\gamma_{10})=\sum\limits_{\gamma_{20},\bs{\delta}_i}\#\widetilde{\cM^1}_{\R\times\La^-_{012}}(\gamma_{20};\bs{\delta}_0,\gamma_{10},\bs{\delta}_1,\gamma_{21},\bs{\delta}_2)\cdot\ep_i^-(\bs{\delta}_i\bs{\delta}_i')\cdot \gamma_{20}
\end{alignat*}
\begin{figure}[ht]  
	\begin{center}\includegraphics[width=13cm]{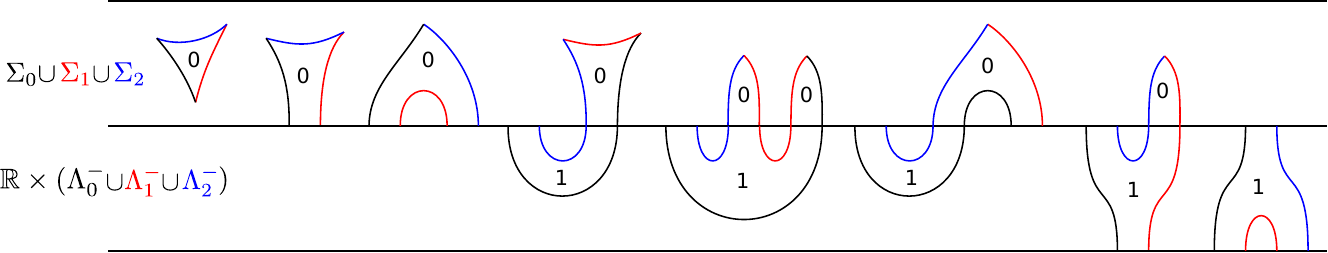}\end{center}
	\caption{Pseudo-holomorphic discs contributing to $\mfm_2^{-\infty}$.}
	\label{m_infini}
\end{figure}

Then, let us define the remaining components of the map $\mfm_2$, involving Reeb chords in the positive end. First, the components $\mfm_{00}^+$, $\mfm_{0-}^+$, $\mfm_{-0}^+$, and $\mfm_{--}^+$ vanish. It remains to define $\mfm_{++}^k$, $\mfm_{+i}^k$, and $\mfm_{i+}^k$ for $i\in\{0,-\}$ and $k\in\{+,0,-\}$.
Given a pair $(\gamma_{12},\gamma_{01})\in C(\La_2^+,\La_1^+)\otimes C(\La_1^+,\La_0^+)$, we have first
\small
\begin{alignat*}{1}
\mfm_{++}^+(\gamma_{12},\gamma_{01})&=\sum\limits_{\gamma_{02},\bs{\zeta}_i}\#\widetilde{\cM^1}_{\R\times\La_{012}^+}(\gamma_{02};\bs{\zeta}_0,\gamma_{01},\bs{\zeta}_1,\gamma_{12},\bs{\zeta}_2)\ep_i^+(\bs{\zeta}_i)\cdot\gamma_{02}\\
&+\sum\limits_{\substack{\gamma_{02},\gamma_{21}\\\bs{\zeta}_i,\bs{\delta}_i}}\#\widetilde{\cM^1}_{\R\times\La_{012}^+}(\gamma_{02};\bs{\zeta}_0,\gamma_{01},\bs{\zeta}_1,\gamma_{21},\bs{\zeta}_2)\#\cM^0_{\Sigma_{12}}(\gamma_{21};\bs{\delta}_1,\gamma_{12},\bs{\delta}_2)\cdot\ep_i^+(\bs{\zeta}_i)\ep_i^-(\bs{\delta}_i)\cdot\gamma_{02}\\
&+\sum\limits_{\substack{\gamma_{02},\gamma_{10}\\\bs{\zeta}_i,\bs{\delta}_i}}\#\widetilde{\cM^1}_{\R\times\La_{012}^+}(\gamma_{02};\bs{\zeta}_0,\gamma_{10},\bs{\zeta}_1,\gamma_{12},\bs{\zeta}_2)\#\cM^0_{\Sigma_{01}}(\gamma_{10};\bs{\delta}_0,\gamma_{01},\bs{\delta}_1)\cdot\ep_i^+(\bs{\zeta}_i)\ep_i^-(\bs{\delta}_i)\cdot\gamma_{02}\end{alignat*}
\normalsize
summing over $\gamma_{ij}\in C(\La_j^+,\La_i^+)$, $\bs{\zeta}_i$ words of Reeb chords of $\La_i^+$, for $i=0,1,2$, and $\bs{\delta}_i$ words of Reeb chords of $\La_i^-$, for $i=0,1,2$.
Then we have
\begin{alignat*}{1}
\mfm_{++}^0(\gamma_{12},\gamma_{01})=\sum\limits_{p_{20},\bs{\delta}_i}\#\cM^0_{\Sigma_{012}}(p_{20};\bs{\delta}_0,\gamma_{01},\bs{\delta}_1,\gamma_{12},\bs{\delta}_2)\ep_i^-(\bs{\delta}_i)\cdot p_{20}
\end{alignat*}
summing over $p_{02}\in\Sigma_0\cap\Sigma_2$ and $\bs{\delta}_i$ as above.
And finally the last component of the product for this pair of generators is:
\small
\begin{alignat*}{1}
&\mfm_{++}^-(\gamma_{12},\gamma_{01})=\sum\limits_{\substack{\gamma_{20},\xi_{02}\\\bs{\delta}_i,\bs{\delta}_i'}}\#\widetilde{\cM^1}_{\R\times\La_{02}^-}(\gamma_{20};\bs{\delta}_0,\xi_{02},\bs{\delta}_2)\#\cM^0_{\Sigma_{012}}(\xi_{02};\bs{\delta}_0',\gamma_{01},\bs{\delta}_1',\gamma_{12},\bs{\delta}_2')\ep_i^-(\bs{\delta}_i)\ep_i^-(\bs{\delta}_i')\cdot \gamma_{20}\\
&+\sum\limits_{\substack{\gamma_{20},\xi_{01},\xi_{12}\\ \bs{\delta}_i,\bs{\delta}_i',\bs{\delta}_i''}}\#\widetilde{\cM^1}_{\R\times\La_{012}^-}(\gamma_{20};\bs{\delta}_0,\xi_{01},\bs{\delta}_1,\xi_{12},\bs{\delta}_2)\#\cM^0_{\Sigma_{01}}(\xi_{01};\bs{\delta}_0',\gamma_{01},\bs{\delta}_1')\#\cM^0_{\Sigma_{12}}(\xi_{12};\bs{\delta}_1'',\gamma_{12},\bs{\delta}_2'')\cdot\ep_i^-(\bs{\delta}_i\bs{\delta}_i'\bs{\delta}_i'')\cdot\gamma_{20}
\end{alignat*}
\normalsize
summing over $\gamma_{20}\in C(\La_0^-,\La_2^-)$, $\xi_{ij}\in C(\La_j^-,\La_i^-)$, and $\bs{\delta}_i$, $\bs{\delta}_i'$ words of Reeb chords of $\La_i^-$.
Then, for a pair of generators $(\gamma_{12},x_{01})\in C(\La_2^+,\La_1^+)\otimes CF(\Sigma_0,\Sigma_1)$ we define:
\small
\begin{alignat*}{1}
&\mfm_{+0}^+(\gamma_{12},x_{01})=\sum\limits_{\substack{\gamma_{02},\gamma_{10}\\\bs{\zeta}_i,\bs{\delta}_i}}\#\widetilde{\cM^1}_{\R\times\La_{012}^+}(\gamma_{02};\bs{\zeta}_0,\gamma_{10},\bs{\zeta}_1,\gamma_{12},\bs{\zeta}_2)\#\cM^0_{\Sigma_{01}}(\gamma_{10};\bs{\delta}_0,x_{01},\bs{\delta}_1)\cdot\ep_i^+(\bs{\zeta}_i)\ep_i^-(\bs{\delta}_i)\cdot\gamma_{02}\\
&\mfm_{+0}^0(\gamma_{12},x_{01})=\sum\limits_{p_{20},\bs{\delta}_i}\#\cM^0_{\Sigma_{012}}(p_{20};\bs{\delta}_0,x_{01},\bs{\delta}_1,\gamma_{12},\bs{\delta}_2)\ep_i^-(\bs{\delta}_i)\cdot p_{02}\\
&\mfm_{+0}^-(\gamma_{12},x_{01})=\sum\limits_{\substack{\gamma_{20},\xi_{02}\\\bs{\delta}_i,\bs{\delta}_i'}}\#\widetilde{\cM^1}_{\R\times\La_{02}^-}(\gamma_{20};\bs{\delta}_0,\xi_{02},\bs{\delta}_2)\#\cM^0_{\Sigma_{012}}(\xi_{02};\bs{\delta}_0',x_{01},\bs{\delta}_1',\gamma_{12},\bs{\delta}_2')\ep_i^-(\bs{\delta}_i)\ep_i^-(\bs{\delta}_i')\cdot \gamma_{20}\\
&+\sum\limits_{\substack{\gamma_{20},\xi_{01},\xi_{12}\\ \bs{\delta}_i,\bs{\delta}_i',\bs{\delta}_i''}}\#\widetilde{\cM^1}_{\R\times\La_{012}^-}(\gamma_{20};\bs{\delta}_0,\xi_{01},\bs{\delta}_1,\xi_{12},\bs{\delta}_2)\#\cM^0_{\Sigma_{01}}(\xi_{01};\bs{\delta}_0',x_{01},\bs{\delta}_1')\#\cM^0_{\Sigma_{12}}(\xi_{12};\bs{\delta}_1'',\gamma_{12},\bs{\delta}_2'')\cdot\ep_i^-(\bs{\delta}_i\bs{\delta}_i'\bs{\delta}_i'')\cdot\gamma_{20}
\end{alignat*}
\normalsize
We finish by defining the product for a pair $(\gamma_{12},\gamma_{10})\in C(\La_2^+,\La_1^+)\otimes C(\La^-_0,\La^-_1)$ as follows:
\small
\begin{alignat*}{1}
&\mfm_{+-}^+(\gamma_{12},\gamma_{10})=\sum\limits_{\substack{\gamma_{02},\xi_{10}\\\bs{\zeta}_i,\bs{\delta}_i}}\#\widetilde{\cM^1}_{\R\times\La_{012}^+}(\gamma_{02};\bs{\zeta}_0,\xi_{10},\bs{\zeta}_1,\gamma_{12},\bs{\zeta}_2)\#\cM^0_{\Sigma_{01}}(\xi_{10};\bs{\delta}_0,\gamma_{10},\bs{\delta}_1)\cdot\ep_i^+(\bs{\zeta}_i)\ep_i^-(\bs{\delta}_i)\cdot\gamma_{02}\\
&\mfm_{+-}^0(\gamma_{12},\gamma_{10})=\sum\limits_{p_{20},\bs{\delta}_i}\#\cM^0_{\Sigma_{012}}(p_{20};\bs{\delta}_0,\gamma_{10},\bs{\delta}_1,\gamma_{12},\bs{\delta}_2)\ep_i^-(\bs{\delta}_i)\cdot p_{02}\\
&	\mfm_{+-}^-(\gamma_{12},\gamma_{10})=\sum\limits_{\substack {\gamma_{20},\xi_{02}\\\bs{\delta}_i,\bs{\delta}_i'}}\#\widetilde{\cM^1}_{\R\times\La_{02}^-}(\gamma_{20};\bs{\delta}_0,\xi_{02},\bs{\delta}_2)\#\cM^0_{\Sigma_{012}}(\xi_{02};\bs{\delta}_0',\gamma_{10},\bs{\delta}_1',\gamma_{12},\bs{\delta}_2')\ep_i^-(\bs{\delta}_i)\ep_i^-(\bs{\delta}_i')\cdot \gamma_{20}\\
&\hspace{2cm}+\sum\limits_{\substack{\gamma_{20},\xi_{12}\\ \bs{\delta}_i,\bs{\delta}_i'}}\#\widetilde{\cM^1}_{\R\times\La_{012}^-}(\gamma_{20};\bs{\delta}_0,\gamma_{10},\bs{\delta}_1,\xi_{12},\bs{\delta}_2)\#\cM^0_{\Sigma_{12}}(\xi_{12};\bs{\delta}_1',\gamma_{12},\bs{\delta}_2')\cdot\ep_i^-(\bs{\delta}_i\bs{\delta}_i')\cdot\gamma_{20}
\end{alignat*}
\normalsize
The components $\mfm_{0+}^k$ and $\mfm_{-+}^k$ for $k=+,0,-$ are defined analogously as  $\mfm_{+0}^k$ and $\mfm_{+-}^k$. See Figures \ref{m++}, \ref{m+0} and \ref{m+-}.
\begin{figure}[ht]  
	\begin{center}\includegraphics[width=11cm]{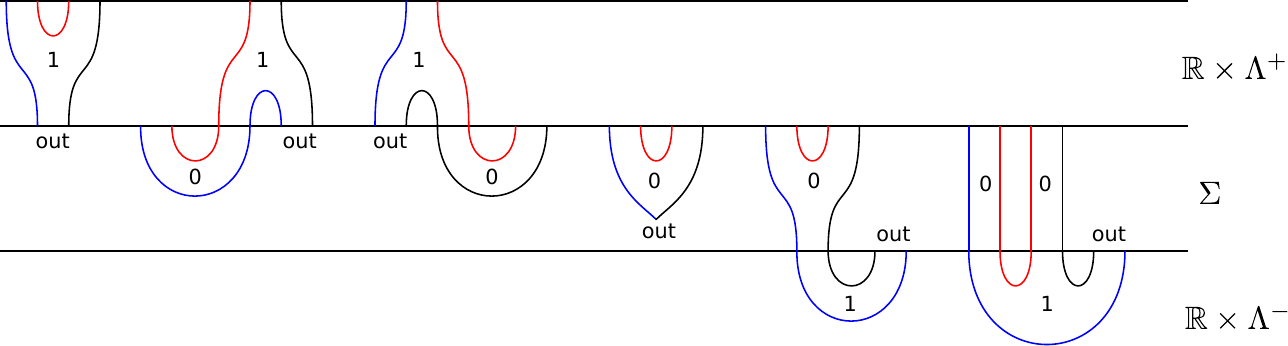}\end{center}
	\caption{Pseudo-holomorphic discs contributing to $\mfm_{++}^k$, $k=+,0,-$.}
	\label{m++}
\end{figure}
\begin{figure}[ht]  
	\begin{center}\includegraphics[width=8cm]{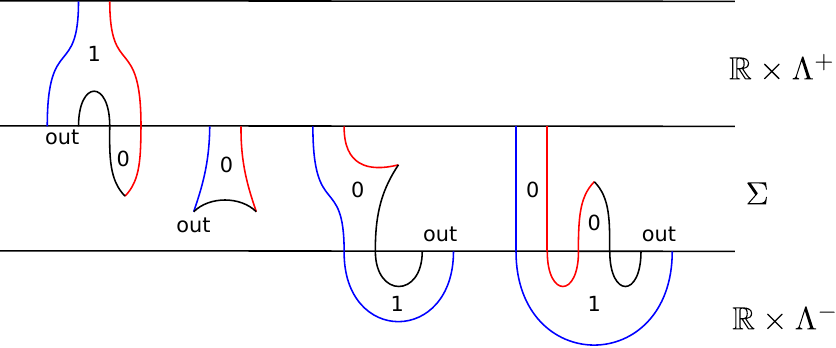}\end{center}
	\caption{Pseudo-holomorphic discs contributing to $\mfm_{+0}^k$, $k=+,0,-$.}
	\label{m+0}
\end{figure}
\begin{figure}[ht]  
	\begin{center}\includegraphics[width=8cm]{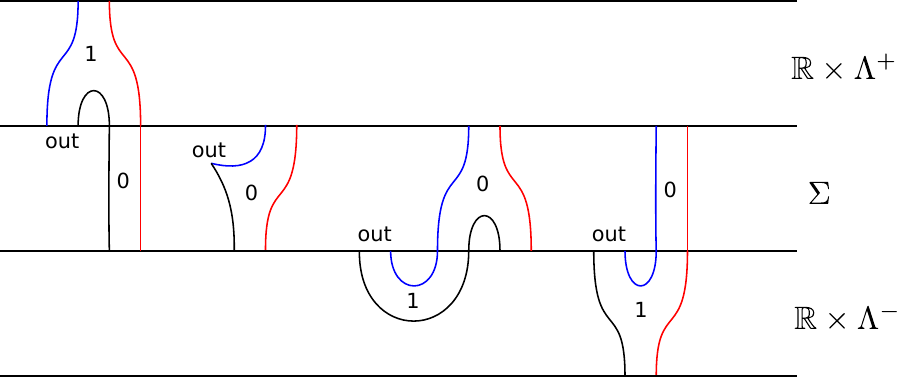}\end{center}
	\caption{Pseudo-holomorphic discs contributing to $\mfm_{+-}^k$, $k=+,0,-$.}
	\label{m+-}
\end{figure}
\begin{teo}\label{teo_Leibniz_rule}
	The map $\mfm_2$ satisfies the Leibniz rule, i.e. given three exact pairwise transverse Lagrangian cobordisms $\La_i^-\prec_{\Sigma_i}\La_i^+$ with augmentations $\ep_i^-$ of $\Ac(\La_i^-)$ for $i=0,1,2$, we have:
	\begin{alignat*}{1}
	\mfm_2(\mfm_1^{\ep^-_1,\ep^-_2},\cdot)+\mfm_2(\cdot,\mfm_1^{\ep^-_0,\ep^-_1})+\mfm_1^{\ep_0^-,\ep_2^-}\circ\mfm_2(\cdot,\cdot)=0
	\end{alignat*}
\end{teo}
\begin{rem}
	A \textquotedblleft complete\textquotedblright\, notation for the product would be something like $\mfm^{\Sigma_0,\Sigma_1,\Sigma_2}_{\ep_0^-,\ep_1^-,\ep_2^-}$ as it depends on the choice of cobordisms and on the choice of augmentations of the negative ends. However, to simplify, we will just write it $\mfm_2$ as the choices mentioned  are clear from the context.
\end{rem}
As for $\mfm_1$, we can write the components $\mfm_2^+$ and $\mfm_2^-$ as a composition of maps, it will be convenient when describing the boundary of the compactification of 1-dimensional moduli spaces. First, we introduce the maps
\begin{alignat*}{1}
&\Delta_2^+:\mathfrak{C}^*(\La_1^+,\La_2^+)\otimes\mathfrak{C}^*(\La_0^+,\La_1^+)\to C_{n-1-*}(\La_2^+,\La_0^+)\\
&\Delta_2^\Sigma:\Cth_+(\Sigma_1,\Sigma_2)\otimes\Cth_+(\Sigma_0,\Sigma_1)\to C_{n-1-*}(\La_2^-,\La_0^-)\\
&b_2^-:\mathfrak{C}^*(\La_1^-,\La_2^-)\otimes\mathfrak{C}^*(\La_0^-,\La_1^-)\to C^{*-1}(\La_0^-,\La_2^-)	
\end{alignat*}
defined by
\begin{alignat*}{1}
&\Delta_2^+(\gamma_2,\gamma_1)=\sum\limits_{\gamma_{02},\bs{\zeta}_i}\#\widetilde{\cM^1}_{\R\times\La_{012}^+}(\gamma_{02};\bs{\zeta}_0,\gamma_1,\bs{\zeta}_1,\gamma_2,\bs{\zeta}_2)\ep_i^+(\bs{\zeta}_i)\cdot\gamma_{02}\\
&\Delta_2^\Sigma(a_2,a_1)=\sum\limits_{\gamma_{02},\bs{\delta}_i}\#\cM^0_{\Sigma_{012}}(\gamma_{02};\bs{\delta}_0,a_1,\bs{\delta}_1,a_2,\bs{\delta}_2)\ep_i^-(\bs{\delta}_i)\cdot\gamma_{02}\\
&b_2^-(\gamma_2,\gamma_1)=\sum\limits_{\gamma_{20},\bs{\delta}_i}\#\widetilde{\cM^1}_{\R\times\La_{012}^-}(\gamma_{20};\bs{\delta}_0,\gamma_1,\bs{\delta}_1,\gamma_2,\bs{\delta}_2)\ep_i^-(\bs{\delta}_i)\cdot\gamma_{20}
\end{alignat*}
and observe that $\Delta_2^\Sigma$ vanishes on $C(\La_1,\La_2)\otimes C(\La_0,\La_1)$ for energy reasons.
Using these maps, we have
\begin{alignat}{2}
	&\mfm_2^+=\Delta_2^+(\bs{b}_1^\Sigma\otimes\bs{b}_1^\Sigma)\label{defm+}\\
	&\mfm_2^-=b_1^-\circ\Delta_2^\Sigma+b_2^-(\bs{\Delta}_1^\Sigma\otimes\bs{\Delta}_1^\Sigma)\label{defm-}
\end{alignat}
where the maps $b_1^-, \bs{\Delta}_1^\Sigma$ are defined in Section \ref{section:def_complex} and $\bs{b}_1^\Sigma$ in Section \ref{sec:conc} (see also Section \ref{special_case}).

\subsection{Leibniz rule} \label{sec:Leibniz}
The map $\mfm_2$ restricted to $CF_{-\infty}(\Sigma_1,\Sigma_2)\otimes CF_{-\infty}(\Sigma_0,\Sigma_1)$ satisfies the Leibniz rule because $\mfm_2^{-\infty}$ satisfies it with respect to the differential $\mfm_1^{-\infty}$ (see \cite{L}) and there is no component of the differential $\mfm_1^{\ep_0^-,\ep_1^-}$ from the subcomplex $CF_{-\infty}(\Sigma_0,\Sigma_1)$ to $C(\La_0^+,\La_1^+)$. It remains to check the Leibniz rule for each pair of generators containing at least one Reeb chord in the positive end:
	\begin{enumerate}
		\item[(a)] $(\gamma_{12},\gamma_{01})\in C(\La_2^+,\La_1^+)\otimes C(\La_1^+,\La_0^+)$,
		\item[(b)] $(\gamma_{12},x_{01})\in C(\La_2^+,\La_1^+)\otimes CF(\Sigma_0,\Sigma_1)$ and $(x_{12},\gamma_{01})\in CF(\Sigma_1,\Sigma_2)\otimes C(\La_1^+,\La_0^+)$,
		\item[(c)] $(\gamma_{12},\gamma_{10})\in C(\La_2^+,\La_1^+)\otimes C(\La_0^-,\La_1^-)$ and $(\xi_{12},\gamma_{10})\in C(\La_1^-,\La_2^-)\otimes C(\La_1^+,\La_0^+)$,
	\end{enumerate}
As usual, the Leibniz rule will follow from the study of the boundary of the compactification of some (product of) moduli spaces. Recall that we described in Section \ref{sec:structure} the different types of broken discs arising in this boundary. We focus now on some particular moduli spaces and specify the algebraic contribution of each broken disc. \\

\noindent\textbf{Leibniz rule for a pair of type (a):}
For the pair of generators of type (a), we will show that the following three relations are satisfied: 
	\begin{alignat}{1}
	&\mfm_2^+(\mfm_1^{\ep_1^-,\ep_2^-}(\gamma_{12}),\gamma_{01})+\mfm_2^+(\gamma_{12},\mfm_1^{\ep_0^-,\ep_1^-}(\gamma_{01}))+\mfm_1^+\circ\mfm_2^+(\gamma_{12},\gamma_{01})=0\label{rel+++}\\
	&\mfm_2^0(\mfm_1^{\ep_1^-,\ep_2^-}(\gamma_{12}),\gamma_{01})+\mfm_2^0(\gamma_{12},\mfm_1^{\ep_0^-,\ep_1^-}(\gamma_{01}))+\mfm_1^0\circ\mfm_2(\gamma_{12},\gamma_{01})=0\label{rel++0}\\
	&\mfm_2^-(\mfm_1^{\ep_1^-,\ep_2^-}(\gamma_{12}),\gamma_{01})+\mfm_2^-(\gamma_{12},\mfm_1^{\ep_0^-,\ep_1^-}(\gamma_{01}))+\mfm_1^-\circ\mfm_2(\gamma_{12},\gamma_{01})=0\label{rel++-}
	\end{alignat}
After adding in \eqref{rel+++} the vanishing terms $\mfm_1^+\circ\mfm_2^0(\gamma_{12},\gamma_{01})=\mfm_1^+\circ\mfm_2^-(\gamma_{12},\gamma_{01})=0$, the sum of these three relations gives the Leibniz rule for the pair $(\gamma_{12},\gamma_{01})\in C(\La_1^+,\La_2^+)\otimes C(\La_0^+,\La_1^+)$.
First, we see that relation \eqref{rel+++} follows from the study of the boundary of the compactification of the following products of moduli spaces:
	\begin{alignat}{1}
	&\widetilde{\cM^2}_{\R\times\La_{012}^+}(\gamma_{02};\bs{\zeta}_0,\gamma_{01},\bs{\zeta}_1,\gamma_{12},\bs{\zeta}_2)\\
	&\widetilde{\cM^2}_{\R\times\La_{012}^+}(\gamma_{02};\bs{\zeta}_0,\gamma_{01},\bs{\zeta}_1,\gamma_{21},\bs{\zeta}_2)\times\cM^0_{\Sigma_{12}}(\gamma_{21};\bs{\delta}_1,\gamma_{12},\bs{\delta}_2)\label{mod2}\\
	&\widetilde{\cM^1}_{\R\times\La_{012}^+}(\gamma_{02};\bs{\zeta}_0,\gamma_{01},\bs{\zeta}_1,\gamma_{21},\bs{\zeta}_2)\times\cM^1_{\Sigma_{12}}(\gamma_{21};\bs{\delta}_1,\gamma_{12},\bs{\delta}_2)\label{mod3}\\
	&\widetilde{\cM^2}_{\R\times\La_{012}^+}(\gamma_{02};\bs{\zeta}_0,\gamma_{10},\bs{\zeta}_1,\gamma_{12},\bs{\zeta}_2)\times\cM^0_{\Sigma_{01}}(\gamma_{10};\bs{\delta}_0,\gamma_{01},\bs{\delta}_1)\label{mod4}\\
	&\widetilde{\cM^1}_{\R\times\La_{012}^+}(\gamma_{02};\bs{\zeta}_0,\gamma_{10},\bs{\zeta}_1,\gamma_{12},\bs{\zeta}_2)\times\cM^1_{\Sigma_{01}}(\gamma_{10};\bs{\delta}_0,\gamma_{01},\bs{\delta}_1)\label{mod5}
	\end{alignat}
\begin{figure}[ht]  
	\begin{center}\includegraphics[width=11cm]{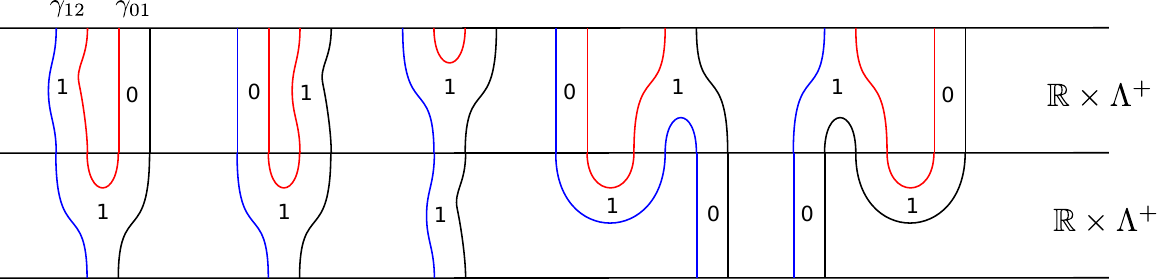}\end{center}
	\caption{Types of broken discs in the boundary of $\overline{\cM^2}_{\R\times\La_{012}^+}(\gamma_{02};\bs{\zeta}_0,\gamma_{01},\bs{\zeta}_1,\gamma_{12},\bs{\zeta}_2)$.}
	\label{bris++1}
\end{figure}
The broken discs in $\partial\overline{\cM^2}_{\R\times\La_{012}^+}(\gamma_{02};\bs{\zeta}_0,\gamma_{01},\bs{\zeta}_1,\gamma_{12},\bs{\zeta}_2)$ are schematized on Figure \ref{bris++1}. The sum of their algebraic contributions vanishes, and thus gives:
\begin{alignat}{1}
	\Delta_2^+(\mfm_1^+(\gamma_{12}),\gamma_{01})+\Delta_2^+(\gamma_{12},\mfm_1^+(\gamma_{01}))&+\mfm_1^+\circ\Delta_2^+(\gamma_{12},\gamma_{01})\label{rel++1}\\
	&+\Delta_2^+(b_1^+(\gamma_{12}),\gamma_{01})+\Delta_2^+(\gamma_{12},b_1^+(\gamma_{01}))=0\nonumber
\end{alignat}
The boundary of the compactification of \eqref{mod2}, see Figure \ref{bris++2}, gives the algebraic relation:
\begin{alignat}{1}
\Delta_2^+(b_1^\Sigma(\gamma_{12}),\mfm_1^+(\gamma_{01}))+\mfm_1^+\circ\Delta_2^+(b_1^\Sigma(\gamma_{12}),\gamma_{01})+\Delta_2^+(b_1^+\circ b_1^\Sigma(\gamma_{12}),\gamma_{01})=0\label{rel++2}
\end{alignat}
\begin{figure}[ht]  
	\begin{center}\includegraphics[width=8cm]{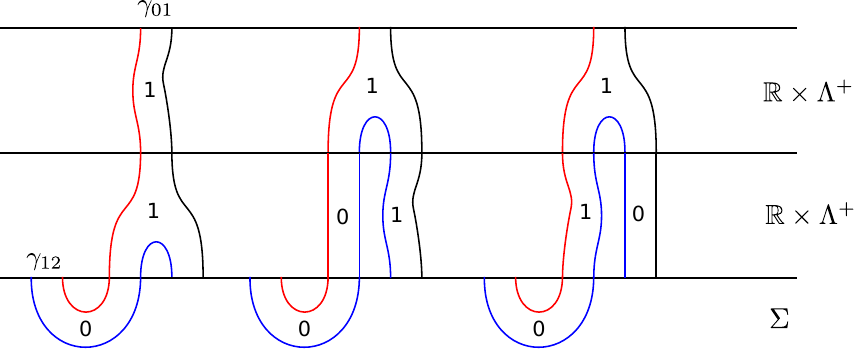}\end{center}
	\caption{Broken discs in $\partial\overline{\cM^2}_{\R\times\La_{012}^+}(\gamma_{02};\bs{\zeta}_0,\gamma_{01},\bs{\zeta}_1,\gamma_{21},\bs{\zeta}_2)\times\cM^0_{\Sigma_{12}}(\gamma_{21};\bs{\delta}_1,\gamma_{12},\bs{\delta}_2)$.}
	\label{bris++2}
\end{figure}

One gets the symmetric relation
\begin{alignat}{1}
\Delta_2^+(\mfm_1^+(\gamma_{12}),b_1^\Sigma(\gamma_{01}))+\mfm_1^+\circ\Delta_2^+(\gamma_{12},b_1^\Sigma(\gamma_{01}))+\Delta_2^+(\gamma_{12},b_1^+\circ b_1^\Sigma(\gamma_{01}))=0\label{rel++2bis}
\end{alignat}
by studying the boundary of \eqref{mod4}.
Finally, one gets the relation
\begin{alignat}{1}
\Delta_2^+(b_1^\Sigma\circ\mfm_1^+(\gamma_{12}),\gamma_{01})&+\Delta_2^+(b_1^+\circ b_1^\Sigma(\gamma_{12}),\gamma_{01})+\Delta_2^+(b_1^+(\gamma_{12}),\gamma_{01})\label{rel++3}\\
&+\Delta_2^+(b_1^\Sigma\circ\mfm_1^0(\gamma_{12}),\gamma_{01})+\Delta_2^+(b_1^\Sigma\circ\mfm_1^-(\gamma_{12}),\gamma_{01})=0\nonumber
\end{alignat}
and the symmetric
\begin{alignat}{1}
\Delta_2^+(\gamma_{12},b_1^\Sigma\circ\mfm_1^+(\gamma_{01}))&+\Delta_2^+(\gamma_{12},b_1^+\circ b_1^\Sigma(\gamma_{01}))+\Delta_2^+(\gamma_{12},b_1^+(\gamma_{01}))\label{rel++3bis}\\
&+\Delta_2^+(\gamma_{12},b_1^\Sigma\circ\mfm_1^0(\gamma_{01}))+\Delta_2^+(\gamma_{12},b_1^\Sigma\circ\mfm_1^-(\gamma_{01}))=0\nonumber
\end{alignat}
by studying first \eqref{mod3} and then \eqref{mod5} (see Figure \ref{bris++3}). Observe that for these last two, we consider the boundary of the compactification of moduli spaces of bananas with boundary on non cylindrical parts of the cobordisms and with two positive Reeb chord asymptotics, as we have done already in the proof of Lemma \ref{lem:rel}.
\begin{figure}[ht]  
	\begin{center}\includegraphics[width=10cm]{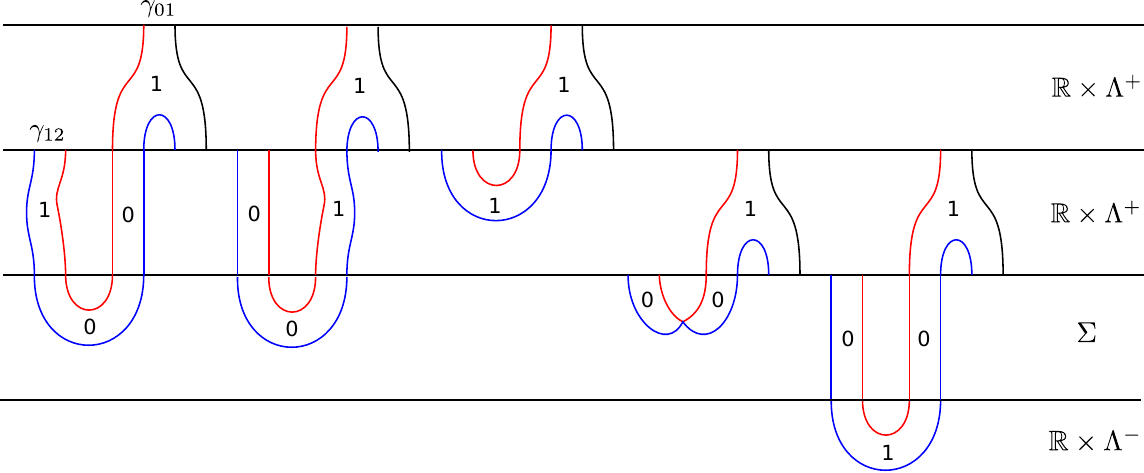}\end{center}
	\caption{Broken discs in $\widetilde{\cM^1}_{\R\times\La_{012}^+}(\gamma_{02};\bs{\zeta}_0,\gamma_{01},\bs{\zeta}_1,\gamma_{21},\bs{\zeta}_2)\times\partial\overline{\cM^1}_{\Sigma_{12}}(\gamma_{21};\bs{\delta}_1,\gamma_{12},\bs{\delta}_2)$.}
	\label{bris++3}
\end{figure}
Summing \eqref{rel++1}, \eqref{rel++2}, \eqref{rel++2bis}, \eqref{rel++3} and \eqref{rel++3bis}, cancelling terms appearing twice and using the definition of $\bs{b}_1^\Sigma$ and $\mfm_2^+$ given in \eqref{defm+}, one obtains relation \eqref{rel+++}.

Then, the study of the boundary of the compactification of 
\begin{alignat*}{1}
	\cM^1_{\Sigma_{012}}(p_{20};\bs{\delta}_0,\gamma_{01},\bs{\delta}_1,\gamma_{12},\bs{\delta}_2)
\end{alignat*}
gives relation \eqref{rel++0}, see Figure \ref{bris++0} for a description of broken discs. Indeed, the algebraic contributions of those discs are (from left to right and top to bottom on the figure):
\begin{alignat*}{1}
	&\mfm_2^0(\mfm_1^+(\gamma_{12}),\gamma_{01})+\mfm_2^0(\gamma_{12},\mfm_1^+(\gamma_{01}))+\mfm_1^0\circ\Delta_2^+(\gamma_{12},\gamma_{01})+\mfm_1^0\circ\Delta_2^+(b_1^\Sigma(\gamma_{12}),\gamma_{01})\\
	&+\mfm_1^0\circ\Delta_2^+(\gamma_{12},b_1^\Sigma(\gamma_{01}))+\mfm_2^0(\mfm_1^0(\gamma_{12}),\gamma_{01})+\mfm_2^0(\gamma_{12},\mfm_1^0(\gamma_{01}))+\mfm_1^0\circ\mfm_2^0(\gamma_{12},\gamma_{01})\\
	&+\mfm_2^0(\mfm_1^-(\gamma_{12}),\gamma_{01})+\mfm_2^0(\gamma_{12},\mfm_1^-(\gamma_{01}))+\mfm_1^0\circ b_1^-\circ\Delta_2^\Sigma(\gamma_{12},\gamma_{01})+\mfm_1^0\circ b_2^-(\Delta_1^\Sigma(\gamma_{12}),\Delta_1^\Sigma(\gamma_{01}))=0
\end{alignat*}
And using the definitions of $\mfm_2^+$ and $\mfm_2^-$ given in \eqref{defm+} and \eqref{defm-} one deduces relation \eqref{rel++0}.
\begin{figure}[ht]  
	\begin{center}\includegraphics[width=12cm]{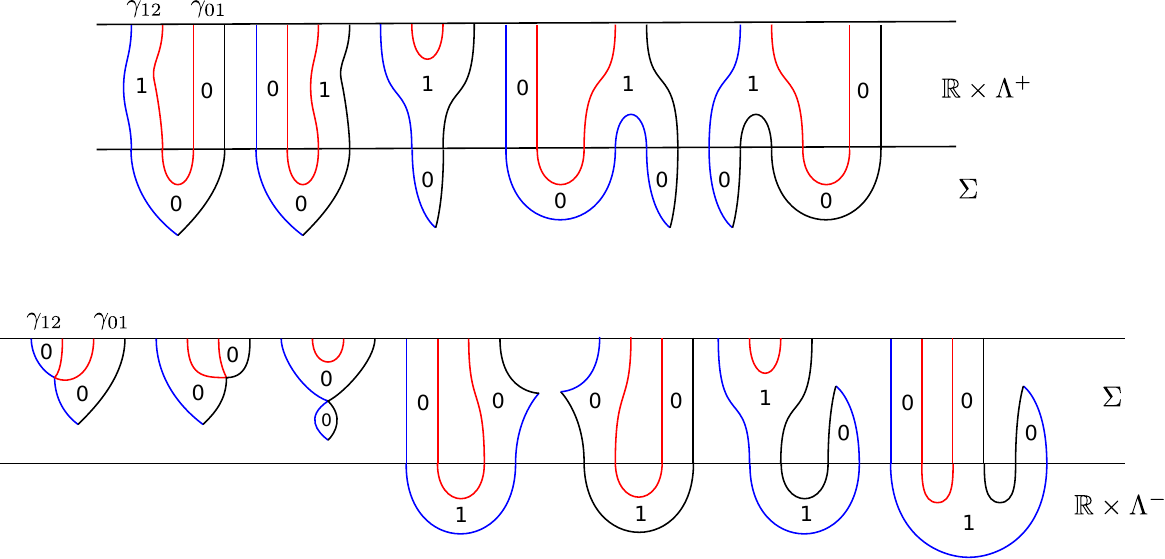}\end{center}
	\caption{Broken discs in $\partial\overline{\cM^1}_{\Sigma_{012}}(p_{20};\bs{\delta}_0,\gamma_{01},\bs{\delta}_1,\gamma_{12},\bs{\delta}_2)$.}
	\label{bris++0}
\end{figure}

Finally, analogously to the previous cases, the broken curves in the boundary of the compactification of
\begin{alignat}{1}
	&\widetilde{\cM^2}_{\R\times\La_{02}^-}(\gamma_{20};\bs{\delta}_0,\gamma_{02},\bs{\delta}_2)\times\cM^0_{\Sigma_{012}}(\gamma_{02};\bs{\delta}_0',\gamma_{01},\bs{\delta}_1',\gamma_{12},\bs{\delta}_2')\label{prod0}\\
	&\widetilde{\cM^1}_{\R\times\La_{02}^-}(\gamma_{20};\bs{\delta}_0,\gamma_{02},\bs{\delta}_2)\times\cM^1_{\Sigma_{012}}(\gamma_{02};\bs{\delta}_0',\gamma_{01},\bs{\delta}_1',\gamma_{12},\bs{\delta}_2')\label{prod0bis}\\
	&\widetilde{\cM^2}_{\R\times\La_{012}^-}(\gamma_{20};\bs{\delta}_0,\xi_{01},\bs{\delta}_1,\xi_{12},\bs{\delta}_2)\times\cM^0_{\Sigma_{01}}(\xi_{01};\bs{\delta}_0',\gamma_{01},\bs{\delta}_1')\times\cM^0_{\Sigma_{12}}(\xi_{12};\bs{\delta}_1'',\gamma_{12},\bs{\delta}_2'')\label{prod1}\\
	&\widetilde{\cM^1}_{\R\times\La_{012}^-}(\gamma_{20};\bs{\delta}_0,\xi_{01},\bs{\delta}_1,\xi_{12},\bs{\delta}_2)\times\cM^1_{\Sigma_{01}}(\xi_{01};\bs{\delta}_0',\gamma_{01},\bs{\delta}_1')\times\cM^0_{\Sigma_{12}}(\xi_{12};\bs{\delta}_1'',\gamma_{12},\bs{\delta}_2'')\label{prod2}\\
	&\widetilde{\cM^1}_{\R\times\La_{012}^-}(\gamma_{20};\bs{\delta}_0,\xi_{01},\bs{\delta}_1,\xi_{12},\bs{\delta}_2)\times\cM^0_{\Sigma_{01}}(\xi_{01};\bs{\delta}_0',\gamma_{01},\bs{\delta}_1')\times\cM^1_{\Sigma_{12}}(\xi_{12};\bs{\delta}_1'',\gamma_{12},\bs{\delta}_2'')\label{prod3}
\end{alignat}
give relation \eqref{rel++-}. First, let us consider \eqref{prod0} and \eqref{prod0bis}. There are two types of broken discs arising in $\partial\overline{\cM^2}_{\R\times\La_{02}^-}(\gamma_{20};\bs{\delta}_0,\gamma_{02},\bs{\delta}_2)$ giving the algebraic relation 
\begin{alignat*}{1}
b_1^-\circ b_1^-(\gamma_{02})+b_1^-\circ\Delta_1^-(\gamma_{02})=0
\end{alignat*}
that we have already considered in the proof of Theorem \ref{diff}.
Then, on Figure \ref{leibniz++-} are schematized the broken discs in $\partial\overline{\cM^1}_{\Sigma_{012}}(\gamma_{02};\bs{\delta}_0',\gamma_{01},\bs{\delta}_1',\gamma_{12},\bs{\delta}_2')$. From this, the broken discs in the boundary of the compactification of \eqref{prod0}
contribute algebraically to
\begin{alignat}{1}
	b_1^-\circ b_1^-\circ\Delta_2^\Sigma(\gamma_{12},\gamma_{01})+b_1^-\circ\Delta_1^-\circ\Delta_2^\Sigma(\gamma_{12},\gamma_{01})\label{rel_banane}
\end{alignat}
and the ones in the boundary of the compactification of \eqref{prod0bis}
contribute to (from top to bottom and left to right for discs on Figure \ref{leibniz++-})
\small
\begin{alignat}{1}
	&b_1^-\circ\Delta_2^\Sigma(\mfm_1^+(\gamma_{12}),\gamma_{01})+b_1^-\circ\Delta_2^\Sigma(\gamma_{12},\mfm_1^+(\gamma_{01}))+\mfm_1^-\circ\Delta_2^+(\gamma_{12},\gamma_{01})+\mfm_1^-\circ\Delta_2^+(b_1^\Sigma(\gamma_{12}),\gamma_{01})+\mfm_1^-\circ\Delta_2^+(\gamma_{12},b_1^\Sigma(\gamma_{01}))\nonumber\\
	&+\,b_1^-\circ\Delta_2^\Sigma(\mfm_1^0(\gamma_{12}),\gamma_{01})+b_1^-\circ\Delta_2^\Sigma(\gamma_{12},\mfm_1^0(\gamma_{01}))+\mfm_1^-\circ\mfm_2^0(\gamma_{12},\gamma_{01})\label{rel_autre}\\
	&+\,b_1^-\circ\Delta_2^\Sigma(\mfm_1^-(\gamma_{12}),\gamma_{01})+b_1^-\circ\Delta_2^\Sigma(\gamma_{12},\mfm_1^-(\gamma_{01}))+b_1^-\circ\Delta_2^-(\Delta_1^\Sigma(\gamma_{12}),\Delta_1^\Sigma(\gamma_{01}))+b_1^-\circ \Delta_1^-\circ\Delta_2^\Sigma(\gamma_{12},\gamma_{01})\nonumber
	\end{alignat}
\normalsize
\begin{figure}[ht]  
	\begin{center}\includegraphics[width=12cm]{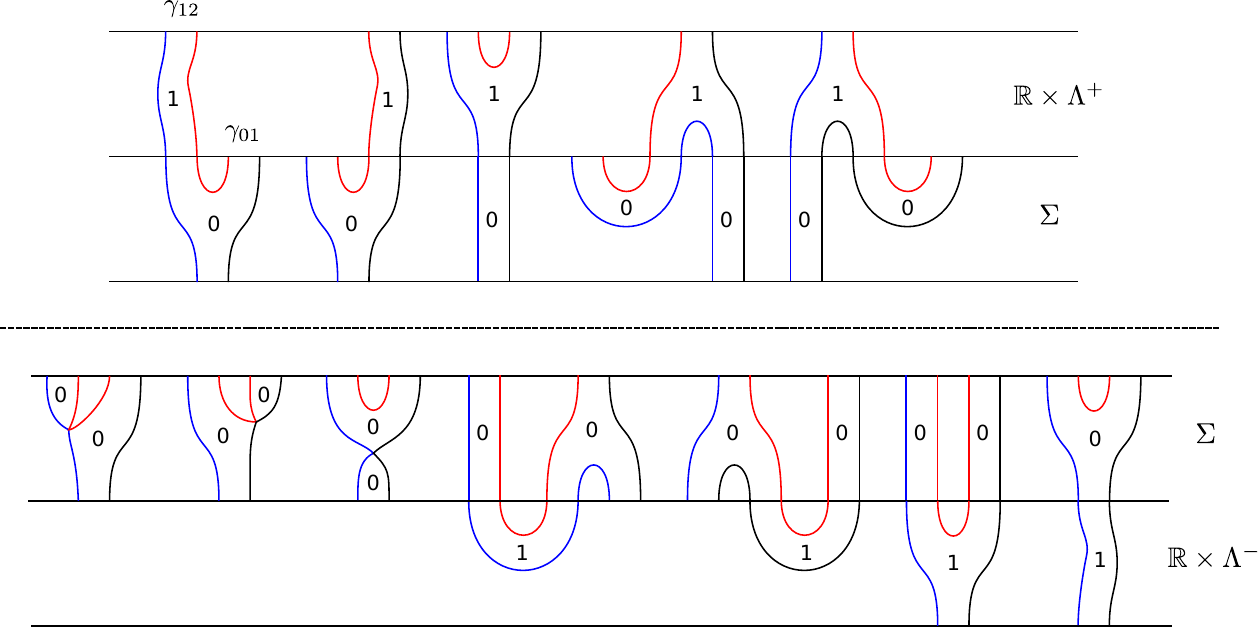}\end{center}
	\caption{Broken discs in $\partial\overline{\cM^1}_{\Sigma_{012}}(\gamma_{02};\bs{\delta}_0',\gamma_{01},\bs{\delta}_1',\gamma_{12},\bs{\delta}_2')$.}
	\label{leibniz++-}
\end{figure}
Note that the three last terms on the first line give $\mfm_1^-\circ\mfm_2^+(\gamma_{12},\gamma_{01})$ by definition of $\mfm_2^+$. Moreover, observe that the last term of \eqref{rel_banane} is the same as the last term of \eqref{rel_autre}. Thus, the boundary of the compactifications of \eqref{prod0} and \eqref{prod0bis} provides us with the following relation:
\begin{alignat}{1}
b_1^-\circ b_1^-\circ\Delta_2^\Sigma(\gamma_{12},\gamma_{01})&+b_1^-\circ\Delta_2^\Sigma(\mfm_1(\gamma_{12}),\gamma_{01})+b_1^-\circ\Delta_2^\Sigma(\gamma_{12},\mfm_1(\gamma_{01}))\label{rel_again1}\\
&+\mfm_1^-\circ(\mfm_2^++\mfm_2^0)(\gamma_{12},\gamma_{01})+b_1^-\circ\Delta_2^-(\Delta_1^\Sigma(\gamma_{12}),\Delta_1^\Sigma(\gamma_{01}))=0\nonumber
\end{alignat}

Let us now consider the products \eqref{prod1}, \eqref{prod2} and \eqref{prod3}. The broken discs arising in the boundary of the compactification of moduli spaces of bananas with three positive mixed asymptotics are schematized on Figure \ref{fig:bris_triple_banana}. We deduce from this that the broken discs in the boundary of the compactification of \eqref{prod1} give the relation:
\begin{alignat}{1}
b_1^-\circ b_2^-\big(&\Delta_1^\Sigma(\gamma_{12}),\Delta_1^\Sigma(\gamma_{01})\big)+b_1^-\circ\Delta_2^-\big(\Delta_1^\Sigma(\gamma_{12}),\Delta_1^\Sigma(\gamma_{01})\big)\label{rel_again2}\\
&+b_2^-\Big(\big(\Delta_1^-+b_1^-\big)\big(\Delta_1^\Sigma(\gamma_{12})\big),\Delta_1^\Sigma(\gamma_{01})\Big)+b_2^-\Big(\Delta_1^\Sigma(\gamma_{12}),\big(\Delta_1^-+b_1^-\big)\big(\Delta_1^\Sigma(\gamma_{01})\big)\Big)=0\nonumber
\end{alignat}
\begin{figure}[ht]  
	\begin{center}\includegraphics[width=12cm]{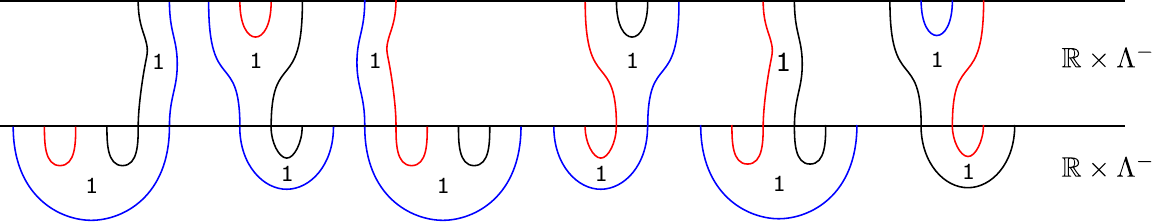}\end{center}
	\caption{Broken discs in $\partial\overline{\cM^2}_{\R\times\La_{012}^-}(\gamma_{20};\bs{\delta}_0,\xi_{01},\bs{\delta}_1,\xi_{12},\bs{\delta}_2)$.}
	\label{fig:bris_triple_banana}
\end{figure}
The last moduli spaces to study are moduli spaces of discs with boundary on the non-cylindrical parts of the cobordisms, with a positive and a negative mixed Reeb chord asymptotic. We have already considered the boundary of the compactification of such moduli spaces in the proof of Theorem \ref{diff} as well as in the proof of Lemma \ref{lem:rel}, see also Figure \ref{broken_lemma}. The algebraic contributions of broken discs in $\partial\overline{\cM^1}_{\Sigma_{01}}(\xi_{01};\bs{\delta}_0',\gamma_{01},\bs{\delta}_1')$ and $\partial\overline{\cM^1}_{\Sigma_{12}}(\xi_{12};\bs{\delta}_1'',\gamma_{12},\bs{\delta}_2'')$ give the following relations:
\begin{alignat}{1}
&\Delta_1^\Sigma\circ\Delta_1^+(\gamma_{01})+\Delta_1^\Sigma\circ\mfm_1^0(\gamma_{01})+\Delta_1^-\circ\Delta_1^\Sigma(\gamma_{01})=0\label{rel_again3}\\
&\Delta_1^\Sigma\circ\Delta_1^+(\gamma_{12})+\Delta_1^\Sigma\circ\mfm_1^0(\gamma_{12})+\Delta_1^-\circ\Delta_1^\Sigma(\gamma_{12})=0\label{rel_again4}
\end{alignat}
Observe now that in  the two last terms of the sum \eqref{rel_again2} we have $b_1^-\circ\Delta_1^\Sigma(\gamma_{12})=\mfm_1^-(\gamma_{12})$ and $b_1^-\circ\Delta_1^\Sigma(\gamma_{01})=\mfm_1^-(\gamma_{01})$ by definition of $\mfm_1^-$: Moreover in the same terms one can replace $\Delta_1^-\circ\Delta_1^\Sigma(\gamma_{12})$ and $\Delta_1^-\circ\Delta_1^\Sigma(\gamma_{01})$ by $\Delta_1^\Sigma\circ(\mfm_1^++\mfm_1^0)(\gamma_{12})$ and $\Delta_1^\Sigma\circ(\mfm_1^++\mfm_1^0)(\gamma_{01})$ respectively, using the relations \eqref{rel_again3} and \eqref{rel_again4} and the definition of $\mfm_1^+$. Finally, recall that by definition of $\bs{\Delta}_1^\Sigma$ we have $\Delta_1^\Sigma\circ(\mfm_1^++\mfm_1^0)=\bs{\Delta}_1^\Sigma\circ(\mfm_1^++\mfm_1^0)$ and $\mfm_1^-=\bs{\Delta}_1^\Sigma\circ\mfm_1^-$.  Thus, relation \eqref{rel_again2} can be rewritten:
\begin{alignat}{1}
b_2^-\big(\bs{\Delta}_1^\Sigma\otimes\bs{\Delta}_1^\Sigma\big)&\big(\mfm_1(\gamma_{12}),\gamma_{01}\big)+b_2^-\big(\bs{\Delta}_1^\Sigma\otimes\bs{\Delta}_1^\Sigma\big)\big(\gamma_{12},\mfm_1(\gamma_{01})\big)\label{rel_again5}\\
&+b_1^-\circ b_2^-(\bs{\Delta}_1^\Sigma\otimes\bs{\Delta}_1^\Sigma)(\gamma_{12},\gamma_{01})+b_1^-\circ\Delta_2^-(\Delta_1^\Sigma(\gamma_{12}),\Delta_1^\Sigma(\gamma_{01}))=0\nonumber
\end{alignat}
In order to get the Leibniz rule relation \eqref{rel++-}, we sum relations \eqref{rel_again1} and \eqref{rel_again5}, removing the term which appears twice (the last term in each of them), and get:
\small
\begin{alignat*}{1}
&b_1^-\circ b_1^-\circ\Delta_2^\Sigma(\gamma_{12},\gamma_{01})+b_1^-\circ\Delta_2^\Sigma(\mfm_1(\gamma_{12}),\gamma_{01})+b_1^-\circ\Delta_2^\Sigma(\gamma_{12},\mfm_1(\gamma_{01}))+\mfm_1^-\circ(\mfm_2^++\mfm_2^0)(\gamma_{12},\gamma_{01})\\
&+b_2^-\big(\bs{\Delta}_1^\Sigma\otimes\bs{\Delta}_1^\Sigma\big)\big(\mfm_1(\gamma_{12}),\gamma_{01}\big)+b_2^-\big(\bs{\Delta}_1^\Sigma\otimes\bs{\Delta}_1^\Sigma\big)\big(\gamma_{12},\mfm_1(\gamma_{01})\big)+b_1^-\circ b_2^-(\bs{\Delta}_1^\Sigma\otimes\bs{\Delta}_1^\Sigma)(\gamma_{12},\gamma_{01})=0
\end{alignat*}
\normalsize
By definition of $\mfm_1^-$ and $\mfm_2^-$, the sum of the first and the last term gives $\mfm_1^-\circ\mfm_2^-(\gamma_{12},\gamma_{01})$, the sum of the second and fifth term gives $\mfm_2^-(\mfm_1(\gamma_{12}),\gamma_{01})$, and the sum of the third and sixth term gives $\mfm_2^-(\gamma_{12},\mfm_1(\gamma_{01}))$. We have thus shown that relation \eqref{rel++-} holds.\\

\noindent\textbf{Leibniz rule for a pair of type (b):}

Let us consider a pair $(\gamma_{12},x_{01})$ of generators of type (b). The Leibniz rule for such a pair decomposes into the following three relations:
\begin{alignat}{1}
&\mfm_2^+(\Delta_1^+(\gamma_{12}),x_{01})+\mfm^+_2(\gamma_{12},\mfm_1(x_{01}))+\mfm_1^+\circ\mfm^+_2(\gamma_{12},x_{01})=0\label{rel+0+}\\
&\mfm_2^0(\mfm_1(\gamma_{12}),x_{01})+\mfm_2^0(\gamma_{12},\mfm_1(x_{01}))+\mfm_1^0\circ\mfm_2(\gamma_{12},x_{01})=0\label{rel+00}\\
&\mfm_2^-(\mfm_1(\gamma_{12}),x_{01})+\mfm_2^-(\gamma_{12},\mfm_1(x_{01}))+\mfm_1^-\circ\mfm_2(\gamma_{12},x_{01})=0\label{rel+0-}
\end{alignat}
where for \eqref{rel+0+} we make use of the fact that $\mfm_2^+(\mfm_1^0(\gamma_{12}),x_{01}))$ and $\mfm_2^+(\mfm_1^-(\gamma_{12}),x_{01})$ vanish by definition ($\mfm_{00}^+=\mfm_{-0}^+=0$).
The study of the boundary of the compactification of the products
\begin{alignat*}{1}
	&\widetilde{\cM^2}_{\R\times\La_{012}^+}(\gamma_{02};\bs{\zeta}_0,\xi_{10},\bs{\zeta}_1,\gamma_{12},\bs{\zeta}_2)\times\cM^0_{\Sigma_{01}}(\xi_{10};\bs{\delta}_0,x_{01},\bs{\delta}_1)\\
	&\widetilde{\cM^1}_{\R\times\La_{012}^+}(\gamma_{02};\bs{\zeta}_0,\xi_{10},\bs{\zeta}_1,\gamma_{12},\bs{\zeta}_2)\times\cM^1_{\Sigma_{01}}(\xi_{10};\bs{\delta}_0,x_{01},\bs{\delta}_1)
\end{alignat*}
gives relation \eqref{rel+0+}. In order to get relation \eqref{rel+00} we need to study the boundary of
\begin{alignat*}{1}
	\cM^1_{\Sigma_{012}}(p_{20};\bs{\delta}_0,x_{01},\bs{\delta}_1,\gamma_{12},\bs{\delta}_2)
\end{alignat*}
and finally for relation \eqref{rel+0-}, we study
\begin{alignat*}{1}
	&\widetilde{\cM^2}_{\R\times\La_{02}^-}(\gamma_{20};\bs{\delta}_0,\xi_{02},\bs{\delta}_2)\times\cM^0_{\Sigma_{012}}(\xi_{02};\bs{\delta}_0',x_{01},\bs{\delta}_1',\gamma_{12},\bs{\delta}_2')\\
	&\widetilde{\cM^1}_{\R\times\La_{02}^-}(\gamma_{20};\bs{\delta}_0,\xi_{02},\bs{\delta}_2)\times\cM^1_{\Sigma_{012}}(\xi_{02};\bs{\delta}_0',x_{01},\bs{\delta}_1',\gamma_{12},\bs{\delta}_2')\\
	&\widetilde{\cM^2}_{\R\times\La_{012}^-}(\gamma_{20};\bs{\delta}_0,\xi_{01},\bs{\delta}_1,\xi_{12},\bs{\delta}_2)\times\cM^0_{\Sigma_{01}}(\xi_{01};\bs{\delta}_0',x_{01},\bs{\delta}_1')\times\cM^0_{\Sigma_{12}}(\xi_{12};\bs{\delta}_1'',\gamma_{12},\bs{\delta}_2'')\\
	&\widetilde{\cM^1}_{\R\times\La_{012}^-}(\gamma_{20};\bs{\delta}_0,\xi_{01},\bs{\delta}_1,\xi_{12},\bs{\delta}_2)\times\cM^1_{\Sigma_{01}}(\xi_{01};\bs{\delta}_0',x_{01},\bs{\delta}_1')\times\cM^0_{\Sigma_{12}}(\xi_{12};\bs{\delta}_1'',\gamma_{12},\bs{\delta}_2'')\\
	&\widetilde{\cM^1}_{\R\times\La_{012}^-}(\gamma_{20};\bs{\delta}_0,\xi_{01},\bs{\delta}_1,\xi_{12},\bs{\delta}_2)\times\cM^0_{\Sigma_{01}}(\xi_{01};\bs{\delta}_0',x_{01},\bs{\delta}_1')\times\cM^1_{\Sigma_{12}}(\xi_{12};\bs{\delta}_1'',\gamma_{12},\bs{\delta}_2'')
\end{alignat*}

\noindent\textbf{Leibniz rule for a pair of type (c):}

Finally, for a pair $(\gamma_{12},\gamma_{10})$ of generators of type (c), we decompose the Leibniz rule into:

\begin{alignat}{1}
&\mfm_2^+(\Delta_1^+(\gamma_{12}),\gamma_{10})+\mfm_2^+(\gamma_{12},\mfm_1(\gamma_{10}))+\mfm_1^+\circ\mfm_2^+(\gamma_{12},\gamma_{10})=0\label{rel+-+}\\
&\mfm_2^0(\mfm_1(\gamma_{12}),\gamma_{10})+\mfm_2^0(\gamma_{12},\mfm_1(\gamma_{10}))+\mfm_1^0\circ\mfm_2(\gamma_{12},\gamma_{10})=0\label{rel+-0}\\
&\mfm_2^-(\mfm_1(\gamma_{12}),\gamma_{10})+\mfm_2^-(\gamma_{12},\mfm_1(\gamma_{10}))+\mfm_1^-\circ\mfm_2(\gamma_{12},\gamma_{10})=0\label{rel+--}
\end{alignat}
and observe that one of the two terms contributing to $\mfm_2^-(\gamma_{12},\mfm_1^0(\gamma_{10}))$, namely $b_2^-(\Delta_1^\Sigma(\gamma_{12}),\Delta_1^\Sigma\circ\mfm_1^0(\gamma_{10}))$, vanishes for energy reasons. Relations \eqref{rel+-+}, \eqref{rel+-0} and \eqref{rel+--} are obtained respectively by studying the boundary of the compactification of
\begin{alignat*}{1}
	&\widetilde{\cM^2}_{\R\times\La_{012}^+}(\gamma_{02};\bs{\zeta}_0,\xi_{10},\bs{\zeta}_1,\gamma_{12},\bs{\zeta}_2)\times\cM^0_{\Sigma_{01}}(\xi_{10};\bs{\delta}_0,\gamma_{10},\bs{\delta}_1)\\
	&\widetilde{\cM^1}_{\R\times\La_{012}^+}(\gamma_{02};\bs{\zeta}_0,\xi_{10},\bs{\zeta}_1,\gamma_{12},\bs{\zeta}_2)\times\cM^1_{\Sigma_{01}}(\xi_{10};\bs{\delta}_0,\gamma_{10},\bs{\delta}_1)
\end{alignat*}
of	$\cM^1_{\Sigma_{012}}(p_{20};\bs{\delta}_0,\gamma_{10},\bs{\delta}_1,\gamma_{12},\bs{\delta}_2)$,
and of
\begin{alignat*}{1}	
	&\widetilde{\cM^2}_{\R\times\La_{02}^-}(\gamma_{20};\bs{\delta}_0,\xi_{02},\bs{\delta}_2)\times\cM^0_{\Sigma_{012}}(\xi_{02};\bs{\delta}_0',\gamma_{10},\bs{\delta}_1',\gamma_{12},\bs{\delta}_2')\\
	&\widetilde{\cM^1}_{\R\times\La_{02}^-}(\gamma_{20};\bs{\delta}_0,\xi_{02},\bs{\delta}_2)\times\cM^1_{\Sigma_{012}}(\xi_{02};\bs{\delta}_0',\gamma_{10},\bs{\delta}_1',\gamma_{12},\bs{\delta}_2')\\
	&\widetilde{\cM^2}_{\R\times\La_{012}^-}(\gamma_{20};\bs{\delta}_0,\gamma_{10},\bs{\delta}_1,\xi_{12},\bs{\delta}_2)\times\cM^0_{\Sigma_{12}}(\xi_{12};\bs{\delta}_1',\gamma_{12},\bs{\delta}_2')\\
	&\widetilde{\cM^1}_{\R\times\La_{012}^-}(\gamma_{20};\bs{\delta}_0,\gamma_{10},\bs{\delta}_1,\xi_{12},\bs{\delta}_2)\times\cM^1_{\Sigma_{12}}(\xi_{12};\bs{\delta}_1',\gamma_{12},\bs{\delta}_2')
\end{alignat*}

\section{Product in the concatenation}\label{sec:prod_conc}

\subsection{Definition of the product}

Given a pair of concatenation $(V_0\odot W_0,V_1\odot W_1)$, we denote $\mfm_1^V,\mfm_1^W$ the differentials of the complexes $\Cth_+(V_0,V_1)$ and $\Cth_+(W_0,W_1)$ respectively. Given a third concatenation $V_2\odot W_2$, we denote again $\mfm_1^V$ and $\mfm_1^W$ the differentials on complexes $\Cth_+(V_i,V_j)$ and $\Cth_+(W_i,W_j)$ respectively, for $0\leq i\neq j\leq2$, without specifying the pair of cobordisms when it is clear from the context.
Moreover, we will use the transfer maps $\bs{b}_1^{V_i,V_j}:\Cth_+(V_i\odot W_i,V_j\odot W_j)\to\Cth_+(W_i,W_j)$ and $\bs{\Delta}_1^{W_i,W_j}:\Cth_+(V_i\odot W_i,V_j\odot W_j)\to\Cth_+(V_i,V_j)$ and will shorten the notations to $\bs{b}_1^V$ and $\bs{\Delta}_1^W$ as there should not be any risk of confusion about which pair of cobordisms is involved in the domain and codomain.
Finally, we denote $\mfm_2^V,\mfm_2^W$ the products $\Cth_+(V_1,V_2)\otimes\Cth_+(V_0,V_1)\to\Cth_+(V_0,V_2)$ and $\Cth_+(W_1,W_2)\otimes\Cth_+(W_0,W_1)\to\Cth_+(W_0,W_2)$ respectively. We now define a product:
\begin{alignat*}{1}
\mfm_2^{V\odot W}:\Cth_+(V_1\odot W_1,V_2\odot W_2)\otimes\Cth_+(V_0\odot W_0,V_1\odot W_1)\to\Cth_+(V_0\odot W_0,V_2\odot W_2)
\end{alignat*}
Using maps we already defined before, as well as the two inputs banana $b_2^V$ with boundary on $V_0\cup V_1\cup V_2$ (we encountered in Section \ref{section:def_product} the two inputs banana $b_2^-$ with boundary on cylindrical ends), defined by 
\begin{alignat*}{1}
	&b_2^V:\Cth_+(V_1,V_2)\otimes\Cth_+(V_0,V_1)\to C^{*-1}(\La_0,\La_1)\\
	&b_2^V(a_2,a_1)=\sum_{\gamma_{20},\bs{\delta}_i}\#\cM^0_{V_{012}}(\gamma_{20};\bs{\delta}_0,a_1,\bs{\delta}_1,a_2,\bs{\delta}_2)\cdot\ep_i^-(\bs{\delta}_i)\cdot\gamma_{20}
\end{alignat*}
we set:
\begin{alignat*}{2}
\mfm_2^{V\odot W}&=\mfm_2^{W,+0}\big(\bs{b}_1^V\otimes\bs{b}_1^V\big)+\mfm_1^{W,+0}\circ\,b_1^V\circ\Delta_2^W\big(\bs{b}_1^V\otimes\bs{b}_1^V\big)+\mfm_1^{W,+0}\circ\,b_2^V\big(\bs{\Delta}_1^W\otimes\bs{\Delta}_1^W\big)\\
&\,+\mfm_2^{V,0-}\big(\bs{\Delta}_1^W\otimes\bs{\Delta}_1^W\big)+\mfm_1^{V,0-}\circ\,\Delta_2^W\big(\bs{b}_1^V\otimes\bs{b}_1^V\big)
\end{alignat*}
where $\mfm_i^{W,+0}=\mfm_i^{W,+}+\mfm_i^{W,0}$, $i=1,2$ is the component of $\mfm_i^W$ with values in $C(\La_2^+,\La_0^+)\oplus CF(W_0,W_2)$, and $\mfm_i^{V,0-}=\mfm_i^{V,0}+\mfm_i^{V,-}$, $i=1,2$, is the component of $\mfm_i^V$ with values in $CF(V_0,V_2)\oplus C(\La_0^-,\La_2^-)$. Observe that $\mfm_1^{W,+}\circ\,b_1^V\circ\Delta_2^W\big(\bs{b}_1^V\otimes\bs{b}_1^V\big)$ and $\mfm_1^{W,+}\circ\,b_2^V\big(\bs{\Delta}_1^W\otimes\bs{\Delta}_1^W\big)$ vanish, but we keep it in the formula to make it look more homogeneous, which helps a bit to check the Leibniz rule in the next section.

\subsection{Leibniz rule}\label{sec:LCconc}
This section is dedicated in proving that the map $\mfm_2^{V\odot W}$ satisfies the Leibniz rule with respect to $\mfm_1^{V\odot W}$. This is just computation.
We want to show
\begin{alignat*}{1}
\mfm_2^{V\odot W}(\mfm_1^{V\odot W}\otimes\id)+\mfm_2^{V\odot W}(\id\otimes\mfm_1^{V\odot W})+\mfm_1^{V\odot W}\circ\mfm_2^{V\odot W}=0
\end{alignat*}
We will actually decompose it into two equations:
\begin{alignat}{1}
&\mfm_2^{V\odot W,+0_W}(\mfm_1^{V\odot W}\otimes\id)+\mfm_2^{V\odot W,+0_W}(\id\otimes\mfm_1^{V\odot W})+\mfm_1^{V\odot W,+0_W}\circ\mfm_2^{V\odot W}=0\label{Leibniz_conc1}\\
&\mfm_2^{V\odot W,0_V-}(\mfm_1^{V\odot W}\otimes\id)+\mfm_2^{V\odot W,0_V-}(\id\otimes\mfm_1^{V\odot W})+\mfm_1^{V\odot W,0_V-}\circ\mfm_2^{V\odot W}=0\label{Leibniz_conc2}
\end{alignat}
The first one corresponds to the components of the Leibniz rule taking values in $C(\La_2^+,\La_0^+)\oplus CF(W_0,W_2)$, and the second one to the components taking values in $CF(V_0,V_2)\oplus C(\La_0^-,\La_2^-)$.\\

In the proof of the Leibniz rule, we will refer to the following equations:
\begin{alignat}{1}
&\mfm_2^{W,+0}(\mfm_1^W\otimes\id)+\mfm_2^{W,+0}(\id\otimes\mfm_1^W)+\mfm_1^{W,+0}\circ\mfm_2^W=0\label{eq000}\\
&\mfm_2^{V,0-}(\mfm_1^V\otimes\id)+\mfm_2^{V,0-}(\id\otimes\mfm_1^V)+\mfm_1^{V,0-}\circ\mfm_2^V=0\label{eq0000}\\
&\Delta_2^W(\mfm_1^W\otimes\id)+\Delta_2^W(\id\otimes\mfm_1^W)+\Delta_1^W\circ\mfm_2^W+\Delta_2^\Lambda(\bs{\Delta}_1^W\otimes\bs{\Delta}_1^W)+\Delta_1^\Lambda\circ\Delta_2^W=0\label{eq3}\\
&b_2^V(\mfm_1^V\otimes\id)+b_2^V(\id\otimes\mfm_1^V)+b_1^V\circ\mfm_2^V+\,b_2^\La(\bs{b}_1^V\otimes\bs{b}_1^V)+b_1^\La\circ b_2^V=0\label{eq4}\\
&\bs{b}_1^V\circ\bs{\Delta}_1^W=\bs{\Delta}_1^W\circ\bs{b}_1^V\label{eq5}
\end{alignat} 
Equations \eqref{eq000} and \eqref{eq0000} come from the fact that $\mfm_2^W$ and $\mfm_2^V$ satisfy the Leibniz rule. 
Equations \eqref{eq3} and \eqref{eq4} (for other Lagrangian boundary conditions) appear implicitly in Section \ref{sec:Leibniz}: they come respectively from the study the boundary of the compactification of moduli spaces
\begin{alignat*}{1}
\cM^1_{W_{012}}(\gamma_{02};\bs{\delta}_0^W,a_1^W,\bs{\delta}_1^W,a_2^W,\bs{\delta}_2^W)\,\,\mbox{ and }\,\,\cM^1_{V_{012}}(\gamma_{20},\bs{\delta}_0^V,a_1^V,\bs{\delta}_1^V,a_2^V,\bs{\delta}_2^V), 
\end{alignat*}
for $\gamma_{02}\in C(\La_2,\La_0)$, $\gamma_{20}\in C(\La_0,\La_2)$, $(a_2^W,a_1^W)\in\Cth_+(W_1,W_2)\otimes\Cth_+(W_0,W_1)$, $(a_2^V,a_1^V)\in\Cth_+(V_1,V_2)\otimes\Cth_+(V_0,V_1)$, $\bs{\delta}_i^W$ words of pure Reeb chords of $\La_i$, $\bs{\delta}_i^V$ words of pure Reeb chords of $\La_i^-$.
Finally, Equation \eqref{eq5} is the content of Lemma \ref{bcircdelta}.

\subsubsection{Equation \eqref{Leibniz_conc1}}\label{sec:LC1}

Let us write the left-hand side of Equation \eqref{Leibniz_conc1} as (LR1), i.e. \eqref{Leibniz_conc1} $\Leftrightarrow$ (LR1)=0.
We start by developing the first term of (LR1), using the definition of $\mfm_2^{V\odot W}$ and the fact that $\bs{b}_1^V$ and $\bs{\Delta}_1^W$ are chain maps:

\begin{alignat*}{1}
\mfm&_2^{V\odot W,+0_W}(\mfm_1^{V\odot W}\otimes\id)\\
=&\big(\mfm_2^{W,+0}+\mfm_1^{W,+0}\circ\, b_1^V\circ\Delta_2^W\big)\Big[\bs{b}_1^V\circ\mfm_1^{V\odot W}\otimes\,\bs{b}_1^V\Big]+\mfm_1^{W,+0}\circ\,b_2^V\Big[\bs{\Delta}_1^W\circ\mfm_1^{V\odot W}\otimes\bs{\Delta}_1^W\Big]\\
=&\big(\mfm_2^{W,+0}+\mfm_1^{W,+0}\circ\, b_1^V\circ\Delta_2^W\big)\Big[\mfm_1^W\circ\,\bs{b}_1^V\otimes\,\bs{b}_1^V\Big]+\mfm_1^{W,+0}\circ\,b_2^V\Big[\mfm_1^V\circ\,\bs{\Delta}_1^W\otimes\bs{\Delta}_1^W\Big]\\
=&\mfm_2^{W,+0}\big(\mfm_1^{W}\otimes\id\big)\big[\bs{b}_1^V\otimes\,\bs{b}_1^V\big]+\mfm_1^{W,+0}\circ\, b_1^V\circ\Delta_2^W\big(\mfm_1^{W}\otimes\id\big)\big[\bs{b}_1^V\otimes\,\bs{b}_1^V\big]\\
&+\mfm_1^{W,+0}\circ\,b_2^V\big(\mfm_1^{V}\otimes\id\big)\big[\bs{\Delta}_1^W\otimes\bs{\Delta}_1^W\big]
\end{alignat*}
One decomposes similarly the symmetric term $\mfm_2^{V\odot W,+0_W}(\id\otimes\mfm_1^{V\odot W})$.
Now let us take a look at $\mfm_1^{V\odot W,+0_W}\circ\mfm_2^{V\odot W}$. We have
\begin{alignat*}{1}
\mfm&_1^{V\odot W,+0_W}\circ\mfm_2^{V\odot W}=\mfm_1^{W,+0}\circ\bs{b}_1^V\circ\mfm_2^{V\odot W}\\
=&\mfm_1^{W,+0}\circ\,\bs{b}_1^V\big(\mfm_2^{V\odot W,+0_W}+\mfm_2^{V\odot W,0_V-}\big)\\
=&\mfm_1^{W,+0}\big(\mfm_2^{V\odot W,+0_W}+b_1^V\circ\Delta_1^W\circ\mfm_2^{V\odot W,+0_W}+b_1^V\circ\mfm_2^{V\odot W,0_V-}\big)\\
=&\big(\mfm_1^{W,+0}+\mfm_1^{W,+0}\circ\,b_1^V\circ\Delta_1^W\big)\Big[\mfm_2^{W,+0}\big(\bs{b}_1^V\otimes\bs{b}_1^V\big)+\mfm_1^{W,+0}\circ\, b_1^V\circ\Delta_2^W\big(\bs{b}_1^V\otimes\bs{b}_1^V\big)\\
&+\mfm_1^{W,+0}\circ \,b_2^V\big(\bs{\Delta}_1^W\otimes\bs{\Delta}_1^W\big)\Big]+\mfm_1^{W,+0}\circ\,b_1^V\Big[\mfm_2^{V,0-}\big(\bs{\Delta}_1^W\otimes\bs{\Delta}_1^W\big)+\mfm_1^{V,0-}\circ\,\Delta_2^W\big(\bs{b}_1^V\otimes\bs{b}_1^V\big)\Big]
\end{alignat*}
The term $\mfm_1^{W,+0}\circ\,b_1^V\circ\Delta_1^W\circ\mfm_1^{W,+0}\circ\, b_1^V\circ\Delta_2^W\big(\bs{b}_1^V\otimes\bs{b}_1^V\big)$ vanishes for energy reasons, as well as $\mfm_1^{W,+0}\circ\,b_1^V\circ\Delta_1^W\circ\mfm_1^{W,+0}\circ \,b_2^V\big(\bs{\Delta}_1^W\otimes\bs{\Delta}_1^W\big)$, hence we finally get:
\begin{alignat*}{1}
\mfm_1^{V\odot W,+0_W}\circ\mfm_2^{V\odot W}=&\mfm_1^{W,+0}\circ\mfm_2^{W,+0}\big(\bs{b}_1^V\otimes\bs{b}_1^V\big)+\mfm_1^{W,+0}\circ\,b_1^V\circ\Delta_1^W\circ\mfm_2^{W,+0}\big(\bs{b}_1^V\otimes\bs{b}_1^V\big)\\
+&\mfm_1^{W,+0}\circ\mfm_1^{W,+0}\circ\, b_1^V\circ\Delta_2^W\big(\bs{b}_1^V\otimes\bs{b}_1^V\big)+\mfm_1^{W,+0}\circ\mfm_1^{W,+0}\circ \,b_2^V\big(\bs{\Delta}_1^W\otimes\bs{\Delta}_1^W\big)\\
+&\mfm_1^{W,+0}\circ\,b_1^V\circ\mfm_2^{V,0-}\big(\bs{\Delta}_1^W\otimes\bs{\Delta}_1^W\big)+\mfm_1^{W,+0}\circ\,b_1^V\circ\mfm_1^{V,0-}\circ\,\Delta_2^W\big(\bs{b}_1^V\otimes\bs{b}_1^V\big)
\end{alignat*}
Summing all together gives:
\begin{alignat*}{1}
&\text{(LR1)}=\\
&\big[\mfm_1^{W,+0}\circ\mfm_2^{W,+0}+\mfm_2^{W,+0}(\mfm_1^{W}\otimes\id)+\mfm_2^{W,+0}(\id\otimes\mfm_1^{W})\big]\big(\bs{b}_1^V\otimes\bs{b}_1^V\big)\tag{L1}\\
&+\big[\mfm_1^{W,+0}\circ\,b_1^V\big]\big[\Delta_1^W\circ\mfm_2^{W,+0}+\Delta_2^W(\mfm_1^{W}\otimes\id)+\Delta_2^W(\id\otimes\mfm_1^{W})\big]\big(\bs{b}_1^V\otimes\bs{b}_1^V\big)\tag{L2}\\
&+\mfm_1^{W,+0}\circ\mfm_1^{W,+0}\circ\, b_1^V\circ\Delta_2^W\big(\bs{b}_1^V\otimes\bs{b}_1^V\big)\tag{L3}\\
&+\mfm_1^{W,+0}\circ\mfm_1^{W,+0}\circ \,b_2^V\big(\bs{\Delta}_1^W\otimes\bs{\Delta}_1^W\big)\tag{L4}\\
&+\mfm_1^{W,+0}\big[b_1^V\circ\mfm_2^{V,0-}+b_2^V(\mfm_1^{V}\otimes\id)+b_2^V(\id\otimes\mfm_1^{V})\big]\big(\bs{\Delta}_1^W\otimes\bs{\Delta}_1^W\big)\tag{L5}\\
&+\mfm_1^{W,+0}\circ\,b_1^V\circ\mfm_1^{V,0-}\circ\,\Delta_2^W\big(\bs{b}_1^V\otimes\bs{b}_1^V\big)\tag{L6}
\end{alignat*}
Now we use Equation \eqref{eq000} on (L1), Equation \eqref{eq3} on (L2), the fact that $\mfm_1^{W,+0}\circ\mfm_1^{W,+0}=\mfm_1^{W,+0}\circ\mfm_1^{W,-}$ on (L3) and (L4) and finally Equation \eqref{eq4} on (L5), to write
\begin{alignat*}{1}
&\text{(LR1)}=\\
&\mfm_1^{W,+0}\circ\mfm_2^{W,-}\big(\bs{b}_1^V\otimes\bs{b}_1^V\big)\tag{L1'}\\
&+\mfm_1^{W,+0}\circ\,b_1^V\big[\Delta_1^W\circ\mfm_2^{W,-}+\Delta_2^\La\big(\bs{\Delta}_1^W\otimes\bs{\Delta}_1^W)+\Delta_1^\La\circ\Delta_2^W\big]\big(\bs{b}_1^V\otimes\bs{b}_1^V\big)\tag{L2'}\\
&+\mfm_1^{W,+0}\circ\mfm_1^{W,-}\circ\, b_1^V\circ\Delta_2^W\big(\bs{b}_1^V\otimes\bs{b}_1^V\big)\tag{L3'}\\
&+\mfm_1^{W,+0}\circ\mfm_1^{W,-}\circ \,b_2^V\big(\bs{\Delta}_1^W\otimes\bs{\Delta}_1^W\big)\tag{L4'}\\
&+\mfm_1^{W,+0}\big[b_1^V\circ\mfm_2^{V,+}+b_2^\La(\bs{b}_1^V\otimes\bs{b}_1^V)+\,b_1^\La\circ b_2^V\big]\big(\bs{\Delta}_1^W\otimes\bs{\Delta}_1^W\big)\tag{L5'}\\
&+\mfm_1^{W,+0}\circ\,b_1^V\circ\mfm_1^{V,0-}\circ\,\Delta_2^W\big(\bs{b}_1^V\otimes\bs{b}_1^V\big)\tag{L6'}
\end{alignat*}
We apply then the following modifications:
\begin{enumerate}
	\item on (L1') we write $\mfm_2^{W,-}=b_1^\La\circ\Delta_2^W+b_2^\La(\bs{\Delta}_1^W\otimes\bs{\Delta}_1^W)$,
	\item on (L2') observe that $\Delta_1^W\circ\mfm_2^{W,-}$ vanishes for energy reasons,
	\item on (L3') and (L4') we have $\mfm_1^{W,-}\circ\,b_1^V=b_1^\La\circ\bs{\Delta}_1^W\circ b_1^V=b_1^\La\circ b_1^V$ and $\mfm_1^{W,-}\circ\,b_2^V=b_1^\La\circ\bs{\Delta}_1^W\circ b_2^V=b_1^\La\circ b_2^V$,
	\item on (L5'), we write $\mfm_2^{V,+}=\Delta_2^\La(\bs{b}_1^V\otimes\bs{b}_1^V)$,
	\item finally, in the last term of (L2') we have $\Delta_1^\La=\mfm_1^{V,+}$ so adding it to (L6') gives $\mfm_1^{W,+0}\circ\,b_1^V\circ\mfm_1^{V}\circ\,\Delta_2^W\big(\bs{b}_1^V\otimes\bs{b}_1^V\big)$. But observe that  by definition of $\bs{\Delta}_1^W$ one has $b_1^V\circ\mfm_1^{V}\circ\,\Delta_2^W=b_1^V\circ\mfm_1^{V}\circ\,\bs{\Delta}_1^W\circ\Delta_2^W$, which gives by Lemma \ref{lem:rel} $b_1^\La\circ\bs{\Delta}_1^W\circ\bs{b}_1^V\circ\Delta_2^W$. We thus get
	$$\mfm_1^{W,+0}\circ\,b_1^V\circ\mfm_1^{V}\circ\,\Delta_2^W\big(\bs{b}_1^V\otimes\bs{b}_1^V\big)=\mfm_1^{W,+0}\circ\,b_1^\La\circ\bs{\Delta}_1^W\circ\bs{b}_1^V\circ\Delta_2^W\big(\bs{b}_1^V\otimes\bs{b}_1^V\big)$$
	 which, using Equation \eqref{eq5}, is equal to:
	 \begin{alignat*}{1}
	 \mfm_1^{W,+0}\circ\,b_1^\La\circ\bs{b}_1^V\circ\bs{\Delta}_1^W\circ\Delta_2^W\big(\bs{b}_1^V\otimes\bs{b}_1^V\big)&=\mfm_1^{W,+0}\circ\,b_1^\La\circ\bs{b}_1^V\circ\Delta_2^W\big(\bs{b}_1^V\otimes\bs{b}_1^V\big)\\
	 &=\big[\mfm_1^{W,+0}\circ\,b_1^\La\big]\big[\Delta_2^W+b_1^V\circ\Delta_2^W\big]\big(\bs{b}_1^V\otimes\bs{b}_1^V\big)
	 \end{alignat*}
\end{enumerate}
So finally we have:
\begin{alignat*}{1}
\text{(LR1)}&=\mfm_1^{W,+0}\circ b_1^\La\circ\Delta_2^W\big(\bs{b}_1^V\otimes\bs{b}_1^V\big)\tag{R1}\\
&+\mfm_1^{W,+0}\circ b_2^\La\big(\bs{\Delta}_1^W\circ\bs{b}_1^V\otimes\bs{\Delta}_1^W\circ\bs{b}_1^V\big)\tag{R2}\\
&+\mfm_1^{W,+0}\circ\,b_1^V\circ\Delta_2^\La\big(\bs{\Delta}_1^W\circ\bs{b}_1^V\otimes\bs{\Delta}_1^W\circ\bs{b}_1^V\big)\tag{R3}\\
&+\mfm_1^{W,+0}\circ b_1^\La\circ\, b_1^V\circ\Delta_2^W\big(\bs{b}_1^V\otimes\bs{b}_1^V\big)\tag{R4}\\
&+\mfm_1^{W,+0}\circ b_1^\La\circ \,b_2^V\big(\bs{\Delta}_1^W\otimes\bs{\Delta}_1^W\big)\tag{R5}\\
&+\mfm_1^{W,+0}\circ\,b_1^V\circ\Delta_2^\La\big(\bs{b}_1^V\circ\bs{\Delta}_1^W\otimes\bs{b}_1^V\circ\bs{\Delta}_1^W\big)\tag{R6}\\
&+\mfm_1^{W,+0}\circ\,b_2^\La\big(\bs{b}_1^V\circ\bs{\Delta}_1^W\otimes\bs{b}_1^V\circ\bs{\Delta}_1^W\big)\tag{R7}\\
&+\mfm_1^{W,+0}\circ\,b_1^\La\circ b_2^V\big(\bs{\Delta}_1^W\otimes\bs{\Delta}_1^W\big)\tag{R8}\\
&+\mfm_1^{W,+0}\circ\,b_1^\La\circ\,\Delta_2^W\big(\bs{b}_1^V\otimes\bs{b}_1^V\big)+\mfm_1^{W,+0}\circ\,b_1^\La\circ\,b_1^{V}\circ\,\Delta_2^W\big(\bs{b}_1^V\otimes\bs{b}_1^V\big)\tag{R9}
\end{alignat*}
We have (R1)+(R4)+(R9)=0 and (R5)+(R8)=0. Then, using Equation \eqref{eq5} gives (R2)+(R7)=0 and (R3)+(R6)=0. Thus (LR1)=0.

\subsubsection{Equation \eqref{Leibniz_conc2}}\label{sec:LC2}

Denote (LR2) the left-hand side of Equation \eqref{Leibniz_conc2} so that this equation is equivalent to (LR2)=0. Using again the fact that $\bs{b}_1^V$ and $\bs{\Delta}_1^W$ are chain maps, the first term of (LR2) is:
\begin{alignat*}{1}
\mfm&_2^{V\odot W,0_V-}(\mfm_1^{V\odot W}\otimes\id)\\
=&\mfm_2^{V,0-}\big(\bs{\Delta}_1^W\circ\mfm_1^{V\odot W}\otimes\,\bs{\Delta}_1^W\big)
+\mfm_1^{V,0-}\circ\,\Delta_2^W\big(\bs{b}_1^V\circ\mfm_1^{V\odot W}\otimes\bs{b}_1^V\big)\\
=&\mfm_2^{V,0-}\big(\mfm_1^V\circ\,\bs{\Delta}_1^W\otimes\,\bs{\Delta}_1^W\big)
+\mfm_1^{V,0-}\circ\,\Delta_2^W\big(\mfm_1^W\circ\,\bs{b}_1^V\otimes\bs{b}_1^V\big)\\
=&\mfm_2^{V,0-}\big(\mfm_1^V\otimes\id\big)\big(\bs{\Delta}_1^W\otimes\bs{\Delta}_1^W\big)
+\mfm_1^{V,0-}\circ\,\Delta_2^W\big(\mfm_1^W\otimes\id\big)\big(\bs{b}_1^V\otimes\bs{b}_1^V\big)
\end{alignat*}
One writes analogously the symmetric term $\mfm_2^{V\odot W,0_V-}(\id\otimes\mfm_1^{V\odot W})$.
Now let us consider the third term of (LR2):
\begin{alignat*}{2}
\mfm&_1^{V\odot W,0_V-}\circ\mfm_2^{V\odot W}=\mfm_1^{V,0-}\circ\bs{\Delta}_1^W\circ\mfm_2^{V\odot W}\\
=&\mfm_1^{V,0-}\circ\,\Delta_1^W\circ\mfm_2^{V\odot W,+0_W}+\mfm_1^{V,0-}\circ\mfm_2^{V\odot W,0_V-}\\
=&\mfm_1^{V,0-}\circ\,\Delta_1^W\Big[\mfm_2^{W,+0}\big(\bs{b}_1^V\otimes\bs{b}_1^V\big)+\mfm_1^{W,+0}\circ\,b_1^V\circ\Delta_2^W\big(\bs{b}_1^V\otimes\bs{b}_1^V\big)+\mfm_1^{W,+0}\circ\,b_2^V\big(\bs{\Delta}_1^W\otimes\bs{\Delta}_1^W\big)\Big]\\
&+\mfm_1^{V,0-}\circ\mfm_2^{V,0-}\big[\bs{\Delta}_1^W\otimes\bs{\Delta}_1^W\big]+\mfm_1^{V,0-}\circ\mfm_1^{V,0-}\circ\Delta_2^W\big[\bs{b}_1^V\otimes\bs{b}_1^V\big]
\end{alignat*}
The term $\mfm_1^{V,0-}\circ\bs{\Delta}_1^W\Big[\mfm_1^{W,+0}\circ\,b_1^V\circ\Delta_2^W\big(\bs{b}_1^V\otimes\bs{b}_1^V\big)+\mfm_1^{W,+0}\circ\,b_2^V\big(\bs{\Delta}_1^W\otimes\bs{\Delta}_1^W\big)\Big]$ vanishes for energy reasons. Then, observe that $\mfm_1^{V,0-}\circ\mfm_1^{V,0-}=\mfm_1^{V,0-}\circ\mfm_1^{V,+}$ and $\Delta_1^W\circ\mfm_2^{W,+0}=\Delta_1^W\circ\mfm_2^{W}$ because $\Delta_1^W\circ\mfm_2^{W,-}=0$, so we have:
\begin{alignat*}{2}
\mfm_1^{V\odot W,0_V-}\circ\mfm_2^{V\odot W}&=\mfm_1^{V,0-}\circ\bs{\Delta}_1^W\circ\mfm_2^{W}\big(\bs{b}_1^V\otimes\bs{b}_1^V\big)+\mfm_1^{V,0-}\circ\mfm_2^{V,0-}\big[\bs{\Delta}_1^W\otimes\bs{\Delta}_1^W\big]\\
&+\mfm_1^{V,0-}\circ\mfm_1^{V,+}\circ\,\Delta_2^W\big[\bs{b}_1^V\otimes\bs{b}_1^V\big]
\end{alignat*}
Summing all together gives
\begin{alignat*}{1}
\text{(LR2)}=&
\mfm_1^{V,0-}\big[\bs{\Delta}_1^W\circ\mfm_2^{W}+\Delta_2^W(\mfm_1^W\otimes\id)+\Delta_2^W(\id\otimes\mfm_1^W)\big]\big(\bs{b}_1^V\otimes\bs{b}_1^V\big)\tag{L1}\\
+&\big[\mfm_1^{V,0-}\circ\mfm_2^{V,0-}+\mfm_2^{V,0-}(\mfm_1^V\otimes\id)+\mfm_2^{V,0-}(\id\otimes\mfm_1^V)\big]\big(\bs{\Delta}_1^W\otimes\bs{\Delta}_1^W\big)\tag{L2}\\
+&\mfm_1^{V,0-}\circ\mfm_1^{V,+}\circ\,\Delta_2^W\big(\bs{b}_1^V\otimes\bs{b}_1^V\big)
\end{alignat*}
Use then Equation \eqref{eq3} on (L1) and Equation \eqref{eq0000} on (L2) to get
\begin{alignat*}{1}
\text{(LR2)}=&\mfm_1^{V,0-}\big[\Delta_2^\La(\bs{\Delta}_1^W\otimes\bs{\Delta}_1^W)+\Delta_1^\La\circ\Delta_2^W\big]\big(\bs{b}_1^V\otimes\bs{b}_1^V\big)\tag{L1'}\\
&+\mfm_1^{V,0-}\circ\mfm_2^{V,+}\big[\bs{\Delta}_1^W\otimes\bs{\Delta}_1^W\big]\tag{L2'}\\
&+\mfm_1^{V,0-}\circ\mfm_1^{V,+}\circ\Delta_2^W\big[\bs{b}_1^V\otimes\bs{b}_1^V\big]\tag{T3}
\end{alignat*}
But remark that $\mfm_1^{V,+}=\Delta_1^\La$ and  $\mfm_2^{V,+}=\Delta_2^\La(\bs{b}_1^V\otimes\bs{b}_1^V)$ by definition so then by Equation \eqref{eq5}, we get (LR2)=0.

\subsection{Functoriality of the transfer maps}\label{sec:functoriality}

In this section, we prove that the product structures behave well under the transfer maps $\bs{b}_1^V$ and $\bs{\Delta}_1^W$. Namely, we have first the following:

\begin{prop}\label{prop:transferf1}
	The map induced by $\bs{b}_1^V$ in homology preserves the product structures, in other words we have:
	$\bs{b}_1^{V_{02}}\circ\mfm_2^{V\odot W}=\mfm_2^W\big(\bs{b}_1^{V_{12}}\otimes\bs{b}_1^{V_{01}}\big)$ in homology.
\end{prop}
\begin{proof}
Given a triple $(V_0\odot W_0,V_1\odot W_1,V_2\odot W_2)$, we define a map
\begin{alignat*}{1}
	\bs{b}_2^V:\Cth_+(V_1\odot W_1,V_2\odot W_2)\otimes \Cth_+(V_0\odot W_0,V_1\odot W_1)\to\Cth_+(W_0,W_2)
\end{alignat*}
by $$\bs{b}_2^V=b_1^V\circ\Delta_2^W\big(\bs{b}_1^V\otimes\bs{b}_1^V\big)+b_2^V\big(\bs{\Delta}_1^W\otimes\bs{\Delta}_1^W\big)$$
In order to prove the proposition, we prove that the following relation is satisfied:
\begin{alignat}{1}
	\bs{b}_2^V\big(\mfm_1^{V\odot W}\otimes\id\big)+\bs{b}_2^V\big(\id\otimes\mfm_1^{V\odot W}\big)+\bs{b}_1^V\circ\mfm_2^{V\odot W}+\mfm_2^W\big(\bs{b}_1^V\otimes\bs{b}_1^V\big)+\mfm_1^W\circ\,\bs{b}_2^V=0\label{phi_2}
\end{alignat}
Let us first consider $\bs{b}_2^V\big(\mfm_1^{V\odot W}\otimes\id\big)$. We have
\begin{alignat*}{2}
	\bs{b}_2^V\big(\mfm_1^{V\odot W}\otimes\id\big)=&\,b_1^V\circ\Delta_2^W\big[\bs{b}_1^V\circ\mfm_1^{V\odot W}\otimes\,\bs{b}_1^V\big]+b_2^V\big[\bs{\Delta}_1^W\circ\mfm_1^{V\odot W}\otimes\bs{\Delta}_1^W\big]\\
	=&b_1^V\circ\Delta_2^W\big(\mfm_1^W\circ\,\bs{b}_1^V\otimes\,\bs{b}_1^V\big)+b_2^V\big(\mfm_1^V\circ\bs{\Delta}_1^W\otimes\bs{\Delta}_1^W\big)\\
	=&b_1^V\circ\Delta_2^W\big(\mfm_1^{W}\otimes\id\big)\big(\bs{b}_1^V\otimes\bs{b}_1^V\big)+b_2^V\big(\mfm_1^{V}\otimes\id\big)\big(\bs{\Delta}_1^W\otimes\bs{\Delta}_1^W\big)
\end{alignat*}
Then we consider $\bs{b}_1^V\circ\mfm_2^{V\odot W}$. Observe that we have already computed this term in Section \ref{sec:LC1} when considering the term $\mfm_1^{V\odot W,+0_W}\circ\mfm_2^{V\odot W}$. So recall that we have
\begin{alignat*}{2}
	\bs{b}_1^V\circ\mfm_2^{V\odot W}&=\mfm_2^{W,+0}\big(\bs{b}_1^V\otimes\bs{b}_1^V\big)+\mfm_1^{W,+0}\circ\,b_1^V\circ\Delta_2^W\big(\bs{b}_1^V\otimes\bs{b}_1^V\big)+\mfm_1^{W,+0}\circ\,b_2^V\big(\bs{\Delta}_1^W\otimes\bs{\Delta}_1^W\big)\\
	&+b_1^V\circ\Delta_1^W\circ\mfm_2^{W,+0}\big(\bs{b}_1^V\otimes\bs{b}_1^V\big)+\,b_1^V\circ\mfm_2^{V,0-}\big(\bs{\Delta}_1^W\otimes\bs{\Delta}_1^W\big)+b_1^V\circ\mfm_1^{V,0-}\circ\Delta_2^W\big(\bs{b}_1^V\otimes\bs{b}_1^V\big)
\end{alignat*}
The left-hand side of \eqref{phi_2}, rearranging terms according to the decompositions above is thus given by
\begin{alignat*}{1}
&b_1^V\big(\Delta_2^W\big(\mfm_1^{W}\otimes\id\big)+\Delta_2^W\big(\id\otimes\mfm_1^{W}\big)+\Delta_1^W\circ\mfm_2^{W,+0}\big)\big(\bs{b}_1^V\otimes\bs{b}_1^V\big)\tag{L1}\\
&+\big(b_2^V\big(\mfm_1^{V}\otimes\id\big)+b_2^V\big(\id\otimes\mfm_1^{V}\big)+b_1^V\circ\mfm_2^{V,0-}\big)\big(\bs{\Delta}_1^W\otimes\bs{\Delta}_1^W\big)\tag{L2}\\
&+\big(\mfm_2^{W,+0}+\mfm_2^W\big)\big(\bs{b}_1^V\otimes\bs{b}_1^V\big)\tag{L3}\\
&+\big(\mfm_1^{W,+0}\circ\,b_1^V\circ\Delta_2^W+\mfm_1^W\circ\,b_1^V\circ\Delta_2^W\big)\big(\bs{b}_1^V\otimes\bs{b}_1^V\big)\tag{L4}\\
&+\big(\mfm_1^{W,+0}\circ\,b_2^V+\mfm_1^W\circ\,b_2^V\big)\big(\bs{\Delta}_1^W\otimes\bs{\Delta}_1^W\big)\tag{L5}\\
&+b_1^V\circ\mfm_1^{V,0-}\circ\Delta_2^W\big(\bs{b}_1^V\otimes\bs{b}_1^V\big)\tag{L6}\\
\end{alignat*}
We use now Equation \eqref{eq3} on (L1), Equation \eqref{eq4} on (L2) and the same modification as the point 5. on Section \ref{sec:LC1} on line (L6) to rewrite:
\begin{alignat*}{1}
&\big(b_1^V\circ\Delta_2^\La\big(\bs{\Delta}_1^W\otimes\bs{\Delta}_1^W\big)+b_1^V\circ\Delta_1^\La\circ\Delta_2^W\big)\big(\bs{b}_1^V\otimes\bs{b}_1^V\big)\\
&+\big(b_1^V\circ\mfm_2^{V,+}+\,b_2^\La(\bs{b}_1^V\otimes\bs{b}_1^V)+b_1^\La\circ b_2^V\big)\big(\bs{\Delta}_1^W\otimes\bs{\Delta}_1^W\big)\\
&+\mfm_2^{W,-}\big(\bs{b}_1^V\otimes\bs{b}_1^V\big)\\
&+\mfm_1^{W,-}\circ\,b_1^V\circ\Delta_2^W\big(\bs{b}_1^V\otimes\bs{b}_1^V\big)\\
&+\mfm_1^{W,-}\circ\,b_2^V\big(\bs{\Delta}_1^W\otimes\bs{\Delta}_1^W\big)\\	&+\big(b_1^V\circ\mfm_1^{V,+}\circ\,\Delta_2^W+\,b_1^\La\circ\bs{b}_1^V\circ\Delta_2^W\big)\big(\bs{b}_1^V\otimes\bs{b}_1^V\big)
\end{alignat*}
Finally, using
\begin{enumerate}
	\item $\mfm_2^{V,+}=\Delta_2^\La(\bs{b}_1^V\otimes\bs{b}_1^V)$ and $\mfm_1^{V,+}=\Delta_1^\La$,
	\item $\mfm_1^{W,-}\circ\,b_1^V=b_1^\La\circ\bs{\Delta}_1^W\circ\,b_1^V=b_1^\La\circ b_1^V$, and also $\mfm_1^{W,-}\circ\,b_2^V=b_1^\La\circ b_2^V$,
	\item $b_1^\La\circ\bs{b}_1^V\circ\Delta_2^W=b_1^\La\circ\Delta_2^W+b_1^\La\circ b_1^V\circ\Delta_2^W$
\end{enumerate}
we rewrite
\begin{alignat*}{1}
&\big(b_1^V\circ\Delta_2^\La\big(\bs{\Delta}_1^W\otimes\bs{\Delta}_1^W\big)+b_1^V\circ\Delta_1^\La\circ\Delta_2^W\big)\big(\bs{b}_1^V\otimes\bs{b}_1^V\big)\\
&+\big(b_1^V\circ\Delta_2^\La(\bs{b}_1^V\otimes\bs{b}_1^V)+\,b_2^\La(\bs{b}_1^V\otimes\bs{b}_1^V)+b_1^\La\circ b_2^V\big)\big(\bs{\Delta}_1^W\otimes\bs{\Delta}_1^W\big)\\
&+\big(b_1^\La\circ\Delta_2^W+b_2^\La(\bs{\Delta}_1^W\otimes\bs{\Delta}_1^W)\big)\big(\bs{b}_1^V\otimes\bs{b}_1^V\big)\\
&+b_1^\La\circ\,b_1^V\circ\Delta_2^W\big(\bs{b}_1^V\otimes\bs{b}_1^V\big)\\
&+b_1^\La\circ\,b_2^V\big(\bs{\Delta}_1^W\otimes\bs{\Delta}_1^W\big)\\	&+\big(b_1^V\circ\Delta_1^\La\circ\,\Delta_2^W+b_1^\La\circ\Delta_2^W+b_1^\La\circ b_1^V\circ\Delta_2^W\big)\big(\bs{b}_1^V\otimes\bs{b}_1^V\big)
\end{alignat*}
and making use of Equation \eqref{eq5}, all the terms in the sum cancel by pair.
\end{proof}

The same functorial property applies for the map $\bs{\Delta}_1^W:\Cth_+(V_0\odot W_0,V_1\odot W_1)\to\Cth_+(V_0,V_1)$. Indeed, we have
\begin{prop}\label{prop:transferf2}
	The map induced by $\bs{\Delta}_1^W$ in homology preserves the product structures, that is to say:
	$\bs{\Delta}_1^{W_{02}}\circ\mfm_2^{V\odot W}=\mfm_2^V\big(\bs{\Delta}_1^{W_{12}}\otimes\bs{\Delta}_1^{W_{01}}\big)$ in homology.
\end{prop}
\begin{proof}
Given a triple $(V_0\odot W_0,V_1\odot W_1,V_2\odot W_2)$, we define a map
\begin{alignat*}{1}
	\bs{\Delta}_2^W:\Cth_+(V_1\odot W_1,V_2\odot W_2)\otimes \Cth_+(V_0\odot W_0,V_1\odot W_1)\to\Cth_+(V_0,V_2)
\end{alignat*}
by
$$\bs{\Delta}_2^W=\Delta_2^W\big(\bs{b}_1^V\otimes\bs{b}_1^V\big)$$
where the map $\Delta_2^W$ was defined in Section \ref{section:def_product} for the case of three pairwise transverse Lagrangian cobordisms. In order to prove the proposition, we prove that the following relation is satisfied:
\begin{alignat}{1}
	\bs{\Delta}_2^W\big(\mfm_1^{V\odot W}\otimes\id\big)+\bs{\Delta}_2^W\big(\id\otimes\mfm_1^{V\odot W}\big)+\bs{\Delta}_1^W\circ\mfm_2^{V\odot W}+\mfm_2^V\big(\bs{\Delta}_1^W\otimes\bs{\Delta}_1^W\big)+\mfm_1^V\circ\bs{\Delta}_2^W=0\label{phi_22}
\end{alignat}
First, we have
\begin{alignat*}{1}
	\bs{\Delta}_2^W\big(\mfm_1^{V\odot W}\otimes\id\big)=&\Delta_2^W\big(\bs{b}_1^V\circ\mfm_1^{V\odot W}\otimes\bs{b}_1^V\big)=\Delta_2^W\big(\mfm_1^W\otimes\id\big)\big(\bs{b}_1^V\otimes\bs{b}_1^V\big)	
\end{alignat*}
Then, again we have already computed the term $\bs{\Delta}_1^W\circ\mfm_2^{V\odot W}$ when considering $\mfm_1^{V\odot W,0_V-}\circ\mfm_2^{V\odot W}$ in Section \ref{sec:LC2}. Recall that we have
\begin{alignat*}{1}
	\bs{\Delta}_1^W\circ\mfm_2^{V\odot W}=&\Delta_1^W\circ\mfm_2^{W,+0}\big(\bs{b}_1^V\otimes\bs{b}_1^V\big)+\mfm_2^{V,0-}\big(\bs{\Delta}_1^W\otimes\bs{\Delta}_1^W\big)+\mfm_1^{V,0-}\circ\,\Delta_2^W\big(\bs{b}_1^V\otimes\bs{b}_1^V\big)
\end{alignat*}
Hence, the left-hand side of Equation \eqref{phi_22} is equal to:
\begin{alignat*}{2}
	&\big(\Delta_2^W(\mfm_1^{W}\otimes\id)+\Delta_2^W(\id\otimes\mfm_1^{W})+\Delta_1^W\circ\mfm_2^{W,+0}\big)\big(\bs{b}_1^V\otimes\bs{b}_1^V\big)\tag{L1}\\
	+&\mfm_2^{V,0-}\big(\bs{\Delta}_1^W\otimes\bs{\Delta}_1^W\big)+\mfm_1^{V,0-}\circ\,\Delta_2^W\big(\bs{b}_1^V\otimes\bs{b}_1^V\big)\tag{L2}\\
	+&\mfm_2^V\big(\bs{\Delta}_1^W\otimes\bs{\Delta}_1^W\big)+\mfm_1^V\circ\Delta_2^W\big(\bs{b}_1^V\otimes\bs{b}_1^V\big)\tag{L3}
\end{alignat*}
Using Equation \eqref{eq3} on line (L1) and summing (L2) and (L3) gives
\begin{alignat*}{2}
	&\Delta_1^W\circ\mfm_2^{W,-}\big(\bs{b}_1^V\otimes\bs{b}_1^V\big)+\Delta_2^\La\big(\bs{\Delta}_1^W\circ\bs{b}_1^V\otimes\bs{\Delta}_1^W\circ\bs{b}_1^V\big)+\Delta_1^\La\circ\Delta_2^W\big(\bs{b}_1^V\otimes\bs{b}_1^V\big)\\
	&+\mfm_2^{V,+}\big(\bs{\Delta}_1^W\otimes\bs{\Delta}_1^W\big)+\mfm_1^{V,+}\circ\Delta_2^W\big(\bs{b}_1^V\otimes\bs{b}_1^V\big)
\end{alignat*}
Observe that $\Delta_1^W\circ\mfm_2^{W,-}\big(\bs{b}_1^V\otimes\bs{b}_1^V\big)=0$ for energy reasons. Then, $\mfm_1^{V,+}=\Delta_1^\La$ and $\mfm_2^{V,+}=\Delta_2^\La\big(\bs{b}_1^V\otimes\bs{b}_1^V\big)$, so using Equation \eqref{eq5} one gets that the terms sum to $0$.
\end{proof}

Observe that given the maps $\bs{b}_2^V$ and $\bs{\Delta}_2^V$ defined in the proofs of Propositions \ref{prop:transferf1} and \ref{prop:transferf2}, we can rewrite the formula of the product $\mfm_2^{V\odot W}$ as follows:
\begin{alignat*}{2}
\mfm_2^{V\odot W}&=\mfm_2^{W,+0}\big(\bs{b}_1^V\otimes\bs{b}_1^V\big)+\mfm_1^{W,+0}\circ\,\bs{b}_2^V+\mfm_2^{V,0-}\big(\bs{\Delta}_1^W\otimes\bs{\Delta}_1^W\big)+\mfm_1^{V,0-}\circ\,\bs{\Delta}_2^W
\end{alignat*}

Moreover, if we restrict again to the special cases where the triples $(W_0,W_1,W_2)$ or $(V_0,V_1,V_2)$ are trivial cylinders, one has
\begin{enumerate}
	\item $(W_0,W_1,W_2)=(\R\times\La_0,\R\times\La_1,\R\times\La_2)$: The map $\bs{b}_2^V$ becomes a map
	\begin{alignat*}{1}
		\bs{b}_2^V:\Cth_+(V_1,V_2)\otimes\Cth_+(V_0,V_1)\to C(\La_0,\La_2)
	\end{alignat*}
	which is equal to the map $b_2^V$ for the case of three pairwise transverse Lagrangian cobordisms $(V_0,V_1,V_2)$, and $\bs{\Delta}_2^W$ vanishes.
	\item $(V_0,V_1,V_2)=(\R\times\La_0,\R\times\La_1,\R\times\La_2)$: in this case the map $\bs{b}_2^V$ vanishes and $\bs{\Delta}_2^W$ is a map
	\begin{alignat*}{1}
		\bs{\Delta}_2^W:\Cth_+(W_1,W_2)\otimes\Cth_+(W_0,W_1)\to C(\La_2,\La_0)
	\end{alignat*}
	which is actually equal to the map $\Delta_2^W$ for the case of three pairwise transverse Lagrangian cobordisms $(W_0,W_1,W_2)$.
\end{enumerate}

\section{Continuation element}\label{sec:unit}

Again let $\Sigma_0$ be an exact Lagrangian cobordism from $\La_0^-$ to $\La_0^+$ with $\Ac(\La_0^-)$ admitting an augmentation $\ep_0^-$.
In this section we prove that there is a \textit{continuation element} $e\in\Cth_+(\Sigma_0,\Sigma_1)$, where $\Sigma_1$ is a suitable small Hamiltonian perturbation of $\Sigma_0$.
Assume $\Sigma_0$ is cylindrical outside $[-T,T]\times Y$. Fix $\eta>0$ smaller than the length of any chord of $\La_0^-$ and $\La_0^+$, and $N>0$. Then we set $\Sigma_1:=\varphi^\eta_{\widetilde{H}}(\Sigma_0)$ for a Hamiltonian $\widetilde{H}:\R\times(P\times\R)\to\R$ being a small perturbation of $H(t,p,z)=h_{T,N}(t)$ for $h_{T,N}:\R\to\R$ satisfying
$$\left\{\begin{array}{l}
h_{T,N}(t)=-e^t\,\mbox{ for}\,t<-T-N\\
h_{T,N}(t)=-e^t+C\,\mbox{ for }\,t>T+N\\
h_{T,N}'(t)\leq 0\\

[-T,T]\subset(h')^{-1}(0)\\
\end{array}\right.$$
for a positive constant $C$, and whose corresponding Hamiltonian vector field is given by $\rho_{T,N}\partial_z$, with $\rho_{T,N}:\R\to\R$ satisfying $\rho_{T,N}(t)=-1$ for $t\leq -T-N$ and $t\geq T+N$, $\rho_{T,N}(t)=0$ for $t\in[-T,T]$, $\rho_{T,N}'\geq0$ in $[-T-N,-T]$ and $\rho_{T,N}'\leq0$ in $[T,T+N]$, see Figure \ref{fig:graph}.
\begin{figure}[ht] 
	\begin{center}\includegraphics[width=13cm]{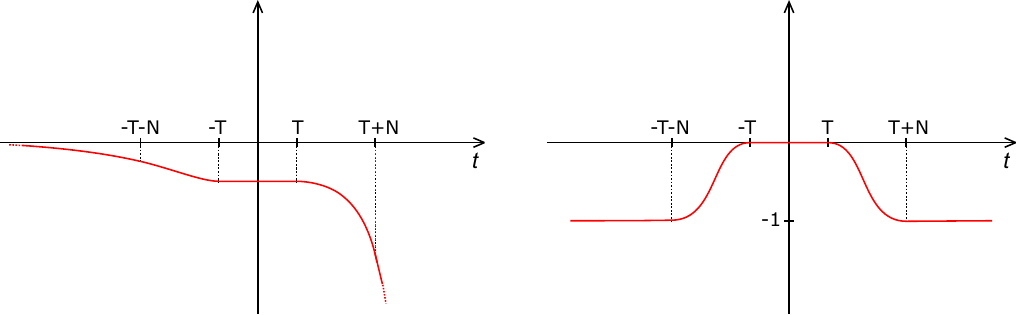}\end{center}
	\caption{Graphs of $h_{T,N}$ on the left and $\rho_{T,N}$ on the right.}
	\label{fig:graph}
\end{figure}
Moreover, under an appropriate identification of a tubular neighborhood of $\Sigma_0$ with a standard neighborhood of the $0$-section in $T^*\Sigma_0$, see \cite[Section 6.2.2]{DR}, we assume that $\Sigma_1$ is given by the graph of $dF$ in $T^*\Sigma_0$, with $F:\Sigma_0\to\R$ Morse function satisfying under this identification the following properties:

\begin{itemize}
	\item the critical points of $F$ (in one-to-one correspondence with intersection points in $\Sigma_0\cap\Sigma_1$) are all contained in $(-T,T)\times Y$.
	\item on the cylindrical ends of $\Sigma_0$, $F$ is equal to $e^t(f_\pm-\eta)$, for $f_\pm:\La_0^\pm\to\R$ Morse functions such that the $C^0$-norm of $f_\pm$ is much smaller than $\eta$. In other words, it means that the cylindrical ends of $\Sigma_0\cup\Sigma_1$ are cylinders over the $2$-copy $\La_0^\pm\cup\La_1^\pm$ where $\La_1^\pm$ is a Morse perturbation of $\La_0^\pm-\eta\partial_z$ (translation of $\La_0^\pm$ by $\eta$ in the negative Reeb direction). Moreover, we assume that $f_\pm$ admit a unique maximum on each connected component.
	\item  we assume that $F$ admits a unique maximum on each filling component of $\Sigma_0$ (observe that $F$ is decreasing with respect to the coordinate t) and has no maximum on each component of $\Sigma_0$ with a non empty negative end.
\end{itemize}
See Figure \ref{fig:unit_cycle} for a schematic picture of the $2$-copy $\Sigma_0\cup\Sigma_1$.
\begin{figure}[ht]  
	\begin{center}\includegraphics[width=6cm]{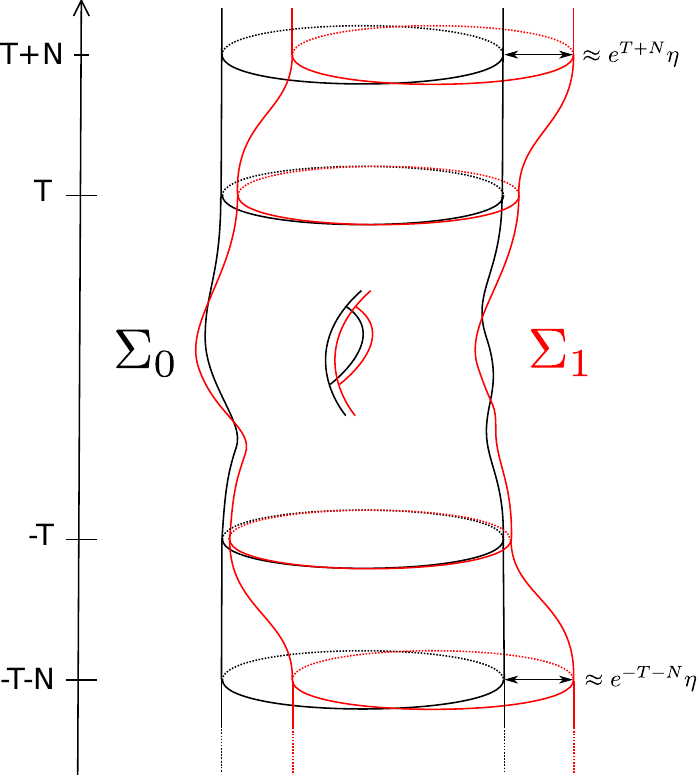}\end{center}
	\caption{Schematic picture of the perturbation $\Sigma_1$ of $\Sigma_0$.}
	\label{fig:unit_cycle}
\end{figure}
The CE-algebras $\Ac(\La_0^-)$ and $\Ac(\La_1^-)$ are canonically identified and thus an augmentation $\ep_0^-$ of $\Ac(\La_0^-)$ can be seen as an augmentation of $\Ac(\La_1^-)$. Moreover, for $\eta$ small enough and Morse functions $F,f_\pm$ such that $\Sigma_1$ is sufficiently $C_1$-close to $\Sigma_0$, one has $\ep_0^-\circ\Phi_{\Sigma_0}=\ep_0^-\circ\Phi_{\Sigma_1}$, see \cite[Theorem 2.15]{CDGG1}.

\begin{rem}
	Observe that in order to define the continuation element in $\Cth_+(\Sigma_0,\Sigma_1)$ we choose to view $\Sigma_1$ as a \textit{negative wrapping} of $\Sigma_0$ at infinity. We could also choose to view $\Sigma_0$ as a \textit{positive wrapping} of $\Sigma_1$, as done in \cite{GPS} for Lagrangians in Liouville sectors. Of course both points of view are valid, in the first case our continuation element will be represented by the sum of the maxima of the Morse functions $F$ and $f_-$ (see below) while in the second case it would correspond to a sum of minima.
\end{rem}

\begin{rem}\label{rem:morse}
According to the choice of perturbation we take we have the following isomorphism of complexes $CF^*(\Sigma_0,\Sigma_1)\simeq C_{n+1-*}^M(F)$, see \cite[Theorem 7.9]{CDGG2}, and recall that a critical point of $f_-$ of Morse index $k$ corresponds to a Morse chord of LCH index $n-k-1$ from $\La_1^-$ to $\La_0^-$, see Example \ref{ex1}. 
\end{rem}

Let us denote $e=e^0+e^-$, where $e^0=\sum e^0_i$ is the sum of the maxima $e_i^0$ of $F$ and $e^-=\sum e_i^-$is the sum of the maxima $e_i^-$ of $f_-$, where the sum is indexed over the connected components of $\Sigma$. Each $e_i^-$ corresponds to a Reeb chord from $\La_1^-$ to $\La_0^-$. Note that $e$ is of degree $0$ in the complex $\Cth_+(\Sigma_0,\Sigma_1)$.

\begin{prop}\label{unit_cycle}
	We have $\mfm_1^{\Sigma_{01}}(e)=0$, i.e. $e$ is a cycle.
\end{prop}
\begin{proof}
We develop	$\mfm_1^{\Sigma_{01}}(e)=\sum\mfm_1^0(e_i^0)+\sum\mfm_1^0(e_i^-)+\sum\mfm_1^-(e_i^0)+\mfm_1^-(e_i^-)$.
	
First, for all $i$ $\mfm_1^-(e_i^0)=0$ by assumption. Indeed the components of $\Sigma_0$ on which there is a maximum of $F$ are assumed to have an empty negative end. This is also true for action reasons because according to the perturbation we perform to construct $\Sigma_1$ from $\Sigma_0$, all intersection points in $\Sigma_0\cap\Sigma_1$ have positive action.

Then, we prove that $\sum\mfm_1^-(e_i^-)=0$. 
The strategy is the following. We will make use of the isomorphism in Example \ref{ex1} in order to view pseudo-holomorphic discs contributing to $\mfm_1^-(e_i^-)$ with boundary on $\R\times(\La_0^-\cup\La_1^-)$ as discs with boundary on $\R\times(\La_0^-\cup\overline{\La_1^-})$. Then we apply results in \cite{EESa} in order to interpret the later as negative gradient flow lines of $f_-$ or generalised discs with boundary on $\R\times\La_0^-$, and conclude.

Let $u$ be a pseudo-holomorphic disc contributing to $m_1^-(e_i^-)$. By definition, $u$ has boundary on $\R\times(\La_0^-\cup\La_1^-)$, a negative asymptotic to the maximum Reeb chord $e_i^-$ and a positive asymptotic to an output Reeb chord from $\La_1^-$ to $\La_0^-$, call it $\beta_{10}$. Observe that as $\mfm_1^-$ is of degree $1$, $|\beta_{10}|_{\Cth_+}=1$. By the isomorphism in Example \ref{ex1}, the disc $u$ is in bijective correspondence with a disc $\overline{u}$ with boundary on $\R\times(\La_0^-\cup\overline{\La_1^-})$, a positive asymptotic to $\overline{e_i^-}$ and a negative asymptotic to $\overline{\beta_{10}}$. We consider two cases: either $\beta_{10}$ (and thus also $\overline{\beta_{10}}$) is a Morse chord, or it is not.                                                                                                                                                                                                                                                                                                                                                                                                                                                                       

If $\beta_{10}$ is Morse it corresponds to an index $n-1$ critical point of $f_-$. Moreover, by \cite[Theorem 3.6]{EESa} $\overline{u}$ has no pure Reeb chords asymptotics for action reasons and corresponds to a negative gradient flow line of $f_-$ from the maximum $e_i^-$ to the critical point $\beta_{10}$ (here we abuse notation and denote the same way Morse chords and critical points of $f_-$ to which they correspond). Observe also that for each index $n-1$ critical point $\beta_{10}$ of $f_-$ there are exactly two flow lines of $f_-$ flowing to the (unique on the $i$-th connected component!) maximum  $e_i^-$. So the contribution of such a critical point $\beta_{10}$ to $\mfm_-(e_i^-)$ vanishes.

If $\beta_{10}$ is not Morse, then by \cite[Theorem 3.6]{EESa} again $\overline{u}$  corresponds to a rigid generalised disc with boundary on $\R\times\La_0^-$, which consists of a pseudo-holomorphic disc $v$ with a negative gradient flow line of $f_-$ flowing from the maximum $e_i^-$ to the boundary of $v$. By rigidity, $v$ is a constant disc at $\beta_0$ which is the pure chord of $\La_0^-$ corresponding to $\beta_{10}$. Following the proof of \cite[Theorem 5.5]{EESa} there are two ways this negative gradient flow line can be attached to $u$: either on the starting point of $\beta_0$, or on its ending point.
Thus we get that the contribution of $\beta_{10}$ to $\sum\mfm_1^-(e_i^-)$ is given by $\ep_0^-(\beta_0)+\ep_1^-(\beta_0)=0$.

Finally, we prove $\sum\mfm_1^0(e_i^0)=\sum\mfm_1^0(e_i^-)=0$.
As observe in Remark \ref{rem:morse}, $\mfm_1^0(e_i^0)$ counts negative gradient flow lines of $F$ from the maximum $e_i^0$ to a critical point of Morse index $n$. From such a point, there are exactly two flow lines of $F$ flowing out and as the gradient of $F$ points inward in the positive end, these two flow lines must flow to the (unique on this connected component!) maximum $e_i^0$. Hence $\mfm_1^0(e_i^0)=0$.
In order to show $\sum\mfm_1^0(e_i^-)=0$, wrap the negative end of (the non-empty negative ends components of) $\Sigma_1$ slightly in the positive Reeb direction using the Hamiltonian vector field $\rho^-_{T+N,N}\partial_z$ (see Section \ref{wrapping}). Let $V_1$ be the image of $\R\times\La_1^-$ by the corresponding time-$s_-$ flow where $s_-$ is bigger than the longest Morse chord from $\La_1^-$ to $\La_0^-$ but much smaller than the shortest non Morse chord from $\La_1^-$ to $\La_0^-$. We set $V_0=\R\times\La_0^-$.
Observe that each Morse chord (from $\La_1^-$ to $\La_0^-$) becomes an intersection point in $V_0\cap V_1$. We denote $m_i$ the intersection point corresponding to $e_i^-$, see Figure \ref{wrap_unit3}. Assume that the perturbation is sufficiently small and generic so that $V_1\odot\Sigma_1$ can be seen as a perturbation of $V_0\odot\Sigma_0$ by a Morse function $\widetilde{F}$ which equals $F$ on $[-T,T]\times Y \cap\Sigma_0$. Moreover, observe that the gradient of $\widetilde{F}$ points inward in the negative end.
\begin{figure}[ht]  
	\begin{center}\includegraphics[width=4cm]{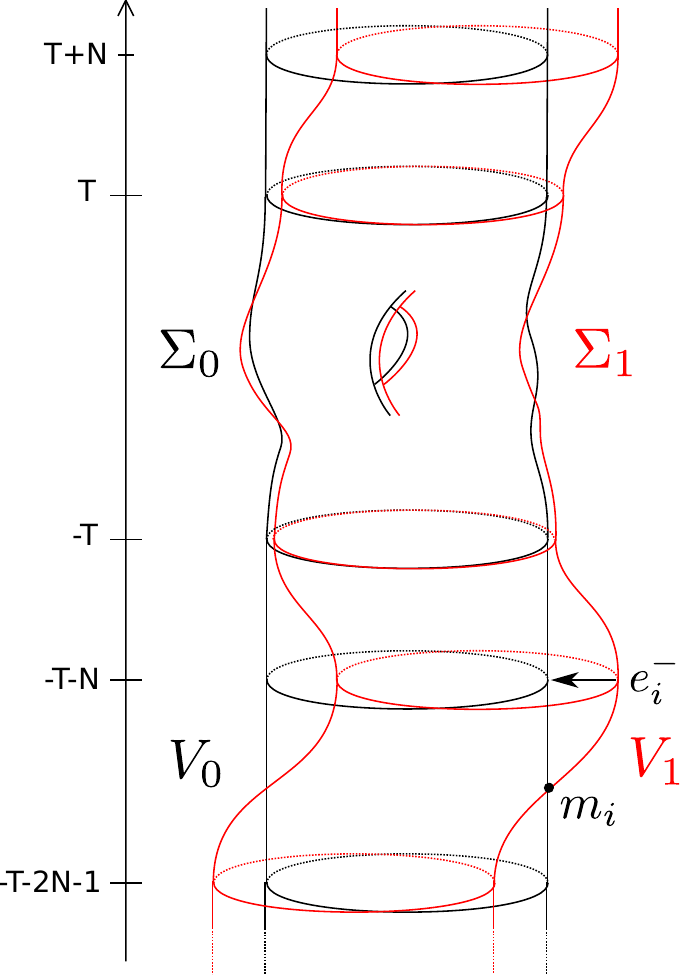}\end{center}
	\caption{Schematization of the wrapping of the negative of $\Sigma_1$.}
	\label{wrap_unit3}
\end{figure}
Consider the pair of concatenations $(V_0\odot\Sigma_0,V_1\odot\Sigma_1)$.
By projecting curves on $P$ as done in Section \ref{wrapping}, one can prove that $b_1^V(m_i)=e_i^-$. Then, by definition of the differential in a concatenation,
\begin{alignat*}{1}
	\mfm_1^{V\odot\Sigma}(m_i)=\mfm_1^{\Sigma,+0}\circ\bs{b}_1^V(m_i)+\mfm_1^{V,0-}\circ\bs{\Delta}_1^\Sigma(m_i)
\end{alignat*}
which gives $\mfm_1^{V\odot\Sigma,0_\Sigma}(m_i)=\mfm_1^{\Sigma,0}\circ\,b_1^V(m_i)=\mfm_1^{\Sigma,0}(e_i^-)$ and curves contributing to $\mfm_1^{V\odot\Sigma,0_\Sigma}(m_i)$ are in one-to-one correspondence with negative gradient flow lines of $\widetilde{F}$ from the maximum $m_i$ to a critical point in $\Sigma_0$.
Now, each intersection point in $\Sigma_0\cap\Sigma_1$ (of a non empty negative end component) which corresponds to an index $n$ critical point of $\widetilde{F}$, is the starting point of two gradient flow lines flowing to $m_i$. 
Thus, $\sum\mfm_1^0(e_i^-)=0$.
\end{proof}

\begin{teo}\label{teo_unit}
	Consider $\Sigma_0$ and $\Sigma_1$ as above. Take $\La_2^-\prec_{\Sigma_2}\La_2^+$ another exact Lagrangian cobordism such that the intersection of $\Sigma_2$ with a small standard neighborhood of $\Sigma_0$ identified with $D_\varepsilon T^*\Sigma_0$ and containing also $\Sigma_1$, consists of a union of fibres, then we have:
	\begin{alignat}{1}
	\mfm_2^{\Sigma_{012}}(\,\cdot\,,e):\Cth_+(\Sigma_1,\Sigma_2)\to\Cth_+(\Sigma_0,\Sigma_2)\label{map_prod}
	\end{alignat} is an isomorphism.
\end{teo}
\begin{rem}
	Given $\Sigma_0$ and a transverse cobordism $\Sigma_2$, one can always find a sufficiently small perturbation $\Sigma_1$ of $\Sigma_0$ such that the intersection of $\Sigma_2$ with $D_\varepsilon T^*\Sigma_0$ which contains $\Sigma_1$, consists of a union of fibres. This way, there is a canonical identification of vector spaces $\Cth_+(\Sigma_1,\Sigma_2)\cong\Cth_+(\Sigma_0,\Sigma_2)$, and for a generator $\gamma_{12}$, $x_{12}$ or $\gamma_{21}$ in $\Cth_+(\Sigma_1,\Sigma_2)$, one denotes respectively $\gamma_{02}$, $x_{02}$ or $\gamma_{20}$ the corresponding generator in $\Cth_+(\Sigma_0,\Sigma_2)$.
\end{rem}

\begin{rem}
	Proposition \ref{unit_cycle} states that the element $e$ is a cycle. Given Theorem \ref{teo_unit} one gets that it is a boundary if and only if $\Cth_+(\Sigma_1,\Sigma_2)$ is acyclic for every cobordism $\Sigma_2$ satisfying the hypothesis of the theorem, as proved in \cite[Lemma 4.17]{CDGG3}.
\end{rem}

\begin{proof}[Proof of Theorem \ref{teo_unit}]
First we write
	\begin{alignat}{1}
	\Cth_+(\Sigma_1,\Sigma_2)=C(\La_2^+,\La_1^+)^\dagger[n-1]\oplus CF^-(\Sigma_1,\Sigma_2)\oplus C(\La_1^-,\La_2^-)\oplus CF^+(\Sigma_1,\Sigma_2)\label{decomp}
	\end{alignat}
	where $CF^\pm(\Sigma_1,\Sigma_2)\subset CF(\Sigma_1,\Sigma_2)$ is the sub-vector space generated by positive, resp. negative action intersection points. According to this decomposition, ordering Reeb chords in $C(\La_2^+,\La_1^+)^\dagger[n-1]$ from biggest to smallest action and intersection points in $CF^-(\Sigma_1,\Sigma_2)\oplus CF^+(\Sigma_1,\Sigma_2)$ from smallest to biggest action, we will show that the matrix of the map \eqref{map_prod} is lower triangular with identity terms on the diagonal.\\

\noindent\textbf{1.} For $\gamma_{12}\in C(\La_2^+,\La_1^+)$, we have
\begin{alignat*}{2}
\mfm_2(\gamma_{12},e)&=\mfm_2^+(\gamma_{12},e)+\mfm_2^0(\gamma_{12},e)+\mfm_2^-(\gamma_{12},e)\\
&=\Delta_2^+\big(\gamma_{12},b_1^\Sigma(e)\big)+\mfm_2^0(\gamma_{12},e)+b_1^-\circ\Delta_2^\Sigma(\gamma_{12},e)+b_2^-\big(\Delta_1^\Sigma(\gamma_{12}),e^-\big)+b_2^-\big(\Delta_1^\Sigma(\gamma_{12}),\Delta_1^\Sigma(e^0)\big)\\
&=\Delta_2^+\big(\gamma_{12},b_1^\Sigma(e)\big)+\mfm_2^0(\gamma_{12},e)+b_1^-\circ\Delta_2^\Sigma(\gamma_{12},e)+b_2^-\big(\Delta_1^\Sigma(\gamma_{12}),e^-\big)
\end{alignat*}
where the last equality is because $\Delta_1^\Sigma(e_i^0)$ vanishes for energy reasons. On Figure \ref{unit3} we schematized pseudo-holomorphic configurations contributing to $\mfm_2(\gamma_{12},e)$. Let $i$ denote the index of the connected component of $\Sigma_1$ containing the starting point of $\gamma_{12}$. Note that by the hypothesis on the Morse function $F$, if this component has a non-empty negative end only the configurations A, B, C, and D are relevant, whereas if it is a filling component then only the configurations A', B' and C' are. We start by considering the first case and explain at the end of this part how the second case is treated in a similar way. We will prove that $$\mfm_2(\gamma_{12},e)=\gamma_{02}+\bs{\zeta}_{02}+\bs{y}_{02}^-+\bs{\xi}_{20}+\bs{y}_{02}^+$$
where $\gamma_{02}\in C(\La_2^+,\La_0^+)$ is the Reeb chord canonically identified to $\gamma_{12}$, $\bs{\zeta}_{02}\in C(\La_2^+,\La_0^+)$ is a linear combination of Reeb chords whose action are smaller than the action of $\gamma_{02}$, $\bs{y}_{02}^\pm\in CF^\pm(\Sigma_0,\Sigma_2)$ and $\bs{\xi}_{20}\in C(\La_0^-,\La_2^-)$.
\begin{figure}[ht]  
	\begin{center}\includegraphics[width=12cm]{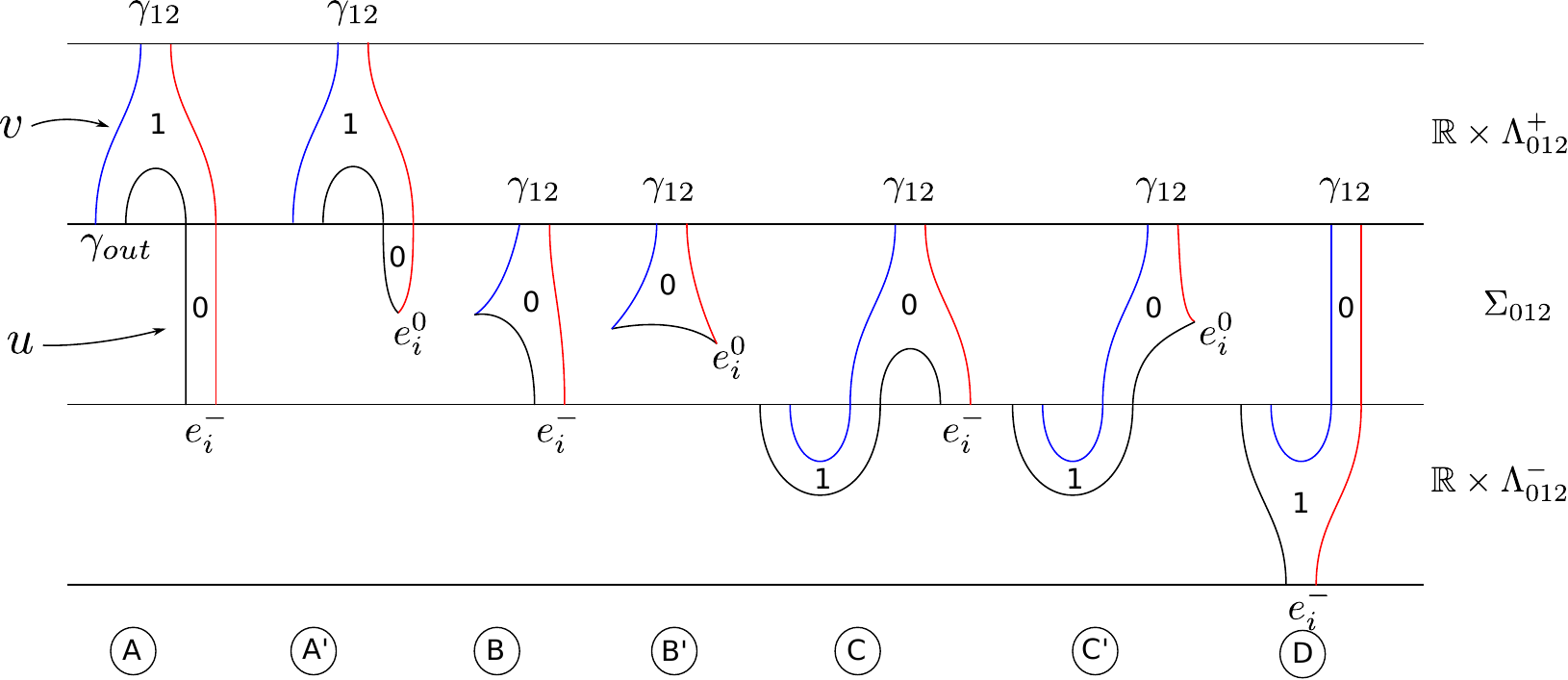}\end{center}
	\caption{Types of curves (potentially) contributing to $\mfm_2(\gamma_{12},e)$.}
	\label{unit3}
\end{figure}

Let us consider the configurations of type A. Denote $v$ a rigid disc with boundary on the positive cylindrical ends, with a positive asymptotic to $\gamma_{12}$, a negative Reeb chord asymptotic $\beta_{10}\in C(\La_0^+,\La_1^+)$, and an output negative Reeb chord asymptotic $\gamma_{out}\in C(\La_2^+,\La_0^+)$, and $u$ a rigid disc with boundary on $\Sigma_0\cup\Sigma_1$ with a positive asymptotic to $\beta_{10}$ and a negative asymptotic to a maximum Morse Reeb chord $e_i^-$.
We distinguish two cases: either $\beta_{10}$ is a Morse chord, or it is not a Morse chord.

\begin{enumerate}
	\item[(a)] If $\beta_{10}$ is a Morse chord. First, rigidity implies that $|\beta_{10}|=|e_i^-|=-1$ (LCH grading), and thus $\beta_{10}$ corresponds to the (only one by assumption) maximum of $f_+$ on the component of $\La_0^+$ containing the starting point of $\gamma_{02}$. For action reasons the disc $u$ has no pure Reeb chords asymptotics.
Similarly as in the proof of Proposition \ref{unit_cycle} we show that the count of such discs $u$ coincides with the count of some rigid gradient flow lines of a Morse function $\widetilde{F}$ which equals $F$ on $\Sigma_0\cap([-T,T]\times Y)$.
To get this correspondence, wrap the negative and the positive ends of $\Sigma_1$ slightly in the positive Reeb direction:
take $V_1$, resp $W_1$, to be the image of $\R\times\La_1^-$, resp $\R\times\La_1^+$, by the time $s_-$, resp $s_+$, flow of the Hamiltonian vector field $\rho^-_{T+N,N}\partial_z$, resp $-\rho^+_{T+N,N}\partial_z$, with $s_\pm$ bigger than the longest Morse chord from $\La_1^\pm$ to $\La_0^\pm$ but smaller than the shortest non Morse chord from $\La_1^\pm$ to $\La_0^\pm$. See Figure \ref{wrap_unit2} for a schematized picture of the perturbation.
\begin{figure}[ht]  
	\begin{center}\includegraphics[width=4cm]{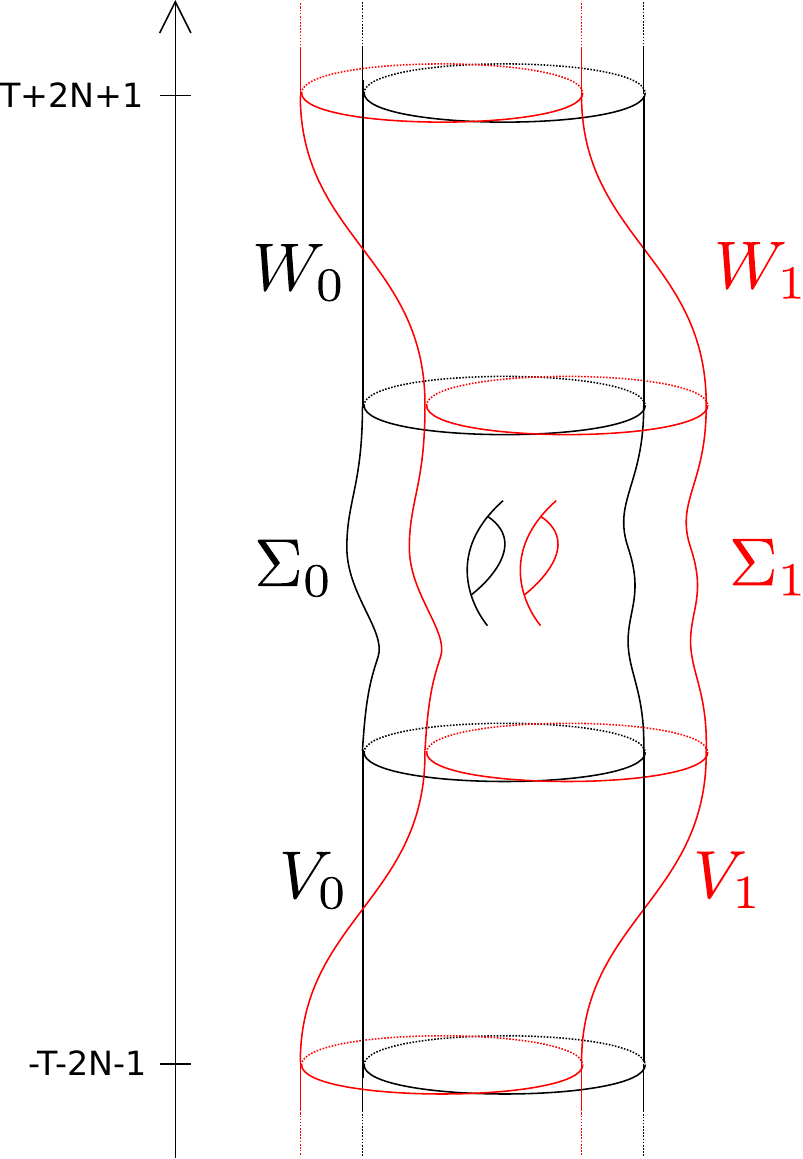}\end{center}
	\caption{Schematized picture of wrapping of the negative and positive ends of $\Sigma_1$.}
	\label{wrap_unit2}
\end{figure}
This way, $e_i^-$ corresponds canonically to an intersection point $m_i\in CF(V_0,V_1)$ and $\beta_{10}$ corresponds to an intersection point $x_\beta\in CF(W_0,W_1)$.
As before, by projecting discs on $P$ one can prove that 
$b_1^V(m_i)=e_i^-$ and $\mfm_1^{W,0}(\beta_{10})=x_\beta$.
By definition of the differential for the pairs of concatenated cobordisms $(V_0\odot(\Sigma_0\odot W_0),V_1\odot(\Sigma_1\odot W_1))$ and  $(\Sigma_0\odot W_0,\Sigma_1\odot W_1)$ one has
\begin{alignat}{1}
	\mfm_1^{V\odot(\Sigma\odot W)}(m_i)=\mfm_1^{\Sigma\odot W,+0}\circ\bs{b}_1^V(m_i)+\mfm_1^{V,0-}\circ\bs{\Delta}_1^{\Sigma\odot W}(m_i)\label{correspond}
\end{alignat}
and 
\begin{alignat*}{1}
	\mfm_1^{\Sigma\odot W,+0}\circ\bs{b}_1^V(m_i)=\mfm_1^{W,+0}\circ\bs{b}_1^\Sigma\circ\bs{b}_1^V(m_i)+\mfm_1^{\Sigma,0}\circ\bs{\Delta}_1^{W}\circ\bs{b}_1^V(m_i)
\end{alignat*}
Considering the components with values in $CF(W_0,W_1)$ on both sides of \eqref{correspond} gives
$$\mfm_1^{V\odot(\Sigma\odot W),0_W}(m_i)=\mfm_1^{W,0}\circ\bs{b}_1^\Sigma\circ\bs{b}_1^V(m_i)=\mfm_1^{W,0}\circ b_1^\Sigma\circ b_1^V(m_i)=\mfm_1^{W,0}\circ b_1^\Sigma(e_i^-)$$	
The disc $u$ contributes to the coefficient of $\beta_{10}$ in $b_1^\Sigma(e_i^-)$ which is thus equal to the coefficient of $x_\beta$ in $\mfm_1^{V\odot(\Sigma\odot W),0_W}(m_i)$, as $\mfm_1^{W,0}(\beta_{10})=x_\beta$.
Similarly as before, one can view $V_1\odot\Sigma_1\odot W_1$ as a Morse perturbation of $\Sigma_0$ by a Morse function $\widetilde{F}$ which is equal to $F$  on $([-T,T]\times Y)\cap\Sigma_0$.
Thus, pseudo-holomorphic strips asymptotic to $m_i$ and $x_\beta$ are in one-to-one correspondence with negative gradient flow lines of $\widetilde{F}$ from $m_i$ to $x_\beta$. Moreover, there exists exactly one such disc because $x_\beta$ is the starting point of two flow lines of $d\widetilde{F}$, but one escapes in the positive end (note that $d\widetilde{F}$ points outward in the positive end) while the other flows to $m_i$.

It remains to understand the pseudo-holomorphic disc $v$ with boundary on $\R\times(\La_0^+\cup\La_1^+\cup\La_2^+)$ with a positive asymptotic to $\gamma_{12}$ and negative asymptotics to a maximum Morse chord $\beta_{10}$ and a chord $\gamma_{out}\in C(\La_2^+,\La_0^+)$. 
In order to do so, we will use the same strategy as in the proof of Proposition \ref{unit_cycle} when we showed $\sum\mfm_1^-(e_i^-)=0$. Namely, we use the isomorphism recalled in Example \ref{ex1} relating $\Cth_+$ complexes of two different 2-copies, and then \cite[Theorem 3.6]{EESa} to identify discs with boundary on a 2-copy with generalised discs with boundary on one copy. The two different 2-copies we consider are the following. 
The first is $(\Gamma_0,\Gamma_1)$, with $\Gamma_0=\La_0^+\cup\La_2^+$ and $\Gamma_1=\La_1^+\cup(\La_2^+)'$ where $(\La_2^+)'$ is a perturbation by a Morse function of a small push-off of $\La_2^+$ in the negative Reeb direction. The second 2-copy is $(\Gamma_0,\overline{\Gamma_1})$ where $\overline{\Gamma_1}$ is a translation of $\Gamma_1$ far in the positive Reeb direction. Now the disc $v$ we consider with boundary on $\R\times(\La_0^+\cup\La_1^+\cup\La_2^+)$ is a disc with boundary on $\R\times(\Gamma_0\cup\Gamma_1)$ with a mixed positive asymptotic $\gamma_{12}$ which is a chord from $\Gamma_1$ to $\Gamma_0$, a mixed negative asymptotic $\beta_{10}$ from $\Gamma_1$ to $\Gamma_0$ and a negative asymptotic $\gamma_{out}$ which is a pure chord of $\Gamma_0$. According to the isomorphism in Example \ref{ex1}, this disc corresponds to a disc with boundary on $\R\times(\Gamma_0\cup\overline{\Gamma_1})$ with a mixed positive asymptotic at $\overline{\beta_{10}}$, a mixed negative asymptotic at $\overline{\gamma_{12}}$ and a pure negative asymptotic at $\gamma_{out}$. By \cite[Theorem 3.6]{EESa} it corresponds to a rigid generalised disc with boundary on $\R\times\Gamma_0$ consisting of a disc and a negative gradient flow line of $f_+$ flowing from the maximum of $f_+$ to the boundary of the disc (on $\R\times\La_0^+$). By rigidity this last disc is constant, implying $\gamma_{out}=\gamma_{02}$. Conversely, following the flow of $df_+$ from the starting point of $\gamma_{02}$ leads to the maximum of $f_+$ on the corresponding connected component. Such a flow line is a generalised disc which corresponds to a disc $v$ with boundary on $\R\times(\La_0^+\cup\La_1^+\cup\La_2^+)$ as considered above.

Thus we have proved that the coefficient of $\gamma_{02}$ in $\mfm_2(\gamma_{12},e)$ is $1$.

\item[(b)] If $\beta_{10}$ is not a Morse chord. Given $R\geq0$ such that the three cobordisms $\Sigma_0,\Sigma_1$ and $\Sigma_2$ are cylindrical outside of $[-R,R]\times Y$, the energy of the disc $v$ with boundary on the positive cylindrical ends is given by
\begin{alignat*}{1}
E(v)=\mathfrak{a}(\gamma_{12})-\mathfrak{a}(\beta_{10})-\mathfrak{a}(\gamma_{out})-\sum\limits_{i=0}^2\mathfrak{a}(\bs{\delta_i})
\end{alignat*}
with $\mathfrak{a}(\gamma_{12})=e^R\ell(\gamma_{12})+\mathfrak{c}_2-\mathfrak{c}_1$, $\mathfrak{a}(\beta_{10})=e^R\ell(\beta_{10})+\mathfrak{c}_0-\mathfrak{c}_1$ and $\mathfrak{a}(\gamma_{out})=e^R\ell(\gamma_{out})+\mathfrak{c}_2-\mathfrak{c}_0$.
One can check that
\begin{alignat*}{1}
	&|\mathfrak{a}(\gamma_{12})-\mathfrak{a}(\gamma_{02})|\leq e^R(\max\|f_+\|_{\mathcal{C}^0}+\eta)+\mathfrak{c}_0-\mathfrak{c}_1\\
	&|\mathfrak{a}(\beta_{10})-\mathfrak{a}(\beta_0)|\leq e^R(\max\|f_+\|_{\mathcal{C}^0}+\eta)+\mathfrak{c}_0-\mathfrak{c}_1
\end{alignat*}
and thus
\begin{alignat*}{1}
	\mathfrak{a}(\gamma_{02})-\mathfrak{a}(\gamma_{out})\geq E(v)+\mathfrak{a}(\beta_0)-2\big(e^R(\max\|f_+\|_{\mathcal{C}^0}+\eta)+\mathfrak{c}_0-\mathfrak{c}_1\big)
\end{alignat*}
and for $\eta$ sufficiently small, the term on the right hand side is strictly positive, so the action of $\gamma_{out}$ is strictly smaller than that of $\gamma_{02}$.
\end{enumerate}

Thus, together with the curves of type B, C, and D, we obtain as expected
\begin{alignat*}{1}
\mfm_2(\gamma_{12},e)=\gamma_{02}+\bs{\zeta}_{02}+\bs{y}_{02}^-+\bs{\xi}_{20}+\bs{y}_{02}^+
\end{alignat*}

In case the connected component of $\Sigma_1$ containing the starting point of $\gamma_{12}$ has an empty negative end, as observed at the beginning of \textbf{1.} the configurations A', B' and C' of Figure \ref{unit3} are the only one which will contribute to $\mfm_2(\gamma_{12},e)$. As for A, the configurations of type A' consist of two discs, one is like the disc $v$ studied in the case of A, positively asymptotic to $\gamma_{12}$ and negatively asymptotic to $\beta_{10}$ and $\gamma_{out}$, and the other disc is a disc $u'$ with boundary on $\Sigma_0\cup\Sigma_1$ with a positive asymptotic to $\beta_{10}$ and an intersection point asymptotic at a maximum $e_i^0$ of $F$. As before, the term $\gamma_{02}$ in $\mfm_2(\gamma_{12},e)$ will come from a configuration A' where $\beta_{10}$ is a Morse chord while the other terms in $\mfm_2(\gamma_{12},e)$ will come from $A'$ where $\beta_{10}$ is not Morse, and from the configurations $B'$ and $C'$. If $\beta_{10}$ is Morse, we just mimic the proof of \textbf{1.}(a) in order to show that $u'$ can be identified with a negative gradient flow line of a Morse function $\widetilde{F}$ (equal to $F$ on $([-T,T]\times Y)\cap\Sigma_0$) from $e_i^0$ to $x_\beta$. The only difference is that we only need to wrap slightly the positive end of $\Sigma_1$ but not the negative end because there is no negative end on the $i$-th connected component of $\Sigma_0$.
\vspace{5mm}

\textbf{2.} For $x_{12}\in CF(\Sigma_1,\Sigma_2)=CF^+(\Sigma_1,\Sigma_2)\oplus CF^-(\Sigma_1\cup\Sigma_2)$ we have
\begin{alignat*}{1}
	\mfm_2(x_{12},e)&=\mfm_2^0(x_{12},e)+\mfm_2^-(x_{12},e)\\
	&=\mfm_2^0(x_{12},e)+b_1^-\circ\Delta_2^\Sigma(x_{12},e)+b_2^-\big(\Delta_1^\Sigma(x_{12}),e^-\big)+b_2^-\big(\Delta_1^\Sigma(x_{12}),\Delta_1^\Sigma(e^0)\big)\\
	&=\mfm_2^0(x_{12},e)+b_1^-\circ\Delta_2^\Sigma(x_{12},e)+b_2^-\big(\Delta_1^\Sigma(x_{12}),e^-\big)
\end{alignat*}	
see Figure \ref{unit6}, and we will prove that for $x_{12}^+\in CF^+(\Sigma_1,\Sigma_2)$ and $x_{12}^-\in CF^-(\Sigma_1,\Sigma_2)$, one has	
\begin{alignat*}{1}
&\mfm_2(x^+_{12},e)=x^+_{02}+\bs{y}^+_{02},\mbox{ and }\\
&\mfm_2(x^-_{12},e)=x^-_{02}+\bs{z}^-_{02}+\bs{\xi}_{20}+\bs{z}^+_{02}
\end{alignat*}
where $\bs{y}_{02}^+, \bs{z}_{02}^+\in CF^+(\Sigma_0,\Sigma_2)$, $\bs{z}_{02}^-\in CF^-(\Sigma_0,\Sigma_2)$ and $\bs{\xi}_{20}\in C(\La_0^-,\La_2^-)$, and each intersection point in $\bs{y}_{02}^+$, resp. $\bs{z}^-_{02}$, has action strictly bigger than $x_{02}^+$, resp. $x_{02}^-$.

\begin{figure}[ht]  
	\begin{center}\includegraphics[width=10cm]{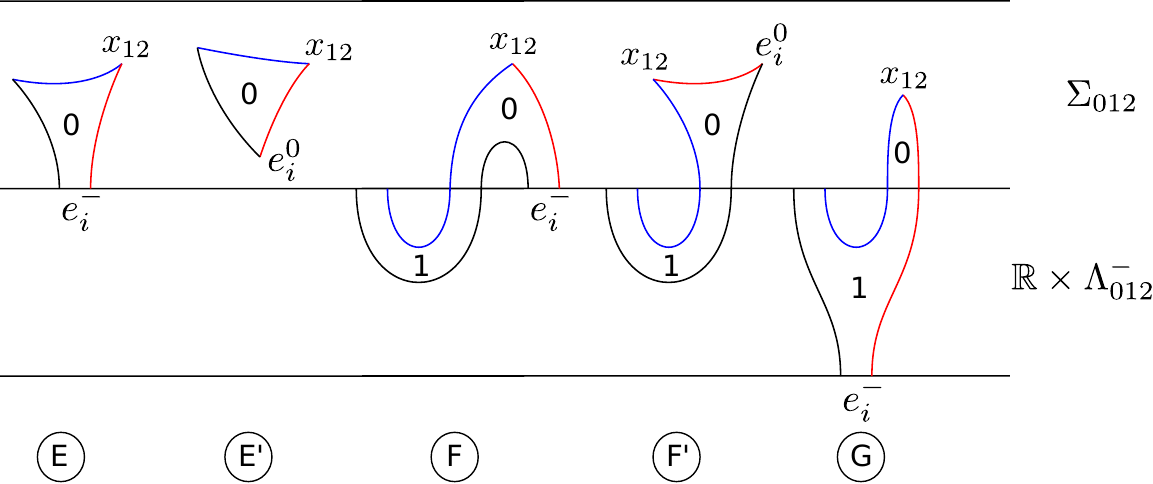}\end{center}
	\caption{Curves contributing to $\mfm_2(x_{12},e)$.}
	\label{unit6}
\end{figure}
As for the previous case, we start by assuming that $x_{12}$ is an intersection point of $\Sigma_1\cap\Sigma_2$ with the $i$-th connected component of $\Sigma_1$ having a non empty negative end. Then, only configurations E, F and G are relevant.
Let us consider configurations of type E. Let $u$ be a pseudo-holomorphic disc with boundary on $\Sigma_0\cup\Sigma_1\cup\Sigma_2$ negatively asymptotic to $e_i^-$, and asymptotic to $x_{12}\in CF(\Sigma_1,\Sigma_2)$ and an output $x_{out}\in CF(\Sigma_0,\Sigma_2)$.
The action of intersection points are assumed to be much smaller than that of pure Reeb chords, so $u$ has no pure Reeb chord asymptotics.

To understand what can be the output of such a disc, we wrap as before the negative end of $\Sigma_1$ slightly in the positive Reeb direction to get the pair $(V_0\odot\Sigma_0,V_1\odot\Sigma_1)$ where the Morse Reeb chords $e_j^-$ in $C(\La_0^-,\La_1^-)$ correspond to intersection points $m_j$ in $CF(V_0,V_1)$ and $\bs{b}_1^V(m_j)=b_1^V(m_j)=e_j^-$. By definition of the product for a pair of concatenated cobordisms (see Section \ref{sec:prod_conc}) we have:
\begin{alignat*}{2}
\mfm_2&^{V\odot \Sigma,0_\Sigma}(x_{12},m_i)\\
&=\mfm_2^{\Sigma,0}(\bs{b}_1^V(x_{12}),\bs{b}_1^V(m_i))
+\mfm_1^{\Sigma,0}\circ\,b_1^V\circ\Delta_2^\Sigma(\bs{b}_1^V(x_{12}),\bs{b}_1^V(m_i))+\mfm_1^{\Sigma,0}\circ\,b_2^V(\bs{\Delta}_1^\Sigma(x_{12}),\bs{\Delta}_1^\Sigma(m_i))\\
&=\mfm_2^{\Sigma,0}(\bs{b}_1^V(x_{12}),e_i^-)
+\mfm_1^{\Sigma,0}\circ\,b_1^V\circ\Delta_2^\Sigma(\bs{b}_1^V(x_{12}),e_i^-)+\mfm_1^{\Sigma,0}\circ\,b_2^V(\Delta_1^\Sigma(x_{12}),m_i)\\
&=\mfm_2^{\Sigma,0}(x_{12},e_i^-)+\mfm_2^{\Sigma,0}(b_1^V\circ\Delta_1^\Sigma(x_{12}),e_i^-)
+\mfm_1^{\Sigma,0}\circ\,b_1^V\circ\Delta_2^\Sigma(x_{12},e_i^-)+\mfm_1^{\Sigma,0}\circ\,b_2^V(\Delta_1^\Sigma(x_{12}),m_i)
\end{alignat*}
See Figure \ref{unit2}. All these terms except the first one involve bananas with two positive Reeb chord asymptotics and with boundary on $V_0\cup V_1\cup V_2$ where $V_0=\R\times\La_0^-$, $V_1$ is a wrapping of $\R\times\La_1^-$ and $V_2:=\R\times\La_2^-$. These rigid bananas project to rigid discs with boundary on $\pi_P(\La_0^-\cup\La_1^-\cup\La_2^-)$ and for dimension reasons they must be constant. This is not possible as they all have two distinct positive Reeb chord asymptotics (a constant curve with boundary on $\pi_P(\La_0^-\cup\La_1^-\cup\La_2^-)$ does not lift to a banana with two positive asymptotics but to a trivial strip). 
So we are left with $\mfm_2^{V\odot \Sigma,0_\Sigma}(x_{12},m_i)=\mfm_2^{\Sigma,0}(x_{12},e_i^-)$.

Let us denote again $\widetilde{F}$ a Morse function such that $V_1\odot\Sigma_1$ is viewed as a $1$-jet perturbation of $\Sigma_0$ by $\widetilde{F}$, and $\widetilde{F}$ equals $F$ on $\Sigma_0\cap([-T,T]\times Y)$.  The intersection point $m_i$ is a maximum of $\widetilde{F}$ and the gradient flow line of $\widetilde{F}$ flowing from $x_{02}$ to $m_i$ corresponds to a pseudo-holomorphic triangle asymptotic to $x_{02}$, $m_i$ and $x_{12}$. Thus the coefficient of $x_{02}$ in $\mfm_2^{\Sigma,0}(x_{12},e_i^-)$ is $1$.
Note also that the energy of this triangle is given by
\begin{alignat}{1}
E(u)=\mathfrak{a}(x_{02})-\mathfrak{a}(e_i^-)-\mathfrak{a}(x_{12})\label{energy}
\end{alignat}
and by definition of the action one can check that it can be made as small as possible by taking smaller $\eta$.

Now suppose there is another pseudo-holomorphic triangle with asymptotics $x_{12}$, $e_i^-$ and $y_{02}\neq x_{02}$, contributing to the coefficient of $y_{02}$ in $\mfm_2(x_{12},e)$. This triangle necessary leaves a small neighborhood of the gradient flow line from $x_{02}$ to $m_i$ and thus according to the relation \eqref{energy} between the energy of such a triangle and the action of its asymptotics, the action of $y_{02}$ is strictly bigger than the action of $x_{02}$, independently of how small $\eta$ is.

\begin{figure}[ht]  
	\begin{center}\includegraphics[width=10cm]{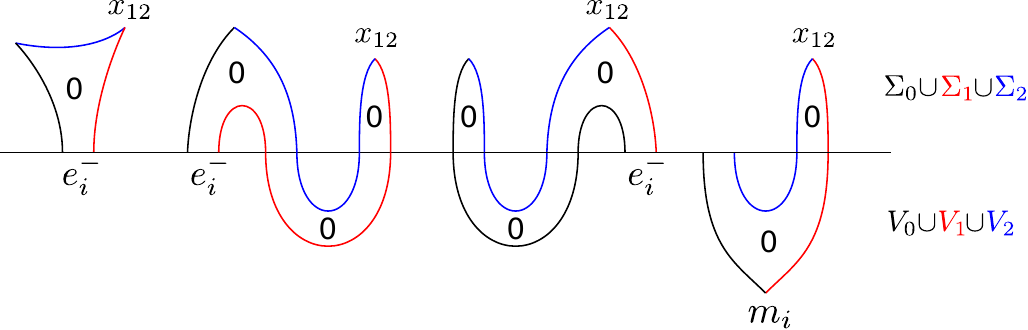}\end{center}
	\caption{Curves contributing to $\mfm_2^{V\odot \Sigma,0_\Sigma}(x,m_i)$.}
	\label{unit2}
\end{figure}

Then, about configurations of type F and G, observe that a disc with boundary on the non-cylindrical parts in such configurations exists only if the action of $x_{12}$ is negative.

To sum up, the configurations of type E, F, and G give that
\begin{alignat*}{1}
&\mfm_2(x^+_{12},e)=x^+_{02}+\bs{y}^+_{02},\mbox{ and }\\
&\mfm_2(x^-_{12},e)=x^-_{02}+\bs{z}^-_{02}+\bs{\xi}_{20}+\bs{z}^+_{02}
\end{alignat*}
If the component of $\Sigma_1$ containing $x_{12}$ is a filling, we have to consider configurations E' and F' only. For E', we don't need to wrap the negative end of $\Sigma_1$ and consider directly the one-to-one correspondence between gradient flow lines of $F$ from $x_{02}$ to $e_i^0$ and pseudo-holomorphic triangles with vertices $x_{02}$, $e_i^0$ and $x_{12}$. For F', observe that a disc with boundary on the non-cylindrical parts in such a configuration exists also only if the action of $x_{12}$ is negative (remember $e_i^0$ has positive action).
\vspace{5mm}

\textbf{3.} Finally, for $\xi_{21}\in C(\La_1^-,\La_2^-)$ we have
		\begin{alignat*}{1}
		\mfm_2(\xi_{21},e)&=\mfm_2^0(\xi_{21},e)+\mfm_2^-(\xi_{21},e)\\
		&=\mfm_2^0(\xi_{21},e^-)+b_2^-(\xi_{21},e^-)
		\end{alignat*}
because the Morse function $F$ has no maxima $e_i^0$ on the component of $\Sigma_1$ involved as this component has a non empty negative end. See Figure \ref{unit7}. For energy reasons, if a disc of type H exists then the output intersection point must have positive action. Then a disc of type I is such that $\xi_{out}$ is the chord in $C(\La_0^-,\La_2^-)$ canonically identified with $\xi_{21}\in C(\La_1^-,\La_2^-)$. In order to prove this, one can use the same type of argument as in \textbf{1.}. Let $\Gamma_0=\La_0^-\cup(\La_2^-)'$ where $(\La_2^-)'$ is a Morse perturbation of a small push-off of $\La_2^-$ in the positive Reeb direction, and $\Gamma_1=\La_1^-\cup\La_2^-$. The disc of type I with boundary on $\Gamma_0\cup\Gamma_1$ corresponds to a disc with boundary on $\Gamma_0\cup\overline{\Gamma_1}$ with a positive asymptotic at $\overline{e_i^-}\in C(\overline{\La_1^-},\La_0^-)$, a negative asymptotic at $\overline{\xi_{out}}\in C(\overline{\La_2^-},\La_0^-)$ and a negative asymptotic at $\overline{\xi_{21}}$ which is a pure chord of $\overline{\Gamma_1}$. By \cite[Theorem 3.6]{EESa} this last disc corresponds to a rigid generalised disc with boundary on $\R\times\Gamma_0$ consisting of a constant disc at $\xi'_{20}$ (chord in $C(\La_0^-,(\La_2^-)')$ canonically identified with $\xi_{21}\in C(\La_1^-,\La_2^-)$) and a negative gradient flow line of $f_-$ from the maximum $e_i^-$ to the ending point of $\xi'_{20}$. Translating this back to our setting, the configuration of type I contributes $\xi_{20}$ to $\mfm_2(\xi_{21},e)$, and thus we have $\mfm_2(\xi_{21},e)=\xi_{20}+\bs{y}^+_{02}$.

\begin{figure}[ht]  
	\begin{center}\includegraphics[width=5cm]{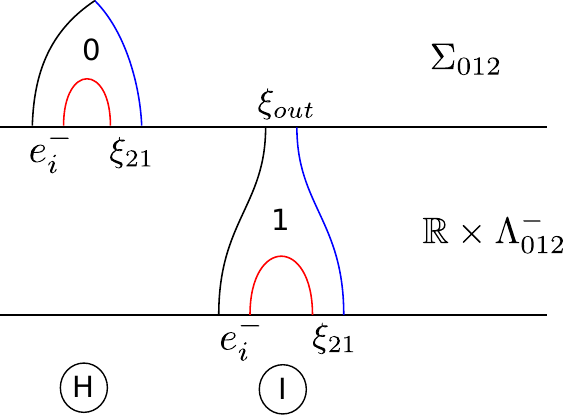}\end{center}
	\caption{Curves contributing to $\mfm_2(\xi_{21},e)$}
	\label{unit7}
\end{figure}
\end{proof} 
\begin{rem}
	Generalizing the conjectural \cite[Lemma 4.10]{E1} to the case of cobordisms, one could probably prove that with the choice of basis given above the matrix of $\mfm_2(\,\cdot\,,e)$ is actually the identity matrix. However we don't need such a strong statement on the chain level here, but what we get is that it is the identity in homology (see details at end of the current section).
\end{rem}

We will apply now the previous theorem to a $3$-copy $(\Sigma_0,\Sigma_1,\Sigma_2)$ of $\Sigma_0$. 
By $3$-copy we mean that $\Sigma_1=\varphi_{\widetilde{H}_1}^{\eta_1}(\Sigma_0)$ and $\Sigma_2=\varphi_{\widetilde{H}_2}^{\eta_2}(\Sigma_0)$ for $\eta_2>\eta_1$ and where $\widetilde{H}_1$ and $\widetilde{H}_2$ are small perturbations of the Hamiltonian $H$ from the beginning of this section.
Moreover, we assume that there exist Morse functions $F_{1}$, $F_{2}$ on $\Sigma_0$ such that $F_{2}-F_{1}$ is also Morse, and 
\begin{itemize}
	\item for $i=1,2$, $\Sigma_i$ is viewed as the graph of $dF_i$ in a standard neighborhood of $\Sigma_0$,
	\item on the cylindrical ends of $\Sigma_0$ we have $F_i=e^t(f_i^\pm-\eta_i)$, $i=1,2$, for $f_i^\pm:\La_0^\pm\to\R$ Morse functions whose $\Cc_0$-norm are strictly smaller than $\min\{\eta_1,\eta_2-\eta_1\}/2$ and such that $f_2-f_1$ is also Morse.
	\item we assume that $F_1$, $F_2$ and $F_2-F_1$ admit a unique maximum on each filling component of $\Sigma_0$ and no maximum on each component with a non-empty negative end; and that $f_1$, $f_2$ and $f_2-f_1$ admit a unique maximum on each connected component.
\end{itemize}
Note that the critical points of $F_i$ are in one-to-one correspondence with the intersection points in $\Sigma_0\cap\Sigma_i$ for $i=1,2$ and the critical points of $F_2-F_1$ are in one-to-one correspondence with intersection points in $\Sigma_1\cap\Sigma_2$.
Then we have the following:
\begin{coro}\label{coro:unit}
	Given the $3$-copy $(\Sigma_0,\Sigma_1,\Sigma_2)$ described above, we have
	$$\mfm_2(e_{\Sigma_1,\Sigma_2},e_{\Sigma_0,\Sigma_1})=e_{\Sigma_0,\Sigma_2}$$
\end{coro}
\begin{proof}
	It is enough to consider the case where $\Sigma_0$ is connected. The case of a filling is already known, see for example \cite{GPS}. We recall a proof in our setting. Assume $\Sigma_0$ is a connected filling of $\La_0^+$, then we have $e_{\Sigma_0,\Sigma_1}=e_{\Sigma_0,\Sigma_1}^0$ and $e_{\Sigma_1,\Sigma_2}=e_{\Sigma_1,\Sigma_2}^0$ and according to Theorem \ref{teo_unit}:
	\begin{alignat*}{1}
		\mfm_2(e_{\Sigma_1,\Sigma_2}^0,e_{\Sigma_0,\Sigma_1}^0)=e_{\Sigma_0,\Sigma_2}^0+\bs{y}_{02}^+
	\end{alignat*}
where $\bs{y}_{02}^+\in CF^+(\Sigma_0,\Sigma_2)$, and each element in $\bs{y}_{02}^+$ has action bigger than the action of $e_{\Sigma_0,\Sigma_2}^0$. Observe then that any triangle asymptotic to $y_{02}^+\neq e_{\Sigma_0,\Sigma_2}^0, e_{\Sigma_0,\Sigma_1}^0$, and $e_{\Sigma_1,\Sigma_2}^0$ would have to leave a small neighborhood of a gradient flow line of $F_{1}$ from $e_{\Sigma_1,\Sigma_2}^0$ to the maximum $e_{\Sigma_0,\Sigma_1}^0$. But for a sufficiently small perturbation no $y_{02}^+\neq e_{\Sigma_0,\Sigma_2}^0$ has action big enough for such a triangle to exist.

Suppose now that $\Sigma_0$ is a connected cobordism from $\La_0^-\neq\emptyset$ to $\La_0^+$. Then $e_{\Sigma_0,\Sigma_1}=e_{\Sigma_0,\Sigma_1}^-$ and $e_{\Sigma_1,\Sigma_2}=e_{\Sigma_1,\Sigma_2}^-$ and according to Theorem \ref{teo_unit}:
\begin{alignat*}{1}
\mfm_2(e_{\Sigma_1,\Sigma_2}^-,e_{\Sigma_0,\Sigma_1}^-)=e_{\Sigma_0,\Sigma_2}^-+\bs{y}_{02}^+
\end{alignat*}
The proof is the same as in the filling case after wrapping slightly the negative ends of $\Sigma_1$ and $\Sigma_2$ in the positive Reeb direction. We wrap so that the negative end becomes a cylinder over $\La_0^-\cup\widetilde{\La}_1^-\cup\widetilde{\La}_2^-$ where $\widetilde{\La}_1^-$ is a perturbation of a small push-off of $\La_0^-$ in the positive Reeb direction and $\widetilde{\La}_2^-$ is a perturbation of a small push-off of $\widetilde{\La}_1^-$ in the positive Reeb direction. The pseudo-holomorphic disc asymptotic to $e_{\Sigma_0,\Sigma_2}^-$, $e_{\Sigma_0,\Sigma_1}^-$ and $e_{\Sigma_1,\Sigma_2}^-$ corresponds after wrapping to a triangle asymptotic to the corresponding intersection points. So then,  for the same reasons as before a disc asymptotic to $y_{02}^+,e_{\Sigma_0,\Sigma_1}^-$ and $e_{\Sigma_1,\Sigma_2}^-$ can not exist.
\end{proof}

We end this section by proving that the transfer maps preserve the continuation element. Consider a pair $(V_0\odot W_0,V_1\odot W_1)$ such that $V_1\odot W_1$ is a small perturbation of $(V_0\odot W_0)$ the  same way as we perturbed $\Sigma_0$ to get $\Sigma_1$ previously, in particular $\La_1^\pm$ is a perturbation of a push-off of $\La_0^\pm$ in the negative Reeb direction. We assume moreover that the Morse function $F$ used to perturb the compact part of $V_0\odot W_0$ is such that $(V_0,V_1)$ and $(W_0,W_1)$ are also pairs of cobordisms of the same type, so $\La_1$ is a perturbation of a push-off of $\La_0$ in the negative Reeb direction. 

Giving this, by what we did previously, there are continuation elements $e_V\in\Cth_+(V_0,V_1)$, $e_W\in\Cth_+(W_0,W_1)$ and $e_{V\odot W}\in\Cth_+(V_0\odot W_0,V_1\odot W_1)$, described as follows
\begin{alignat*}{1}
	&e_V=e_V^0+e_V^-=\sum e_{V,i}^0+\sum e_{V,i}^-\\
	&e_W=e_W^0+e_W^-=\sum e_{W,i}^0+\sum e_{W,i}^-\\
	&e_{V\odot W}=e_W^0+e_V^0+e_V^-
\end{alignat*}

\begin{prop}\label{prop:cont1}
	The transfer map $\bs{\Delta}_1^W$ preserves the continuation element.
\end{prop}

\begin{proof}
Directly from the definition, one has
\begin{alignat*}{1}
	\bs{\Delta}_1^W(e_{V\odot W})&=\bs{\Delta}_1^W(e_W^0+e_V^0+e_V^-)=\Delta_1^W(e_W^0)+e_V^0+e_V^-=e_V^0+e_V^-
\end{alignat*}
where the last equality holds for energy reason.
\end{proof}

\begin{prop}\label{prop:cont2}
	The transfer map $\bs{b}_1^V$ preserves the continuation element in homology, i.e. $[\bs{b}_1^V(e_{V\odot W})]=[e_W]$.
\end{prop}

\begin{proof}
Observe first that
\begin{alignat*}{1}
\bs{b}_1^V(e_{V\odot W})&=\bs{b}_1^V(e_{W}^0+e_{V}^0+e_{V}^-)=e_{W}^0+b_1^V\circ\Delta_1^W(e_{W}^0)+b_1^V(e_{V}^0+e_{V}^-)=e_{W}^0+b_1^V(e_{V}^0+e_{V}^-)
\end{alignat*}
for energy reasons. Now, wrapping slightly the positive and negative cylindrical ends of $V_1$ in the positive Reeb direction one can prove that $b_1^V(e_{V}^0+e_{V}^-)=e_{W}^-+E_{W}^-$,
where $E_{W}^-\in C(\La_0,\La_1)$ is a linear combination of non Morse chords (same type of argument as in the proof of Theorem \ref{teo_unit}).

Now, take a third copy $V_2\odot W_2$ as in Corollary \ref{coro:unit} such that $(V_0,V_1,V_2)$ and $(W_0,W_1,W_2)$ are also $3$-copies. 
We claim that the map
\begin{alignat}{1}
	\mfm_2^W(\,\cdot\,,\bs{b}_1^V(e_{V\odot W_{01}})):\Cth_+(W_1,W_2)\to\Cth_+(W_0,W_2)
	\label{map1}
\end{alignat}
is a quasi-isomorphism. Again this follows from studying the pseudo-holomorphic curves involved, repeating some arguments of the proof of Theorem \ref{teo_unit}. As we are working over a field, it admits an inverse.
	
Consider finally a fourth copy $V_3\odot W_3$ being a perturbation of $V_2\odot W_2$ using the same type of perturbation as before. From the third $A_\infty$-relation satisfied by $\mfm^W$  (see Section \ref{A-inf maps}), the fact that $\mfm_2^{V\odot W}(e_{V\odot W_{12}},e_{V\odot W_{01}})=e_{V\odot W_{02}}$, and the fact that $\bs{b}_1^V$ preserves the product structures in homology, we get that the maps
	\begin{alignat*}{1}
	&\mfm_2^W\big(\mfm_2^W(\,\cdot\,,\bs{b}_1^V(e_{V\odot W_{12}})),\bs{b}_1^V(e_{V\odot W_{01}})\big):\Cth_+(W_2,W_3)\to\Cth_+(W_0,W_3)\\
	&\mfm_2^W(\,\cdot\,,\bs{b}_1^V(e_{V\odot W_{02}})):\Cth_+(W_2,W_3)\to\Cth_+(W_0,W_3)
	\end{alignat*}
are homotopic.
It implies that the map \eqref{map1} is homotopic to the identity map (after canonical identification of the generators of the complexes $\Cth_+(W_1,W_2)$ and $\Cth_+(W_0,W_1)$).
Finally, as $e_{W_{12}}$ is a continuation element we have
\begin{alignat*}{1}
\mfm_2^W(e_{W_{12}},\bs{b}_1^V(e_{V\odot W_{01}}))&=\mfm_2^W(e_{W_{12}},e_{W_{01}}+E_{W_{01}}^-)=\mfm_2^W(e_{W_{12}},e_{W_{01}})+\mfm_2^W(e_{W_{12}},E_{W_{01}}^-)\\
&=e_{W_{02}}+\mfm_2^{W,0}(e_{W_{12}},E_{W_{01}}^-)+\mfm_2^{W,-}(e_{W_{12}},E_{W_{01}}^-)\\
&=e_{W_{02}}+\mfm_2^{W,-}(e_{W_{12}}^-,E_{W_{01}}^-)\\
&=e_{W_{02}}+E_{W_{02}}^-
\end{alignat*}
where the second to last equality comes from the fact that the connected components of the cobordisms $W_0,W_1$ have non empty negative end so there is no maximum of the perturbation Morse function, so $\mfm_2^{W,0-}(e_W^0,E_W^-)=0$, and then $\mfm_2^{W,0}(e_W^-,E_W^-)=0$ for action reasons. We have thus $\mfm_2^W(e_W,\bs{b}_1^V(e_{V\odot W}))=\bs{b}_1^V(e_{V\odot W})$. As \eqref{map1} is the identity in homology, we get $[e_W]=[\bs{b}_1^V(e_{V\odot W})]$.
\end{proof}

\begin{rem} Observe that the same arguments show that
	\begin{alignat*}{1}
	\mfm_2^{\Sigma_{012}}(\,\cdot\,,e):\Cth_+(\Sigma_1,\Sigma_2)\to\Cth_+(\Sigma_0,\Sigma_2)
	\end{alignat*}
is the identity in homology as we have proved that it is an isomorphism (Theorem \ref{teo_unit}) and that $\mfm_2(e_{\Sigma_1,\Sigma_2},e_{\Sigma_0,\Sigma_1})=e_{\Sigma_0,\Sigma_2}$ (Corollary \ref{coro:unit}).
\end{rem}

\begin{rem}
	In Sections \ref{sec:higher} and \ref{sec:higher_conc} we will extend the algebraic structures we have encountered to $A_\infty$ ones. In particular we will define an $A_\infty$-category of cobordisms in $\R\times Y$, $\Fuk(\R\times Y)$, and generalize the transfer maps to families of maps satisfying the $A_\infty$-functor equations. 
	Once the technical details to extend our algebraic constructions to Lagrangian cobordisms in a more general Liouville cobordism are carried out, the transfer maps will provide $A_\infty$-functors $\Fuk^{dec}(X_0\odot X_1)\to\Fuk(X_i)$ from the full subcategory $\Fuk^{dec}(X_0\odot X_1)\subset\Fuk(X_0\odot X_1)$ generated by decomposable Lagrangian cobordisms to the Fukaya category of each cobordism. By Proposition \ref{prop:cont1} and \ref{prop:cont2} these functors will be cohomologically unital. 
\end{rem}

\section{An $A_\infty$-category of Lagrangian cobordisms}\label{sec:higher}

\subsection{Higher order maps}\label{A-inf maps}

In this section, we extend the differential $\mfm_1^\Sigma$ and the product $\mfm_2^\Sigma$ to families of maps $\mfm_d^\Sigma$ defined for each $(d+1)$-tuple of pairwise transverse exact Lagrangian cobordisms $(\Sigma_0,\dots,\Sigma_d)$ for all $d\geq1$. Remember that we denote $\mathfrak{C}(\La^\pm_i,\La^\pm_j)=C_{n-1-*}(\La^\pm_i,\La^\pm_j)\oplus C^{*-1}(\La^\pm_j,\La^\pm_i)$. We define first six families of maps, $b_d^+$, $b_d^-$, $\Delta_d^+$, $\Delta_d^-$, $b_d^\Sigma$ and $\Delta_d^\Sigma$:
\begin{alignat*}{1}
	&b_d^\pm:\mathfrak{C}(\La_{d-1}^\pm,\La_d^\pm)\otimes\mathfrak{C}(\La_{d-2}^\pm,\La_{d-1}^\pm)\otimes\dots\otimes\mathfrak{C}(\La_0^\pm,\La_1^\pm)\to C^{*-1}(\La_0^\pm,\La_d^\pm)\\
	&\Delta_d^\pm:\mathfrak{C}(\La_{d-1}^\pm,\La_d^\pm)\otimes\mathfrak{C}(\La_{d-2}^\pm,\La_{d-1}^\pm)\otimes\dots\otimes\mathfrak{C}(\La_0^\pm,\La_1^\pm)\to C_{n-1-*}(\La_d^\pm,\La_0^\pm)\\
	&b^\Sigma_d:\Cth_+(\Sigma_{d-1},\Sigma_d)\otimes\Cth_+(\Sigma_{d-2},\Sigma_{d-1})\otimes\dots\otimes\Cth_+(\Sigma_0,\Sigma_1)\to C^{*-1}(\La_0^+,\La_d^+)\\
	&\Delta^\Sigma_d:\Cth_+(\Sigma_{d-1},\Sigma_d)\otimes\Cth_+(\Sigma_{d-2},\Sigma_{d-1})\otimes\dots\otimes\Cth_+(\Sigma_0,\Sigma_1)\to C_{n-1-*}(\La_d^-,\La_0^-)
\end{alignat*}
as follows:
\begin{alignat*}{1}
	&b_d^+(a_d,\dots,a_1)=\sum\limits_{\gamma_{d,0}}\sum\limits_{\bs{\zeta}_i}\#\widetilde{\cM^1}_{\R\times\La_{0,\dots,d}^+}(\gamma_{d,0};\bs{\zeta}_0,a_1,\dots,a_d,\bs{\zeta}_d)\cdot\ep_i^+(\bs{\zeta}_i)\cdot\gamma_{d,0}\\
	&b_d^-(a_d,\dots,a_1)=\sum\limits_{\gamma_{d,0}}\sum\limits_{\bs{\delta}_i}\#\widetilde{\cM^1}_{\R\times\La_{0,\dots,d}^-}(\gamma_{d,0};\bs{\delta}_0,a_1,\dots,a_d,\bs{\delta}_d)\cdot\ep_i^-(\bs{\delta}_i)\cdot\gamma_{d,0}\\
	&\Delta_d^+(a_d,\dots,a_1)=\sum\limits_{\gamma_{0,d}}\sum\limits_{\bs{\zeta}_i}\#\widetilde{\cM^1}_{\R\times\La_{0,\dots,d}^+}(\gamma_{0,d};\bs{\zeta}_0,a_1,\dots,a_d,\bs{\zeta}_d)\cdot\ep_i^+(\bs{\zeta}_i)\cdot\gamma_{0,d}\\
	&\Delta_d^-(a_d,\dots,a_1)=\sum\limits_{\gamma_{0,d}}\sum\limits_{\bs{\delta}_i}\#\widetilde{\cM^1}_{\R\times\La_{0,\dots,d}^-}(\gamma_{0,d};\bs{\delta}_0,a_1,\dots,a_d,\bs{\delta}_d)\cdot\ep_i^-(\bs{\delta}_i)\cdot\gamma_{0,d}\\
	&b_d^\Sigma(a_d,\dots,a_1)=\sum\limits_{\gamma_{d,0}}\sum\limits_{\bs{\delta}_i}\#\cM^0_{\Sigma_{0,\dots,d}}(\gamma_{d,0};\bs{\delta}_0,a_1,\dots,a_d,\bs{\delta}_d)\cdot\ep_i^-(\bs{\delta}_i)\cdot\gamma_{d,0}\\
	&\Delta_d^\Sigma(a_d,\dots,a_1)=\sum\limits_{\gamma_{0,d}}\sum\limits_{\bs{\delta}_i}\#\cM^0_{\Sigma_{0,\dots,d}}(\gamma_{0,d};\bs{\delta}_0,a_1,\dots,a_d,\bs{\delta}_d)\cdot\ep_i^-(\bs{\delta}_i)\cdot\gamma_{0,d}
\end{alignat*}
Observe that these maps for the case $d=1$ have already been considered in Section \ref{sec:Cth}, and $\Delta_2^+$, $\Delta_2^\Sigma$ and $b_2^-$ have been defined in Section \ref{section:def_product} already.

Given these families of maps, we define the higher order maps $\mfm_d$ as being the sum $\mfm_d=\mfm_d^++\mfm_d^0+\mfm_d^-$, where each component is defined by:
\begin{alignat}{1}
	&\mfm_d^+(a_d,\dots,a_1)=\sum\limits_{j=1}^d\sum\limits_{i_1+\dots+ i_j=d}\Delta^+_j\big(\bs{b}_{i_j}^\Sigma(a_d,\dots,a_{d-i_j+1}),\dots,\bs{b}_{i_1}^\Sigma(a_{i_1+1},\dots,a_1)\big)\label{defmap+}\\
	&\mfm_d^0(a_d,\dots,a_1)=\sum\limits_{x\in\Sigma_0\cap\Sigma_d}\sum\limits_{\bs{\delta}_i}\#\cM^0_{\Sigma_{0,...,d}}(x;\bs{\delta}_0,a_1,\bs{\delta}_1,\dots,a_d,\bs{\delta}_d)\cdot\ep_i^-(\bs{\delta}_i)\cdot x\label{defmap0}\\
	&\mfm_d^-(a_d,\dots,a_1)=\sum\limits_{j=1}^d\sum\limits_{i_1+\dots+ i_j=d}b^-_j\big(\bs{\Delta}_{i_j}^\Sigma(a_d,\dots,a_{d-i_j+1}),\dots,\bs{\Delta}_{i_1}^\Sigma(a_{i_1+1},\dots,a_1)\big)\label{defmap-}
\end{alignat}
where the maps $\bs{b}_1^\Sigma$ and $\bs{\Delta}_1^\Sigma$ are special cases of transfer maps as explained in Section \ref{special_case}, and for $j\geq2$ one has $\bs{b}_j^\Sigma:=b_j^\Sigma$ and $\bs{\Delta}_j^\Sigma:=\Delta_j^\Sigma$. 
In the formulas above, for $1\leq j\leq d$ fixed and an index $i_s$, the maps $\bs{b}_{i_s}^\Sigma$ and $\bs{\Delta}_{i_s}^\Sigma$ are defined on (with convention $i_0=-1$):
\begin{alignat*}{1}
	\Cth_+(\Sigma_{i_{s}+\dots+i_1},\Sigma_{1+i_{s}+\dots+i_1})\otimes\dots\otimes\Cth_+(\Sigma_{1+i_{s-1}+\dots+i_1},\Sigma_{2+i_{s-1}+\dots+ i_1})
\end{alignat*}
and the maps $\Delta_j^+$ and $b_j^-$ on 
\begin{alignat*}{1}
	\Cth_+(\Sigma_{1+i_{j-1}+\dots+i_1},\Sigma_d)\otimes\dots\otimes\Cth_+(\Sigma_0,\Sigma_{i_1+1})
\end{alignat*}
For $d=1,2$, the Formulas \ref{defmap+}, \ref{defmap0} and \ref{defmap-} recover the definitions of the differential $\mfm_1$ and the product $\mfm_2$ given in Sections \ref{section:def_complex} and \ref{section:def_product}.
\begin{rem}
	Observe that for energy reasons, depending on the $d$-tuple of asymptotics, it can happen that a lot of terms in the Formulas \eqref{defmap+} and \eqref{defmap-} vanish, but for example, if $a_i$ is a Reeb chord in $C(\La_{i+1}^+,\La_i^+)$ for $i=0,\dots d$, then none of them vanish.
\end{rem}
\begin{rem}
	The maps $b_d^+$ and $\Delta_d^-$ defined previously are not useful to define the maps $\mfm_d$ but they naturally  appear in the proof of the $A_\infty$-equations, see Sections \ref{proofrelinf+} and \ref{proofrelinf-} below.
\end{rem}
Now we want to show that the maps $\{\mfm_d\}_{d\geq1}$ satisfy the $A_\infty$-equations, i.e. for all $k\geq1$ and all $(k+1)$-tuple of transverse cobordisms $(\Sigma_0,\dots,\Sigma_k)$, we want to check that for every $1\leq d\leq k$ and $(d+1)$-sub-tuple $(\Sigma_{i_0},\dots,\Sigma_{i_d})$ with $i_0<\dots<i_d$, we have:
\begin{alignat*}{1}
	\sum\limits_{j=1}^d\sum\limits_{n=0}^{d-j}\mfm_{d-j+1}(\id^{\otimes d-j-n}\otimes\mfm_j\otimes\id^{\otimes n})=0
\end{alignat*}
To simplify notations in the following we assume that the $(d+1)$-tuple $(\Sigma_{i_0},\dots,\Sigma_{i_d})$ is $(\Sigma_0,\dots,\Sigma_d)$.
As usual, we decompose this equation into three equations to check:
\begin{alignat}{1}
	&\sum\limits_{j=1}^d\sum\limits_{n=0}^{d-j}\mfm^+_{d-j+1}(\id^{\otimes d-j-n}\otimes\mfm_j\otimes\id^{\otimes n})=0\label{relinf+}\\
	&\sum\limits_{j=1}^d\sum\limits_{n=0}^{d-j}\mfm^0_{d-j+1}(\id^{\otimes d-j-n}\otimes\mfm_j\otimes\id^{\otimes n})=0\label{relinf0}\\
	&\sum\limits_{j=1}^d\sum\limits_{n=0}^{d-j}\mfm^-_{d-j+1}(\id^{\otimes d-j-n}\otimes\mfm_j\otimes\id^{\otimes n})=0\label{relinf-}
\end{alignat}

\subsubsection{Proof of Equation \eqref{relinf+}}\label{proofrelinf+}
Consider the boundary of the compactification of $\widetilde{\cM^2}_{\R\times\La_{0,\dots,d}^+}(\gamma_{0,d};\bs{\zeta}_0,a_1,\dots,a_d,\bs{\zeta}_d)$. According to the compactness results for one dimensional moduli spaces of pseudo-holomorphic discs with cylindrical Lagrangian boundary conditions as recalled in Section \ref{sec:structure}, the non trivial components of broken discs in the boundary consist of two index $1$ discs glued along a node asymptotic to a Reeb chord. If it is a positive asymptotic for the index $1$ disc not containing the output puncture, this disc contributes to a map $b_j^+$, and if it is a negative asymptotic, this disc contributes to a map $\Delta_j^+$. Hence we get the following:
\begin{lem}\label{relDelta-}
	For all $1\leq d\leq k$, we have $\sum\limits_{j=1}^d\sum\limits_{n=0}^{d-j}\Delta^+_{d-j+1}\big(\id^{\otimes d-j-n}\otimes(b_j^++\Delta_j^+)\otimes\id^{\otimes n}\big)=0$
\end{lem}
Then, we also have:
\begin{lem}\label{relbSigma}
	For all $1\leq d\leq k$, we have
	\begin{alignat*}{1}
	\sum_{j=1}^d\sum_{n=0}^{d-j}b_{d-j+1}^\Sigma\big(\id^{\otimes d-j-n}\otimes\mfm_j\otimes\id^{\otimes n}\big)+\sum\limits_{j=1}^d\sum\limits_{i_1+\dots+i_j=d}b^+_{j}\big(\bs{b}^\Sigma_{i_j}\otimes\dots\otimes\bs{b}_{i_1}^\Sigma\big)=0
	\end{alignat*}
\end{lem}
\begin{proof}
This time we have to consider the boundary of the compactification of a moduli space
$\cM^1_{\Sigma_{0,\dots,d}}(\gamma_{d,0};\bs{\delta}_0,a_1,\dots,a_d,\bs{\delta}_d)$. Again as recalled in Section \ref{sec:structure}, the broken discs are of two types. It can first consist of two index $0$ discs glued at a common intersection point. In this case, the one not containing the output puncture asymptotic to $\gamma_{d,0}$ contributes $\mfm_j^0$, and the disc containing the output contributes to a banana $b_{d-j+1}^\Sigma$. The other type of possible broken disc consists of several (possibly 0 !) non trivial index $0$ components and an index $1$ disc with boundary on the negative or positive cylindrical ends, such that each index $0$ disc is connected to the index $1$ one via a Reeb chord. Observe that the output puncture is asymptotic to a chord in the positive end, so there are two subcases:
\begin{enumerate}
	\item the output puncture is contained in the index $1$ disc. In this case, this disc has boundary on the positive ends and contributes to $b_j^+$ while the index $0$ discs must then have at least one positive Reeb chord asymptotic (connecting it to the index $1$ disc) and so each of them contributes to a banana $b^\Sigma$. Note that if among the asymptotics $a_1,a_2,\dots,a_d$ there is a chord in the positive end, this chord could be an asymptotic of the index $1$ disc or of an index $0$ banana $b^\Sigma$, this is why we have the bold symbols maps $\bs{b}^\Sigma$ in the formula.
	\item the output puncture is contained in an index $0$ disc. This disc contributes thus to a map $b_{d-j+1}^\Sigma$. Then, if the index $1$ disc has boundary in the positive ends, it contributes, with the index $0$ discs not containing the output, to $\mfm_j^+$. If the index $0$ disc has boundary on the negative ends, it will contribute, with the index $0$ discs not containing the output, to $\mfm_j^-$.
\end{enumerate}
Summing the algebraic contributions of all the different types of broken discs described above gives the relation.
\end{proof}

Now we can compute
\begin{alignat*}{1}
&\sum\limits_{j=1}^d\sum\limits_{n=0}^{d-j}\mfm^+_{d-j+1}(\id^{\otimes d-j-n}\otimes\mfm_j\otimes\id^{\otimes n})\\
&=\sum_{j=1}^d\sum_{n=0}^{d-j}\sum_{k=1}^{d-j+1}\sum_{s=1}^k\sum_{\substack{i_1+...+i_k=d-j+1\\0\leq r=n-i_1-...-i_{s-1}\leq i_s}}\Delta^+_{k}\big(\bs{b}^\Sigma_{i_k}\otimes\dots\otimes\bs{b}^\Sigma_{i_s}(\id^{\otimes i_s-j-r}\otimes\mfm_j\otimes\id^{\otimes r})\otimes\dots\otimes\bs{b}^\Sigma_{i_1}\big)
\end{alignat*}
In this sum, we fix first the number $j$ of entries for the map $\mfm_j$, and then a partition of $d-j+1$ for the maps $\bs{b}^\Sigma$. Note that if $i_s<j$ the terms $\bs{b}^\Sigma_{i_s}(\id^{\otimes i_s-j-r}\otimes\mfm_j\otimes\id^{\otimes r})$ vanish. We could also choose first a partition of $d$ and then the number of entries for the $\mfm$ "in the middle". Thus, the sum above is equal to:
\begin{alignat*}{1}
&=\sum_{k=1}^d\sum_{i_1+...+i_k=d}\sum_{s=1}^k\sum_{j=1}^{i_s}\sum_{n=0}^{i_s-j}\Delta^+_{k}\big(\bs{b}^\Sigma_{i_k}\otimes\dots\otimes\bs{b}^\Sigma_{i_s-j+1}(\id^{\otimes i_s-j-n}\otimes\mfm_j\otimes\id^{\otimes n})\otimes\dots\otimes\bs{b}^\Sigma_{i_1}\big)\\
&=\sum_{k=1}^d\sum_{i_1+...+i_k=d}\sum_{s=1}^k\Delta^+_{k}\big(\bs{b}^\Sigma_{i_k}\otimes\dots\otimes\sum_{j=1}^{i_s}\sum_{n=0}^{i_s-j}\bs{b}^\Sigma_{i_s-j+1}(\id^{\otimes i_s-j-n}\otimes\mfm_j\otimes\id^{\otimes n})\otimes\dots\otimes\bs{b}^\Sigma_{i_1}\big)
\end{alignat*}
Using the definition of $\bs{b}$ and then Lemma \ref{relbSigma}, we have
\begin{alignat*}{1}
	\sum_{j=1}^{i_s}\sum_{n=0}^{i_s-j}\bs{b}^\Sigma_{i_s-j+1}(\id^{\otimes i_s-j-n}\otimes\mfm_j\otimes\id^{\otimes n})&=\sum_{j=1}^{i_s}\sum_{n=0}^{i_s-j}b^\Sigma_{i_s-j+1}(\id^{\otimes i_s-j-n}\otimes\mfm_j\otimes\id^{\otimes n})+\mfm_{i_s}^+\\	
	&=\sum\limits_{u=1}^{i_s}\sum\limits_{t_1+\dots+t_u=i_s}b^+_{u}\big(\bs{b}^\Sigma_{t_u}\otimes\dots\otimes\bs{b}_{t_1}^\Sigma\big)+\mfm_{i_s}^+
\end{alignat*}
Given this, we rewrite
\begin{alignat*}{1}
&\sum\limits_{j=1}^d\sum\limits_{n=0}^{d-j}\mfm^+_{d-j+1}(\id^{\otimes d-j-n}\otimes\mfm_j\otimes\id^{\otimes n})\\
&=\sum_{k=1}^d\sum_{i_1+...+i_k=d}\sum_{s=1}^k\Delta^+_{k}\big(\bs{b}^\Sigma_{i_k}\otimes\dots\otimes\sum\limits_{u=1}^{i_s}\sum\limits_{t_1+\dots+t_u=i_s}b^+_{u}\big(\bs{b}^\Sigma_{t_u}\otimes\dots\otimes\bs{b}_{t_1}^\Sigma\big)\otimes\dots\otimes\bs{b}^\Sigma_{i_1}\big)\\
&\hspace{45mm}+\sum_{k=1}^d\sum_{i_1+...+i_k=d}\sum_{s=1}^k\Delta^+_{k}\big(\bs{b}^\Sigma_{i_k}\otimes\dots\otimes\mfm_{i_s}^+\otimes\dots\otimes\bs{b}^\Sigma_{i_1}\big)
\end{alignat*}
and we finally use Lemma \ref{relDelta-} to obtain
\begin{alignat*}{1}
&\sum\limits_{j=1}^d\sum\limits_{n=0}^{d-j}\mfm^+_{d-j+1}(\id^{\otimes d-j-n}\otimes\mfm_j\otimes\id^{\otimes n})\\
&=\sum_{k=1}^d\sum_{i_1+...+i_k=d}\sum_{s=1}^k\Delta^+_{k}\big(\bs{b}^\Sigma_{i_k}\otimes\dots\otimes\sum\limits_{u=1}^{i_s}\sum\limits_{t_1+\dots+t_u=i_s}\Delta^+_{u}\big(\bs{b}^\Sigma_{t_u}\otimes\dots\otimes\bs{b}_{t_1}^\Sigma\big)\otimes\dots\otimes\bs{b}^\Sigma_{i_1}\big)\\
&\hspace{3cm}+\sum_{k=1}^d\sum_{i_1+...+i_k=d}\sum_{s=1}^k\Delta^+_{k}\big(\bs{b}^\Sigma_{i_k}\otimes\dots\otimes\mfm_{i_s}^+\otimes\dots\otimes\bs{b}^\Sigma_{i_1}\big)\\
&=\sum_{k=1}^d\sum_{i_1+...+i_k=d}\sum_{s=1}^k\Delta^+_{k}\big(\bs{b}^\Sigma_{i_k}\otimes\dots\otimes\mfm_{i_s}^+\otimes\dots\otimes\bs{b}^\Sigma_{i_1}\big)\\
&\hspace{3cm}+\sum_{k=1}^d\sum_{i_1+...+i_k=d}\sum_{s=1}^k\Delta^+_{k}\big(\bs{b}^\Sigma_{i_k}\otimes\dots\otimes\mfm_{i_s}^+\otimes\dots\otimes\bs{b}^\Sigma_{i_1}\big)\\
&=0
\end{alignat*}

\subsubsection{Proof of Equation \eqref{relinf0}}
This equation is obtained after describing the broken discs in the boundary of the compactification of $\cM^1_{\Sigma_{0,...,d}}(x;\bs{\delta}_0,a_1,\bs{\delta}_1,\dots,a_d,\bs{\delta}_d)$. As well as in the proof of Lemma \ref{relbSigma}, there are different types of broken discs, depending on if it contains an index $1$ component or not, but the total algebraic contribution of them gives the relation \eqref{relinf0}.

\subsubsection{Proof of Equation \eqref{relinf-}}\label{proofrelinf-}
Finally, to get Equation \eqref{relinf-}, we study the broken discs in the boundary of the compactification of the moduli spaces 
\begin{alignat}{1}
\widetilde{\cM^2}_{\R\times\La_{0,\dots,d}^-}(\gamma_{d0};\bs{\delta}_0,a_1,\dots,a_d,\bs{\delta}_d)
\end{alignat}
and
\begin{alignat}{1}
\cM^1_{\Sigma_{0,\dots,d}}(\gamma_{0d};\bs{\delta}_0,a_1,\dots,a_d,\bs{\delta}_d)
\end{alignat}
This gives us respectively the following lemmas:
\begin{lem}\label{relb-}
	For all $d\geq1$, we have $\sum\limits_{j=1}^d\sum\limits_{n=0}^{d-j}b^-_{d-j+1}\big(\id^{\otimes d-j-n}\otimes(b_j^-+\Delta_j^-)\otimes\id^{\otimes n}\big)=0$
\end{lem}
and
\begin{lem}\label{relDeltaSigma}
	For all $d\geq1$, we have
	\begin{alignat*}{1}
		\sum\limits_{j=1}^d\sum\limits_{n=0}^{d-j}\Delta^\Sigma_{d-j+1}\big(\id^{\otimes d-j-n}\otimes\mfm_j\otimes\id^{\otimes n}\big)+\sum\limits_{j=1}^d\sum\limits_{i_1+\dots+i_j=d}\Delta^-_{j}\big(\bs{\Delta}^\Sigma_{i_j}\otimes\dots\otimes\bs{\Delta}_{i_1}^\Sigma\big)=0
	\end{alignat*}
\end{lem}
We can now prove Equation \eqref{relinf-} for $d\geq1$ in a direct way:
\begin{alignat*}{1}
	&\sum_{j=1}^d\sum_{n=0}^{d-j}\mfm^-_{d-j+1}\big(\id^{\otimes d-j-n}\otimes\mfm_j\otimes\id^{\otimes n}\big)\\
	&=\sum_{j=1}^d\sum_{n=0}^{d-j}\sum_{k=1}^{d-j+1}\sum_{s=1}^k\sum_{\substack{i_1+...+i_k=d-j+1\\0\leq r=n-i_1-...-i_{s-1}\leq i_s}}b^-_{k}\big(\bs{\Delta}^\Sigma_{i_k}\otimes\dots\otimes\bs{\Delta}^\Sigma_{i_s}(\id^{\otimes i_s-j-r}\otimes\mfm_j\otimes\id^{\otimes r})\otimes\dots\otimes\bs{\Delta}^\Sigma_{i_1}\big)\\
	&=\sum_{k=1}^d\sum_{i_1+...+i_k=d}\sum_{s=1}^k\sum_{j=1}^{i_s}\sum_{n=0}^{i_s-j}b^-_{k}\big(\bs{\Delta}^\Sigma_{i_k}\otimes\dots\otimes\bs{\Delta}^\Sigma_{i_s-j+1}(\id^{\otimes i_s-j-n}\otimes\mfm_j\otimes\id^{\otimes n})\otimes\dots\otimes\bs{\Delta}^\Sigma_{i_1}\big)\\
	&=\sum_{k=1}^d\sum_{i_1+...+i_k=d}\sum_{s=1}^kb^-_{k}\big(\bs{\Delta}^\Sigma_{i_k}\otimes\dots\otimes\sum_{j=1}^{i_s}\sum_{n=0}^{i_s-j}\bs{\Delta}^\Sigma_{i_s-j+1}(\id^{\otimes i_s-j-n}\otimes\mfm_j\otimes\id^{\otimes n})\otimes\dots\otimes\bs{\Delta}^\Sigma_{i_1}\big)
\end{alignat*}	
Observe that using the definition of $\bs{\Delta}$, then adding the vanishing term $\Delta_1^\Sigma\circ\mfm_{i_s}^-$, and then applying Lemma \ref{relDeltaSigma}, we have the following consecutive equalities:
\begin{alignat*}{1}
	\sum_{j=1}^{i_s}&\sum_{n=0}^{i_s-j}\bs{\Delta}^\Sigma_{i_s-j+1}(\id^{\otimes i_s-j-n}\otimes\mfm_j\otimes\id^{\otimes n})\\
	&=\sum_{j=1}^{i_s-1}\sum_{n=0}^{i_s-j}\Delta^\Sigma_{i_s-j+1}(\id^{\otimes i_s-j-n}\otimes\mfm_j\otimes\id^{\otimes n})+\Delta_1^\Sigma\circ\mfm_{i_s}^++\Delta_1^\Sigma\circ\mfm_{i_s}^0+\mfm_{i_s}^-\\
	&=\sum_{j=1}^{i_s}\sum_{n=0}^{i_s-j}\Delta^\Sigma_{i_s-j+1}(\id^{\otimes i_s-j-n}\otimes\mfm_j\otimes\id^{\otimes n})+\mfm_{i_s}^-\\
	&=\sum\limits_{u=1}^{i_s}\sum\limits_{t_1+\dots+t_u=i_s}\Delta^-_{u}\big(\bs{\Delta}^\Sigma_{t_u}\otimes\dots\otimes\bs{\Delta}_{t_1}^\Sigma\big)+\mfm_{i_s}^-
\end{alignat*}
If we plug it into the expression above, we get:
\begin{alignat*}{1}
&\sum_{j=1}^d\sum_{n=0}^{d-j}\mfm^-_{d-j+1}\big(\id^{\otimes d-j-n}\otimes\mfm_j\otimes\id^{\otimes n}\big)\\
&=\sum_{k=1}^d\sum_{i_1+...+i_k=d}\sum_{s=1}^kb^-_{k}\big(\bs{\Delta}^\Sigma_{i_k}\otimes\dots\otimes\big(\sum\limits_{u=1}^{i_s}\sum\limits_{t_1+\dots+t_u=i_s}\Delta^-_{u}\big(\bs{\Delta}^\Sigma_{t_u}\otimes\dots\otimes\bs{\Delta}_{t_1}^\Sigma\big)\big)\otimes\dots\otimes\bs{\Delta}^\Sigma_{i_1}\big)\\
&\hspace{3cm}+\sum_{k=1}^d\sum_{i_1+...+i_k=d}\sum_{s=1}^kb^-_{k}\big(\bs{\Delta}^\Sigma_{i_k}\otimes\dots\otimes\mfm_{i_s}^-\otimes\dots\otimes\bs{\Delta}^\Sigma_{i_1}\big)
\end{alignat*}
Finally we apply Lemma \ref{relb-} and we obtain
\begin{alignat*}{1}
&\sum_{j=1}^d\sum_{n=0}^{d-j}\mfm^-_{d-j+1}\big(\id^{\otimes d-j-n}\otimes\mfm_j\otimes\id^{\otimes n}\big)\\
&=\sum_{k=1}^d\sum_{i_1+...+i_k=d}\sum_{s=1}^kb^-_{k}\big(\bs{\Delta}^\Sigma_{i_k}\otimes\dots\otimes\big(\sum\limits_{u=1}^{i_s}\sum\limits_{t_1+\dots+t_u=i_s}b^-_{u}\big(\bs{\Delta}^\Sigma_{t_u}\otimes\dots\otimes\bs{\Delta}_{t_1}^\Sigma\big)\big)\otimes\dots\otimes\bs{\Delta}^\Sigma_{i_1}\big)\\
&\hspace{3cm}+\sum_{k=1}^d\sum_{i_1+...+i_k=d}\sum_{s=1}^kb^-_{k}\big(\bs{\Delta}^\Sigma_{i_k}\otimes\dots\otimes\mfm_{i_s}^-\otimes\dots\otimes\bs{\Delta}^\Sigma_{i_1}\big)\\
&=\sum_{k=1}^d\sum_{i_1+...+i_k=d}\sum_{s=1}^kb^-_{k}\big(\bs{\Delta}^\Sigma_{i_k}\otimes\dots\otimes\mfm_{i_s}^-\otimes\dots\otimes\bs{\Delta}^\Sigma_{i_1}\big)\\
&\hspace{3cm}+\sum_{k=1}^d\sum_{i_1+...+i_k=d}\sum_{s=1}^kb^-_{k}\big(\bs{\Delta}^\Sigma_{i_k}\otimes\dots\otimes\mfm_{i_s}^-\otimes\dots\otimes\bs{\Delta}^\Sigma_{i_1}\big)=0
\end{alignat*}

\subsection{Fukaya category of Lagrangian cobordisms}\label{sec:Fuk}

We define an $A_\infty$-category $\Fuk(\R\times Y)$ whose objects are exact Lagrangian cobordisms whose negative ends are cylinders over Legendrian admitting augmentations. We define this category by localization, in the same spirit as the definition of the wrapped Fukaya category of a Liouville sector in \cite{GPS} to which we refer for details about quotient and localization, as well as \cite{LO}.

\begin{defi}
	A Hamiltonian isotopy $\varphi_h^s$ of $\R\times Y$ is called \textit{cylindrical at infinity} if there exists a $R>0$ such that $\varphi_h^s$ does not depend on the symplectization coordinate $t$ in $(-\infty,-R)\times Y$ and $(R,\infty)\times Y$.
\end{defi}

Let $E$ be a countable set of exact Lagrangian cobordisms in $\R\times Y$, with negative cylindrical ends on Legendrian submanifolds of $Y$ admitting an augmentation. Assume that any exact Lagrangian cobordism $\La^-\prec_\Sigma\La^+$ such that $\La^-$ admits an augmentation is isotopic to one in $E$ through a cylindrical at infinity Hamiltonian isotopy.
For each cobordism $\La^-\prec_\Sigma\La^+$ in $E$, we choose a  sequence $\Sigma^\bullet$ of cobordisms
\begin{alignat*}{1}
\Sigma^\bullet=(\Sigma^{(0)},\Sigma^{(1)},\Sigma^{(2)},\dots)
\end{alignat*}
as follows. First $\Sigma^{(0)}=\Sigma$, and then we need to make several choices:
\begin{enumerate}
	\item a sequence $\{\eta_i\}_{i\geq1}$ of real numbers such that $\sum\limits_{j>0}\eta_j$ is strictly smaller than the length of the shortest Reeb chord of $\La^+\cup\La^-$, and denote $\tau_i=\sum\limits_{j=1}^i\eta_j$,
	\item given that $\Sigma$ is cylindrical outside $[-T,T]\times Y$, and given $N>0$, we choose Hamiltonians $H_i:\R\times Y\to\R$, $i\geq1$, being small perturbations of $h_{T,N}$ (see Section \ref{sec:unit}) and set $\Sigma^{(i)}=\varphi_{H_i}^{\tau_i}(\Sigma)$ such that $\Sigma^{(i)}$ is the graph of $dF_i$ in a standard neighborhood of the $0$-section in $T^*\Sigma$, for $F_i:\Sigma\to\R$ Morse function satisfying the following:
\begin{enumerate}
	\item on $\Sigma\cap\big([T+N,\infty)\times Y\big)$, resp. $\Sigma\cap\big((-\infty,-T-N]\times Y\big)$, $F_i$ is equal to $e^t(f_i^+-\tau_i)$, resp $e^t(f_i^--\tau_i)$ where $f_i^\pm:\La^\pm\to\R$ are Morse functions such that the $\mathcal{C}^0$-norm of $f_i^\pm$ is strictly smaller than $\min\{\eta_i,\eta_{i+1}\}/2$,
	\item the functions $F_i-F_j$, $f_i^\pm-f_j^\pm$ are Morse for $i\neq j$,
	\item the functions $f_i^\pm$ and $f_i^\pm-f_j^\pm$ admit a unique maximum on each connected component while the functions $F_i$ and $F_i-F_j$ admit a unique maximum on each filling connected component and no maximum on each connected component of $\Sigma$ admitting a non empty negative end.
\end{enumerate}
\end{enumerate}
We call such a sequence of cobordisms \textit{cofinal}. Note that an augmentation of $\La^-$ gives canonically an augmentation of the negative end of $\Sigma^{(i)}$ for $i\geq1$.
The construction is inductive and the different choices above are made so that for any $(d+1)$-tuple of cobordisms $\Sigma_0,\Sigma_1,\dots,\Sigma_d$ in $E$ (not necessarily distinct), and any strictly increasing sequence of integers $i_0<i_1<\dots<i_d$, the cobordisms $\Sigma_0^{(i_0)},\Sigma_1^{(i_1)},\dots,\Sigma_d^{(i_d)}$ are pairwise transverse.
Let us construct now a strictly unital $A_\infty$-category $\mathcal{O}$ as follows:
\begin{itemize}
	\item Obj($\mathcal{O}$): pairs $(\Sigma^{(i)},\ep^-)$ where $\Sigma\in E$ is an exact Lagrangian cobordism from $\La^-$ to $\La^+$, and $\ep^-$ is an augmentation of $\La^-$,
	\item $\hom_\mathcal{O}\big((\Sigma_0^{(i)},\ep_0^-),(\Sigma_1^{(j)},\ep_1^-)\big)=\left\{
   \begin{array}{ll}
	\Cth_+\big(\Sigma_0^{(i)},\Sigma_1^{(j)}\big)&\mbox{ if }i<j\\
	\mathbb{Z}_2e_{\ep_0^-}^{(i)}&\mbox{ if }\Sigma_0=\Sigma_1, i=j \mbox{ and } \ep_0^-=\ep_1^-\\
	0&\mbox{otherwise}
	\end{array}\right.$
\end{itemize}
where $e_{\ep_0^-}^{(i)}$ is a formal degree $0$ element.
The $A_\infty$-operations are given by the maps defined in Section \ref{A-inf maps} for each $(d+1)$-tuple of cobordisms $\Sigma_0^{(i_0)},\Sigma_1^{(i_1)},\dots,\Sigma_d^{(i_d)}$ with $i_0<i_1<\dots<i_d$, i.e. for such a tuple we have a map
\begin{alignat*}{1}
\mfm_d:\Cth_+(\Sigma_{d-1}^{(i_{d-1})},\Sigma_d^{(i_d)})\otimes\dots\otimes\Cth_+(\Sigma_0^{(i_0)},\Sigma_1^{(i_1)})\to\Cth_+(\Sigma_0^{(i_0)},\Sigma_d^{(i_d)})
\end{alignat*}
These maps extend to maps defined for any $(d+1)$-tuple $\Sigma_0^{(i_0)},\Sigma_1^{(i_1)},\dots,\Sigma_d^{(i_d)}$ with the condition that the elements $e_{\ep_j^-}^{(i)}\in\hom_\mathcal{O}((\Sigma_j^{(i)},\ep_j^-),(\Sigma_j^{(i)},\ep_j^-))$ behave as strict units.\\

We finally define the Fukaya category $\Fuk(\R\times Y)$ of Lagrangian cobordisms in $\R\times Y$ as a quotient of $\mathcal{O}$ by the set of continuation elements, as follows.
Consider $\Sigma\in E$ together with an augmentation $\ep^-$ of $\La^-$. For all $i<j$, there is a continuation element $e_{\Sigma^{(i)},\Sigma^{(j)}}\in\hom_\mathcal{O}\big((\Sigma^{(i)},\ep^-),(\Sigma^{(j)},\ep^-)\big)$ as described in Section \ref{sec:unit}, which is a cycle in $\mathcal{O}$.
%
%
Let $Tw(\mathcal{O})$ denote the $A_\infty$-category of twisted complexes of $\mathcal{O}$ and $\mathcal{C}$ the full subcategory of $Tw(\mathcal{O})$ generated by cones of the continuation elements. We define $\Fuk(\R\times Y):=\mathcal{O}[\mathcal{C}^{-1}]$ to be the image of $\mathcal{O}$ in the quotient $Tw(\mathcal{O})/\mathcal{C}$.

Defined as follows, the category $\Fuk(\R\times Y)$ depends on various choices, namely:
\begin{enumerate}
	\item the choice for each $\Sigma$ in $E$ of a cofinal sequence $\Sigma^{\bullet}=(\Sigma^{(0)},\Sigma^{(1)},\dots)$.
	\item the choice of the countable set $E$ of representatives of Hamiltonian isotopy classes of exact Lagrangian cobordisms with negative end admitting an augmentation,
\end{enumerate}

The fact that the quasi-equivalence class of the category does not depend on the choice of a cofinal sequence for each element in $E$ is purely algebraic.
Assume $\Sigma^{\widetilde\bullet}$ is a cofinal sequence for $\Sigma$ which is a subsequence of a bigger cofinal sequence $\Sigma^\bullet$. Then, denote $\widetilde{\mathcal{O}}$ the category constructed using the cofinal sequence $\Sigma^{\widetilde\bullet}$ and $\mathcal{O}$ the one constructed using $\Sigma^{\bullet}$. The inclusion functor $\widetilde{\mathcal{O}}\to\mathcal{O}$ is full and faithful, and if $\widetilde{\mathcal{C}}\subset Tw\widetilde{\mathcal{O}}$ denotes the full subcategory generated by cones of continuation elements, one gets a cohomologically full and faithful functor $\widetilde{\mathcal{O}}[\widetilde{\mathcal{C}}^{-1}]\to\mathcal{O}[\mathcal{C}^{-1}]$. Moreover, the continuation elements become quasi-isomorphisms in $\mathcal{O}[\mathcal{C}^{-1}]$ thus this functor is a quasi-equivalence.
Now if $\Sigma^{\bullet,1}$ and $\Sigma^{\bullet,2}$ are two cofinal sequences for $\Sigma$, then one can find a cofinal sequence $\Sigma^{\widetilde{\bullet}}$ such that $\Sigma^{\bullet,i}\cup\Sigma^{\widetilde{\bullet}}$, $i=1,2$, are also cofinal sequences. As $\Sigma^{\widetilde{\bullet}}$ and $\Sigma^{\bullet,i}$ are subsequences of $\Sigma^{\bullet,i}\cup\Sigma^{\widetilde{\bullet}}$, the Fukaya category constructed using the cofinal sequence $\Sigma^{\widetilde{\bullet}}$ is quasi-equivalent to the one using $\Sigma^{\bullet,i}\cup\Sigma^{\widetilde{\bullet}}$ which is quasi-equivalent to the one using $\Sigma^{\bullet,i}$.

Then, the fact that the category does not depend (up to quasi-equivalence) on the choice of representatives of cylindrical at infinity Hamiltonian isotopy classes follows from the invariance result:

\begin{prop}
	Let $\Sigma_0$ be an exact Lagrangian cobordism and $(\varphi^s_h)_{s\in[0,1]}$ a cylindrical at infinity Hamiltonian isotopy such that $\Sigma_0$ and $\Sigma_1:=\varphi_h^1(\Sigma_0)$ are transverse. Then, for any $T$ exact Lagrangian cobordism transverse to $\Sigma_0$ and $\Sigma_1$, the complexes $\Cth_+(\Sigma_0,T)$ and $\Cth_+(\Sigma_1,T)$ are homotopy equivalent.
\end{prop}

\begin{proof}
	All the ingredients to prove this proposition already appeared in Section \ref{sec:acy}.
	Observe first that if $\Sigma_0$ is a cobordism from $\La_0^-$ to $\La_0^+$ then $\Sigma_1$ is a cobordism from $\La_1^-$ to $\La_1^+$ with $\La_1^\pm$ Legendrian isotopic to $\La_0^\pm$.
	The isotopy from $\Sigma_0$ to $\Sigma_1$ can be decomposed as a cylindrical at infinity isotopy of $\Sigma_0$ giving a cobordism $\widetilde{\Sigma}_0$ from $\La_1^-$ to $\La_1^+$, followed by a compactly supported Hamiltonian isotopy from $\widetilde{\Sigma}_0$ to $\Sigma_1$. The proof of the proposition now goes as follows.
	
	Start from $\Sigma_0$ and wrap its positive, resp. negative, end in the positive, resp negative, Reeb direction to obtain the cobordism $\Sigma_0^s$, $s\geq0$, having cylindrical ends over $\La_{0,-s}^-$ and $\La_{0,s}^+$ where $\La_{0,\pm s}^\pm=\La_0^\pm\pm s\partial_z$. Take $s$ big enough so that the Cthulhu complex $\Cth_+(\Sigma_0^s,T)$ has only intersection points generators.

	Denote also $\La_{1,a}^+:=\La_1^++a\partial_z$ and $\La_{1,-b}^-:=\La_1^--b\partial_z$ for $a,b\geq0$ so that $\La_{1,a}^+$ lies entirely above $\La_T^+$ and $\La_{1,-b}^-$ lies entirely below $\La_T^-$.
	Denote $C^+$ a concordance from $\La_{0,s}^+$ to $\La_{1,a}^+$ and $C^-$ a concordance from $\La_{1,-b}^-$ to $\La_{0,-s}^-$. We can assume that $C^+\cap(\R\times\La_T^+)=C^-\cap(\R\times\La_T^-)=\emptyset$.
	Concatenating the concordances $C^-$ and $C^+$ with $\Sigma_0^s$ gives $C^-\odot\Sigma_0^S\odot C^+$ which is an exact Lagrangian cobordism from $\La_{1,-b}^-$ to $\La_{1,a}^+$, satisfying $\Cth_+(C^-\odot\Sigma_0^s\odot C^+,T)=\Cth_+(\Sigma_0^s,T)$ by construction, where it is an equality of complexes. Finally, wrap the ends of $C^-\odot\Sigma_0^s\odot C^+$ in such a way that it \textquotedblleft translates back\textquotedblright\, $\La_{1,a}^+$ to $\La_1^+$ and $\La_{1,-b}^-$ to $\La_1^-$, to obtain a cobordism $\widetilde{\Sigma}_0$ from $\La_1^-$ to $\La_1^+$, Hamiltonian isotopic to $\Sigma_1$ by a compactly supported Hamiltonian isotopy. Invariance of the Cthulhu complex by wrapping the ends and compactly supported Hamiltonian isotopy ends the proof.
\end{proof}

\section{Higher order maps in the concatenation}\label{sec:higher_conc}

In case of a $(d+1)$-tuple of concatenated cobordisms $(V_0\odot W_0,\dots,V_d\odot W_d)$ one can also define higher order maps $\mfm_d^{V\odot W}$, which will recover the maps $\mfm_d^\Sigma$ defined in the previous section in the case of concatenation of a cobordism with a trivial cylinder.
We define recursively, for $d\geq1$, the maps
\begin{alignat*}{1}
 	&\bs{\Delta}_d^W:\Cth_+(V_{d-1}\odot W_{d-1},V_d\odot W_d)\otimes\dots\Cth_+(V_0\odot W_0,V_1\odot W_1)\to\Cth_+(V_0,V_d)\\
 	&\bs{b}_d^V:\Cth_+(V_{d-1}\odot W_{d-1},V_d\odot W_d)\otimes\dots\Cth_+(V_0\odot W_0,V_1\odot W_1)\to\Cth_+(W_0,W_d)
\end{alignat*}
as follows. First, $\bs{b}_1^V$ and $\bs{\Delta}_1^W$ are the transfer maps from Section \ref{sec:conc}, and then for $d\geq2$ one sets:
\begin{alignat*}{1}
	&\bs{\Delta}_d^W=\sum_{s=2}^d\sum_{\substack{1\leq i_1,\dots,i_s\\i_1+\dots+i_s=d}}\Delta_s^W\big(\bs{b}_{i_s}^V\otimes\dots\otimes\bs{b}_{i_1}^V\big)\\
	&\bs{b}_d^V=\sum_{s=1}^d\sum_{\substack{1\leq i_1,\dots,i_s\\i_1+\dots+i_s=d}}b_s^V\big(\bs{\Delta}_{i_s}^W\otimes\dots\otimes\bs{\Delta}_{i_1}^W\big)
\end{alignat*}
Using the maps $\Delta_s^W,b_s^V$ from Section \ref{A-inf maps}. Observe that the maps $\bs{b}_2^V$ and $\bs{\Delta}_2^W$ already appeared in Section \ref{sec:functoriality}. 
Given this, we define
\begin{alignat*}{1}
	\mfm_d^{V\odot W}:\Cth_+(V_{d-1}\odot W_{d-1},V_d\odot W_d)\otimes\dots\Cth_+(V_0\odot W_0,V_1\odot W_1)\to\Cth_+(V_0\odot W_0,V_d\odot W_d)
\end{alignat*}
by
\begin{alignat*}{1}
	\mfm_d^{V\odot W}=\sum_{s=1}^d\sum_{\substack{1\leq i_1,\dots,i_s\\i_1+\dots+i_s=d}}\mfm_s^{W,+0}\big(\bs{b}_{i_s}^V\otimes\dots\otimes\bs{b}_{i_1}^V\big)+\mfm_s^{V,0-}\big(\bs{\Delta}_{i_s}^W\otimes\dots\otimes\bs{\Delta}_{i_1}^W\big)
\end{alignat*}
We will first prove that the maps $\bs{\Delta}_j^W$ and $\bs{b}_j^V$ satisfy the $A_\infty$-functor equations, and then we will prove that the maps $\mfm_j^{V\odot W}$ satisfy the $A_\infty$ equations. We start by proving the following:
\begin{lem}\label{lem:deltabanane}
	For all $d\geq1$,
	\begin{alignat*}{1}
	\sum_{j=1}^d\sum_{i_1+\dots+i_j=d}\bs{b}_j^V\big(\bs{\Delta}_{i_j}^W\otimes\dots\otimes\bs{\Delta}_{i_1}^W\big)+\bs{\Delta}_j^W\big(\bs{b}_{i_j}^V\otimes\dots\otimes\bs{b}_{i_1}^V\big)=0
	\end{alignat*}		
\end{lem}
\begin{proof}
	This holds by definition of the maps. Observe that we have already made use of the case $d=1$ in Section \ref{sec:LCconc}. For $d\geq2$, one has
\begin{alignat*}{1}
	&\sum_{j=1}^d\sum_{i_1+\dots+i_j=d}\bs{b}_j^V\big(\bs{\Delta}_{i_j}^W\otimes\dots\otimes\bs{\Delta}_{i_1}^W\big)+\bs{\Delta}_j^W\big(\bs{b}_{i_j}^V\otimes\dots\otimes\bs{b}_{i_1}^V\big)\\
	&=\bs{b}_1^V\circ\bs{\Delta}_d^W+\bs{\Delta}_1^W\circ\bs{b}_d^V+\sum_{j=2}^d\sum_{i_1+\dots+i_j=d}\bs{b}_j^V\big(\bs{\Delta}_{i_j}^W\otimes\dots\otimes\bs{\Delta}_{i_1}^W\big)+\bs{\Delta}_j^W\big(\bs{b}_{i_j}^V\otimes\dots\otimes\bs{b}_{i_1}^V\big)
\end{alignat*}
Observe that $\bs{\Delta}_{i_j}^W\otimes\dots\otimes\bs{\Delta}_{i_1}^W$ takes values in $\Cth_+(V_{d-i_j},V_d)\otimes\dots\otimes\Cth_+(V_0,V_{i_1})$ and $\bs{b}_{i_j}^V\otimes\dots\otimes\bs{b}_{i_1}^V$ takes values in $\Cth_+(W_{d-i_j},W_d)\otimes\dots\otimes\Cth_+(W_0,W_{i_1})$, hence the maps $\bs{b}_j^V$ and $\bs{\Delta}_j^W$ are the one corresponding to one cobordism only and not the concatenation, as defined in Section \ref{A-inf maps}. So the sum above equals
\begin{alignat*}{1}
	&\bs{\Delta}_d^W+b_1^V\circ\bs{\Delta}_d^W+\bs{b}_d^V+\sum_{j=2}^d\sum_{i_1+\dots+i_j=d}b_j^V\big(\bs{\Delta}_{i_j}^W\otimes\dots\otimes\bs{\Delta}_{i_1}^W\big)+\Delta_j^W\big(\bs{b}_{i_j}^V\otimes\dots\otimes\bs{b}_{i_1}^V\big)\\
	&=\bs{\Delta}_d^W+\bs{b}_d^V+\sum_{j=1}^d\sum_{i_1+\dots+i_j=d}b_j^V\big(\bs{\Delta}_{i_j}^W\otimes\dots\otimes\bs{\Delta}_{i_1}^W\big)+\Delta_j^W\big(\bs{b}_{i_j}^V\otimes\dots\otimes\bs{b}_{i_1}^V\big)
\end{alignat*}
	where note that in the sum on the right we have $\Delta_1^W\circ\bs{b}_d^V=0$. Then, this gives $0$ by definition of $\bs{b}_d^V$ and $\bs{\Delta}_d^W$.
\end{proof}

Now we can prove:
\begin{lem}\label{lem:functoriality} For all $d\geq1$
	\begin{alignat*}{1}
		&\sum_{s=1}^d\sum_{i_1+\dots+i_s=d}\mfm_s^V\big(\bs{\Delta}_{i_s}^W\otimes\dots\otimes\bs{\Delta}_{i_1}^W\big)+\sum_{j=1}^d\sum_{n=0}^{d-j}\bs{\Delta}_{d-j+1}^W\big(\id\otimes\dots\otimes\id\otimes\mfm_j^{V\odot W}\otimes\underbrace{\id\otimes\dots\otimes\id}_{n}\big)=0
	\end{alignat*}
and
	\begin{alignat*}{1}
		&\sum_{s=1}^d\sum_{i_1+\dots+i_s=d}\mfm_s^W\big(\bs{b}_{i_s}^V\otimes\dots\otimes\bs{b}_{i_1}^V\big)+\sum_{j=1}^d\sum_{n=0}^{d-j}\bs{b}_{d-j+1}^V\big(\id\otimes\dots\otimes\id\otimes\mfm_j^{V\odot W}\otimes\underbrace{\id\otimes\dots\otimes\id}_{n}\big)=0
	\end{alignat*}
\end{lem}
\begin{proof}
We prove it by recursion on $d$. For $d=1$, the relations above means that $\bs{b}_1^V$ and $\bs{\Delta}_1^W$ are chain maps, which is the content of Proposition \ref{prop:Phi} and Proposition \ref{prop:Delta}. Note that we have also proved the case $d=2$ in Section \ref{sec:functoriality}.
For $d\geq2$, we have
\begin{alignat*}{1}
	&\sum_{s=1}^d\sum_{i_1+\dots+i_s=d}\mfm_s^V\big(\bs{\Delta}_{i_s}^W\otimes\dots\otimes\bs{\Delta}_{i_1}^W\big)+\sum_{j=1}^d\sum_{n=0}^{d-j}\bs{\Delta}_{d-j+1}^W\big(\id\otimes\dots\otimes\id\otimes\mfm_j^{V\odot W}\otimes\underbrace{\id\otimes\dots\otimes\id}_{n}\big)\\
	&=\sum_{s=1}^d\sum_{i_1+\dots+i_s=d}\mfm_s^V\big(\bs{\Delta}_{i_s}^W\otimes\dots\otimes\bs{\Delta}_{i_1}^W\big)+\bs{\Delta}_1^W\circ\mfm_d^{V\odot W}\\
	&\hspace{5mm}+\sum_{j=1}^{d-1}\sum_{n=0}^{d-j}\bs{\Delta}_{d-j+1}^W\big(\id\otimes\dots\otimes\id\otimes\mfm_j^{V\odot W}\otimes\underbrace{\id\otimes\dots\otimes\id}_{n}\big)
\end{alignat*}
but by definition of $\mfm_d^{V\odot W}$:
\begin{alignat*}{1}
	&\bs{\Delta}_1^W\circ\mfm_d^{V\odot W}\\
	&=\bs{\Delta}_1^W\circ\Big(\sum_{s=1}^d\sum_{i_1+\dots+i_s=d}\mfm_s^{W,+0}\big(\bs{b}_{i_s}^V\otimes\dots\otimes\bs{b}_{i_1}^V\big)+\mfm_s^{V,0-}\big(\bs{\Delta}_{i_s}^W\otimes\dots\otimes\bs{\Delta}_{i_1}^W\big)\Big)\\
	&=\sum_{s=1}^d\sum_{i_1+\dots+i_s=d}\Delta_1^W\circ\mfm_s^{W,+0}\big(\bs{b}_{i_s}^V\otimes\dots\otimes\bs{b}_{i_1}^V\big)+\mfm_s^{V,0-}\big(\bs{\Delta}_{i_s}^W\otimes\dots\otimes\bs{\Delta}_{i_1}^W\big)
\end{alignat*}
So we get
\begin{alignat*}{1}
&\sum_{s=1}^d\sum_{i_1+\dots+i_s=d}\mfm_s^V\big(\bs{\Delta}_{i_s}^W\otimes\dots\otimes\bs{\Delta}_{i_1}^W\big)+\sum_{j=1}^d\sum_{n=0}^{d-j}\bs{\Delta}_{d-j+1}^W\big(\id\otimes\dots\otimes\id\otimes\mfm_j^{V\odot W}\otimes\id\otimes\dots\otimes\id\big)\\
&=\sum_{s=1}^d\sum_{i_1+\dots+i_s=d}\mfm_s^{V,+}\big(\bs{\Delta}_{i_s}^W\otimes\dots\otimes\bs{\Delta}_{i_1}^W\big)+\sum_{s=1}^d\sum_{i_1+\dots+i_s=d}\Delta_1^W\circ\mfm_s^{W,+0}\big(\bs{b}_{i_s}^V\otimes\dots\otimes\bs{b}_{i_1}^V\big)\\
&\hspace{5mm}+\sum_{j=1}^{d-1}\sum_{n=0}^{d-j}\bs{\Delta}_{d-j+1}^W\big(\id\otimes\dots\otimes\id\otimes\mfm_j^{V\odot W}\otimes\id\otimes\dots\otimes\id\big)\\
\end{alignat*}
Then one has
\begin{alignat*}{1}
&\sum_{j=1}^{d-1}\sum_{n=0}^{d-j}\bs{\Delta}_{d-j+1}^W\big(\id\otimes\dots\otimes\id\otimes\mfm_j^{V\odot W}\otimes\id\otimes\dots\otimes\id\big)\\
&=\sum_{j=1}^{d-1}\sum_{n=0}^{d-j}\sum_{s=2}^{d-j+1}\sum_{i_1+\dots+i_s=d-j+1}\Delta_{s}^W\big(\bs{b}_{i_s}^V\otimes\dots\otimes\bs{b}_{i_1}\big)\big(\id\otimes\dots\otimes\id\otimes\mfm_j^{V\odot W}\otimes\id\otimes\dots\otimes\id\big)\\
&=\sum_{j=1}^{d-1}\sum_{n=0}^{d-j}\sum_{s=2}^{d-j+1}\sum_{i_1+\dots+i_s=d-j+1}\sum_{l=1}^s\Delta_{s}^W\big(\bs{b}_{i_s}^V\otimes\dots\otimes\bs{b}_{i_l}^V(\id\otimes\dots\otimes\mfm_j^{V\odot W}\otimes\dots\id)\otimes\dots\otimes\bs{b}_{i_1}^V\big)\\
&=\sum_{s=2}^{d}\sum_{i_1+\dots+i_s=d}\sum_{l=1}^s\sum_{j=1}^{i_l}\sum_{n=0}^{i_l-j}\Delta_{s}^W\big(\bs{b}_{i_s}^V\otimes\dots\otimes\bs{b}_{i_l}^V(\id\otimes\dots\otimes\mfm_j^{V\odot W}\otimes\dots\id)\otimes\dots\otimes\bs{b}_{i_1}^V\big)\\
&=\sum_{s=2}^{d}\sum_{i_1+\dots+i_s=d}\sum_{l=1}^s\Delta_{s}^W\big(\bs{b}_{i_s}^V\otimes\dots\otimes\sum_{j=1}^{i_l}\sum_{n=0}^{i_l-j}\bs{b}_{i_l}^V(\id\otimes\dots\otimes\mfm_j^{V\odot W}\otimes\dots\id)\otimes\dots\otimes\bs{b}_{i_1}^V\big)
\end{alignat*}
Observe that $i_l\leq d-1$ so by recursion we have that
\begin{alignat*}{1}
	&\sum_{j=1}^{d-1}\sum_{n=0}^{d-j}\bs{\Delta}_{d-j+1}^W\big(\id\otimes\dots\otimes\id\otimes\mfm_j^{V\odot W}\otimes\id\otimes\dots\otimes\id\big)\\
	&=\sum_{s=2}^{d}\sum_{i_1+\dots+i_s=d}\sum_{l=1}^s\Delta_{s}^W\big(\bs{b}_{i_s}^V\otimes\dots\otimes\sum_{u=1}^{i_l}\sum_{t_1+\dots+t_u=i_l}\mfm_u^W(\bs{b}_{t_u}^V\otimes\dots\otimes\bs{b}_{t_1}^V)\otimes\dots\otimes\bs{b}_{i_1}\big)
\end{alignat*}
So we get
\begin{alignat*}{1}
&\sum_{s=1}^d\sum_{i_1+\dots+i_s=d}\mfm_s^V\big(\bs{\Delta}_{i_s}^W\otimes\dots\otimes\bs{\Delta}_{i_1}^W\big)+\sum_{j=1}^d\sum_{n=0}^{d-j}\bs{\Delta}_{d-j+1}^W\big(\id\otimes\dots\otimes\id\otimes\mfm_j^{V\odot W}\otimes\id\otimes\dots\otimes\id\big)\\
&=\sum_{s=1}^d\sum_{i_1+\dots+i_s=d}\mfm_s^{V,+}\big(\bs{\Delta}_{i_s}^W\otimes\dots\otimes\bs{\Delta}_{i_1}^W\big)+\sum_{s=1}^d\sum_{i_1+\dots+i_s=d}\Delta_1^W\circ\mfm_s^{W,+0}\big(\bs{b}_{i_s}^V\otimes\dots\otimes\bs{b}_{i_1}^V\big)\\
&\hspace{5mm}+\sum_{s=2}^{d}\sum_{i_1+\dots+i_s=d}\sum_{l=1}^s\Delta_{s}^W\big(\bs{b}_{i_s}^V\otimes\dots\otimes\sum_{u=1}^{i_l}\sum_{t_1+\dots+t_u=i_l}\mfm_u^W(\bs{b}_{t_u}^V\otimes\dots\otimes\bs{b}_{t_1}^V)\otimes\dots\otimes\bs{b}_{i_1}\big)\\
&=\sum_{s=1}^d\sum_{i_1+\dots+i_s=d}\mfm_s^{V,+}\big(\bs{\Delta}_{i_s}^W\otimes\dots\otimes\bs{\Delta}_{i_1}^W\big)\\
&\hspace{5mm}+\sum_{s=1}^{d}\sum_{i_1+\dots+i_s=d}\sum_{l=1}^s\Delta_{s}^W\big(\bs{b}_{i_s}^V\otimes\dots\otimes\sum_{u=1}^{i_l}\sum_{t_1+\dots+t_u=i_l}\mfm_u^W(\bs{b}_{t_u}^V\otimes\dots\otimes\bs{b}_{t_1}^V)\otimes\dots\otimes\bs{b}_{i_1}\big)\\
\end{alignat*}
where we have used the fact that $\Delta_1^W\circ\mfm_s^{W,-}=0$.
By definition of $\mfm_s^{V,+}$, it gives
\begin{alignat*}{1}
&=\sum_{s=1}^d\sum_{i_1+\dots+i_s=d}\sum_{u=1}^s\sum_{t_1+\dots+t_u=s}\Delta_u^\La\big(\bs{b}_{t_u}^V\otimes\dots\otimes\bs{b}_{t_1}^V\big)\big(\bs{\Delta}_{i_s}^W\otimes\dots\otimes\bs{\Delta}_{i_1}^W\big)\\
&\hspace{5mm}+\sum_{s=1}^{d}\sum_{i_1+\dots+i_s=d}\sum_{l=1}^s\Delta_{s}^W\big(\bs{b}_{i_s}^V\otimes\dots\otimes\sum_{u=1}^{i_l}\sum_{t_1+\dots+t_u=i_l}\mfm_u^W(\bs{b}_{t_u}^V\otimes\dots\otimes\bs{b}_{t_1}^V)\otimes\dots\otimes\bs{b}_{i_1}\big)
\end{alignat*}
and finally, using Lemma \ref{relDeltaSigma} which states in this case that
\begin{alignat*}{1}
\sum\limits_{j=1}^d\sum\limits_{n=0}^{d-j}\Delta^W_{d-j+1}\big(\id^{\otimes d-j-n}\otimes\mfm_j^W\otimes\id^{\otimes n}\big)+\sum\limits_{j=1}^d\sum\limits_{i_1+\dots+i_j=d}\Delta^\La_{j}\big(\bs{\Delta}^W_{i_j}\otimes\dots\otimes\bs{\Delta}_{i_1}^W\big)=0
\end{alignat*}
and Lemma \ref{lem:deltabanane}, one obtains that the sum vanishes, and the maps $\bs{\Delta}_j^W$ satisfy the $A_\infty$-functor equations.

One proves analogously that the maps $\bs{b}_j^V$ satisfy the $A_\infty$-functor equations.
\end{proof}

\begin{prop}
	 The maps $\mfm_d^{V\odot W}$ satisfy the $A_\infty$-equations.
\end{prop}
\begin{proof} For all $d\geq1$, one has
\begin{alignat*}{1}
	&\sum_{j=1}^d\sum_{n=0}^{d-j}\mfm_{d-j+1}^{V\odot W}\big(\id\otimes\dots\otimes\id\otimes\mfm_j^{V\odot W}\otimes\overbrace{\id\otimes\dots\otimes\id}^{n}\big)\\
	&=\sum_{j,n}\Big(\sum_{k=1}^{d-j+1}\sum_{i_1+\dots+i_k=d-j+1}\mfm_k^{W,+0}\big(\bs{b}_{i_k}^V\otimes\dots\otimes\bs{b}_{i_1}^V\big)+\mfm_k^{V,0-}\big(\bs{\Delta}_{i_k}^W\otimes\dots\otimes\bs{\Delta}_{i_1}^W\big)\Big)\big(\id\otimes\dots\otimes\mfm_j^{V\odot W}\otimes\dots\otimes\id\big)\\
	&=\sum_{k=1}^d\sum_{i_1+...+i_k=d}\sum_{s=1}^k\sum_{j=1}^{i_s}\sum_{n=0}^{i_s-j}\mfm_k^{W,+0}\big(\bs{b}_{i_k}^V\otimes\dots\otimes\bs{b}_{i_s-j+1}^V(\id\otimes\dots\otimes\mfm_j^{V\odot W}\otimes\dots\otimes\id)\otimes\dots\otimes\bs{b}_{i_1}^V\big)\\
	&\hspace{4cm}+\mfm_k^{V,0-}\big(\bs{\Delta}_{i_k}^W\otimes\dots\otimes\bs{\Delta}_{i_s-j+1}^V(\id\otimes\dots\otimes\mfm_j^{V\odot W}\otimes\dots\otimes\id)\otimes\dots\otimes\bs{\Delta}_{i_1}^W\big)\\
	&=\sum_{k=1}^d\sum_{i_1+...+i_k=d}\sum_{s=1}^k\mfm_k^{W,+0}\big(\bs{b}_{i_k}^V\otimes\dots\otimes\sum_{j=1}^{i_s}\sum_{n=0}^{i_s-j}\bs{b}_{i_s-j+1}^V(\id\otimes\dots\otimes\mfm_j^{V\odot W}\otimes\dots\otimes\id)\otimes\dots\otimes\bs{b}_{i_1}^V\big)\\
	&+\sum_{k=1}^d\sum_{i_1+...+i_k=d}\sum_{s=1}^k\mfm_k^{V,0-}\big(\bs{\Delta}_{i_k}^W\otimes\dots\otimes\sum_{j=1}^{i_s}\sum_{n=0}^{i_s-j}\bs{\Delta}_{i_s-j+1}^V(\id\otimes\dots\otimes\mfm_j^{V\odot W}\otimes\dots\otimes\id)\otimes\dots\otimes\bs{\Delta}_{i_1}^W\big)
\end{alignat*}
Using Lemma \ref{lem:functoriality}, the sum above is equal to:
\begin{alignat*}{1}
	&\sum_{k=1}^d\sum_{i_1+...+i_k=d}\sum_{s=1}^k\mfm_k^{W,+0}\big(\bs{b}_{i_k}^V\otimes\dots\otimes\sum_{u=1}^{i_s}\sum_{t_1+\dots+t_u=i_s}\mfm_u^W(\bs{b}_{t_u}^V\otimes\dots\otimes\bs{b}_{t_1}^V)\otimes\dots\otimes\bs{b}_{i_1}^V\big)\\
	&+\sum_{k=1}^d\sum_{i_1+...+i_k=d}\sum_{s=1}^k\mfm_k^{V,0-}\big(\bs{\Delta}_{i_k}^W\otimes\dots\otimes\sum_{u=1}^{i_s}\sum_{t_1+\dots+t_u=i_s}\mfm_u^V(\bs{\Delta}_{t_u}^W\otimes\dots\otimes\bs{\Delta}_{t_1}^W)\otimes\dots\otimes\bs{\Delta}_{i_1}^W\big)\\
	&=\sum_{j=1}^d\sum_{r_1+\dots+r_j=d}\sum_{u=1}^j\sum_{s=1}^{j-u+1}\mfm_{j-u+1}^{W,+0}\big(\underbrace{\id\otimes\dots\otimes\id}_{k-s}\otimes\mfm_u^W\otimes\underbrace{\id\otimes\dots\otimes\id}_{s-1}\big)\big(\bs{b}_{r_j}^V\otimes\dots\otimes\bs{b}_{r_1}^V\big)\\
	&+\sum_{j=1}^d\sum_{r_1+\dots+r_j=d}\sum_{u=1}^j\sum_{s=1}^{j-u+1}\mfm_{j-u+1}^{V,0-}\big(\underbrace{\id\otimes\dots\otimes\id}_{k-s}\otimes\mfm_u^V\otimes\underbrace{\id\otimes\dots\otimes\id}_{s-1}\big)\big(\bs{\Delta}_{r_j}^V\otimes\dots\otimes\bs{\Delta}_{r_1}^V\big)\\
	&=0	
\end{alignat*}
as the maps $\mfm_j^W$ and $\mfm_j^V$ satisfy the $A_\infty$-equations.
\end{proof}

\normalsize
\bibliographystyle{alpha}
\bibliography{references2.bib}
\end{document}